\documentclass[reqno]{amsart}
\usepackage{amsmath}
\usepackage{paralist}
\usepackage{amsfonts}
\usepackage{amssymb}
\usepackage{amsthm,amsmath}
\usepackage{amscd}
\usepackage{float}
\usepackage{tikz}
\usepackage{graphicx}
\usepackage{hyperref}
\usepackage{mathabx}
\usepackage[shortlabels]{enumitem}

\newtheorem{theorem}{Theorem}[section]
\newtheorem{corollary}[theorem]{Corollary}

\newtheorem{lemma}[theorem]{Lemma}
\newtheorem{proposition}[theorem]{Proposition}

\theoremstyle{definition}
\newtheorem{definition}[theorem]{Definition}
\newtheorem{remark}[theorem]{Remark}

\newcommand{\R}{\mathbb{R}}

\newcommand{\N}{\mathbb{N}}

\newcommand{\C}{\mathbb{C}}

\title[Maxwell equations on domains]{Strichartz estimates for Maxwell equations on domains with perfectly conducting boundary conditions}

\author{Nicolas Burq}
\address{Laboratoire de math'emathiques d'Orsay, CNRS, Universit'e Paris-Saclay,
Bat. 307, 91405 Orsay Cedex, France, and Institut Universitaire de France (IUF)}
\email{nicolas.burq@universite-paris-saclay.fr}
\author[Robert Schippa]{Robert Schippa*}
\address{Karlsruhe Institute of Technology, Englerstrasse 2, 76131 Karlsruhe, Germany
}
\email{robert.schippa@kit.edu}
\thanks{*Corresponding author}

\makeatletter
\@namedef{subjclassname@2020}{%
  \textup{2020} Mathematics Subject Classification}
\makeatother

\keywords{Maxwell equations, quasilinear wave equations, Strichartz estimates}

\subjclass[2020]{35B45, 35L03, 35Q61}

\numberwithin{equation}{section}

\begin{document}

\maketitle

\begin{abstract}
We consider Maxwell equations on a smooth domain with perfectly conducting boundary conditions in isotropic media in two and three dimensions. In the charge-free case we recover Strichartz estimates due to Blair--Smith--Sogge for wave equations on domains up to endpoints. For the proof we suitably extend Maxwell equations over the boundary, which introduces coefficients on the full space with codimension-$1$ Lipschitz singularity. This system can be diagonalized to half-wave equations amenable to the results of Blair--Smith--Sogge. In two dimensions, we improve the local well-posedness of the Maxwell system with Kerr nonlinearity via Strichartz estimates.
\end{abstract}
\section{Introduction}
We discuss dispersive properties for Maxwell equations on bounded domains $\Omega \subseteq \R^3$ with compact boundary $\partial \Omega \in C^\infty$\footnote{Certainly, the present arguments extend to $\partial \Omega \in C^N$ for $N$ large enough corresponding to a generalization of the results due to Blair--Smith--Sogge \cite{BlairSmithSogge2009} to the $C^N$-category. We are not attempting to minimize the required regularity.}.
The system relates \emph{electric} and \emph{displacement field} $(\mathcal{E},\mathcal{D}): \R \times \Omega \to \R^3 \times \R^3$ and \emph{magnetic} and \emph{magnetizing field} $(\mathcal{B},\mathcal{H}): \R \times \Omega \to \R^3 \times \R^3$. The system of equations reads
\begin{equation}
\label{eq:MaxwellDomainsIntroduction}
\left\{ \begin{array}{cllcl}
\partial_t \mathcal{D} &= \nabla \times \mathcal{H} - \mathcal{J}_e, &\qquad \nabla \cdot \mathcal{D} &=& \rho_e, \qquad (t,x') \in \R \times \Omega, \\
\partial_t \mathcal{B} &= -\nabla \times \mathcal{E}, &\qquad \nabla \cdot \mathcal{B}&=& 0, 
\end{array} \right.
\end{equation}
with initial conditions $(\mathcal{E},\mathcal{H})(0) = (\mathcal{E}_0, \mathcal{H}_0)$.
$\mathcal{J}_e: \R \times \Omega \to \R^3$ denotes the \emph{electric current}, which is regarded as source term. Throughout we denote space-time variables by $x = (t,x') \in \R \times \Omega$.
In the first part of the paper, we supplement the Maxwell system with pointwise time-independent material laws for isotropic media
\begin{equation}
\label{eq:MaterialLaws}
\mathcal{D}(t,x') = \varepsilon(x') \mathcal{E}(t,x'), \qquad \mathcal{B}(t,x') = \mu(x') \mathcal{H}(t,x')
\end{equation}
with $\varepsilon$, $\mu \in C^\infty(\Omega;\R_{>0})$ denoting \emph{permittivity} and \emph{permeability}, which satisfy the uniform ellipticity conditions
\begin{equation}
\label{eq:Ellipticity}
\exists \lambda, \Lambda > 0: \; \forall x' \in \Omega: \; \lambda \leq \varepsilon(x'), \mu(x') \leq \Lambda.
\end{equation}
For technical reasons, we further suppose that for some $N \geq 2$ we have\footnote{This constant is the regularity required for the metric such that the results of Blair--Smith--Sogge hold true. It is conceivable that $N=2$ suffices, but this is currently unclear.}
\begin{equation}
\label{eq:ContinuityBoundary}
\varepsilon, \, \partial \varepsilon, \, \ldots \partial^N \varepsilon \in C(\overline{\Omega}) \cap L^\infty(\Omega), \qquad \mu, \, \partial \mu, \, \ldots, \partial^N \mu \in C(\overline{\Omega}) \cap L^\infty(\Omega).
\end{equation}

Maxwell equations in media describe the electromagnetism of matter and are of great physical importance. We refer to the physics' literature for a detailed explanation (cf. \cite{FeynmanLeightonSands1964,LandauLifschitz1984}). We also refer occasionally to the lecture notes surveying basic results by Schnaubelt \cite{Schnaubelt2022}.

\medskip

Let $\nu \in C^\infty(\partial \Omega, \R^3)$ denote the outer unit normal. Below $[\cdot]_{x' \in \partial \Omega}$ denotes the trace of a function at the boundary. Here we consider the \emph{perfectly conducting boundary conditions}
\begin{equation}
\label{eq:PerfectConductor}
[ \mathcal{E} \times \nu ]_{x' \in \partial \Omega} = 0, \qquad [\mathcal{B} \cdot \nu ]_{x' \in \partial \Omega} = 0.
\end{equation}
The boundary conditions of the perfect electric conductor are among the physically most relevant ones (cf. \cite{SpitzPhdThesis2017,Schnaubelt2022}).
We define \emph{surface charges} $\rho_\Sigma$ and \emph{surface currents} $J_{\Sigma}$ by (cf. \cite[Eq.~(2.3)]{Schnaubelt2022}):
\begin{equation}
\label{eq:BoundaryConditionsII}
[\mathcal{D} \cdot \nu]_{x' \in \partial \Omega} = \rho_{\Sigma}, \qquad [ \mathcal{H} \times \nu]_{x' \in \partial \Omega} = J_{\Sigma}.
\end{equation}
Furthermore, we require the normal component of $\mathcal{J}_e$ to vanish at the boundary, which is physically sensible:
\begin{equation}
\label{eq:NormalComponentJ}
[\mathcal{J}_e \cdot \nu]_{x' \in \partial \Omega} = 0.
\end{equation}

\bigskip

The Maxwell equations satisfy finite speed of propagation (see \cite[Chapter~6]{SpitzPhdThesis2017}). Hence, in the interior of the domain we can use previously established results on the whole space for local-in-time results (see previous works by Dumas--Sueur \cite{DumasSueur2012} and the second author \cite{SchippaSchnaubelt2022,Schippa2021Maxwell3d}). Thus, it suffices to work close to the boundary, at which we resolve the Maxwell system in geodesic normal coordinates; see Section \ref{section:MaxwellManifolds}. At the boundary, we write the equation in geodesic normal coordinates to localize to the half-space $\R^3_{>0} = \{ x' \in \R^3 : x'_3 > 0 \}$. The cometric is given by
\begin{equation*}
g^{-1} =
\begin{pmatrix}
g^{11} & g^{12} & 0 \\
g^{21} & g^{22} & 0 \\
0 & 0 & 1
\end{pmatrix}
.
\end{equation*}
As short-hand notation, we write $\sqrt{g} := \sqrt{ \det g}$. This effectively gives rise to anisotropic permittivity $\sqrt{g} g^{-1} \varepsilon$ and permeability $\sqrt{g} g^{-1} \mu$: 
\begin{equation}
\label{eq:MaxwellGeodesicIntroduction}
\left\{ \begin{array}{clcll}
\partial_t (\sqrt{g} g^{-1} \varepsilon \mathcal{E} ) &= \nabla \times \mathcal{H}, &\qquad (\mathcal{E} \times e_3) \vert_{x'_3 = 0} &=& 0, \quad (t,x') \in \R \times \R^3_{>0}, \\
\partial_t(\sqrt{g} g^{-1} \mu \mathcal{H}) &= -\nabla \times \mathcal{E},  &\qquad (\mathcal{H} \cdot e_3) \vert_{x'_3 = 0} &=& 0
\end{array} \right.
\end{equation}
with the divergence conditions now reading
\begin{equation*}
\nabla \cdot ( \sqrt{g} g^{-1} \varepsilon \mathcal{E} ) = \sqrt{g} \rho_e, \qquad \nabla \cdot ( \sqrt{g} g^{-1} \varepsilon \mathcal{H} ) = 0.
\end{equation*}
It is important to note that the boundary conditions \eqref{eq:PerfectConductor} are respected by $g^{-1}$.

\bigskip

Note that taking time derivatives in \eqref{eq:PerfectConductor} and plugging in \eqref{eq:MaxwellDomainsIntroduction} yields compatibility conditions. In order to maintain a less technical introduction, we postpone the discussion of compatibility conditions to Section \ref{section:MaxwellManifolds} after we have localised Maxwell's equations to \eqref{eq:MaxwellGeodesicIntroduction}. The second compatibility condition simplifies under the assumption
\begin{equation}
\label{eq:VanishingBoundaryDerivativePermeability}
\partial \mu \vert_{x' \in \partial \Omega} = 0.
\end{equation}

\medskip

Spitz \cite{Spitz2019,Spitz2022} showed existence and local well-posedness in $H^3(\Omega)$ (also in the quasilinear case) provided that the compatibility conditions up to second order are satisfied. These are precisely the conditions, which are meaningful in the sense of traces. First, we tend to homogeneous solutions with $\mathcal{J}_e = 0$. Accordingly, we let 
\begin{equation*}
\begin{split}
\mathcal{H}^3(\Omega) = \{ (\mathcal{E}_0,\mathcal{H}_0) \in H^3(\Omega)^2 \; : \; &(\mathcal{E}_0,\mathcal{H}_0) \text{ satisfies homogeneous } \\
&\quad \text{ compatibility conditions up to second order } \}.
\end{split}
\end{equation*}
With the solutions existing, we can show Strichartz estimates for homogeneous solutions
\begin{equation}
 \label{eq:PreStrichartz}
\| (\mathcal{E},\mathcal{H}) \|_{L_T^p L^q} \lesssim_T \| (\mathcal{E}_0,\mathcal{H}_0) \|_{\mathcal{H}^{\gamma+\delta}(\Omega)} + \| \rho_e(0) \|_{H^{\gamma-1 + \frac{1}{p}+\delta}(\Omega)}
\end{equation}
for certain $2 \leq p,q \leq \infty$, $q<\infty$ with $\gamma$ determined by scaling, and $\delta > 0$\footnote{Note that $\delta$ is chosen small enough such that boundary conditions are not relevant for the Sobolev space $H^{\gamma - 1 + \frac{1}{p} + \delta}$.} such that
\begin{equation}
\label{eq:DerivativeScaling}
\gamma = 3 \big( \frac{1}{2} - \frac{1}{q} \big) - \frac{1}{p}, \quad \delta < \frac{3}{q}.
\end{equation}
All Strichartz estimates established in this paper are local in time. For $0<T<\infty$, we write $L_T^p L^q := L_T^p L_{x'}^q(\Omega) := L_t^p([0,T],L^q(\Omega))$ with
\begin{equation*}
\| A \|_{L_T^p L_{x'}^q} := \big( \int_0^T \big( \int_{\Omega} |A(t,x')|^q dx' \big)^{\frac{p}{q}} \big)^{\frac{1}{p}}
\end{equation*}
for $1 \leq p,q < \infty$ with the usual modification for $p=\infty$ or $q=\infty$.

\medskip

\eqref{eq:PreStrichartz} is proved in two steps: First, we show
\begin{equation*}
 \| (\mathcal{E},\mathcal{H}) \|_{L_T^p L^q(\Omega)} \lesssim \| (\mathcal{E},\mathcal{H}) \|_{L_T^\infty \mathcal{H}^{\gamma+\delta}(\Omega)} + \| \rho_e(0) \|_{H^{\gamma-1+\frac{1}{p}+\delta}(\Omega)}.
\end{equation*}
Then it suffices to prove energy estimates for homogeneous solutions for $0 \leq s \leq 3$:
\begin{equation*}
 \| (\mathcal{E},\mathcal{H}) \|_{L_T^\infty \mathcal{H}^s} \lesssim_T \| (\mathcal{E}_0, \mathcal{H}_0) \|_{\mathcal{H}^s}.
\end{equation*}
 
Linearity and boundedness allows us to extend the linear solution mapping from the subspace $\mathcal{H}^3(\Omega)$ of $H^\gamma(\Omega)$ to its closure in the $H^\gamma$-norm:
\begin{equation}
 \label{eq:Density}
\mathcal{H}^\gamma(\Omega) = \overline{\mathcal{H}^3(\Omega)}^{\| \cdot \|_{H^\gamma(\Omega)}}.
\end{equation}
 We denote Sobolev spaces (of real-valued functions) on $\Omega$ with Dirichlet boundary condition with $H_D^\gamma(\Omega)$; the Sobolev spaces with Neumann boundary conditions are denoted by $H_N^\gamma(\Omega)$. 

 Since we shall estimate the regularity of $(\mathcal{E}_0,\mathcal{H}_0)$ only in $H^\gamma$ for $\gamma < \frac{3}{2}$, the compatibility conditions involving derivatives are not relevant. This means we actually only require the Dirichlet conditions for $\mathcal{H}^\gamma$, $\gamma < \frac{3}{2}$. We shall then recover inhomogeneous estimates by Duhamel's formula. Roughly speaking, $\mathcal{H}^\gamma(\Omega)$ is the Sobolev space with relevant compatibility conditions; see Proposition \ref{prop:Characterization}.
For $\gamma < \frac{1}{2}$, this means there are no boundary conditions. For $\frac{1}{2} < \gamma < \frac{3}{2}$, we only have Dirichlet conditions. For $\frac{3}{2} < \gamma < \frac{5}{2}$, we have to take into account first order compatibility conditions, which are Neumann boundary conditions for $\mathcal{H} \times \nu$.

\medskip

On the full space, Maxwell equations with rough coefficients and also quasilinear Maxwell equations were considered in \cite{SchippaSchnaubelt2022} (the two-dimensional case) and the partially anisotropic case in three dimensions was analyzed in \cite{Schippa2021Maxwell3d}. In these works, it was pointed out how Maxwell equations (at least in the case of isotropic media) admit diagonalization to two degenerate half-wave equations and four non-degenerate half-wave equations. The contribution of the degenerate components, i.e., stationary solutions, is quantified by the charges. Here we extend Maxwell equations \eqref{eq:MaxwellGeodesicIntroduction} over the boundary via suitable reflections to carry out the diagonalization afterwards. Since the coefficients of the cometric and the permittivity and permeability are extended evenly, the extension introduces a codimension-1 Lipschitz singularity. After paradifferential decomposition, we can still carry out the diagonalization to half-wave equations similar to the more regular case covered in \cite{Schippa2021Maxwell3d} (see \cite{SchippaSchnaubelt2022} for the previously established two-dimensional case). After diagonalization, we can apply the Strichartz estimates for wave equations with structured Lipschitz singularity due to Blair--Smith--Sogge \cite{BlairSmithSogge2009}. 
We find local-in-time Strichartz estimates for inhomogeneous Maxwell equations by Duhamel's formula:
\begin{equation}
\label{eq:StrichartzEstimatesMaxwell}
\begin{split}
\| (\mathcal{E},\mathcal{H}) \|_{L^p([0,T],L^q(\Omega))} &\lesssim_T \| (\mathcal{E},\mathcal{H})(0) \|_{\mathcal{H}^{\gamma+\delta}(\Omega)} + \| \mathcal{J}_e \|_{L^1(0,T;\mathcal{H}^{\gamma+\delta}(\Omega))} \\
 &\quad + \| \rho_e(0) \|_{H^{\gamma-1 + \frac{1}{p}+\delta}(\Omega)} + \| \nabla \cdot \mathcal{J}_e \|_{L_T^1 H^{\gamma - 1 + \frac{1}{p}+\delta}(\Omega)}.
 \end{split}
\end{equation}
The use of Duhamel's formula in $\mathcal{H}^{\gamma+\delta}$ requires us to impose Dirichlet boundary conditions on $\mathcal{J}_e$.

\medskip

Maxwell equations were previously diagonalized with pseudo-differential operators to prove Strichartz estimates in \cite{SchippaSchnaubelt2022,Schippa2021Maxwell3d}. However, in the corresponding diagonalizatons, the coefficients were more regular. Presently, the coefficients are only Lipschitz regular after reflection, but still the diagonalization only loses $\delta$-regularity compared to the scalar case. We believe our arguments to be general enough to apply to other first-order systems as well. We stress that it is not enough to diagonalize the symbol, but one also has to carefully analyze the mapping properties of the conjugation matrices (see \cite{SchippaSchnaubelt2022Anisotropic}).



\medskip

We digress for a moment to recall Strichartz estimates for the wave equation on domains: Strichartz estimates for wave equations on (general) manifolds with boundary for Dirichlet as well as Neumann boundary conditions were first investigated by the first author \emph{et al.} \cite{BurqLebeauPlanchon2008,BurqPlanchon2009} and Blair--Smith--Sogge \cite{BlairSmithSogge2009} based on the seminal contribution by Smith--Sogge \cite{SmithSogge2007} regarding spectral cluster estimates. Notably, there are more refined results and counterexamples on special domains due to Ivanovici \emph{et al.} \cite{IvanoviciLebeauPlanchon2014,IvanoviciLebeau2017,IvanoviciLebeauPlanchon2021,IvanoviciLebeauPlanchon2021II}.
For exterior convex domains, Smith--Sogge \cite{SmithSogge1995} recovered the Euclidean Strichartz estimates (local-in-time) much earlier by the Melrose--Taylor parametrix.

For Maxwell equations with perfectly conducting boundary conditions, we prove the following theorem:
\begin{theorem}
\label{thm:StrichartzMaxwellDomains}
Let $\Omega \subseteq \R^3$ be a smooth domain with compact boundary and $\varepsilon$, $\mu \in C^\infty(\R^3;\R_{>0})$ satisfy \eqref{eq:Ellipticity}. Let $2 \leq p,q < \infty$, and let $(\mathcal{E},\mathcal{H}): \R \times \Omega \to \R^3 \times \R^3$ denote solutions to \eqref{eq:MaxwellDomainsIntroduction} with material laws \eqref{eq:MaterialLaws}, which satisfy the perfectly conducting boundary conditions \eqref{eq:PerfectConductor}.
Then \eqref{eq:StrichartzEstimatesMaxwell} holds with $\gamma$ and $\delta$ in \eqref{eq:DerivativeScaling} provided that
\begin{equation}
\label{eq:ParametersStrichartz}
\frac{3}{p} + \frac{2}{q} \leq 1.
\end{equation}
\end{theorem}
Recall that the boundary conditions are indistinguishable at low regularities. We have $H^s_D(\Omega) = H^s(\Omega)$ for $s<1/2$ and $H^s_N(\Omega) = H^s(\Omega)$ for $s<\frac{3}{2}$. Since we estimate $\mathcal{J}_e$ in Sobolev spaces with boundary conditions, we have to require
\begin{equation*}
[\mathcal{J}_e]_{x' \in \partial \Omega} = 0
\end{equation*}
for $\gamma \geq \frac{1}{2}$. Note that because $\gamma -1 + \frac{1}{p} + \delta < \frac{1}{2}$ the boundary condition of $\rho_e$ is not relevant.

\bigskip

We shall also discuss the two-dimensional case:
\begin{equation}
\label{eq:Maxwell2dIntroduction}
\left\{ \begin{array}{clcll}
\partial_t (\varepsilon \mathcal{E}) &= \nabla_{\perp} \mathcal{H} - \mathcal{J}_e, &\qquad (t,x') &\in& \R \times \Omega, \\
\partial_t (\mu \mathcal{H}) &= -(\nabla \times \mathcal{E})_3 = -(\partial_1 \mathcal{E}_2 - \partial_2 \mathcal{E}_1), &\quad \nabla \cdot (\varepsilon \mathcal{E}) &=& \rho_e
\end{array} \right.
\end{equation}
with $\nabla_\perp = (\partial_2, -\partial_1)$. Here $\Omega \subseteq \R^2$ denotes a smooth domain in $\R^2$ with compact boundary, and
$\mathcal{E}: \R \times \Omega \to \R^2$, $\mathcal{J}_e : \R \times \Omega \to \R^2$,  $\mathcal{H} : \R \times \Omega \to \R$. We let $\varepsilon,\mu \in C^\infty(\Omega)$. We require $\varepsilon : \Omega \to \R$ and $\mu: \Omega \to \R$ to satisfy
\begin{equation}
\label{eq:Ellipticity2d}
\exists \lambda, \Lambda > 0: \forall x' \in \Omega: \lambda \leq \varepsilon(x'), \mu(x') \leq \Lambda.
\end{equation}
Like above, we require uniform bounds for finitely many derivatives up to the boundary for large $N \geq 2$:
\begin{equation}
 \label{eq:BoundaryRegularity2d}
\varepsilon, \partial \varepsilon, \ldots, \partial^N \varepsilon \in C(\overline{\Omega}) \cap L^\infty(\Omega), \quad \mu, \partial \mu, \ldots, \partial^N \mu \in C(\overline{\Omega}) \cap L^\infty(\Omega).
\end{equation}

The perfectly conducting boundary condition for \eqref{eq:Maxwell2dIntroduction} is given by
\begin{equation}
\label{eq:PerfectlyConductingBoundaryCondition2d}
[ \mathcal{E} \wedge \nu ]_{x' \in \partial \Omega} = 0.
\end{equation}
In Appendix \ref{appendix:LWPHighRegularity} we shall see how Spitz's local well-posedness in three dimensions descends to two dimensions. In the following $\mathcal{H}^3(\Omega)$ denotes the Sobolev space $H^3(\Omega)$ with boundary and compatibility conditions taken into account as in the three-dimensional case. We abuse notation and define $\mathcal{H}^\gamma(\Omega)$ as closure of $\mathcal{H}^3(\Omega)$ in the $H^\gamma(\Omega)$-topology like in \eqref{eq:Density}.
 We prove the following:
 \begin{theorem}
 \label{thm:StrichartzEstimatesMaxwell2d}
 Let $\Omega \subseteq \R^2$ be a smooth domain with compact boundary, $2 \leq p,q \leq \infty$, and suppose that
\begin{equation}
\label{eq:StrichartzAdmissibility2d}
\frac{3}{p} + \frac{1}{q} \leq \frac{1}{2}, \quad q < \infty, \quad \gamma = 2 \big( \frac{1}{2} - \frac{1}{p} \big) - \frac{1}{q}, \quad 0<\delta < \frac{1}{2}.
\end{equation} 
Let $\varepsilon \in C^\infty(\Omega;\R)$, $\mu \in C^\infty(\Omega;\R)$ satisfy \eqref{eq:Ellipticity2d} and \eqref{eq:BoundaryRegularity2d}.

 Then the following estimate holds for solutions to \eqref{eq:Maxwell2dIntroduction} with initial data $(\mathcal{E}_0,\mathcal{H}_0) \in \mathcal{H}^\gamma(\Omega)$ satisfying boundary conditions \eqref{eq:PerfectlyConductingBoundaryCondition2d}:
 \begin{equation*}
 \begin{split}
 \| (\mathcal{E},\mathcal{H}) \|_{L_T^p L^q(\Omega)} &\lesssim_T \| (\mathcal{E}_0,\mathcal{H}_0) \|_{\mathcal{H}^{\gamma+\delta}} + \| \mathcal{J}_e \|_{L_T^1 \mathcal{H}^{\gamma+\delta}} \\
 &\quad + \| \rho_e(0) \|_{H^{\gamma-1+\frac{1}{p}+\delta}(\Omega)} + \| \nabla \cdot \mathcal{J}_e \|_{L_T^1 H^{\gamma-1+\frac{1}{p}+\delta}(\Omega)}.
 \end{split}
 \end{equation*}
 \end{theorem}
 
 \bigskip
 
 In the last section of the paper, we analyze Maxwell equations with Kerr nonlinearity $\varepsilon(\mathcal{E}) = 1 + |\mathcal{E}|^2$ in two dimensions. In the following let $\Omega \subseteq \R^2$ be a smooth domain with compact boundary. We use Strichartz estimates to improve local well-posedness of the following system:
 \begin{equation}
\label{eq:Maxwell2dKerrIntroduction}
\left\{
\begin{array}{clcll}
\partial_t (\varepsilon \mathcal{E}) &= \nabla_{\perp} \mathcal{H}, &\quad [ \mathcal{E} \wedge \nu ]_{x' \in \partial \Omega} &=& 0, \quad (t,x') \in \R \times \Omega, \\
\partial_t \mathcal{H} &= -( \partial_1 \mathcal{E}_2 - \partial_2 \mathcal{E}_1), &\quad [\rho_e]_{x' \in \partial \Omega} &=& 0
\end{array}
\right.
\end{equation}
with $\nabla_\perp = (\partial_2, - \partial_1)$. It turns out that the diagonalization can still be applied if 
\begin{equation*}
\| \partial_x \varepsilon \|_{L_t^2 L_{x'}^\infty} + \| \partial_x \mu \|_{L_t^2 L_{x'}^\infty} \lesssim 1.
\end{equation*}
By $\partial$ we denote space-time derivatives $\partial = \partial_x = \partial_{t,x'}$.
 In this case we cannot use the estimates from \cite{BlairSmithSogge2009}, but rely instead on Strichartz estimates due to Tataru for quasilinear wave equations (see \cite{Tataru2002} and \cite{SchippaSchnaubelt2022} for an elaboration on the half-wave equation). 
 
\begin{theorem}[Low regularity well-posedness of the Kerr system]
\label{thm:KerrLWP}
For $s \in (\frac{11}{6},2]$, \eqref{eq:Maxwell2dKerrIntroduction} is locally well-posed for small initial data $\|(\mathcal{E}_0,\mathcal{H}_0) \|_{\mathcal{H}_0^s} \leq \delta \ll 1$ and finite charges $\| \rho_e(0) \|_{H^{\tilde{s}}} \leq D < \infty$ for some $\tilde{s} > \frac{13}{12}$.
\end{theorem}
By continuous dependence we mean local existence, uniqueness, and continuous dependence of solutions. By $\mathcal{H}_0^s(\Omega)$ we denote the subspace of $\mathcal{H}^s(\Omega)$ with $\mathcal{E}_0$ satisfying zero boundary conditions. This ensures compatibility conditions to hold and facilitates regularization. We refer to Theorem \ref{thm:ImprovedLWPKerr2d} for a more precise statement.  Moreover, we note that in particular for vanishing charges we infer low regularity local well-posedness.

We remark that energy arguments in two dimensions yield local well-posedness for $s>2$. Here we improve on this regularity making use of the dispersive properties.
 
\bigskip

\emph{Outline of the paper.} In Section \ref{section:MaxwellManifolds} we write Maxwell equations in terms of differential forms to facilitate change of variables. We use this to formulate Maxwell equations on the half-space. We reduce Strichartz estimates to homogeneous estimates for the reflected solutions. A key step is the proof of energy estimates. However, these we prove on the level of the original equations posed on domains in Section \ref{section:EnergyEstimates}. In Section \ref{section:Preliminaries} we collect facts on pseudo-differential operators. In Section \ref{section:Diagonalization} we diagonalize three-dimensional Maxwell equations after localization to the half-space, and in Section \ref{section:Diagonalization2d} we diagonalize two-dimensional Maxwell equations. In Section \ref{section:ImprovedLWP} we use Strichartz estimates to improve the local well-posedness of the Kerr system in two dimensions. In Appendix \ref{appendix:LWPHighRegularity}, we relate the two-dimensional Maxwell system to the Maxwell system in three dimensions via cylindrical extension. This allows us to transfer the local well-posedness results at high regularity due to Spitz to two dimensions without revisiting the arguments in three dimensions. In Appendix \ref{appendix:StrichartzRevisited} we revisit Strichartz estimates for Maxwell equations posed on the full space to formulate the estimates in a form, which is suitable for the purposes of the paper. In Appendix \ref{appendix:Helmholtz} we collect facts on Helmholtz decompositions, which are crucial for energy estimates proved in Section \ref{section:EnergyEstimates}.

\section{Maxwell equations on manifolds}
\label{section:MaxwellManifolds}
To investigate the behavior of Maxwell equations under coordinate transformations, we set up Maxwell equations on smooth Riemannian manifolds with boundary $(M,g)$. In this context, the fields are given at any time $t \in \R$ as covectorfields $X(t):M \to T^*M$, $X \in \{\mathcal{E},\mathcal{D},\mathcal{H},\mathcal{B},\mathcal{J}_e\}$, i.e., sections of the cotangent bundle. Permittivity and permeability are given by $\kappa(t): M \to \text{Sym}(T^*M \to T^* M)$, $\kappa \in \{\varepsilon,\mu\}$, and $\rho_e(t) : M \to \R$.
Let $*,d: \Lambda T^* M \to \Lambda T^* M$ denote the Hodge dual and exterior derivative. We localize Maxwell equations to the half-space via geodesic normal coordinates. This facilitates finding the compatibility conditions. These in turn allow us to find suitable extensions of the fields from the half-space to the full space. The extension respects the Sobolev regularity $0 \leq \gamma \leq 2$, which suffices for the presently considered Strichartz estimates, and the extended fields moreover satisfy Maxwell equations on the full space, albeit with coefficients with Lipschitz singularity. We first consider the more involved three-dimensional case and then shall be brief for the two-dimensional case.

\subsection{3d manifolds}
With the aid of Hodge dual and the exterior derivative, we can write for the curl and divergence of vector fields $F: \Omega \to \R^3$:
\begin{equation*}
\nabla \times F = *dF, \qquad \nabla \cdot F = *d*F.
\end{equation*}

Consequently, the Maxwell system of equations reads
\begin{equation}
\label{eq:MaxwellEquationsManifolds}
\left\{ \begin{array}{cllcl}
\partial_t(\varepsilon \mathcal{E}) &= *d \mathcal{H} - \mathcal{J}_e, &\quad *d* (\varepsilon \mathcal{E}) &=& \rho_e, \qquad (t,x') \in \R \times M, \\
\partial_t( \mu \mathcal{H}) &= - * d \mathcal{E}, &\quad *d* (\mu \mathcal{H}) &=& 0.
\end{array} \right.
\end{equation}
Let $\#:TM \to T^*M$ and $\flat: T^*M \to TM$ denote the musical isomorphisms. The boundary conditions are given by
\begin{equation}
\label{eq:PerfectConductorManifold}
[ (\mathcal{E}^\flat)_{||}  ]_{x' \in \partial M} = 0, \qquad [(\mathcal{B}^\flat)_\perp]_{x' \in \partial M} = 0
\end{equation}
We define surface current $\mathcal{J}_\Sigma$ and surface charges $\rho_\Sigma$ on the boundary by
\begin{equation}
\label{eq:BoundaryConditionsIIManifolds}
[(\mathcal{H}^{\flat})_{||}]_{x' \in \partial M} = [\mathcal{J}_\Sigma]_{x' \in \partial M} \text{ and } [(\mathcal{D}^{\flat})_{\perp}]_{x' \in \partial M} = \rho_\Sigma.
\end{equation}

\subsubsection{Finite speed of propagation and Strichartz estimates in the interior}
\label{subsubsection:FinitePropagation}
In this section, we show how we can reduce the local-in-time analysis to charts. We recall the notion of finite speed of propagation.
Let $(\mathcal{E},\mathcal{H})$ be homogeneous solutions to 
\begin{equation*}
 \left\{ \begin{array}{clcll}
  \partial_t (\varepsilon \mathcal{E}) &= \nabla \times \mathcal{H}, &\quad [ \mathcal{E} \times \nu]_{x' \in \partial \Omega} &=& 0, \quad (t,x') \in \R \times \Omega, \\
  \partial_t (\mu \mathcal{H}) &= - \nabla \times \mathcal{E}, &\quad [\mathcal{H} \cdot \nu]_{x' \in \partial \Omega} &=& 0.
 \end{array} \right.
\end{equation*}
Note that we do not require divergence conditions here (cf. \cite[Section~6]{SpitzPhdThesis2017}).

\medskip

For $X \subseteq \Omega$ let $\mathcal{N}_r(X) = \{ x' \in \Omega : \text{dist}(x',X) < r \}$. By Maxwell equations having finite speed of propagation, we mean that there
is $0<c<\infty$ such that for $0<t<\infty$ it holds
\begin{equation*}
 \text{supp}_{x'}( (\mathcal{E},\mathcal{H})(t)) \subseteq \mathcal{N}_{ct} (\text{supp}_{x'}(\mathcal{E}_0,\mathcal{H}_0)).
\end{equation*}
We refer to \cite[Theorem~6.1]{SpitzPhdThesis2017} for a more precise statement in terms of the backwards light cone.

Let $d: \Omega \to \R_{>0}$, $d(x') = \text{dist}(x',\partial \Omega)$ denote the distance function away from the boundary, and $H_\tau = d^{-1}(\tau)$ denote corresponding level sets. By the implicit function theorem, $H_\tau$ is a smooth hypersurface with metric $g_\tau$ and we can write
\begin{equation*}
 g = d\tau^2 + g_\tau \text{ for } 0 \leq t \leq \tilde{\delta}.
\end{equation*}
By compactness of $\partial \Omega$, finitely many geodesics charts suffice to cover a set $\{ x \in \Omega : d(x) < \varepsilon \}$ close to the boundary. Shrinking the charts allows us to restrict to local-in-time solutions, which do not leave the geodesic chart.

\medskip

In Appendix B.1 we show Strichartz estimates in the interior based on Strichartz estimates in the full space. We find $T$ small enough such that $(\mathcal{E},\mathcal{H})(t)$ within $\Omega^{\text{int}} = \{ x' \in \Omega : d(x') > \varepsilon /2 \}$ only depends on $(\mathcal{E},\mathcal{H})(0)$ in $\tilde{\Omega}^{\text{int}} = \{ x' \in \Omega : d(x') > \varepsilon /4 \}$, and the solution does not reach the boundary for times $t \leq T$. 

We prove that
\begin{equation}
\label{eq:StrichartzInterior}
 \| (\mathcal{E},\mathcal{H}) \|_{L_T^p L^q(\Omega^{\text{int}})} \lesssim \| (\mathcal{E}_0,\mathcal{H}_0) \|_{H^s(\Omega)} + \| \rho_e(0) \|_{H^{s-1+\frac{1}{p}}}.
\end{equation}

\subsubsection{Geodesic normal coordinates}

Let $g= (g_{ij})$ denote the metric tensor and $g^{-1} = (g^{ij})$ the cometric. In this work, we only consider isotropic $\varepsilon$ and $\mu$ on the original domain $(\Omega,\delta^{ij})$. We endow a chart in $(\Omega,\delta^{ij})$ with geodesic normal coordinates derived from the height function. Let $x' = (x^*,x_3')$, and
\begin{equation}\label{geodesic}
g = {dx'_3}^2 + r(x^*,x'_3,(dx^*)^2).
\end{equation}
The Hodge dual transforms by
\begin{equation*}
* ({dx'}^{i_1} \wedge \ldots \wedge {dx'}^{i_k}) = \frac{\sqrt{g}}{(n-k) !} g^{i_1 j_1} \ldots g^{i_k j_k} \varepsilon_{j_1 \ldots j_n} {dx'}^{j_{k+1}} \wedge \ldots \wedge {dx'}^{j_n}.
\end{equation*}
Above $\varepsilon_{j_1 \ldots j_n}$ denotes the $n$-Levi--Civita tensor, i.e.,
\begin{equation*}
\varepsilon_{j_1 \ldots j_n} = 
\begin{cases}
1, \qquad (j_1 \ldots j_n) \text{ is an even permutation}, \\
-1, \qquad (j_1 \ldots j_n) \text{ is an odd permutation}, \\
0, \qquad (j_1 \ldots j_n) \text{ is no permutation}.
\end{cases}
\end{equation*}
and $(g^{ij})$ denotes the inverse metric. Recall that we let $\sqrt{g} = \sqrt{\det g}$. Consequently, in geodesic normal coordinates,
\begin{equation*}
*_g d A = \sqrt{g} \, ad(g^{-1}) \nabla \times A, \qquad *_g d *_g A = \frac{1}{\sqrt{g}} \nabla \cdot ( \sqrt{g} g^{-1} (\varepsilon \mathcal{E}')).
\end{equation*}
In the above display $\text{ad}(B)$ denotes the adjugate matrix, i.e., 
\begin{equation*}
\text{ad} (B) = ((-1)^{i+j} B_{ji})_{i,j}
\end{equation*}
with $B_{ji}$ denoting the $(j,i)$-minor of $B$.
By Cramer's rule, Maxwell equations become on the half-space $(t,x') \in \R \times \R^3_{>0}$:
\begin{equation*}
\left\{ \begin{array}{clrl}
\partial_t ( \varepsilon(x') \mathcal{E}') &= (\sqrt{g})^{-1} g \nabla \times \mathcal{H}' - \mathcal{J}'_e, &\qquad \nabla \cdot \big( \sqrt{g} g^{-1} \mu \mathcal{H} \big) &= 0, \\
\partial_t ( \mu(x') \mathcal{H}') &= - \big( \sqrt{g} \big)^{-1} g \nabla \times \mathcal{E'}, &\qquad \frac{1}{\sqrt{g}} \nabla \cdot \big( \sqrt{g} g^{-1} \varepsilon \mathcal{E} \big) &= \rho_e'.
\end{array} \right.
\end{equation*}
In a sense, $\sqrt{g} g^{-1} \varepsilon$ now plays the role of $\varepsilon$ and $\sqrt{g} g^{-1} \mu$ the role of $\mu$. Also, we redefine $\rho_e' = \nabla \cdot (\sqrt{g} g^{-1} \varepsilon \mathcal{E})$, which does not effect regularity questions because $\sqrt{g}$ is smooth. Moreover, we write $\mathcal{J}'_e := \sqrt{g} g^{-1} \mathcal{J}_e'$. Below we shall see that this is consistent with the compatibility conditions.
We rearrange the equations to
\begin{equation}
\label{eq:GeodesicResolution3d}
\left\{ \begin{array}{cl}
\partial_t (\sqrt{g} g^{-1} \varepsilon \mathcal{E}') &= \nabla \times \mathcal{H}' - \mathcal{J}'_e, \qquad \nabla \cdot \big( \sqrt{g} g^{-1} \mu \mathcal{H} \big) = 0, \qquad (t,x') \in \R \times \Omega,  \\
\partial_t ( \sqrt{g} g^{-1} \mu \mathcal{H}') &= - \nabla \times \mathcal{E}', \qquad \qquad  \nabla \cdot \big( \sqrt{g} g^{-1} \varepsilon \mathcal{E} \big) = \rho_e'.
\end{array} \right.
\end{equation}

\subsubsection{Compatibility conditions}

On the half-space $x' \in \R^3_{>0}$, the boundary conditions are given as follows:
\begin{equation}
\label{eq:CompatibilityConditions0}
[\mathcal{E}_1]_{x'_3 = 0} = [\mathcal{E}_2]_{x'_3 = 0} = [\mathcal{H}_3]_{x'_3 = 0} = 0.
\end{equation}
We call a relation 
\begin{equation*}
\text{tr} (F(\partial^\alpha \mathcal{E},\partial^\beta \mathcal{H})) = 0,
\end{equation*}
which follows from \eqref{eq:CompatibilityConditions0} by taking $k$ time derivatives a compatibility condition of order $k$. Hence, \eqref{eq:CompatibilityConditions0} are of order zero.
For \eqref{eq:MaxwellGeodesicIntroduction}, the tangential derivatives are $\partial_t$, $\partial_1$, and $\partial_2$, which allows us to express the compatibitility conditions explicitly.

It is important to observe from~\eqref{geodesic} that the (possibly non-diagonal) metric tensor only mixes the first and second component:
\begin{equation*}
g^{-1} =
\begin{pmatrix}
g^{11} & g^{12} & 0 \\
g^{21} & g^{22} & 0 \\
0		& 0		& 1
\end{pmatrix}
.
\end{equation*}
We give the first order compatibility conditions in the homogeneous case: Applying tangential derivatives $\partial_1$, $\partial_2$ to $\mathcal{H}_3$ gives
\begin{equation*}
[\partial_1 \mathcal{H}_3]_{x'_3 = 0} = [\partial_2 \mathcal{H}_3]_{x'_3 = 0} = 0.
\end{equation*}
The equation for the first and second component of the equation
\begin{equation*}
\partial_t (\sqrt{g} g^{-1} \varepsilon \mathcal{E}) = \nabla \times \mathcal{H}
\end{equation*}
yields
\begin{equation*}
[\partial_3 \mathcal{H}_1]_{x'_3 = 0} = [\partial_3 \mathcal{H}_2]_{x'_3 = 0} = 0.
\end{equation*}

 Moreover, tangential derivatives $\partial_1$, $\partial_2$ applied to $\mathcal{E}_1$ and $\mathcal{E}_2$ and the charge condition yields
\begin{equation*}
[\partial_i \mathcal{E}_j]_{x'_3 = 0} = 0 \text{ for } i,j=1,2  \text{ and } [\nabla \cdot (\varepsilon \sqrt{g} g^{-1} \mathcal{E})]_{x'_3 = 0} = \text{tr}(\rho_e).
\end{equation*}
Let $\varepsilon' = \sqrt{g} \varepsilon$. The above display becomes
\begin{equation*}
 [\partial_3 (\varepsilon' \mathcal{E}_3)]_{x'_3 = 0} = \text{tr}(\rho_e) \Leftrightarrow [(\partial_3 \varepsilon') \mathcal{E}_3]_{x'_3 = 0} + [\varepsilon' \partial_3 \mathcal{E}_3]_{x'_3 = 0} = \text{tr}(\rho_e).
\end{equation*}
This yields a Robin boundary condition for $\mathcal{E}_3$ in terms of $\text{tr}(\rho_e)$ and $\rho_\Sigma$. If these are vanishing, we obtain Neumann boundary conditions for $\mathcal{E}_3$. But note that this is not a compatibility condtion.

 We extend the equations to the full space as follows:
Reflect $\varepsilon$, $\mu$, and $g^{ij}$ evenly. Let
\begin{equation*}
\tilde{\kappa}(x'_1,x'_2,x'_3) =
\begin{cases}
\kappa(x'_1,x'_2,x'_3), \quad \; \; x'_3 \geq 0, \\
\kappa(x'_1,x'_2,-x'_3), \quad x'_3 < 0,
\end{cases}
\qquad \kappa \in \{ \varepsilon, \mu, g^{ij} \}.
\end{equation*}
On the other hand, $\mathcal{E}_1$, $\mathcal{E}_2$, and $\mathcal{H}_3$ are reflected oddly, and $\mathcal{H}_1$, $\mathcal{H}_2$, and $\mathcal{E}_3$ are reflected evenly. $\mathcal{J}_{e1}$, $\mathcal{J}_{e2}$ are reflected oddly and $\mathcal{J}_{e3}$ evenly. $\rho_e$ is reflected oddly. Denoting the reflected quantities with $\tilde{X}$ and $\sqrt{\tilde{g}} = \sqrt{ \det \tilde{g}}$ the following system of equations holds on $\R^3$:
\begin{equation}
\label{eq:ReflectedMaxwellEquations}
\left\{ \begin{array}{cl}
\partial_t (\sqrt{\tilde{g}} \tilde{g}^{-1} \tilde{\varepsilon} \tilde{\mathcal{E}}) &= \nabla \times \tilde{\mathcal{H}} -\tilde{\mathcal{J}_e}, \quad \nabla \cdot (\sqrt{\det \tilde{g}} \tilde{g}^{-1} \tilde{\mu} \tilde{\mathcal{H}}) = 0, \\
\partial_t (\sqrt{\tilde{g}} \tilde{g}^{-1} \tilde{\mu} \tilde{\mathcal{H}}) &= - \nabla \times \tilde{\mathcal{E}}, \qquad \quad \nabla \cdot (\sqrt{\det \tilde{g}} \tilde{g}^{-1} \tilde{\varepsilon} \tilde{\mathcal{E}}) = \tilde{\rho}_e.
\end{array} \right.
\end{equation}

We give the zeroth and first order compatibility conditions under assumptions \eqref{eq:VanishingBoundaryDerivativePermeability}:
\begin{align}
\label{eq:CompatibilityI}
[\mathcal{E} \times \nu]_{x' \in \partial \Omega} = 0, \quad [\mathcal{H}.\nu]_{x' \in \partial \Omega} &= 0, \\
\label{eq:CompatibilityII}
[\partial_\nu \mathcal{H}_{\text{tang}}]_{x' \in \partial \Omega} &= 0.
\end{align}

We find the second compatibility condition by taking two time derivatives in geodesic normal coordinates:\small
\begin{equation*}
\begin{split}
&\; (\partial_t^2 (\sqrt{g} g^{-1} \varepsilon \mathcal{E}) )_{1,2} = \\
& \begin{pmatrix}
\partial_2 (\mu^{-1} \sqrt{g}^{-1} (\partial_1 \mathcal{E}_2 - \partial_2 \mathcal{E}_1)) - \partial_3 (\mu^{-1} \sqrt{g}^{-1} (g_{21} (\partial_2 \mathcal{E}_3 - \partial_3 \mathcal{E}_2) + g_{22} (\partial_3 \mathcal{E}_1 - \partial_1 \mathcal{E}_3))) \\
\partial_3 (\mu^{-1} \sqrt{g}^{-1} (g_{11} (\partial_2 \mathcal{E}_3 - \partial_3 \mathcal{E}_2) + g_{12} (\partial_3 \mathcal{E}_1 - \partial_1 \mathcal{E}_3))) - \partial_1 (\mu^{-1} \sqrt{g}^{-1} (\partial_1 \mathcal{E}_2 - \partial_2 \mathcal{E}_1))
\end{pmatrix}
.
\end{split}
\end{equation*}
\normalsize Clearly,
\begin{equation*}
[ \partial_2 (\mu^{-1} \sqrt{g}^{-1} (\partial_1 \mathcal{E}_2 - \partial_2 \mathcal{E}_1)) ]_{x_3' = 0} = 0
\end{equation*}
because $\partial_1$ and $\partial_2$ are tangential derivatives and the zeroth order compatibility conditions. Similarly,
\begin{equation*}
[ \partial_1 (\mu^{-1} \sqrt{g}^{-1} (\partial_1 \mathcal{E}_2 - \partial_2 \mathcal{E}_1)) ]_{x_3' = 0} = 0.
\end{equation*}
Recall that we required in \eqref{eq:VanishingBoundaryDerivativePermeability}
\begin{equation*}
\partial \mu \vert_{x' \in \partial \Omega} = 0
\end{equation*}
to simplify the compatibility conditions.
Thus, we obtain for the second order compatibility conditions by \eqref{eq:VanishingBoundaryDerivativePermeability}
\begin{equation}
\label{eq:CompatibilityIII}
\begin{split}
[ \partial_3 ( \sqrt{g}^{-1} (g_{21} (\partial_2 \mathcal{E}_3 - \partial_3 \mathcal{E}_2) + g_{22} (\partial_3 \mathcal{E}_1 - \partial_1 \mathcal{E}_3))) ]_{x_3' = 0} = 0, \\
[ \partial_3 ( \sqrt{g}^{-1} (g_{11} (\partial_2 \mathcal{E}_3 - \partial_3 \mathcal{E}_2) + g_{12} (\partial_3 \mathcal{E}_1 - \partial_1 \mathcal{E}_3))) ]_{x_3' = 0} = 0.
\end{split}
\end{equation}


\begin{proposition}
\label{prop:Characterization}
Let $0 \leq \gamma \leq 3$ and $\mathcal{H}^\gamma(\Omega)$ be defined by \eqref{eq:Density} and suppose that \eqref{eq:VanishingBoundaryDerivativePermeability} holds. Then, we have the following characterization:
\begin{equation*}
\begin{array}{clcl}
&\bullet \; 0 \leq \gamma <\frac{1}{2}: &\quad \mathcal{H}^\gamma(\Omega) &= H^\gamma(\Omega)^6, \\
&\bullet \; \frac{1}{2} < \gamma < \frac{3}{2}: &\quad \mathcal{H}^\gamma(\Omega) &\subseteq \{(\mathcal{E}_0,\mathcal{H}_0) \in H^\gamma(\Omega)^6: \eqref{eq:CompatibilityI} \text{ holds.} \}, \\
&\bullet \; \frac{3}{2} < \gamma < \frac{5}{2}: &\quad \mathcal{H}^\gamma(\Omega) &\subseteq \{(\mathcal{E}_0,\mathcal{H}_0) \in H^\gamma(\Omega)^6 : \eqref{eq:CompatibilityI} \text{ and } \eqref{eq:CompatibilityII} \text{ hold.} \}, \\
&\bullet \; \frac{5}{2} < \gamma \leq 3: &\quad \mathcal{H}^\gamma(\Omega) &\subseteq \{ (\mathcal{E}_0,\mathcal{H}_0) \in H^\gamma(\Omega)^6 : \eqref{eq:CompatibilityI} - \eqref{eq:CompatibilityIII} \text{ hold.} \}.
\end{array}
\end{equation*}
\end{proposition}

For the proof we shall change to geodesic normal coordinates. In a chart endowed with geodesic coordinates, i.e., for Maxwell equations localized to the half-space, we have

\begin{align}
\label{eq:CompHalfSpaceI}
&[\mathcal{E}_1]_{x'_3 = 0 } = [\mathcal{E}_2]_{x'_3 = 0} = [\mathcal{H}_3]_{x'_3 = 0} = 0, \\
\label{eq:CompHalfSpaceII}
&[\partial_3 \mathcal{H}_1]_{x'_3 = 0} = [\partial_3 \mathcal{H}_2]_{x'_3 = 0} = 0, \\
\label{eq:CompHalfSpaceIII}
&[ \partial_3 ( \sqrt{g}^{-1} (g_{21} (\partial_2 \mathcal{E}_3 - \partial_3 \mathcal{E}_2) + g_{22} (\partial_3 \mathcal{E}_1 - \partial_1 \mathcal{E}_3))) ]_{x_3' = 0} = 0, \\ 
&[ \partial_3 ( \sqrt{g}^{-1} (g_{11} (\partial_2 \mathcal{E}_3 - \partial_3 \mathcal{E}_2) + g_{12} (\partial_3 \mathcal{E}_1 - \partial_1 \mathcal{E}_3))) ]_{x_3' = 0} = 0. \nonumber
\end{align}

\begin{proof}[Proof~of~Proposition~2.1]
Let $(\Omega_i,\varphi_i)_{i=1,\ldots,n}$ denote a finite covering of a neighbourhood of the boundary with geodesic charts and $(\Omega_0,\varphi_0 = id)$ the trivial chart of the interior. We decompose $u_0 \in H^\gamma(\Omega)$ with a smooth partition of unity subordinate to $(\Omega_j)_{j=0,\ldots,n}$, $1 = \sum_{i=1}^n \psi_i + \psi_0$ and write
\begin{equation*}
u_0 = \sum_{i=1}^n \psi_i u_0 + \psi_0 u_0.
\end{equation*}
It suffices to show the claim for any $u_0^{(i)}$. Within $\Omega_i$ we can endow $\Omega$ with geodesic normal coordinates and it is enough to prove the claim for the transformed fields by invariance of Sobolev spaces under changes of coordinates. For $\Omega_0$ this is trivial because there is no boundary. Within $\Omega_i$ we can use geodesic normal coordinates. Note that with
\begin{equation*}
\mathcal{H}^3(\R^3_{>0}) = \{ (\mathcal{E}_0,\mathcal{H}_0) \in H^3(\R^3_{>0}) : \eqref{eq:CompHalfSpaceI} - \eqref{eq:CompHalfSpaceIII} \text{ hold}. \},
\end{equation*}
we now have to show that $\overline{\mathcal{H}^3(\R^3_{>0})}^{\| \cdot \|_{H^s}} = H^s(\R^3_{>0})^6$ for $0 \leq s < \frac{1}{2}$ and
\begin{equation}
\label{eq:InclusionCompatibility}
\overline{\mathcal{H}^3(\R^3_{>0})}^{\| \cdot \|_{H^s}} \subseteq 
\begin{cases}
\{ (\mathcal{E}_0,\mathcal{H}_0) \in H^s(\R^3_{>0})^6 : \eqref{eq:CompHalfSpaceI} \text{ holds.} \}, \quad \quad \; \, \; \; \;  \frac{1}{2} < s < \frac{3}{2}, \\
\{ (\mathcal{E}_0,\mathcal{H}_0) \in H^s(\R^3_{>0})^6 : \eqref{eq:CompHalfSpaceI}, \eqref{eq:CompHalfSpaceII} \text{ hold.} \}, \, \; \; \frac{3}{2} < s < \frac{5}{2}, \\
\{ (\mathcal{E}_0,\mathcal{H}_0) \in H^s(\R^3_{>0})^6 : \eqref{eq:CompHalfSpaceI} - \eqref{eq:CompHalfSpaceIII} \text{ hold.} \},  \frac{5}{2} < s \leq 3. \\
\end{cases}
\end{equation}
The inclusion $\overline{\mathcal{H}^3(\R^3_{>0})}^{\| \cdot \|_{H^s}} \subseteq H^s(\R^3_{>0})^6$ for $0 \leq s < \frac{1}{2}$ is trivial, and \eqref{eq:InclusionCompatibility} follows from the continuity of the trace. To show the reverse inclusion $\overline{\mathcal{H}^3(\R^3_{>0})}^{\| \cdot \|_{H^s}} \supseteq H^s(\R^3_{>0})^6$, we observe that 
\begin{equation*}
\overline{\mathcal{H}^3(\R^3_{>0})}^{\| \cdot \|_{H^s}} \supseteq \overline{C^\infty_c(\R^3_{>0})^6}^{\| \cdot \|_{H^s}} = H^s(\R^3_{>0})
\end{equation*}
provided that $0 \leq s < \frac{1}{2}$.
\end{proof}
\begin{remark}
We can be more precise about the conditions on $\mathcal{H}_0$. Let
\begin{equation*}
\pi: \overline{\mathcal{H}^3(\R^3_{>0})}^{\| \cdot \|_{H^s}} \to H^s(\Omega)^3, \quad (\mathcal{E}_0,\mathcal{H}_0) \mapsto \mathcal{H}_0
\end{equation*}
denote the projection to the $\mathcal{H}$-initial data. We have the following:
\begin{equation*}
\text{im}(\pi) =
\begin{cases}
H^s(\R_{>0}^3)^3, &\quad 0 \leq s < \frac{1}{2}, \\
H^s(\R_{>0}^3)^2 \times H^s_D(\R_{>0}^3), &\quad \frac{1}{2} < s < \frac{3}{2}, \\
H^s_N(\R_{>0}^3)^2 \times H^s_D(\R_{>0}^3), &\quad \frac{3}{2} < s \leq 3.
\end{cases}
\end{equation*}
We show that $\text{im}(\pi)$ is a superset of the above subsets of $H^s(\R_{>0}^3)^3$.
For $0 \leq s < \frac{1}{2}$ this has was already carried out above. For $\frac{1}{2} < s \leq 3$, we extend $\mathcal{H}_{0i}$ for $i=1,2$ evenly to the full space and $\mathcal{H}_{03}$ oddly. Then we regularize by convolution with a mollifier. The resulting functions are in $H^3(\R^3)$, satisfy the boundary conditions, and approximate the functions in $H^s(\R^3_{>0})$.
This is based on continuity of
\begin{equation}
\label{eq:DirichletExtension}
\begin{split}
\text{ext}_D: H_0^s(\R^3_{>0}) &\to H^s(\R^3) \\
f &\mapsto \bar{f}_o
\end{split}
\end{equation}
for $0 \leq s \leq 2$ with
\begin{equation*}
\bar{f}_o(x') = 
\begin{cases}
f(x_1',x_2', x_3'), \qquad \quad &x'_3 > 0, \\
-f(x_1',x_2',-x_3'), \quad &x'_3 < 0.
\end{cases}
\end{equation*}

Furthermore, even reflection yields a continuous operator for Neumann functions for $0 \leq s \leq 2$:
\begin{equation}
\label{eq:NeumannExtension}
\begin{split}
\text{ext}_N: H_N^s(\R^3_{>0}) &\to H^s(\R^3) \\
f &\mapsto \bar{f}_e
\end{split}
\end{equation}
with
\begin{equation*}
\bar{f}_e(x') = 
\begin{cases}
f(x_1',x_2',x_3'), \quad &x'_3 > 0, \\
f(x_1',x_2',-x_3'), \quad &x'_3 <0.
\end{cases}
\end{equation*}

It seems likely that instead of \eqref{eq:InclusionCompatibility} it holds
\begin{equation*}
\overline{\mathcal{H}^3(\R^3_{>0})}^{\| \cdot \|_{H^s}} =
\begin{cases}
\{ (\mathcal{E}_0,\mathcal{H}_0) \in H^s(\R^3_{>0})^6 : \eqref{eq:CompHalfSpaceI} \text{ holds.} \}, \quad \quad \; \, \; \; \;  \frac{1}{2} < s < \frac{3}{2}, \\
\{ (\mathcal{E}_0,\mathcal{H}_0) \in H^s(\R^3_{>0})^6 : \eqref{eq:CompHalfSpaceI}, \eqref{eq:CompHalfSpaceII} \text{ hold.} \}, \, \; \; \frac{3}{2} < s < \frac{5}{2}, \\
\{ (\mathcal{E}_0,\mathcal{H}_0) \in H^s(\R^3_{>0})^6 : \eqref{eq:CompHalfSpaceI} - \eqref{eq:CompHalfSpaceIII} \text{ hold.} \},  \frac{5}{2} < s \leq 3. \\
\end{cases}
\end{equation*}
To show this, we had to work out a more sophisticated approximation for the components of $\mathcal{E}_0$. We choose to omit the details because this is not the focus of the present work.
\end{remark}

However, we can readily consider a smaller space, for which we have the complete characterization:
\begin{equation*}
\mathcal{H}^s_0(\Omega) = \overline{\mathcal{H}^3(\Omega) \cap (C^\infty_c(\Omega)^3 \times H^3(\Omega)^3)}^{\| \cdot \|_{H^s}}.
\end{equation*}
This subspace ensures zero-boundary conditions on $\mathcal{E}$ and its derivatives (provided the trace makes sense). For $\mathcal{H}_0$ we have the complete characterization by the above argument, for $\mathcal{E}_0$ this follows from the well-known characterization of the Sobolev space with Dirichlet boundary conditions:
\begin{equation*}
H^s_D(\Omega) = 
\begin{cases}
H^s(\Omega), \quad 0 \leq s < \frac{1}{2}, \\
\{ f \in H^s(\Omega) \, : [f]_{x' \in \partial \Omega} = 0 \}, \quad \frac{1}{2} < s \leq 1.
\end{cases}
\end{equation*}
We shall see, that depending on the value of $s$, we find the following boundary conditions for initial data in $\mathcal{H}^s_0$:
\begin{align}
\label{eq:ZeroBoundaryConditionsModified}
&[\mathcal{E}_0]_{x' \in \partial \Omega} = 0, \; [\mathcal{H}_0 \cdot \nu]_{x' \in \partial \Omega} = 0, \\
\label{eq:FirstBoundaryConditionsModified}
&[\mathcal{E}_0]_{x' \in \partial \Omega} = [\partial \mathcal{E}_0]_{x' \in \partial \Omega} = 0, \quad [\mathcal{H}_0.\nu]_{x' \in \partial \Omega} = 0, \; [\partial_\nu (\mathcal{H}_0 \times \nu)]_{x' \in \partial \Omega} = 0 \}, \\
\label{eq:SecondBoundaryConditionsModified}
&[\mathcal{E}_0]_{x' \in \partial \Omega} = [\partial \mathcal{E}_0]_{x' \in \partial \Omega} = [\partial^2 \mathcal{E}_0]_{x' \in \partial \Omega} = 0, \quad [\mathcal{H}_0.\nu]_{x' \in \partial \Omega} = 0, \; [\partial_\nu (\mathcal{H}_0 \times \nu)]_{x' \in \partial \Omega} = 0 \}.
\end{align}

\begin{proposition}
\label{prop:CharacterizationZeroBoundary}
We have the following characterization:
\begin{equation*}
\mathcal{H}^s_0(\Omega) = 
\begin{cases}
\{ (\mathcal{E}_0,\mathcal{H}_0) \in H^s(\Omega)^6 \}, \quad 0 \leq s < \frac{1}{2}, \\
\{ (\mathcal{E}_0,\mathcal{H}_0) \in H^s(\Omega)^6 : \; \eqref{eq:ZeroBoundaryConditionsModified} \; \text{holds.} \, \}, \quad \frac{1}{2} < s < \frac{3}{2}, \\
\{ (\mathcal{E}_0,\mathcal{H}_0) \in H^s(\Omega)^6 : \; \eqref{eq:FirstBoundaryConditionsModified} \; \text{holds.} \, \}, \quad \frac{3}{2} < s <\frac{5}{2}, \\
 \{ (\mathcal{E}_0,\mathcal{H}_0) \in H^s(\Omega)^6 : \; \eqref{eq:SecondBoundaryConditionsModified} \; \text{holds.} \, \}, \quad \frac{5}{2} < s  \leq 3.
\end{cases}
\end{equation*}
\end{proposition}

\subsubsection{Reductions for smooth time-independent coefficients}
As main step in the proof of Theorem \ref{thm:StrichartzMaxwellDomains}, we show the following:
\begin{proposition}
\label{prop:ExtendedMaxwell}
Let $\tilde{u} = (\tilde{\mathcal{E}},\tilde{\mathcal{H}})$, and $\tilde{\varepsilon}$, $\tilde{\mu}$, $\tilde{g}$ like in \eqref{eq:ReflectedMaxwellEquations}. Then the following estimates hold:
\begin{equation}
\label{eq:StrichartzEstimates}
\| \tilde{u} \|_{L^p L^q} \lesssim \| \tilde{u} \|_{L^\infty_T H^{\gamma+\delta}} + \| \tilde{\mathcal{J}}_e \|_{L_t^2 H^{\gamma+\delta}} + \| \tilde{\rho} \|_{L_T^\infty H^{\gamma-1+\frac{1}{p}+\delta}(\Omega)}
\end{equation}
for $p,q \geq 2$, $q<\infty$, $\delta > 0$ satisfying the following
\begin{equation*}
\frac{3}{p} + \frac{2}{q} \leq 1, \qquad \gamma = 3 \big( \frac{1}{2} - \frac{1}{q} \big)- \frac{1}{p}, \qquad  \delta < \frac{3}{q}.
\end{equation*}
\end{proposition}
\begin{remark}
Recall that $\rho$ is reflected oddly. The Dirichlet condition is irrelevant for $s < \frac{1}{2}$, which is ensured with the condition on $\delta$.
\end{remark}

We conclude the section with the following:
\begin{proposition}
\label{prop:ReductionHalfSpace}
Suppose that Proposition \ref{prop:ExtendedMaxwell} holds true and the energy estimate
\begin{equation}
\label{eq:EnergyEstimate}
\| u \|_{L_T^\infty \mathcal{H}^\gamma(\Omega)} \lesssim_T \| u(0) \|_{\mathcal{H}^\gamma(\Omega)}
\end{equation}
is valid for homogeneous solutions $u=(\mathcal{E},\mathcal{H})$ to \eqref{eq:MaxwellDomainsIntroduction}.
 Then, Theorem \ref{thm:StrichartzMaxwellDomains} follows.
\end{proposition}
\begin{proof}
 First, we prove Theorem \ref{thm:StrichartzMaxwellDomains} for homogeneous solutions $u = (\mathcal{E},\mathcal{H})$ with $\mathcal{J}_e = 0$. By virtue of the energy estimate \eqref{eq:EnergyEstimate}, it suffices to show:
\begin{equation}
\label{eq:HomogeneousEstimate}
\| u \|_{L^p([0,T], L^q(\Omega))} \lesssim \| u \|_{L^\infty_T \mathcal{H}^\gamma} + \| \rho_e(0) \|_{H^{\gamma-1+\frac{1}{p}+\delta}(\Omega)}.
\end{equation}
But for homogeneous solutions $u= (\mathcal{E},\mathcal{H})$ to \eqref{eq:MaxwellEquationsManifolds}, the transformed and extended solutions $\tilde{u} = (\tilde{\mathcal{E}},\tilde{\mathcal{H}})$ are likewise homogeneous and satisfy the following estimates by hypothesis:
\begin{equation}
\label{eq:ReflectedStrichartz}
\| \tilde{u} \|_{L^p L^q} \lesssim \| \tilde{u} \|_{L^\infty_T H^\gamma} + \| \tilde{\rho}_e(0) \|_{H^{\gamma-1+ \frac{1}{p}+\delta}}.
\end{equation}
But clearly, $\| u \|_{L^p L^q} \lesssim \| \tilde{u} \|_{L^p L^q}$ and by continuity of $\text{ext}_D$ and $\text{ext}_N$ for $0 \leq s \leq 2$ (see \eqref{eq:DirichletExtension}, \eqref{eq:NeumannExtension}), we have
\begin{equation*}
\| \tilde{u}(t) \|_{H^\gamma} + \| \tilde{\rho}_e(0) \|_{H^{\gamma-1+\frac{1}{p}+\delta}(\R^3)} \lesssim \| u(t) \|_{\mathcal{H}^\gamma} + \| \rho_e(0) \|_{H^{\gamma-1+\frac{1}{p}+\delta}(\Omega)}.
\end{equation*}
This reduces Theorem \ref{thm:StrichartzMaxwellDomains} to Proposition \ref{prop:ExtendedMaxwell} for homogeneous solutions.
Inhomogeneous solutions are covered by the energy estimate \eqref{eq:EnergyEstimate} and superposition.
Indeed, suppose that \eqref{eq:ReflectedStrichartz} holds true. Let $(U(t))_{t \in \R}$ be the $C_0$-group of the Maxwell evolution in $L^2(\Omega)^6$ (cf. \cite[Section~3.2]{HochbruckJahnkeSchnaubelt2015}). Then, we can write the general solution by Duhamel's formula
\begin{equation*}
u(t)= U(t) u_0 + \int_0^t U(t-s) (\tilde{\mathcal{P}}u)(s) ds.
\end{equation*}
We denote
\begin{equation*}
\tilde{\mathcal{P}} = \begin{pmatrix}
\partial_t & - \varepsilon^{-1} \nabla \times \\
\mu^{-1} \nabla \times & \partial_t
\end{pmatrix}
.
\end{equation*}
Changing to $\tilde{\mathcal{P}}$ is necessary as Duhamel's formula has to be applied in conservative form. By smoothness of the coefficients, this is admissible. The proof is complete.

\end{proof}



\subsection{2d manifolds}
It is also useful to treat the two-dimensional case geometrically. In this case we rewrite \eqref{eq:Maxwell2dIntroduction} as
\begin{equation}
\label{eq:2dMaxwellManifolds}
\left\{ \begin{array}{cl}
\partial_t (\varepsilon(x') \mathcal{E}) &= * d \mathcal{H} - \mathcal{J}_e, \qquad * d * (\varepsilon \mathcal{E}) = \rho_e, \\
\partial_t (\mu(x') \mathcal{H}) &= - * d \mathcal{E}, \qquad (t,x') \in \R \times M
\end{array} \right.
\end{equation}
with $\mathcal{E}, \mathcal{J}_e(t) : M \to T^* M$ covectorfields and $\mathcal{H}(t) : M \to \R$ a zero-form. In \eqref{eq:Maxwell2dIntroduction} we have like above $M=(\Omega,\delta^{ij})$. The boundary condition is given by
\begin{equation*}
[(\mathcal{E}^\flat)_{||}]_{x' \in \partial M} = 0.
\end{equation*}

\subsubsection{Finite speed of propagation and Strichartz estimates in the interior}

The interior part
\begin{equation}
\label{eq:InteriorStrichartz2d}
\| (\mathcal{E},\mathcal{H}) \|_{L_T^p L_{x'}^q(\Omega^{\text{int}})} \lesssim \| (\mathcal{E},\mathcal{H}) \|_{L_T^\infty H^\gamma}  + \| \rho_e \|_{L_T^\infty H^{\gamma - 1 + \frac{1}{p}}}
\end{equation}
for homogeneous solutions is handled like in Paragraph \ref{subsubsection:FinitePropagation}. The proof uses finite speed of propagation (see Appendix \ref{appendix:LWPHighRegularity}) to localize the solution away from the boundary. Then we can use appropriate Strichartz estimates in the full space:
Let $(s,p,q)$ be wave Strichartz admissible in two dimensions, i.e.,
\begin{equation*}
2 \leq p \leq \infty, \; 2 \leq q <\infty, \quad \frac{2}{p} + \frac{1}{q} \leq \frac{1}{2}, \quad s = 2 \big( \frac{1}{2} - \frac{1}{q} \big) - \frac{1}{p}.
\end{equation*}
Let $u=(u_1,u_2,u_3) = (u^{(1)},u^{(2)}): \R \times \R^2 \to \R^2 \times \R$, and
\begin{equation*}
\tilde{P} = 
\begin{pmatrix}
\partial_t & 0 & -\partial_2 (\mu_1 \cdot) \\
0 & \partial_t & \partial_1 (\mu_1 \cdot) \\
\partial_1 (\varepsilon_{21} \cdot) - \partial_2 (\varepsilon_{11} \cdot) & \partial_1( \varepsilon_{22} \cdot) - \partial_2(\varepsilon_{12} \cdot) & \partial_t
\end{pmatrix}
.
\end{equation*}
In the Appendix \ref{appendix:StrichartzRevisited} we show Strichartz estimates
\begin{equation*}
\begin{split}
\| |D'|^{-s} u \|_{L_t^p(0,T;L^q_{x'})} &\lesssim_{T,\varepsilon,\mu_1} \| u \|_{L_t^\infty L_{x'}^2} + \| \tilde{P}(x,D) u \|_{L_t^1 L_{x'}^2} \\
&\quad + \big( \| |D'|^{-1 + \frac{1}{p}} \nabla \cdot u^{(1)} \|_{L_t^\infty L_{x'}^2} + \| |D'|^{-1+\frac{1}{p}} \partial_t \nabla \cdot u^{(1)} \|_{L_t^1 L_{x'}^2} \big)
\end{split}
\end{equation*}
under regularity and ellipticity assumptions on $\varepsilon_{ij}$ and $\mu_1$. Above $|D'|^\alpha = (-\Delta)^{\alpha/2}$ denotes the fractional Laplacian defined as Fourier multiplier. Then, in similar spirit to the arguments in Paragraph \ref{subsubsection:FinitePropagation}, we use commutator arguments to conclude \eqref{eq:InteriorStrichartz2d}. The details are omitted to avoid repitition.

\subsubsection{Geodesic normal coordinates}
In the two-dimensional context, geodesic normal coordinates are given by
\begin{equation*}
g^{ij} = g^{11}(x'_1,x'_2) {dx'_1}^2 + {dx'_2}^2.
\end{equation*}
Computing $*d$ and $*d*$ in these coordinates, we find
\begin{equation*}
\left\{ \begin{array}{clrcl}
\partial_t (\varepsilon(x') \mathcal{E}') &= (\sqrt{g})^{-1} g \nabla_{\perp} \mathcal{H}' - \mathcal{J}'_e, &\quad \frac{1}{\sqrt{g}} \nabla \cdot ( \sqrt{g} g^{-1} \varepsilon \mathcal{E}') &=& \rho_e', \\
\partial_t (\mu(x') \mathcal{H}') &= -(\sqrt{g})^{-1} (\partial_1 \mathcal{E}'_2 - \partial_2 \mathcal{E}'_1), &\qquad (t,x') &\in& \R \times \R^2_{>0}.
\end{array} \right.
\end{equation*}
Above $\R^2_{>0} = \{ (x'_1,x'_2) \in \R^2 : x'_2 > 0\}$ denotes the two-dimensional half-plane and $\nabla_\perp = (\partial_2,-\partial_1)$. The boundary condition reads
\begin{equation*}
[ \mathcal{E}_1 ]_{x'_2 = 0} = 0.
\end{equation*}
We rewrite the system by redefining $\mathcal{J}_e := \sqrt{g} g^{-1} \mathcal{J}_e$, $\rho_e := \sqrt{g} \rho_e$ as
\begin{equation*}
\left\{ \begin{array}{clrcl}
\partial_t (\sqrt{g} g^{-1} \varepsilon \mathcal{E}) &= \nabla_\perp \mathcal{H} - \mathcal{J}_e, &\quad \nabla \cdot (\sqrt{g} g^{-1} \varepsilon \mathcal{E}) &=& \rho_e, \\
\partial_t (\sqrt{g} \mu \mathcal{H}) &= -(\partial_1 \mathcal{E}_2 - \partial_2 \mathcal{E}_1), &\quad (t,x') &\in& \R \times \R^2_{>0}.
\end{array} \right.
\end{equation*}

\subsubsection{Compatibility conditions}

Note that the components of $\mathcal{J}_e$ and $\mathcal{E}$ are respected by $g^{-1}$, which is diagonal. Let $\varepsilon' = \sqrt{g} g^{-1} \varepsilon$ for brevity. $\mathcal{E}_1$ is endowed with Dirichlet boundary conditions, and we endow $\mathcal{H}$ with Neumann boundary conditions, which is a first order compatibility condition:
\begin{equation*}
[\partial_2 \mathcal{H}]_{x'_2 = 0} = 0.
\end{equation*}
For $\mathcal{E}$ we obtain from $[\partial_1 \mathcal{E}_1]_{x'_2 =0} = 0$ the following Robin boundary condition by considering the traces of the charges:
\begin{equation*}
 \partial_1(\varepsilon'_{11} \mathcal{E}_1) + \partial_2(\varepsilon'_{22} \mathcal{E}_2) = \rho_e \Rightarrow [(\partial_2 \varepsilon'_{22}) \mathcal{E}_2]_{x' \in \partial \Omega} + [\varepsilon'_{22} \partial_2 \mathcal{E}_2]_{x' \in \partial \Omega} = \text{tr}(\rho_e).
\end{equation*}
With $\gamma < \frac{3}{2}$ in the two-dimensional case, we choose even reflection for $\mathcal{E}_2$ such that the Robin condition is not relevant. In coordinate-free notation, we find the following compatibility conditions in the two-dimensional case:
\begin{align}
\label{eq:0thCompCondition2d}
[\mathcal{E} \wedge \nu]_{ x' \in \partial \Omega} &= 0, \\
\label{eq:1stCompCondition2d}
[\partial_\nu \mathcal{H}]_{x' \in \partial \Omega} &= 0.
\end{align}
In geodesic coordinates the second compatibility condition reads
\begin{equation}
\label{eq:2ndCompCondition2d}
[\partial_2 (\sqrt{g} (\partial_1 \mathcal{E}_2 - \partial_2 \mathcal{E}_1))]_{x_2' = 0} = 0.
\end{equation}

We record the analog of Proposition \ref{prop:Characterization}:

\begin{proposition}
\label{prop:Characterization2d}
Let $0 \leq \gamma \leq 3$ and $\mathcal{H}^\gamma(\Omega)$ be defined by $\mathcal{H}^\gamma(\Omega) = \overline{\mathcal{H}^3(\Omega)}$ and, if $\gamma > \frac{5}{2}$, we suppose that \eqref{eq:VanishingBoundaryDerivativePermeability} holds. Then, we have the following characterization:
\begin{itemize}
\item $0 \leq \gamma <\frac{1}{2}$: $\mathcal{H}^\gamma(\Omega) = H^\gamma(\Omega)^3$,
\item $\frac{1}{2} < \gamma < \frac{3}{2}$: $\mathcal{H}^\gamma(\Omega) \subseteq \{(\mathcal{E}_0,\mathcal{H}_0) \in H^\gamma(\Omega): \eqref{eq:0thCompCondition2d} \text{ holds.} \}$,
\item $\frac{3}{2} < \gamma < \frac{5}{2}$: $\mathcal{H}^\gamma(\Omega) \subseteq \{(\mathcal{E}_0,\mathcal{H}_0) \in H^\gamma(\Omega) : \eqref{eq:0thCompCondition2d} \text{ and } \eqref{eq:1stCompCondition2d} \text{ hold.} \}$, 
\item $\frac{5}{2} < \gamma \leq 3$: $\mathcal{H}^\gamma(\Omega) \subseteq \{ (\mathcal{E}_0,\mathcal{H}_0) \in H^\gamma(\Omega) : \eqref{eq:0thCompCondition2d} - \eqref{eq:2ndCompCondition2d} \text{ hold.} \}$.
\end{itemize}
\end{proposition}

We define like in three dimensions the smaller space
\begin{equation*}
\mathcal{H}_0^s(\Omega) = \overline{\mathcal{H}^3(\Omega) \cap (C^\infty_c(\Omega)^2 \times H^s(\Omega))}^{\| \cdot \|_{H^s}}.
\end{equation*}
For $\mathcal{H}_0^s(\Omega)$ we have the complete characterization given by the following boundary conditions depending on the size of $s$:
\begin{align}
\label{eq:ZeroBoundaryConditionsModified2d}
[\mathcal{E}_0]_{x' \in \partial \Omega} = 0, \\
\label{eq:FirstBoundaryConditionsModified2d}
[\mathcal{E}_0]_{x' \in \partial \Omega} = [ \partial \mathcal{E}_0]_{x' \in \partial \Omega} = 0, \quad [\partial_\nu \mathcal{H}]_{x' \in \partial \Omega} = 0, \\
\label{eq:SecondBoundaryConditionsModified2d}
[\mathcal{E}_0]_{x' \in \partial \Omega} = [\partial \mathcal{E}_0]_{x' \in \partial \Omega} =  [\partial^2 \mathcal{E}_0]_{x' \in \partial \Omega} = 0, \quad [\partial_\nu \mathcal{H}]_{x' \in \partial \Omega} = 0.
\end{align}
The following is the analog of Proposition \ref{prop:CharacterizationZeroBoundary}:
\begin{proposition}
\label{prop:CharacterizationZeroBoundary2d}
We have the following characterization:
\begin{equation*}
\mathcal{H}^s_0(\Omega) = 
\begin{cases}
\{ (\mathcal{E}_0,\mathcal{H}_0) \in H^s(\Omega)^3 \}, \quad 0 \leq s < \frac{1}{2}, \\
\{ (\mathcal{E}_0,\mathcal{H}_0) \in H^s(\Omega)^3 : \; \eqref{eq:ZeroBoundaryConditionsModified2d} \; \text{holds.} \, \}, \quad \frac{1}{2} < s < \frac{3}{2}, \\
\{ (\mathcal{E}_0,\mathcal{H}_0) \in H^s(\Omega)^3 : \; \eqref{eq:FirstBoundaryConditionsModified2d} \; \text{holds.} \, \}, \quad \frac{3}{2} < s <\frac{5}{2}, \\
 \{ (\mathcal{E}_0,\mathcal{H}_0) \in H^s(\Omega)^3 : \; \eqref{eq:SecondBoundaryConditionsModified2d} \; \text{holds.} \, \}, \quad \frac{5}{2} < s  \leq 3.
\end{cases}
\end{equation*}
\end{proposition}

\medskip

\subsubsection{Reductions for smooth time-independent coefficients}

We extend the equations to the plane similar to the three-dimensional case: $\varepsilon$, $\mu$, and $g^{ij}$ are reflected evenly; $\mathcal{E}_1$ and $\rho_e$ are reflected oddly corresponding to Dirichlet boundary conditions; $\mathcal{E}_2$ and $\mathcal{H}$ are reflected evenly. $\mathcal{J}_{ei}$ is reflected like $\mathcal{E}_i$. The extended functions are denoted with a $\tilde{\;}$. We find the following equations on $\R^2$:
\begin{equation}
\label{eq:ExtendedMaxwell2d}
\left\{ \begin{array}{clrcl}
\partial_t (\tilde{\varepsilon} \sqrt{\tilde{g}} \tilde{g}^{-1} \tilde{\mathcal{E}} ) &= \nabla_{\perp} \tilde{\mathcal{H}} - \tilde{\mathcal{J}_e}, &\quad \nabla \cdot (\sqrt{\tilde{g}} \tilde{g}^{-1} \tilde{\varepsilon} \tilde{\mathcal{E}} ) &=& \tilde{\rho}_e, \\
\partial_t ( \tilde{\mu} \sqrt{\tilde{g}} \tilde{\mathcal{H}} ) &= -(\partial_1 \tilde{\mathcal{E}}_2 - \partial_2 \tilde{\mathcal{E}}_1), &\qquad (t,x') &\in& \R \times \R^2.
\end{array} \right.
\end{equation}
For the proof of Theorem \ref{thm:StrichartzEstimatesMaxwell2d} it suffices to prove the following:
\begin{proposition}
\label{prop:ExtendedMaxwell2d}
Let $\tilde{u} = (\tilde{\mathcal{E}}, \tilde{\mathcal{H}})$, and $(\tilde{\varepsilon},\tilde{\mu},\tilde{g})$ like in \eqref{eq:ExtendedMaxwell2d}. Then the following estimate holds:
\begin{equation*}
\| \tilde{u} \|_{L^p L^q} \lesssim \| \tilde{u} \|_{L^\infty_T H^{\gamma+\delta}} + \| \tilde{\mathcal{J}}_e \|_{L_T^2 H^{\gamma+\delta}} + \| \tilde{\rho}_e \|_{L_T^\infty H^{\gamma-1+\frac{1}{p}+\delta}}
\end{equation*}
for $p,q \geq 2$, $q<\infty$, satisfying the following
\begin{equation*}
\frac{3}{p} + \frac{1}{q} \leq \frac{1}{2}, \qquad \gamma = 2 \big( \frac{1}{2} - \frac{1}{q} \big) - \frac{1}{p}, \qquad 0 < \delta < \frac{1}{2}.
\end{equation*}
\end{proposition}
We omit the proof of the following, which is analogous to Proposition \ref{prop:ReductionHalfSpace}:
\begin{proposition}
\label{prop:ReductionHalfSpaceII}
Suppose that Proposition \ref{prop:ExtendedMaxwell} holds true and the energy estimate
\begin{equation}
\label{eq:EnergyEstimateII}
\| u \|_{L_T^\infty \mathcal{H}^\gamma(\Omega)} \lesssim_T \| u(0) \|_{\mathcal{H}^\gamma(\Omega)}.
\end{equation}
is valid for homogeneous solutions $u=(\mathcal{E},\mathcal{H})$ to \eqref{eq:Maxwell2dIntroduction}.
Then, Theorem \ref{thm:StrichartzEstimatesMaxwell2d} follows.
\end{proposition}


\section{Energy estimates}
\label{section:EnergyEstimates}
This section is devoted to the proof of energy estimates, i.e., a priori estimates for the Sobolev norm
\begin{equation}
\label{eq:SobolevEstimates}
\| (\mathcal{E},\mathcal{H})\|_{L_T^\infty \mathcal{H}^s(\Omega)} \lesssim \| (\mathcal{E},\mathcal{H})(0) \|_{\mathcal{H}^s(\Omega)}
\end{equation}
for homogeneous solutions to Maxwell equations on domains with perfectly conducting boundary conditions. For the existence of sufficiently smooth solutions, which make the integration by parts argument licit, we again refer to Spitz's previous works \cite{Spitz2019,Spitz2022,SpitzPhdThesis2017} relying on the energy method. We stress that the a priori estimates in low regularity do not depend on the norms of the solution in high regularity. Spitz also proved a priori estimates, which however do not appear to be suitable for the quasilinear case, as we require a more precise quantification. We take the opportunity to simplify Spitz's argument in the special cases presently considered.

\medskip

It turns out that the $L^2$-norm of the solutions is approximately conserved and also in the quasilinear case we obtain a suitable quantification for a Gr\o nwall argument. To estimate higher regularities, we differentiate the equation in time to see that the time-derivatives still satisfy a Maxwell-like equation. By comparing time and spatial derivatives via the equation, we see that the $L^2$-estimate for the time derivatives can be compared to a Sobolev regularity in space of the same order. Although the strategy is always the same, the arguments are slightly different in each instance, so we opt to give the proofs.

\subsection{The two-dimensional case}
\label{subsection:TwoDimensionalEnergy}
 We begin with the two-dimensional case:
\begin{equation}
\label{eq:Maxwell2}
\left\{ \begin{array}{clrcl}
\partial_t (\varepsilon \mathcal{E}) &= \nabla_\perp \mathcal{H}, &\quad [\nu \wedge \mathcal{E}]_{x' \in \partial \Omega} &=& 0, \quad (t,x') \in \R \times \Omega, \\
\partial_t (\mu \mathcal{H}) &= - (\nabla \times \mathcal{E})_3, &\quad \nabla \cdot (\varepsilon \mathcal{E}) &=& \rho_e.
\end{array} \right.
\end{equation}
$\varepsilon$, $\mu \in C^\infty(\Omega;\R_{>0})$ satisfy the uniform ellipticity condition \eqref{eq:Ellipticity}. We prove the following:
\begin{proposition}
\label{prop:APrioriMaxwell2}
Let $(\mathcal{E},\mathcal{H})$ be $\mathcal{H}^3$-solutions to \eqref{eq:Maxwell2}. Then, for $s \in [0,2]$, we find \eqref{eq:SobolevEstimates} to hold.
\end{proposition}

\medskip
As preliminary, we record the following Helmholtz decomposition on two-di\-men\-sional domains for vector fields with certain boundary conditions. In Appendix C is explained in detail how this follows from results due to Dautray--Lions \cite{DautrayLions1990}.
\begin{proposition}
\label{prop:HelmholtzDecomposition}
Let $s \in [0,1]$, and $\mathcal{E} \in \mathcal{H}^3(\Omega;\R^2)$, which satisfies boundary conditions:
\begin{equation*}
[ \mathcal{E}_{||} ]_{x' \in \partial \Omega} = 0.
\end{equation*}
Then we have the equivalence of norms:
\begin{equation}
\label{eq:HelmholtzDecomposition}
\| \mathcal{E} \|_{H^{s+1}(\Omega)} \sim \| (\nabla \times \mathcal{E})_3 \|_{H^s(\Omega)} + \| \nabla \cdot \mathcal{E} \|_{H^s(\Omega)} + \| \mathcal{E} \|_{L^2(\Omega)}.
\end{equation}
\end{proposition}

\medskip

\begin{proof}[Proof~of~Proposition~\ref{prop:APrioriMaxwell2}]
Let 
\begin{equation*}
M(t) = \int_{\Omega} \mathcal{D}.\mathcal{E} + \mathcal{H}.\mathcal{B} \, dx'
\end{equation*}
with $\mathcal{D} = \varepsilon \mathcal{E}$ and $\mathcal{B} = \mu \mathcal{H}$. We compute
\begin{equation*}
\partial_t M(t)= 2 \int_\Omega \nabla_{\perp} \mathcal{H}. \mathcal{E} \, dx' - 2 \int_\Omega \mathcal{H} (\partial_1 \mathcal{E}_2 - \partial_2 \mathcal{E}_1) \, dx'.
\end{equation*}
By form invariance as argued in Section \ref{section:MaxwellManifolds}, we can suppose that $\Omega = \R^2_{>0}$, $\nu = e_2$. An integration by parts, using the boundary condition for the normal derivative $\partial_2$, gives $\partial_t M(t) = 0$.

The immediate consequence is an $L^2$-a priori estimate:
\begin{equation}
\label{eq:APrioriL2}
\| (\mathcal{E},\mathcal{H})(t) \|_{L^2(\Omega)} \lesssim \| (\mathcal{E},\mathcal{H})(0) \|_{L^2(\Omega)}.
\end{equation}
For higher norms, we consider time derivatives of \eqref{eq:Maxwell2}. We denote $\partial_t A = \dot{A}$ and $\partial_{t}^2 A = \ddot{A}$ for $A \in \{ \mathcal{E}, \mathcal{H} \}$. Taking one time derivative of \eqref{eq:Maxwell2} yields
\begin{equation*}
\left\{ \begin{array}{clrll}
\partial_t (\varepsilon \dot{\mathcal{E}}) &= \nabla_\perp \dot{\mathcal{H}}, &\qquad \quad \nabla \cdot (\varepsilon \dot{\mathcal{E}}) &=& 0, \\
\partial_t (\mu \dot{\mathcal{H}}) &= -( \partial_1 \dot{\mathcal{E}}_2 - \partial_2 \dot{\mathcal{E}}_1), &\qquad [\nu \wedge \dot{\mathcal{E}}]_{x' \in \partial \Omega} &=&0.
\end{array} \right.
\end{equation*}
Hence, $(\dot{\mathcal{E}},\dot{\mathcal{H}})$ solves \eqref{eq:Maxwell2}, and we have the a priori estimates:
\begin{equation*}
\| (\dot{\mathcal{E}},\dot{\mathcal{H}})(t) \|_{L^2(\Omega)} \lesssim \| (\dot{\mathcal{E}},\dot{\mathcal{H}})(0) \|_{L^2(\Omega)}.
\end{equation*}
Note that (again from \eqref{eq:Maxwell2} and ellipticity of $\varepsilon$ and $\mu$), we have
\begin{equation*}
\| ( \dot{\mathcal{E}},\dot{\mathcal{H}})(t) \|_{L^2(\Omega)} \sim \| \mathcal{H}(t) \|_{\dot{H}^1(\Omega)} + \| \mathcal{E}(t) \|_{H_{curl}(\Omega)}.
\end{equation*}
To estimate the full $H^1$-norm by the Helmholtz decomposition \ref{prop:HelmholtzDecomposition}, we observe that $\| (\mathcal{E},\mathcal{H}) (t) \|_{L^2}$ was estimated in the previous step and for $\| \mathcal{E}(t) \|_{H_{div}}$ we find from the condition on the charges
\begin{equation*}
\varepsilon \nabla \cdot \mathcal{E} + (\nabla \varepsilon) \mathcal{E} = \rho_e.
\end{equation*}
The charges are conserved for homogeneous solutions and by \eqref{eq:APrioriL2} we find
\begin{equation*}
\| \mathcal{E}(t) \|_{H_{div}(\Omega)} \lesssim \| \mathcal{E}(t) \|_{L^2(\Omega)} + \| \rho_e(t) \|_{L^2(\Omega)} \lesssim \| (\mathcal{E},\mathcal{H})(0) \|_{L^2(\Omega)} + \| \rho_e(0) \|_{L^2(\Omega)}.
\end{equation*}
This yields
\begin{equation*}
\| (\mathcal{E},\mathcal{H})(t) \|_{H^1(\Omega)} \lesssim \| (\mathcal{E},\mathcal{H})(0) \|_{H^1(\Omega)}.
\end{equation*}
Taking a second time derivative in \eqref{eq:Maxwell2}, we find
\begin{equation*}
\left\{ \begin{array}{clrll}
\partial_t (\varepsilon \ddot{\mathcal{E}} ) &= \nabla_\perp \ddot{\mathcal{H}}, &\qquad \nabla \cdot (\varepsilon \ddot{\mathcal{E}}) &=& 0, \\
\partial_t (\mu \ddot{\mathcal{H}}) &= -(\partial_1 \ddot{\mathcal{E}}_2 - \partial_2 \ddot{\mathcal{E}}_1), &\quad [\nu \wedge \ddot{\mathcal{E}}]_{x' \in \partial \Omega} &=& 0.
\end{array} \right.
\end{equation*}
We use $L^2$-conservation to find
\begin{equation*}
\| ( \ddot{\mathcal{E}},\ddot{\mathcal{H}}) (t) \|_{L^2} \lesssim \| (\ddot{\mathcal{E}},\ddot{\mathcal{H}})(0) \|_{L^2}.
\end{equation*}
Clearly, from iterating \eqref{eq:Maxwell2}, we have
\begin{equation*}
\|  (\ddot{\mathcal{E}},\ddot{\mathcal{H}})(0) \|_{L^2} \lesssim \| (\mathcal{E},\mathcal{H})(0) \|_{H^2}.
\end{equation*}
Secondly, we find
\begin{equation*}
\| \ddot{\mathcal{E}}(t) \|_{L^2(\Omega)} \sim \| \nabla_\perp \dot{\mathcal{H}}(t) \|_{L^2} \text{ with } \nabla_\perp \dot{\mathcal{H}} = O(\partial_{x'} \mu^{-1}) (\partial_1 \mathcal{E}_2 - \partial_2 \mathcal{E}_1) + \mu^{-1} (\Delta \mathcal{E} - \nabla (\nabla \cdot \mathcal{E})).
\end{equation*}
This gives by the conservation of $\| (\ddot{\mathcal{E}}, \ddot{\mathcal{H}})(t) \|_{L^2}$, the previous a priori estimate for the $H^1$-norm and conservation of charges:
\begin{equation}
\label{eq:SecDerivativeETimeInd}
\| \Delta \mathcal{E}(t) \|_{L^2} \lesssim \| \nabla_\perp \mathcal{\dot{H}}(t) \|_{L^2} + \| \rho_e(t) \|_{H^1} + \| (\mathcal{E},\mathcal{H})(t) \|_{H^1} \lesssim \| (\mathcal{E},\mathcal{H})(0) \|_{H^2}.
\end{equation}
For two time derivatives of $\mathcal{H}$ we find
\begin{equation*}
\mu \ddot{\mathcal{H}} = \varepsilon^{-1} \Delta \mathcal{H} + O(\partial_{x'} \varepsilon \, \partial_{x'} \mathcal{H}).
\end{equation*}
It follows from the conservation of $\| (\ddot{\mathcal{E}}, \ddot{\mathcal{H}})(t) \|_{L^2}$ and the previously established a priori estimate for the $H^1$-norm:
\begin{equation}
\label{eq:SecDerivativeHTimeInd}
\| \Delta \mathcal{H}(t) \|_{L^2} \lesssim \| \ddot{\mathcal{H}}(t) \|_{L^2} + \| \mathcal{H}(t) \|_{H^1} \lesssim \| (\mathcal{E},\mathcal{H})(0) \|_{H^2}.
\end{equation}
Taking \eqref{eq:APrioriL2}, \eqref{eq:SecDerivativeETimeInd}, and \eqref{eq:SecDerivativeHTimeInd} together, we find
\begin{equation*}
\| (\mathcal{E},\mathcal{H})(t) \|_{L^2} + \| \Delta \mathcal{E}(t) \|_{L^2} + \| \Delta \mathcal{H}(t) \|_{L^2} \lesssim \| (\mathcal{E},\mathcal{H})(0) \|_{H^2}.
\end{equation*}
The proof is complete.
\end{proof}

We turn to the quasilinear case of the Kerr nonlinearity in two dimensions:
\begin{equation}
\label{eq:Maxwell2Kerr}
\left\{ \begin{array}{clrll}
\partial_t (\varepsilon \mathcal{E}) &= \nabla_\perp \mathcal{H}, &\qquad [\mathcal{E} \times \nu]_{x' \in \partial \Omega} &=& 0, \\
\partial_t \mathcal{H} &= -(\partial_1 \mathcal{E}_2 - \partial_2 \mathcal{E}_1), &\qquad \nabla \cdot ( \varepsilon \mathcal{E}) &=& \rho_e
\end{array} \right.
\end{equation}
with $\varepsilon(\mathcal{E}) = 1 + |\mathcal{E}|^2$ and $\text{tr}(\rho_e) = 0$. In the following we assume that
\begin{equation}
\label{eq:SmallnessFields}
\sup_{t \in [0,T]} \| (\mathcal{E},\mathcal{H})(t) \|_{H^s(\Omega)} \leq \delta \ll 1
\end{equation}
for $(\mathcal{E},\mathcal{H}): [0,T] \times \Omega \to \R^3$ an $\mathcal{H}^3$-solution and $\delta$ to be chosen later.
We prove the following:
\begin{proposition}
\label{prop:APrioriMaxwell2Kerr}
Let $(\mathcal{E},\mathcal{H}) : [0,T] \times \Omega \to \R^2 \times \R$ be an $\mathcal{H}^3$-solution to \eqref{eq:Maxwell2Kerr}, which satisfies \eqref{eq:SmallnessFields}. Then the following estimate holds for $s \in [0,2)$:
\begin{equation}
\label{eq:EnergyEstimateMaxwell2Kerr}
\| (\mathcal{E},\mathcal{H}) \|_{L_T^\infty H^s(\Omega)} \lesssim_{\delta,T} e^{C \int_0^T \| \partial_x \mathcal{E}(s) \|_{L^\infty(\Omega)} ds} \| (\mathcal{E},\mathcal{H})(0) \|_{H^s(\Omega)}.
\end{equation}
\end{proposition}
We remark that no smallness is required to prove \eqref{eq:EnergyEstimateMaxwell2Kerr} for $s \in [0,1]$, but we need smallness for $s >1$. When we apply Proposition \ref{prop:APrioriMaxwell2Kerr} in the proof of improved local well-posedness for the Kerr system, it turns out that it suffices to require smallness of the initial data by a continuity argument. 

As prerequisite, we show a simple fractional Leibniz rule on the domain:
\begin{lemma}
\label{lem:FractionalLeibnizDomain}
Let $\Omega \subseteq \R^d$ be a smooth domain with compact boundary, and $s \in [0,1]$. Let $p,q_1,q_2,r_1,r_2 \in [1,\infty]$ with $\frac{1}{p} = \frac{1}{q_1} + \frac{1}{q_2} = \frac{1}{r_1} + \frac{1}{r_2}$ with $p,q_1,r_2 < \infty$.
For $f,g \in H^s(\Omega)$ we have
\begin{equation*}
\| f g \|_{W^{s,p}(\Omega)} \lesssim \| f \|_{W^{s,q_1}(\Omega)} \| g \|_{L^{q_2}(\Omega)} + \| f \|_{L^{r_1}(\Omega)} \| g \|_{W^{s,r_2}(\Omega)}.
\end{equation*}
\end{lemma}
\begin{proof}
For the interior part this is immediate from the usual fractional Leibniz rule (cf. \cite{GrafakosOh2014,GrafakosTorresAdv}):
\begin{equation*}
\| \langle \partial_x \rangle^s (fg) \|_{L^p(\R^d)} \lesssim \| \langle \partial_x \rangle^{s} f \|_{L^{q_1}(\R^d)} \| g \|_{L^{q_2}(\R^d)} + \| f \|_{L^{r_1}(\R^d)} \| \langle \partial_x \rangle^{s} g \|_{L^{r_2}(\R^d)}.
\end{equation*}
Hence, it suffices to consider (finitely many) charts at the boundary. We change to geodesic coordinates. By invariance of Sobolev spaces under changes of coordinates, it suffices to estimate:
$\| f g \|_{W^{s,p}(\R^d_{>0})}$. We extend $f$ and $g$ evenly and denote the extensions by $\tilde{f}$ and $\tilde{g}$. We have
\begin{equation*}
\| f g \|_{W^{s,p}(\R^d_{>0})} \lesssim \| \tilde{f} \tilde{g} \|_{W^{s,p}(\R^d)}
\end{equation*}
because $\tilde{f} \tilde{g}$ is an even extension of $fg$. The above display is clearly true for $s \in \{0,1\}$ and follows for $s \in (0,1)$ by interpolation. Now we are in the position to apply the usual fractional Leibniz rule on the whole space and find
\begin{equation*}
\| \tilde{f} \tilde{g} \|_{W^{s,p}(\R^d)} \lesssim \| \tilde{f} \|_{W^{s,q_1}(\R^d)} \| \tilde{g} \|_{L^{q_2}(\R^d)} + \| \tilde{f} \|_{L^{r_1}(\R^d)} \| \tilde{g} \|_{W^{s,r_2}(\R^d)}.
\end{equation*}
The proof is concluded by continuity of even extension for $p \in [1,\infty]$ and $s \in [0,1]$:
\begin{equation*}
\begin{split}
\text{ext}_N : W^{s,p}(\R^d_{>0}) &\to W^{s,p}(\R^d) \\
f &\mapsto f_e(x) = 
\begin{cases}
f(x), \quad x_d >0, \\
f(x_1,\ldots,x_{d-1},-x_d), \quad x_d < 0.
\end{cases}
\end{split}
\end{equation*}
\end{proof}

\begin{proof}[Proof~of~Proposition~\ref{prop:APrioriMaxwell2Kerr}]
We change to non-divergence form: Consider
\begin{equation}
\label{eq:LinearizedKerr2}
\left\{ \begin{array}{clrll}
\varepsilon_1 \partial_t \mathcal{E} &= \nabla_\perp \mathcal{H}, &\quad [ \mathcal{E} \wedge \nu ]_{x \in \partial \Omega} &=& 0, \qquad (t,x') \in \R \times \Omega, \\
\partial_t \mathcal{H} &= - (\nabla \times \mathcal{E})_3, &\quad \eta(x,D) \mathcal{E} &=& \rho_e
\end{array} \right.
\end{equation}
with $(\mathcal{E},\mathcal{H})(0) = (\mathcal{E}_0,\mathcal{H}_0)$, $\varepsilon_1(t) = 1 + | \mathcal{E}|^2(t) + 2 \mathcal{E} \otimes \mathcal{E}(t)$, and
\begin{equation*}
\eta(x,D) = ( (1+|\mathcal{E}|^2+ 2 \mathcal{E}_1^2) \partial_1 + 2 \mathcal{E}_1 \mathcal{E}_2 \partial_2 \quad 2 \mathcal{E}_1 \mathcal{E}_2 \partial_1 + (1+|\mathcal{E}|^2+ 2 \mathcal{E}_2^2) \partial_2 ).
\end{equation*}

We have by $\mathcal{H}^3$-well-posedness and Sobolev embedding
\begin{equation}
\label{eq:RegularityEps1}
\varepsilon_1(t) \in W^{1,\infty}(\Omega), \quad \partial_t \varepsilon_1 \in L_T^1 L_{x'}^\infty(\Omega).
\end{equation}
This allows us to define a time-dependent (but linear) evolution operator $\mathbb{T}(t,s): L^2(\Omega) \to L^2(\Omega)$ via linearization with 
\begin{equation*}
\begin{split}
\varepsilon_1(t) &= 1 + |\mathcal{E}'|^2 + 2 \mathcal{E}' \otimes \mathcal{E}', \\
\eta(x,D) &= ( (1+|\mathcal{E}'|^2+ 2 {\mathcal{E}'_1}^2) \partial_1 + 2 \mathcal{E}'_1 \mathcal{E}'_2 \partial_2 \quad 2 \mathcal{E}'_1 \mathcal{E}'_2 \partial_1 + (1+|\mathcal{E}'|^2+ 2 {\mathcal{E}'_2}^2) \partial_2 ),
\end{split}
\end{equation*}
where $\mathcal{E}'$ denotes the $\mathcal{H}^3$-solution to the quasilinear problem. This maps data $(\mathcal{E},\mathcal{H})(s)$ to values of the solution to \eqref{eq:LinearizedKerr2} at time $t$ given by $(\mathcal{E},\mathcal{H})(t)$.

We shall see that for $s \in \{0,1 \}$
\begin{equation*}
\mathcal{T}: H^s(\Omega) \to L_T^\infty H^s(\Omega), \quad (\mathcal{E}_0,\mathcal{H}_0) \mapsto ((\mathcal{E},\mathcal{H})(t))_{t \in [0,T]}
\end{equation*}
satisfies the estimate
\begin{equation}
\label{eq:EstimateLinearEvolution}
\sup_{t \in [0,T]} \| (\mathcal{E},\mathcal{H})(t) \|_{H^s(\Omega)} \lesssim e^{C \int_0^T \| \partial_x \varepsilon_1(s) \|_{L^\infty_{x'}(\Omega)} ds} \| (\mathcal{E}_0,\mathcal{H}_0) \|_{H^s(\Omega)}.
\end{equation}
By linearity and interpolation this implies \eqref{eq:EnergyEstimateMaxwell2Kerr} for $s \in [0,1]$.

 We prove \eqref{eq:EstimateLinearEvolution} for $s=0$. Let $\mathcal{D}(t) = \varepsilon_1(t) \mathcal{E}(t)$ and
\begin{equation*}
M(t) = \int_\Omega \mathcal{D}(t). \mathcal{E}(t) \, dx' + \int_\Omega \mathcal{H}(t). \mathcal{B}(t) \, dx'.
\end{equation*}
We obtain
\begin{equation*}
\begin{split}
\partial_t M(t) &= \int_\Omega \partial_t \varepsilon_1(t) \mathcal{E}(t). \mathcal{E}(t) \, dx' + 2 \int_\Omega \nabla_\perp \mathcal{H}(t). \mathcal{E}(t) dx' - 2 \int_\Omega (\nabla \times \mathcal{E})_3(t) \mathcal{H}(t) \, dx' \\
&\lesssim \| \partial_t \varepsilon_1(t) \|_{L^\infty(\Omega)} \| \mathcal{E}(t) \|^2_{L^2(\Omega)} \\
&\lesssim \| \partial_t \varepsilon_1(t) \|_{L^\infty(\Omega)} M(t).
\end{split}
\end{equation*}
In the first estimate we use that the second and third term cancel each other. This follows from integration by parts using the boundary condition and resolving on the half-space. In the ultimate estimate we use that $\varepsilon_1(t)$ has eigenvalues $1+3|\mathcal{E}'|^2$, $1 + |\mathcal{E}'|^2$. We find $O \in O(2)$ such that
\begin{equation*}
O^t
\begin{pmatrix}
1 + |\mathcal{E}'|^2 + 2 {\mathcal{E}'_1}^2 & 2 \mathcal{E}'_1 \mathcal{E}'_2 \\
2 \mathcal{E}'_1 \mathcal{E}'_2 & 1 + |\mathcal{E}'|^2 + 2 {\mathcal{E}'_2}^2
\end{pmatrix}
O = 
\begin{pmatrix}
1 + 3 |\mathcal{E}'|^2 & 0 \\
 0 & 1 + |\mathcal{E}'|^2
\end{pmatrix}
.
\end{equation*}
This is achieved for $\mathcal{E}' \neq 0$ by requiring
\begin{equation*}
O
\begin{pmatrix}
\mathcal{E}_1' \\
\mathcal{E}_2'
\end{pmatrix}
= 
\begin{pmatrix}
|\mathcal{E}'| \\
0
\end{pmatrix}
.
\end{equation*}
For $\mathcal{E}' = 0$, we can simply choose $O = 1_{2 \times 2}$. We conclude the proof of \eqref{eq:EstimateLinearEvolution} for $s=0$ by $M(t) \sim \| (\mathcal{E},\mathcal{H})(t) \|_{L^2}^2$ and Gr\o nwall's inequality.

We turn to the proof of \eqref{eq:EstimateLinearEvolution} for $s=1$. To this end, we consider the system for the first time-derivatives:
\begin{equation*}
\left\{ \begin{array}{cl}
\varepsilon_1 \partial_t \dot{\mathcal{E}} &= \nabla_\perp \dot{\mathcal{H}} - \dot{\varepsilon}_1 \dot{\mathcal{E}}, \\
\partial_t \dot{\mathcal{H}} &= - (\partial_1 \dot{\mathcal{E}}_2 - \partial_2 \dot{\mathcal{E}}_1 ).
\end{array} \right.
\end{equation*}
Let
\begin{equation*}
\tilde{M}(t) = \int_\Omega \varepsilon_1(t) \dot{\mathcal{E}}(t) . \dot{\mathcal{E}}(t) \, dx' + \int_\Omega \dot{\mathcal{H}}(t). \dot{\mathcal{B}}(t) \, dx'.
\end{equation*}
For the time-derivative of $\tilde{M}(t)$ we find like above by integration by parts
\begin{equation}
\label{eq:TimeDerivativeM}
\begin{split}
\partial_t \tilde{M}(t) &= \int_\Omega (\partial_t \varepsilon_1(t)) \dot{\mathcal{E}}(t) . \dot{\mathcal{E}}(t) \, dx' + 2 \int_\Omega \nabla_\perp \dot{\mathcal{H}}(t). \dot{\mathcal{E}}(t) \, dx' - 2 \int_\Omega \dot{\varepsilon}_1(t) \dot{\mathcal{E}}(t). \dot{\mathcal{E}}(t) \, dx' \\
&\quad - 2 \int_\Omega (\nabla \times \dot{\mathcal{E}})_3 . \dot{\mathcal{H}}(t) dx' \\
&\lesssim \| \partial_t \varepsilon_1(t) \|_{L^\infty_x(\Omega)} M(t).
\end{split}
\end{equation}
Gr\o nwall's inequality yields
\begin{equation*}
\tilde{M}(t) \lesssim e^{\int_0^t \| \partial_t \varepsilon_1(s) \|_{L_{x'}^\infty(\Omega)} ds} \tilde{M}(0).
\end{equation*}
By invoking \eqref{eq:LinearizedKerr2}, we find
\begin{equation*}
\tilde{M}(t) \sim \| \mathcal{H}(t) \|^2_{\dot{H}^1(\Omega)} + \| \mathcal{E}(t) \|^2_{H_{curl}}.
\end{equation*}
We already control $\| (\mathcal{E},\mathcal{H})(t) \|_{L^2}$. For an estimate of $\| \mathcal{E} \|_{H^1(\Omega)}$, by Proposition \ref{prop:HelmholtzDecomposition} we have to estimate $\| \mathcal{E} \|_{H_{div}}$. From the formula for $\rho_e(t)$ we obtain
\begin{equation*}
\rho_e(t) = \nabla \cdot \mathcal{E}(t) + O( (\mathcal{E}')^2 \nabla \mathcal{E}).
\end{equation*}
This yields the estimate
\begin{equation*}
\| \mathcal{E}(t) \|_{H_{div}} \lesssim_\delta \| \rho_e(t) \|_{L^2} + \| \mathcal{E} \|_{H_{curl}}.
\end{equation*}
Furthermore, we find
\begin{equation*}
\dot{\rho}_e(t) = \partial_t \varepsilon_1 \partial_{x'} \mathcal{E} + \eta(x,D) \partial_t \mathcal{E}.
\end{equation*}
Applying the divergence to $\varepsilon_1 \partial_t \mathcal{E} = \nabla_\perp \mathcal{H}$, we find $\eta(x,D) \partial_t \mathcal{E} = O(\partial \varepsilon_1 \partial_t \mathcal{E})$ and therefore,
\begin{equation*}
\partial_t \| \rho_e(t) \|_{L_{x'}^2}^2 \lesssim \| \partial_t \varepsilon_1 \|_{L_{x'}^\infty} \| \partial_{x'} \mathcal{E} \|_{L^2_{x'}} \| \rho_e(t) \|_{L^2_{x'}} + \| \partial_{x'} \varepsilon_1 \|_{L^\infty_{x'}} \| \partial_t \mathcal{E} \|_{L^2_{x'}} \| \rho_{e}(t) \|_{L^2_{x'}}.
\end{equation*}
We find
\begin{equation}
\label{eq:TimeDerivativeCharges}
\partial_t \| \rho_e(t) \|^2_{L^2(\Omega)} \lesssim \| \partial_x \varepsilon_1 \|_{L^\infty_{x'}} (\tilde{M}(t) + \| \rho_e(t) \|^2_{L^2_{x'}(\Omega)}).
\end{equation}
Taking the estimates \eqref{eq:TimeDerivativeM} and \eqref{eq:TimeDerivativeCharges} together, we find by Gr\o nwall's inequality
\begin{equation*}
\tilde{M}(t) + \| \rho_e(t) \|^2_{L^2(\Omega)} \lesssim e^{C \int_0^T \| \partial_x \varepsilon_1(t) \|_{L^\infty_{x'}(\Omega)} dt } (\tilde{M}(0) + \| \rho_e(0) \|^2_{L^2(\Omega)}).
\end{equation*}
This shows \eqref{eq:EstimateLinearEvolution} for $s=1$ and by interpolation we infer \eqref{eq:EstimateLinearEvolution} for $s \in [0,1]$.

\medskip

In the following let $(\mathcal{E},\mathcal{H})(t) = (\mathcal{E}',\mathcal{H}')(t)$ denote the $\mathcal{H}^3$-solution to the quasilinear Maxwell equation. We observe that the time-derivatives satisfy the following equation:
\begin{equation}
\label{eq:TimeDerivativeMaxwell2}
\left\{ \begin{array}{cl}
\varepsilon_1 \partial_t \dot{\mathcal{E}} &= \nabla_\perp \dot{\mathcal{H}} - \dot{\varepsilon}_1 \dot{\mathcal{E}}, \\
\partial_t \dot{\mathcal{H}} &= - (\nabla \times \dot{\mathcal{E}})_3.
\end{array} \right.
\end{equation}
Hence, we can write in $H^s$, $s \in [0,1]$:
\begin{equation*}
(\dot{\mathcal{E}}, \dot{\mathcal{H}})(t) = \mathbb{T}_{\varepsilon_1}(t,0) (\dot{\mathcal{E}},\dot{\mathcal{H}})(0) - \int_0^t \mathbb{T}_{\varepsilon_1}(t,s) \begin{pmatrix}
\varepsilon_1^{-1} \dot{\varepsilon}_1 \dot{\mathcal{E}} \\
0
\end{pmatrix}
ds.
\end{equation*}
We obtain by the estimates for $\mathbb{T}_{\varepsilon_1}(t,s)$:
\begin{equation*}
\begin{split}
\| (\dot{\mathcal{E}},\dot{\mathcal{H}}) \|_{L_T^\infty H^s(\Omega)} &\lesssim e^{\int_0^T \| \partial_x \mathcal{E}(s) \|_{L^\infty_{x'}(\Omega)} ds} \| (\dot{\mathcal{E}},\dot{\mathcal{H}})(0) \|_{H^s(\Omega)} \\
&\quad + \int_0^T e^{\int_0^T \| \partial_x \mathcal{E}(s) \|_{L^\infty_{x'}(\Omega)} ds} \| \varepsilon_1^{-1} \dot{\varepsilon}_1 \dot{\mathcal{E}}(s) \|_{H^s(\Omega)} ds.
\end{split}
\end{equation*}
We evaluate $\| \varepsilon_1^{-1} \dot{\varepsilon}_1 \dot{\mathcal{E}} \|_{H^s(\Omega)}$ by the fractional Leibniz rule proved in Lemma \ref{lem:FractionalLeibnizDomain}. 

There are two cases: Derivatives fall on $\dot{\mathcal{E}}$, which case is handled by the first term, or derivatives fall on $\mathcal{E}$ or $\varepsilon_1^{-1}$ (or on the metric tensor, which results in even lower order terms), which is handled with the second term:
\begin{equation*}
\| \varepsilon_1^{-1} \dot{\varepsilon}_1 \dot{\mathcal{E}} \|_{H^s(\Omega)} \lesssim_\delta \| \dot{\mathcal{E}} \|_{L^\infty(\Omega)} \| \dot{\mathcal{E}} \|_{H^s(\Omega)}+ \| \mathcal{E} \|_{W^{s,p}(\Omega)} \| \dot{\mathcal{E}} \|_{L^\infty(\Omega)} \| \dot{\mathcal{E}} \|_{L^q(\Omega)}.
\end{equation*}
Above we require $\frac{1}{p} + \frac{1}{q} = \frac{1}{2}$ and $s = 2 \big( \frac{1}{2} - \frac{1}{q} \big) < 1$.

\medskip

By smallness assumption \eqref{eq:SmallnessFields} and Sobolev embedding, we obtain
\begin{equation}
\label{eq:EstimateDerivative2dKerr}
\sup_{t \in [0,T]} \| (\dot{\mathcal{E}},\dot{\mathcal{H}})(t) \|_{H^s(\Omega)} \lesssim_\delta e^{C \int_0^T \| \partial_x \mathcal{E}(s) \|_{L^\infty_{x'}(\Omega)} ds} \| (\mathcal{E},\mathcal{H})(0) \|_{H^{s+1}(\Omega)}.
\end{equation}

We have to control $\| (\mathcal{E},\mathcal{H}) \|_{H^{s+1}}$ in terms of $\| (\dot{\mathcal{E}},\dot{\mathcal{H}}) \|_{H^s(\Omega)}$. To this end, we use Proposition \ref{prop:HelmholtzDecomposition} to find
\begin{equation}
\label{eq:HelmholtzEstimateKerr}
\begin{split}
\| (\mathcal{E},\mathcal{H})(t) \|_{H^{s+1}(\Omega)} &\lesssim \| ( \nabla \cdot \mathcal{E} ) (t) \|_{H^s(\Omega)} + \| (\nabla \times \mathcal{E})_3(t) \|_{H^s(\Omega)} + \| \nabla_\perp \mathcal{H}(t) \|_{H^s(\Omega)} \\
&\quad + \| (\mathcal{E},\mathcal{H})(t) \|_{L^2(\Omega)}.
\end{split}
\end{equation}
We have $\| (\nabla \times \mathcal{E})_3(t) \|_{H^s(\Omega)} = \| \partial_t \mathcal{H} \|_{H^s(\Omega)}$ and already control
\begin{equation}
\label{eq:L2Estimate2dKerr}
\| (\mathcal{E}, \mathcal{H})(t) \|_{L^2(\Omega)} \lesssim e^{C \int_0^t \| \partial_x \mathcal{E}(s) \|_{L^\infty_{x'}(\Omega)} ds} \| (\mathcal{E},\mathcal{H})(0) \|_{L^2(\Omega)}.
\end{equation}
By $\rho_e(t) = \eta(x,D) \mathcal{E} = \nabla \cdot \mathcal{E} + O( \mathcal{E}^2 \partial_{x'} \mathcal{E})$, we obtain by the fractional Leibniz rule and Sobolev embedding
\begin{equation}
\label{eq:DivergenceEstimate}
\begin{split}
\| \nabla \cdot \mathcal{E}(t) \|_{H^s(\Omega)} &\lesssim \| \rho_e(t) \|_{H^s(\Omega)} + \| \mathcal{E}(t) \|_{H^{s+1}}^3 = \| \rho_e(0) \|_{H^s(\Omega)} + \| \mathcal{E}(t) \|_{H^{s+1}}^3 \\
&\lesssim \| \nabla \cdot \mathcal{E}(0) \|_{H^s(\Omega)} + \| \mathcal{E}(0) \|_{H^{s+1}(\Omega)}^3 + \| \mathcal{E}(t) \|_{H^{s+1}(\Omega)}^3.
\end{split}
\end{equation}
For $\| \nabla_\perp \mathcal{H} \|_{H^s}$ we estimate by the fractional Leibniz rule and Sobolev embedding
\begin{equation}
\label{eq:EstimateH2dKerr}
\begin{split}
\| \varepsilon_1 \varepsilon_1^{-1} \nabla_\perp \mathcal{H} \|_{H^s} &\lesssim \| \varepsilon_1 \partial_t \mathcal{E} \|_{H^s(\Omega)} \\
&\lesssim \| \mathcal{E}(t) \|_{H^s(\Omega)} \| \partial_t \mathcal{E} \|_{L^2(\Omega)} + \| \varepsilon_1 \|_{L^\infty(\Omega)} \| \partial_t \mathcal{E} \|_{H^s(\Omega)} \\
&\lesssim \| (\mathcal{E},\mathcal{H})(t) \|_{H^{s+1}(\Omega)}^2 + \| \varepsilon_1 \|_{L^\infty(\Omega)} \| \partial_t \mathcal{E} \|_{H^s(\Omega)}.
\end{split}
\end{equation}
By smallness \eqref{eq:SmallnessFields}, we have $\| \varepsilon_1 \|_{L^\infty(\Omega)} \lesssim 1$. Plugging \eqref{eq:L2Estimate2dKerr}, \eqref{eq:DivergenceEstimate}, and \eqref{eq:EstimateH2dKerr} into \eqref{eq:HelmholtzEstimateKerr} yields
\begin{equation*}
\begin{split}
\| (\mathcal{E},\mathcal{H})(t) \|_{H^{s+1}(\Omega)} &\lesssim \| \nabla \cdot \mathcal{E}(0) \|_{H^s(\Omega)} + \| (\mathcal{E},\mathcal{H})(0) \|_{H^{s+1}(\Omega)}^3 + \| \partial_t (\mathcal{E},\mathcal{H})(t) \|_{H^s(\Omega)} \\
&\quad + e^{\int_0^t \| \partial_x \mathcal{E}(s) \|_{L^\infty_{x'}(\Omega)} ds} \| (\mathcal{E},\mathcal{H})(0) \|_{L^2(\Omega)}.
\end{split}
\end{equation*}
Again by \eqref{eq:SmallnessFields} and \eqref{eq:EstimateDerivative2dKerr} we finish the proof.
\end{proof}

We have proved energy estimates for solutions to \eqref{eq:Maxwell2Kerr} at the regularities we shall cover by Strichartz estimates. However, the proof of local well-posedness is anchored at $H^3$. To show existence in $H^3$ for the same times like time of existence in $H^s$, $s \in (11/6,2)$, we prove the following energy estimates:
\begin{proposition}[Energy~estimates~at~high~regularity]
\label{prop:EnergyEstimatesHighKerr}
Let $(\mathcal{E},\mathcal{H})$ be an $\mathcal{H}^3$-solution to \eqref{eq:Maxwell2Kerr} on $[0,T]$ and suppose
that 
\begin{equation}
\label{eq:SmallnessRoughTopology}
\| (\mathcal{E},\mathcal{H}) \|_{L_T^\infty H^{\frac{7}{4}}} \leq \delta \ll 1.
\end{equation}

 Then for $0 \leq t \leq T$ the following estimates hold:
\begin{align}
\label{eq:H2EnergyEstimateKerr}
\|(\mathcal{E},\mathcal{H})(t) \|_{H^2} &\lesssim e^{C \int_0^T \| \partial_{x'} (\mathcal{E},\mathcal{H})(t') \|_{L^\infty_{x'}} dt'} \| (\mathcal{E},\mathcal{H})(0) \|_{H^2}, \\
\label{eq:H3EnergyEstimateKerr}
\| (\mathcal{E}, \mathcal{H})(t) \|_{H^3} &\lesssim e^{C (T+1) e^{\int_0^T \| \partial_{x'} (\mathcal{E},\mathcal{H})(t') \|_{L_{x'}^\infty} dt'} ( \| (\mathcal{E},\mathcal{H})(0) \|_{H^2} +1) } \| (\mathcal{E}, \mathcal{H})(0) \|_{H^3}.
\end{align}
\end{proposition}


Since we construct solutions in $H^s$ for $\frac{11}{6} < s \leq 2$, the smallness in $H^{7/4}$ is guaranteed by smallness of the solution in $H^s$. This in turn is accomplished by requiring smallness of the initial data and a bootstrap argument. For the proof of Proposition \ref{prop:EnergyEstimatesHighKerr} we require a Helmholtz decomposition for $\mathcal{E}$ at higher regularities, which is detailed in Proposition \ref{prop:Helmholtz2d} in Appendix \ref{appendix:Helmholtz}.

\begin{proof}
We turn to the proof of \eqref{eq:H2EnergyEstimateKerr}: This is based on a Gr\o nwall argument applied to $A_1(t) + \| \rho_e(t) \|_{H^1}^2$
with
\begin{equation}
\label{eq:LowerBoundA1}
\begin{split}
A_1(t) &= (\varepsilon_1 \ddot{\mathcal{E}}, \ddot{\mathcal{E}}) + (\ddot{\mathcal{H}}, \ddot{\mathcal{H}}) + (\varepsilon_1 \dot{\mathcal{E}}, \dot{\mathcal{E}}) + (\dot{\mathcal{H}}, \dot{\mathcal{H}}) + (\varepsilon \mathcal{E}, \mathcal{E}) + (\mathcal{H},\mathcal{H}) \\
&\geq c (\| (\nabla \times \mathcal{E})_3 \|_{H^1}^2 + \| \mathcal{E} \|_{L^2}^2 + \| \mathcal{H} \|_{H^2}^2) + O_\delta(\| (\mathcal{E},\mathcal{H}) \|_{H^2}^2).
\end{split}
\end{equation}
Secondly, 
\begin{equation}
\label{eq:UpperBoundA1}
A_1(0) + \| \rho_e(0) \|_{H^1}^2 \lesssim_\delta \| (\mathcal{E}, \mathcal{H})(0) \|^2_{H^2}.
\end{equation}

\medskip

We begin with the proof of \eqref{eq:LowerBoundA1}: Record that $\partial_t (\varepsilon \mathcal{E}) = \varepsilon_1 \dot{\mathcal{E}}$ with $\varepsilon_1 = 1 + |\mathcal{E}|^2 + 2 \mathcal{E} \otimes \mathcal{E}$, and by taking time derivatives of \eqref{eq:Maxwell2Kerr}
\begin{equation}
\label{eq:MaxwellTimeDerivativeI}
\left\{ \begin{array}{cl}
\varepsilon_1 \ddot{\mathcal{E}} &= \nabla_\perp \dot{\mathcal{H}} - \dot{\varepsilon}_1 \dot{\mathcal{E}}, \quad (t,x') \in \R \times \Omega, \\
\ddot{\mathcal{H}} &= - (\nabla \times \dot{\mathcal{E}})_3,
\end{array} \right.
\end{equation}
and
\begin{equation}
\label{eq:MaxwellTimeDerivativeII}
\left\{ \begin{array}{cl}
\varepsilon_1 \mathcal{E}^{(3)} &= \nabla_\perp \ddot{\mathcal{H}} - \ddot{\varepsilon}_1 \dot{\mathcal{E}} - 2 \dot{\varepsilon}_1 \ddot{\mathcal{E}}, \quad (t,x') \in \R \times \Omega, \\
\partial_t \ddot{\mathcal{H}} &= - (\nabla \times \ddot{\mathcal{E}})_3.
\end{array} \right.
\end{equation}
We have
\begin{equation*}
\begin{split}
(\varepsilon_1 \ddot{\mathcal{E}}, \ddot{\mathcal{E}}) &= (\nabla_\perp (\nabla \times \mathcal{E})_3, \varepsilon_1^{-1} \nabla_\perp (\nabla \times \mathcal{E})_3) + 2 (\nabla_\perp (\nabla \times \mathcal{E})_3, \varepsilon_1^{-1} \dot{\varepsilon}_1 \dot{\mathcal{E}}) \\
&\quad + (\dot{\varepsilon}_1 \dot{\mathcal{E}}, \varepsilon_1^{-1} \dot{\varepsilon}_1 \dot{\mathcal{E}}).
\end{split}
\end{equation*}
We prove that the first term is leading order in the sense
\begin{equation}
\label{eq:SecondOrderEstimateI}
(\varepsilon_1 \ddot{\mathcal{E}}, \ddot{\mathcal{E}}) \geq c \| \nabla (\nabla \times \mathcal{E})_3 \|_{L^2}^2 + O_\delta(\| (\mathcal{E},\mathcal{H}) \|_{H^2}^2).
\end{equation}
For $\| \mathcal{E} \|_{L^\infty_{x'}} \lesssim \delta$ we have uniform ellipticity of $\varepsilon_1$.
For the error estimates, note that
\begin{equation*}
\begin{split}
|(\nabla_\perp (\nabla \times \mathcal{E})_3, \varepsilon_1^{-1} \dot{\varepsilon}_1 \dot{\mathcal{E}})| &\lesssim_{\| \mathcal{E} \|_{L^\infty_{x'}}} \| \mathcal{E} \|_{H^2} \| \dot{\mathcal{E}} \|_{L^4}^2 \\
&\lesssim \| \mathcal{E} \|_{H^2} \| (\mathcal{E},\mathcal{H}) \|_{H^{\frac{3}{2}}}^2 \lesssim \delta \| (\mathcal{E},\mathcal{H}) \|_{H^2}^2,
\end{split}
\end{equation*}
and moreover, by Maxwell equations and Sobolev embedding,
\begin{equation*}
\| \dot{\mathcal{E}} \|_{L^4}^4 \lesssim \| \mathcal{H} \|_{H^{\frac{3}{2}}}^4 \lesssim \delta^2 \| \mathcal{H} \|_{H^2}^2.
\end{equation*}

We turn to the second term in \eqref{eq:LowerBoundA1}:
\begin{equation}
\label{eq:SecondHDerivative}
\partial_t \dot{\mathcal{H}} = - (\nabla \times \dot{\mathcal{E}})_3 = - (\nabla \times (\varepsilon_1^{-1} \nabla_\perp \mathcal{H}))_3 = (1+ O(\mathcal{E}^2)) \Delta \mathcal{H} + O( \mathcal{E} \partial_{x'} \mathcal{E} \partial_{x'} \mathcal{H}).
\end{equation}
For this reason,
\begin{equation}
\label{eq:SecondOrderEstimateII}
(\partial_t \dot{\mathcal{H}}, \partial_t \dot{\mathcal{H}}) = \| \Delta \mathcal{H} \|_{L^2}^2 + O_\delta( \| \Delta \mathcal{H} \|_{L^2}^2) + O(\| \mathcal{E} \|_{L^\infty}^2 \| \partial_{x'} \mathcal{E} \|_{L^4}^2 \| \partial_{x'} \mathcal{H} \|_{L^4}^2).
\end{equation}
By Sobolev embedding, we find
\begin{equation*}
(\partial_t \dot{\mathcal{H}}, \partial_t \dot{\mathcal{H}}) = (1+O(\delta)) \| \Delta \mathcal{H} \|_{L^2}^2 + O(\delta^2 \| (\mathcal{E},\mathcal{H}) \|_{H^{\frac{3}{2}}}^4) \geq c \| \Delta \mathcal{H} \|_{L^2}^2 + \delta^2 \| (\mathcal{E}, \mathcal{H}) \|_{H^2}^2.
\end{equation*}
Moreover, again by uniform ellipticity of $\varepsilon_1^{-1}$ for small $\| \mathcal{E} \|_{L^\infty_{x'}}$ and Maxwell equations,
\begin{equation}
\label{eq:FirstOrderEstimate}
(\varepsilon_1 \dot{\mathcal{E}}, \dot{\mathcal{E}}) = (\nabla_\perp \mathcal{H}, \varepsilon_1^{-1} \nabla_\perp \mathcal{H}) \geq c \| \mathcal{H} \|_{\dot{H}^1}^2, \quad (\dot{\mathcal{H}},\dot{\mathcal{H}}) = \| (\nabla \times \mathcal{E})_3 \|_{L^2}^2.
\end{equation}
Taking \eqref{eq:SecondOrderEstimateI} -- \eqref{eq:FirstOrderEstimate} together, we find \eqref{eq:LowerBoundA1} to hold.

Next, we prove that
\begin{equation}
\label{eq:UpperEstimateA1}
A_1(t) + \| \rho_e(t) \|_{H^1}^2 \lesssim (1+ \delta^2) \| (\mathcal{E},\mathcal{H}) \|_{H^2}^2,
\end{equation}
which establishes \eqref{eq:UpperBoundA1}. Above we let
\begin{equation*}
\rho_e(t) = \nabla \cdot (\varepsilon \mathcal{E}) = \nabla \cdot \mathcal{E} + \nabla \cdot (|\mathcal{E}|^2 \mathcal{E}) = \nabla \cdot \mathcal{E} + O(\mathcal{E}^2 \partial_{x'} \mathcal{E}).
\end{equation*}

First, we note that by the estimates, which established \eqref{eq:SecondOrderEstimateI}:
\begin{equation}
\label{eq:UpperBoundA11}
(\varepsilon_1 \ddot{\mathcal{E}}, \ddot{\mathcal{E}}) \leq C \| \nabla_\perp (\nabla \times \mathcal{E})_3 \|_{L^2}^2 + \delta^2 \| (\mathcal{E},\mathcal{H}) \|_{H^2}^2.
\end{equation}
By \eqref{eq:SecondHDerivative}, Sobolev embedding, and \eqref{eq:SmallnessRoughTopology}, we find
\begin{equation}
\label{eq:SecondTimeDerivativeHUpperBound}
\begin{split}
\| \ddot{\mathcal{H}} \|_{L^2}^2 &\lesssim \| \mathcal{H} \|_{H^2}^2 + \delta^2 \| (\mathcal{E},\mathcal{H}) \|_{H^2}^2 + \delta^2 \| (\mathcal{E},\mathcal{H}) \|_{H^{\frac{3}{2}}}^4 \\
&\lesssim \| (\mathcal{E},\mathcal{H}) \|_{H^2}^2 + \delta^2 \| (\mathcal{E},\mathcal{H}) \|_{H^2}^2.
\end{split}
\end{equation}

Regarding the charges note that
\begin{equation}
\label{eq:UpperEstimateChargeH1}
\| O( \mathcal{E}^2 \partial_{x'} \mathcal{E}) \|_{H^1} \lesssim \delta^2 \| \mathcal{E} \|_{H^2}.
\end{equation}

Taking \eqref{eq:UpperBoundA11} -- \eqref{eq:UpperEstimateChargeH1} together yields \eqref{eq:UpperEstimateA1}.

\medskip

Furthermore, \eqref{eq:UpperEstimateChargeH1} gives
\begin{equation}
\label{eq:LowerBoundComplete}
A_1(t) + \| \rho_e(t) \|_{H^1}^2 \geq c \| (\mathcal{E},\mathcal{H}) \|_{H^2}^2
\end{equation}
provided that $\| (\mathcal{E},\mathcal{H}) \|_{H^{\frac{3}{2}}} \leq \delta \ll 1$ and $\delta$ sufficiently small. This is a consequence of the Helmholtz decomposition Proposition \ref{prop:Helmholtz2d}, which is proved in the Appendix:
\begin{equation*}
\| (\nabla \times \mathcal{E})_3 \|_{H^1} + \| \nabla \cdot \mathcal{E} \|_{H^1} + \| \mathcal{E} \|_{L^2} \sim \| \mathcal{E} \|_{H^2}.
\end{equation*}

\medskip

 Next, we compute $\partial_t A_1(t)$ to apply a Gr\o nwall argument. Recall that charges are conserved.
We have seen in Proposition \ref{prop:APrioriMaxwell2Kerr} by integration by parts
\begin{equation*}
\begin{split}
\partial_t ((\mathcal{H},\mathcal{H}) + (\varepsilon \mathcal{E},\mathcal{E})) &= 2 (\partial_t \mathcal{H},\mathcal{H}) + 2 (\partial_t (\varepsilon \mathcal{E}), \mathcal{E}) - (\dot{\varepsilon} \mathcal{E},\mathcal{E}) \\
 &= -(\dot{\varepsilon} \mathcal{E}, \mathcal{E}) \lesssim_\delta \| \partial_t \mathcal{E} \|_{L^\infty} \| (\mathcal{E},\mathcal{H}) \|_{L^2}^2.
 \end{split}
\end{equation*}
Similarly,
\begin{equation*}
\begin{split}
\partial_t ( (\varepsilon_1 \dot{\mathcal{E}}, \dot{\mathcal{E}}) + (\dot{\mathcal{H}}, \dot{\mathcal{H}})) \lesssim \| \partial_x \mathcal{E} \|_{L^\infty} \| \dot{\mathcal{E}} \|_{L^2}^2.
\end{split}
\end{equation*}

We compute further for the highest order terms using \eqref{eq:MaxwellTimeDerivativeII} (again we use integration by parts to find cancellation of the terms with the most derivatives):
\begin{equation*}
\begin{split}
\partial_t (( \ddot{\mathcal{H}}, \ddot{\mathcal{H}}) + (\varepsilon_1 \ddot{\mathcal{E}}, \ddot{\mathcal{E}}))
&= 2 (\partial_t \ddot{\mathcal{H}}, \ddot{\mathcal{H}}) + (\dot{\varepsilon}_1 \ddot{\mathcal{E}}, \ddot{\mathcal{E}}) + 2(\varepsilon_1 \mathcal{E}^{(3)}, \ddot{\mathcal{E}}) \\
&= ( \dot{\varepsilon}_1 \ddot{\mathcal{E}}, \ddot{\mathcal{E}}) - 2 (\ddot{\varepsilon}_1 \dot{\mathcal{E}}, \ddot{\mathcal{E}}) - 4 (\dot{\varepsilon}_1 \ddot{\mathcal{E}}, \ddot{\mathcal{E}}) \\
&= - 3(\dot{\varepsilon}_1 \ddot{\mathcal{E}}, \ddot{\mathcal{E}}) - 2 (\ddot{\varepsilon}_1 \dot{\mathcal{E}}, \ddot{\mathcal{E}}).
\end{split}
\end{equation*}
We have
\begin{equation*}
| (\dot{\varepsilon}_1 \ddot{\mathcal{E}}, \ddot{\mathcal{E}}) | \lesssim_\delta \| \partial_x \mathcal{E} \|_{L^\infty_{x'}} \| \ddot{\mathcal{E}} \|_{L^2}^2
\end{equation*}
and $\ddot{\varepsilon}_1 = O( \mathcal{E} \ddot{\mathcal{E}}) + O(\dot{\mathcal{E}}^2)$, which implies
\begin{equation*}
|(\ddot{\varepsilon}_1 \dot{\mathcal{E}}, \ddot{\mathcal{E}})| \lesssim \| \mathcal{E} \|_{L^\infty_{x'}} \| \partial_x \mathcal{E} \|_{L^\infty_{x'}} \| \ddot{\mathcal{E}} \|_{L^2}^2 + \| \ddot{\mathcal{E}} \|_{L^2} \| \dot{\mathcal{E}} \|_{L^\infty} \| \dot{\mathcal{E}} \|_{L^4}^2.
\end{equation*}
The first term is already in suitable form, for the second term observe by Sobolev embedding and \eqref{eq:LowerBoundA1}:
\begin{equation*}
\begin{split}
\| \dot{\mathcal{E}} \|_{L^\infty_{x'}} \| \ddot{\mathcal{E}} \|_{L^2} \| (\mathcal{E},\mathcal{H}) \|_{H^2} \delta &\lesssim \| \dot{\mathcal{E}} \|_{L^\infty} (\| \ddot{\mathcal{E}} \|_{L^2}^2 + \delta^2 \| (\mathcal{E},\mathcal{H}) \|_{H^2}^2) \\
&\lesssim \| \dot{\mathcal{E}} \|_{L^\infty} (A_1(t) + \| \rho_e(t) \|_{H^1}^2).
\end{split}
\end{equation*}

This shows
\begin{equation*}
\partial_t (A_1(t) + \| \rho_e(t) \|_{H^1}^2) \lesssim_\delta \| \partial_x (\mathcal{E},\mathcal{H}) \|_{L^\infty} ( A_1(t) + \| \rho_e(t) \|_{H^1}^2).
\end{equation*}
Hence, by Gr\o nwall's argument, we find
\begin{equation*}
A_1(t) + \| \rho_e(t) \|_{H^1}^2 \lesssim e^{C \int_0^t \| \partial_x (\mathcal{E},\mathcal{H})(s) \|_{L^\infty_{x'}} ds} (A_1(0) + \| \rho_e(0) \|_{H^1}).
\end{equation*}
We use \eqref{eq:UpperEstimateA1} and \eqref{eq:LowerBoundComplete} to infer \eqref{eq:H2EnergyEstimateKerr}.

\bigskip

We turn to the a priori estimate in $H^3$: Presently, we replace $A_1$ by
\begin{equation*}
A_2(t) = (\varepsilon_1 \mathcal{E}^{(3)}, \mathcal{E}^{(3)}) + (\mathcal{H}^{(3)}, \mathcal{H}^{(3)}) + A_1(t).
\end{equation*}

Record that
\begin{equation}
\label{eq:TimeDerivativeMaxwellIII}
\left\{ \begin{array}{cl}
\varepsilon_1 \mathcal{E}^{(4)} &= \nabla_\perp \mathcal{H}^{(3)} - 3 \dot{\varepsilon}_1 \mathcal{E}^{(3)} - \varepsilon_1^{(3)} \dot{\mathcal{E}} - 3 \ddot{\varepsilon}_1 \mathcal{E}^{(2)}, \quad (t,x') \in \R \times \Omega, \\
\partial_t \mathcal{H}^{(3)} &= -(\nabla \times \mathcal{E}^{(3)})_3.
\end{array} \right.
\end{equation}

We shall first show that
\begin{equation}
\label{eq:LowerBoundA2}
A_2(t) \geq c ( \| (\nabla \times \mathcal{E})_3 \|_{H^2}^2 + \| \mathcal{E} \|_{L^2}^2 + \| \mathcal{H} \|_{H^3}^2 ) + O_\delta(\|(\mathcal{E},\mathcal{H}) \|_{H^3}^2).
\end{equation}

We have
\begin{equation*}
(\varepsilon_1 \mathcal{E}^{(3)}, \mathcal{E}^{(3)}) = (\nabla_\perp \ddot{\mathcal{H}}, \varepsilon_1^{-1} \nabla_\perp \ddot{\mathcal{H}}) - 2(\dot{\varepsilon}_1 \mathcal{E}^{(2)}, \mathcal{E}^{(3)}) - (\ddot{\varepsilon}_1 \dot{\mathcal{E}}, \mathcal{E}^{(3)})
\end{equation*}
and like above we shall see that the first term is of leading order:
\begin{equation}
\label{eq:FirstTermA2}
(\varepsilon_1 \mathcal{E}^{(3)}, \mathcal{E}^{(3)}) \geq c \| \nabla_\perp \ddot{\mathcal{H}} \|_{L^2}^2 + O_\delta(\|(\mathcal{E},\mathcal{H}) \|_{H^3}^2).
\end{equation}
To this end, we find by Maxwell equations and Sobolev embedding for $|(\dot{\varepsilon}_1 \mathcal{E}^{(2)}, \mathcal{E}^{(3)} )|$:
\begin{equation}
\label{eq:FirstTermA2Aux}
| (\dot{\varepsilon}_1 \mathcal{E}^{(2)}, \mathcal{E}^{(3)}) | \lesssim \| \mathcal{E} \|_{L^\infty_{x'}} \| \partial_t \mathcal{E} \|_{L^4} \| \mathcal{E}^{(2)} \|_{L^4} \| \mathcal{E}^{(3)} \|_{L^2} \lesssim \delta^2 \| \mathcal{E}^{(2)} \|_{L^4} \| \mathcal{E}^{(3)} \|_{L^2}.
\end{equation}
Furthermore, by boundedness of $\| \varepsilon_1^{-1} \|_{L^\infty}$ due to smallness of $\| \mathcal{E} \|_{L^\infty}$, Maxwell equations, and H\"older's inequality we find
\begin{equation*}
\begin{split}
\| \ddot{\mathcal{E}} \|_{L^4} &\lesssim \| \varepsilon_1^{-1} \nabla_\perp \dot{\mathcal{H}} \|_{L^4} + \| \varepsilon_1^{-1} \dot{\varepsilon}_1 \dot{\mathcal{E}} \|_{L^4} \\
&\lesssim \| \partial^2_{x'} \mathcal{E} \|_{L^4} + \| \mathcal{E} \|_{L^\infty} \| \partial_{x'} \mathcal{H} \|_{L^8}^2 \lesssim \| (\mathcal{E},\mathcal{H}) \|_{H^3}.
\end{split}
\end{equation*}
We have $\ddot{\varepsilon}_1 = O (\mathcal{E} \ddot{\mathcal{E}}) + O( \dot{\mathcal{E}}^2)$. $|(O(\mathcal{E} \ddot{\mathcal{E}}) \dot{\mathcal{E}}, \mathcal{E}^{(3)})|$ is estimated like in \eqref{eq:FirstTermA2Aux}.
Secondly, by Sobolev embedding and Maxwell equations, we find
\begin{equation*}
\begin{split}
|(\dot{\mathcal{E}}^3, \mathcal{E}^{(3)}) | &\lesssim \| \mathcal{E}^{(3)} \|_{L^2} \| \dot{\mathcal{E}} \|_{L^6}^3 \\
&\lesssim \| \mathcal{E}^{(3)} \|_{L^2} \| (\mathcal{E},\mathcal{H}) \|_{H^3} \delta^2 \lesssim \delta^2 ( \| \mathcal{E}^{(3)} \|_{L^2}^2 + \| (\mathcal{E},\mathcal{H}) \|_{H^3}^2).
\end{split}
\end{equation*}
This concludes \eqref{eq:FirstTermA2}.

\medskip

In \eqref{eq:SecondHDerivative} we had shown that
\begin{equation}
\label{eq:SecondTimeDerivativeH}
\partial_t^2 \mathcal{H} = (1+ O(\mathcal{E}^2)) \Delta \mathcal{H} + O(\mathcal{E} \partial_{x'} \mathcal{E} \partial_{x'} \mathcal{H}).
\end{equation}
Hence,
\begin{equation*}
\nabla_\perp \partial_t^2 \mathcal{H} = \nabla_\perp \Delta \mathcal{H} + O (\mathcal{E} \partial_{x'} \mathcal{E} \Delta \mathcal{H}) + O((\partial_{x'} \mathcal{E})^2 \partial_{x'} \mathcal{H}) + O(\mathcal{E} \partial_{x'}^2 \mathcal{E} \partial_{x'} \mathcal{H}) + O(\mathcal{E} \partial_{x'} \mathcal{E} \partial_{x'}^2 \mathcal{H}).
\end{equation*}
We compute by Maxwell equations, H\"older's inequality and Sobolev embedding:
\begin{equation}
\| O( \mathcal{E} \partial_{x'} \mathcal{E} \Delta \mathcal{H}) \|_{L^2} \lesssim \| \mathcal{E} \|_{L^\infty_{x'}} \| \partial_{x'} \mathcal{E} \|_{L^4} \| \Delta \mathcal{H} \|_{L^4} \lesssim \delta^2 \| (\mathcal{E},\mathcal{H}) \|_{H^3}.
\label{eq:AuxEstimateHI}
\end{equation}
For the second term, we find
\begin{equation}
\label{eq:AuxEstimateHII}
\| O((\partial_{x'} \mathcal{E})^2 \partial_{x'} \mathcal{H}) \|_{L^2_{x'}} \lesssim \| \partial_{x'} \mathcal{H} \|_{L^2} \| \partial_{x'} \mathcal{E} \|_{L^4}^2 \lesssim \delta^2 \| (\mathcal{E},\mathcal{H}) \|_{H^3}.
\end{equation}
By \eqref{eq:AuxEstimateHI}, \eqref{eq:AuxEstimateHII}, and the Cauchy-Schwarz inequality, we find
\begin{equation*}
\| \nabla_\perp \partial_t^2 \mathcal{H} \|_{L^2}^2 = \| \mathcal{H} \|_{\dot{H}^3}^2 + O(\delta^2 \| (\mathcal{E},\mathcal{H}) \|_{H^3}^2).
\end{equation*}
This shows that
\begin{equation}
\label{eq:LowerBoundE3}
(\varepsilon_1 \mathcal{E}^{(3)}, \mathcal{E}^{(3)}) \geq c \| \mathcal{H} \|_{\dot{H}^3}^2 + O(\delta^2 \| (\mathcal{E},\mathcal{H}) \|_{H^3}^2).
\end{equation}

By \eqref{eq:SecondTimeDerivativeH} again, we obtain
\begin{equation*}
\begin{split}
\partial_t^3 \mathcal{H} &= (1+ O(\mathcal{E}^2)) \Delta (\nabla \times \mathcal{E})_3 + O(\mathcal{E} \varepsilon_1^{-1} \partial_{x'} \mathcal{H} \Delta \mathcal{H}) \\
&\quad + O( \varepsilon_1^{-1} \nabla_\perp \mathcal{H} \partial_{x'} \mathcal{E} \partial_{x'} \mathcal{H}) + O(\mathcal{E} \partial_{x'} (\varepsilon_1^{-1} \partial_{x'} \mathcal{H}) \partial_{x'} \mathcal{H}) + O(\mathcal{E} \partial_{x'} \mathcal{E} \partial_{x'}^2 \mathcal{E}).
\end{split}
\end{equation*}
Regarding the error terms, we find by H\"older's inequality, Sobolev embedding, and the smallness condition:
\begin{equation}
\label{eq:AuxEstimateHIB}
\begin{split}
\| O(\mathcal{E} \varepsilon_1^{-1} \partial_{x'} \mathcal{H} \Delta \mathcal{H}) \|_{L^2} &\lesssim \| \mathcal{E} \|_{L^\infty_{x'}} \| \varepsilon_1^{-1} \|_{L^\infty_{x'}} \| \partial_{x'} \mathcal{H} \|_{L^4} \| \Delta \mathcal{H} \|_{L^4} \\
&\lesssim \delta^2 \| \varepsilon_1^{-1} \|_{L^\infty} \| \mathcal{H} \|_{H^3}.
\end{split}
\end{equation}
Secondly,
\begin{equation}
\label{eq:AuxEstimateHIIB}
\| O(\varepsilon_1^{-1} \nabla_\perp \mathcal{H} \partial_{x'} \mathcal{E} \partial_{x'} \mathcal{H}) \|_{L^2_{x'}} \lesssim \| \partial_{x'} \mathcal{H} \|_{L^4}^2 \| \partial_{x'} \mathcal{E} \|_{L^2} \lesssim \delta^2 \| (\mathcal{E},\mathcal{H}) \|_{H^3}.
\end{equation}
The third and fourth error term can be treated as variants. Hence,
\begin{equation}
\label{eq:LowerBoundH3}
\| \partial_t^3 \mathcal{H} \|_{L^2}^2 \geq c \| \Delta (\nabla \times \mathcal{E})_3 \|_{L^2}^2 + O(\delta^2 \| (\mathcal{E},\mathcal{H}) \|_{H^3}^2)
\end{equation}
by the Cauchy-Schwarz inequality for some $c>0$ provided that $\delta$ is chosen small enough.
By \eqref{eq:LowerBoundE3}, \eqref{eq:LowerBoundH3}, and \eqref{eq:LowerBoundA1}, we conclude \eqref{eq:LowerBoundA2}.
In order to establish a lower bound in terms of $\| \mathcal{E} \|_{H^3}$, we use the Helmholtz decomposition in $H^3$:
\begin{equation*}
\| (\nabla \times \mathcal{E})_3 \|_{H^2} + \| \nabla \cdot \mathcal{E} \|_{H^2} + \| \mathcal{E} \|_{L^2} \sim \| \mathcal{E} \|_{H^3}.
\end{equation*}
Like above, we shall add $\| \rho_e(t) \|_{H^2}^2$ to $A_2$. Again we use that
\begin{equation*}
\rho_e(t) = \nabla \cdot \mathcal{E} + O (\mathcal{E}^2 \partial_{x'} \mathcal{E}).
\end{equation*}
Note that
\begin{equation*}
\begin{split}
\| \mathcal{E}^2 \partial_{x'} \mathcal{E} \|_{H^2} &\lesssim \| \mathcal{E} \|_{L^\infty}^2 \| \mathcal{E} \|_{H^3} + \| O(\mathcal{E} \partial_{x'}^2 \mathcal{E} \partial_{x'} \mathcal{E}) \|_{L^2} + \| O(\partial_{x'} \mathcal{E}) \|_{L^2} \\
&\lesssim \delta^2 \| \mathcal{E} \|_{H^3}.
\end{split}
\end{equation*}
Therefore, we find
\begin{equation}
\label{eq:LowerBoundA2Complete}
A_2(t) + \| \rho_e(t) \|_{H^2}^2 \geq c \| (\mathcal{E},\mathcal{H}) \|_{H^3}^2 + O_\delta (\| (\mathcal{E},\mathcal{H}) \|_{H^3}^2).
\end{equation}
Next, we show the estimate
\begin{equation}
\label{eq:GronwallH3-Est}
\partial_t (A_1(t) + A_2(t) + \| \rho_e(t) \|_{H^2}^2) \lesssim B(t) (A_1(t) + A_2(t) + \| \rho_e(t) \|_{H^2}^2)
\end{equation}
for $B(t) = \| \partial_{x'} (\mathcal{E},\mathcal{H})(t) \|_{L^\infty_{x'}} + \| (\mathcal{E},\mathcal{H})(t) \|_{H^2}$ to apply Gr\o nwall's argument.

\medskip

We have already shown $\partial_t (A_1(t) + \| \rho_e(t) \|_{H^1}^2) \lesssim \| \partial_{x'} \mathcal{E}(t) \|_{L^\infty} (A_1(t) + \| \rho_e(t) \|_{H^1}^2)$.
We compute for the highest order derivatives:
\begin{equation*}
\begin{split}
&\quad \partial_t ((\varepsilon_1 \mathcal{E}^{(3)}, \mathcal{E}^{(3)} ) + (\mathcal{H}^{(3)}, \mathcal{H}^{(3)})) \\
&= 2 (\varepsilon_1 \mathcal{E}^{(4)}, \mathcal{E}^{(3)}) + 2 (\mathcal{H}^{(4)}, \mathcal{H}^{(3)}) + (\dot{\varepsilon}_1 \mathcal{E}^{(3)}, \mathcal{E}^{(3)}) \\
&= - 6 (\dot{\varepsilon}_1 \mathcal{E}^{(3)}, \mathcal{E}^{(3)} ) - 2 (\varepsilon_1^{(3)} \dot{\mathcal{E}}, \mathcal{E}^{(3)} ) - 6 (\ddot{\varepsilon}_1 \mathcal{E}^{(2)}, \mathcal{E}^{(3)} ) + (\dot{\varepsilon}_1 \mathcal{E}^{(3)}, \mathcal{E}^{(3)}).
\end{split}
\end{equation*}
For justifying
\begin{equation*}
(\nabla_\perp \mathcal{H}^{(3)}, \mathcal{E}^{(3)}) - ((\nabla \times \mathcal{E}^{(3)})_3,\mathcal{H}^{(3)}) = 0
\end{equation*}
for solutions $(\mathcal{E},\mathcal{H}) \in C([0,T],H^3))$ we refer to \cite{Schnaubelt2022}.

Hence, for \eqref{eq:GronwallH3-Est} we shall show that
\begin{equation}
\label{eq:ErrorTermsGronwallH3}
|(\dot{\varepsilon}_1 \mathcal{E}^{(3)}, \mathcal{E}^{(3)} )| + |(\varepsilon_1^{(3)} \dot{\mathcal{E}}, \mathcal{E}^{(3)} )| + |(\ddot{\varepsilon}_1 \mathcal{E}^{(2)}, \mathcal{E}^{(3)} )| \lesssim B(t) (A_1(t) + A_2(t) + \| \rho_e(t) \|_{H^2}^2 ).
\end{equation}
The first estimate for \eqref{eq:ErrorTermsGronwallH3} is trivial:
\begin{equation*}
|(\dot{\varepsilon}_1 \mathcal{E}^{(3)}, \mathcal{E}^{(3)} )| \lesssim_\delta \| \partial_{x'} (\mathcal{E},\mathcal{H}) \|_{L^\infty_{x'}} \| \mathcal{E}^{(3)} \|_{L^2}^2.
\end{equation*}
Regarding the second estimate for \eqref{eq:ErrorTermsGronwallH3}, we compute
\begin{equation*}
\dot{\varepsilon}_1 = O(\mathcal{E} \dot{\mathcal{E}}), \quad \ddot{\varepsilon}_1 = O(\mathcal{E} \ddot{\mathcal{E}}) + O( \dot{\mathcal{E}}^2), \quad \varepsilon_1^{(3)} = O(\mathcal{E} \mathcal{E}^{(3)}) + O(\dot{\mathcal{E}} \mathcal{E}^{(2)}).
\end{equation*}
We split
\begin{equation}
\label{eq:ErrorTermsGronwallH3-Aux}
|(\varepsilon_1^{(3)} \dot{\mathcal{E}}, \mathcal{E}^{(3)} )| \leq | (O(\mathcal{E} \mathcal{E}^{(3)}) \dot{\mathcal{E}}, \mathcal{E}^{(3)} )| + |(O(\dot{\mathcal{E}} \mathcal{E}^{(2)}) \dot{\mathcal{E}}, \mathcal{E}^{(3)} )|.
\end{equation}
We have for the first term in \eqref{eq:ErrorTermsGronwallH3-Aux}:
\begin{equation*}
|(O(\mathcal{E} \mathcal{E}^{(3)}) \dot{\mathcal{E}}, \mathcal{E}^{(3)} )| \lesssim \| \mathcal{E} \|_{L^\infty} \| \dot{\mathcal{E}} \|_{L^\infty} \| \mathcal{E}^{(3)} \|_{L^2}^2 \lesssim \delta \| \dot{\mathcal{E}} \|_{L^\infty} \| \mathcal{E}^{(3)} \|_{L^2}^2.
\end{equation*}
For the second term in \eqref{eq:ErrorTermsGronwallH3-Aux} we find
\begin{equation*}
\begin{split}
|(O(\dot{\mathcal{E}} \mathcal{E}^{(2)}) \dot{\mathcal{E}}, \mathcal{E}^{(3)} )| &\lesssim \| \dot{\mathcal{E}} \|_{L^\infty} \| \mathcal{E}^{(2)} \|_{L^4} \| \dot{\mathcal{E}} \|_{L^4} \| \mathcal{E}^{(3)} \|_{L^2} \\
&\lesssim \delta \| \dot{\mathcal{E}} \|_{L^\infty} \| \mathcal{E}^{(2)} \|_{L^4} \| \mathcal{E}^{(3)} \|_{L^2}.
\end{split}
\end{equation*}
Moreover, by $\mathcal{E}^{(2)} = - \varepsilon_1^{-1} \nabla_\perp (\nabla \times \mathcal{E})_3 + \varepsilon_1^{-1} \dot{\varepsilon}_1 \dot{\mathcal{E}}$, H\"older's inequality and Sobolev embedding, we find
\begin{equation*}
\begin{split}
\| \mathcal{E}^{(2)} \|_{L^4} &\lesssim \| \partial_{x'}^2 \mathcal{E} \|_{L^4} + \| \dot{\mathcal{E}} \|_{L^8}^2 \delta 
\lesssim (1+\delta^2) \| (\mathcal{E}, \mathcal{H}) \|_{H^3},
\end{split}
\end{equation*}
which gives
\begin{equation*}
\delta \| \dot{\mathcal{E}} \|_{L^\infty} \| \mathcal{E}^{(2)} \|_{L^4} \| \mathcal{E}^{(3)} \|_{L^2} \lesssim \delta \| \dot{\mathcal{E}} \|_{L^\infty} \| (\mathcal{E},\mathcal{H}) \|_{H^3} \| \mathcal{E}^{(3)} \|_{L^2}.
\end{equation*}
By \eqref{eq:LowerBoundA2Complete} this suffices.

\medskip

Lastly, we estimate $|(\ddot{\varepsilon}_1 \ddot{\mathcal{E}}, \mathcal{E}^{(3)} )|$ to complete the proof of \eqref{eq:ErrorTermsGronwallH3}. By the previously established estimate for $\| \mathcal{E}^{(2)} \|_{L^4}$, we find
\begin{equation*}
|(O( (\dot{\mathcal{E}})^2 \ddot{\mathcal{E}}, \mathcal{E}^{((3)})| \lesssim \| \mathcal{E}^{(3)} \|_{L^2} \| \dot{\mathcal{E}} \|_{L^\infty} \| \dot{\mathcal{E}} \|_{L^4} \| \ddot{\mathcal{E}} \|_{L^4} \lesssim \delta \| \mathcal{E}^{(3)} \|_{L^2} \| (\mathcal{E},\mathcal{H}) \|_{H^3} \| \dot{\mathcal{E}} \|_{L^\infty}.
\end{equation*}
Secondly, we find
\begin{equation}
\label{eq:AuxEstimateFinal}
|(\mathcal{E} \ddot{\mathcal{E}}^2, \mathcal{E}^{(3)} )| \lesssim \delta \| \ddot{\mathcal{E}} \|_{L^4}^2 \| \mathcal{E}^{(3)} \|_{L^2}.
\end{equation}
We have
\begin{equation*}
\| \ddot{\mathcal{E}} \|_{L^4} \lesssim \| \Delta \mathcal{E} \|_{L^4} + \| \dot{\mathcal{E}} \|_{L^8}^2 \lesssim \| \Delta \mathcal{E} \|_{L^4} + \| (\mathcal{E},\mathcal{H}) \|_{H^{7/4}} \| \dot{\mathcal{E}} \|_{L^\infty}.
\end{equation*}
By the Gagliardo-Nirenberg-Ladyzhenskaya inequality, we find
\begin{equation*}
\| \Delta \mathcal{E} \|^2_{L^4} \lesssim \| \mathcal{E} \|_{H^3} \| \mathcal{E} \|_{H^2}.
\end{equation*}
Plugging this into \eqref{eq:AuxEstimateFinal}, we obtain
\begin{equation*}
|(\mathcal{E} \ddot{\mathcal{E}}^2, \mathcal{E}^{(3)} )| \lesssim \delta \| (\mathcal{E},\mathcal{H}) \|_{H^3} \| (\mathcal{E}(t), \mathcal{H}(t)) \|_{H^2} \| \mathcal{E}^{(3)} \|_{L^2} + \delta \| \partial_t \mathcal{E} \|_{L^\infty} \| (\mathcal{E},\mathcal{H}) \|_{H^3} \| \mathcal{E}^{(3)} \|_{L^2}.
\end{equation*}
This completes the proof of \eqref{eq:GronwallH3-Est} and an application of Gr\o nwall's argument yields
\begin{equation*}
\| (\mathcal{E}, \mathcal{H})(t) \|_{H^3} \lesssim  e^{C \int_0^T B(s) ds} \| (\mathcal{E}, \mathcal{H})(0) \|_{H^3}.
\end{equation*}
We have by \eqref{eq:H2EnergyEstimateKerr} 
\begin{equation*}
B(s) \lesssim \| \partial_{x'} (\mathcal{E},\mathcal{H})(s) \|_{L^\infty_{x'}} + e^{\int_0^s \| \partial_{x'} (\mathcal{E},\mathcal{H})(s') \|_{L^\infty_{x'}} ds'} \| (\mathcal{E},\mathcal{H})(0) \|_{H^2}, 
\end{equation*}
which gives
\begin{equation*}
\begin{split}
\| (\mathcal{E}, \mathcal{H})(t) \|_{H^3} &\lesssim e^{C \int_0^T \| \partial_{x'} (\mathcal{E},\mathcal{H})(t) \|_{L^\infty_{x'}} dt + T e^{\int_0^T \| \partial_{x'} (\mathcal{E},\mathcal{H})(t') \|_{L_{x'}^\infty} dt'} \| (\mathcal{E}(0),\mathcal{H}(0)) \|_{H^2} } \\
&\quad \times  \| (\mathcal{E}, \mathcal{H})(0) \|_{H^3}.
\end{split}
\end{equation*}
Now \eqref{eq:H3EnergyEstimateKerr} is straight-forward.
\end{proof}

\subsection{The three-dimensional case}

Next, we extend the arguments to the three-dimensional case: Let $\varepsilon, \mu \in C^\infty(\Omega;\R_{>0})$, for which we suppose that \eqref{eq:Ellipticity} and \eqref{eq:ContinuityBoundary} hold. We consider the system of equations:
\begin{equation}
\label{eq:Maxwell3}
\left\{ \begin{array}{clrcl}
\partial_t (\varepsilon \mathcal{E}) &= \nabla \times \mathcal{H}, &\quad \nabla \cdot (\varepsilon \mathcal{E}) &=& \rho_e, \quad (t,x') \in \R \times \Omega; \\
\partial_t (\mu \mathcal{H}) &= - \nabla \times \mathcal{E}, &\quad \nabla \cdot (\mu \mathcal{H}) &=& 0.
\end{array} \right.
\end{equation}
We require the boundary conditions:
\begin{equation}
\label{eq:Maxwell3BoundaryConditions}
[\mathcal{E} \times \nu]_{x' \in \partial \Omega} = 0, \quad [\nu \cdot \mathcal{B}]_{x' \in \partial \Omega} = 0.
\end{equation}

Local existence of $\mathcal{H}^3$-solutions was discussed in \cite{Spitz2019,Spitz2022}. We show a priori estimates in the time-independent case:
\begin{proposition}
\label{prop:APrioriMaxwell3}
For $s \in [0,2]$ the following estimate holds for $\mathcal{H}^3$-solutions to \eqref{eq:Maxwell3} under the above assumptions:
\begin{equation}
\label{eq:APrioriMaxwell3}
\| (\mathcal{E},\mathcal{H}) \|_{L^\infty H^s(\Omega)} \lesssim \| (\mathcal{E},\mathcal{H})(0) \|_{H^s(\Omega)}.
\end{equation}
\end{proposition}
We refer for the suitable Helmholtz decomposition to Proposition \ref{prop:Helmholtz3d} in Appendix \ref{appendix:Helmholtz}: 
\begin{equation}
\label{eq:HelmholtzDecomposition3d}
\| \mathcal{E} \|_{H^1(\Omega)} \sim \| \mathcal{E} \|_{H_{curl}(\Omega)} + \| \mathcal{E} \|_{H_{div}(\Omega)} + \| \mathcal{E} \|_{L^2(\Omega)}.
\end{equation}

\begin{proof}[Proof~of~Proposition~\ref{prop:APrioriMaxwell3}]
We follow the argument from the two-dimensional case and begin with $L^2$-estimates. Let $M(t) = \int_\Omega \mathcal{D}.\mathcal{E} + \mathcal{H}.\mathcal{B} \, dx'$. We have
\begin{equation*}
\frac{dM}{dt} = 2 \int_\Omega \mathcal{E}.\nabla \times \mathcal{H} \, dx' - 2 \int_\Omega \mathcal{H}.\nabla \times \mathcal{E} \, dx' = 0
\end{equation*}
with the ultimate equality a consequence of the boundary conditions (after resolving \eqref{eq:Maxwell3} on $\R^3_{>0}$). This yields $\| (\mathcal{E},\mathcal{H})(t) \|_{L^2} \sim \|(\mathcal{E},\mathcal{H})(0) \|_{L^2}$,
which is \eqref{eq:APrioriMaxwell3} for $s=0$. 

\medskip

To prove \eqref{eq:APrioriMaxwell3} for $s=1$, we consider one time derivative to find that $(\dot{\mathcal{E}}, \dot{\mathcal{H}})$ satisfies \eqref{eq:Maxwell3}. Consequently, $\| (\dot{ \mathcal{E}}, \dot{\mathcal{H}})(t) \|_{L^2} \sim \| (\dot{\mathcal{E}}, \dot{\mathcal{H}})(0) \|_{L^2}$ which yields by \eqref{eq:Maxwell3} that $\| (\nabla \times \mathcal{E}, \nabla \times \mathcal{H})(t) \|_{L^2} \sim \| (\nabla \times \mathcal{E}, \nabla \times \mathcal{H})(0) \|_{L^2}$. By the Helmholtz decomposition, the defect to $H^1$ is given by $\| (\nabla \cdot \mathcal{E}, \nabla \cdot \mathcal{H})(t) \|_{L^2}$ modulo $L^2$.

\medskip

Here we use the divergence conditions:
\begin{equation*}
\nabla \cdot (\varepsilon \mathcal{E}) = (\nabla \varepsilon) . \mathcal{E} + \varepsilon (\nabla \cdot \mathcal{E}) = \rho_e, \quad \nabla \cdot (\mu \mathcal{H}) = (\nabla \mu) . \mathcal{H} + \mu (\nabla \cdot \mathcal{H}) = 0.
\end{equation*}
Since $\rho_e$ is a conserved quantity, and we have already an a priori estimate for the $L^2$-norm, we have
\begin{equation*}
\| \nabla \cdot \mathcal{E}(t) \|_{L^2} \lesssim \| \mathcal{E}(t) \|_{L^2} + \| \rho_e(t) \|_{L^2} \lesssim \| (\mathcal{E},\mathcal{H})(0) \|_{L^2} + \| \rho_e(0) \|_{L^2}.
\end{equation*}
For $\| \mathcal{H} \|_{H_{div}}$ we can argue likewise. We obtain
\begin{equation*}
\| (\mathcal{E}, \mathcal{H})(t) \|_{H^1} \lesssim \| (\mathcal{E},\mathcal{H})(0) \|_{H^1}.
\end{equation*}

\bigskip

For the proof of \eqref{eq:APrioriMaxwell3} with $s=2$, we use that $(\ddot{ \mathcal{E}}, \ddot{ \mathcal{H}})$ solves \eqref{eq:Maxwell3}. We have
\begin{equation}
\label{eq:SecondTimeEEnergyEstimate}
\begin{split}
\partial_t^2 \mathcal{E}(t) &= \frac{1}{\varepsilon} \nabla \times \partial_t \mathcal{H}(t) = - \frac{1}{\varepsilon} \nabla \times (\frac{1}{\mu} \nabla \times \mathcal{E}) \\
&= \frac{1}{\varepsilon \mu} \Delta \mathcal{E} + O(\| \mathcal{E} \|_{H^1}) - \frac{1}{\varepsilon \mu} \nabla (\nabla \cdot \mathcal{E}),
\end{split}
\end{equation}
and from the divergence condition, we find
\begin{equation*}
\nabla ( \nabla \cdot \mathcal{E})(t) = (\nabla \varepsilon^{-1}) \rho_e(t) + \varepsilon^{-1} \nabla \rho_e(t) + O (\| \mathcal{E}(t) \|_{H^1}).
\end{equation*}
This implies the estimate by conservation of charge and previously established a priori estimates
\begin{equation}
\label{eq:SpaceTimeE}
\begin{split}
\| \mathcal{E}(t) \|_{\dot{H}^2} &\lesssim \| \partial_t^2 \mathcal{E}(t) \|_{L^2} + \| \mathcal{E}(t) \|_{H^1} + \| \nabla (\nabla \cdot \mathcal{E})(t) \|_{L^2} \\
&\lesssim \| (\partial_t^2 \mathcal{E}, \partial^2_t \mathcal{H})(0) \|_{L^2} + \| (\mathcal{E},\mathcal{H})(0) \|_{H^1} + \| \rho_e(0) \|_{H^1}.
\end{split}
\end{equation}
 Similarly,
\begin{equation}
\label{eq:SecondTimeHEnergyEstimate}
\begin{split}
\partial_t^2 \mathcal{H}(t) &= - \frac{1}{\mu} \nabla \times \partial_t \mathcal{E}(t) = - \frac{1}{\mu} \nabla \times \big( \frac{1}{\varepsilon} \nabla \times \mathcal{H}) \\
&= \frac{1}{\varepsilon \mu} \Delta \mathcal{H} - \frac{1}{\varepsilon \mu} \nabla (\nabla \cdot \mathcal{H}) + O(\| \mathcal{H} \|_{H^1}).
\end{split}
\end{equation}
so that by previously established a priori estimates
\begin{equation}
\label{eq:SpaceTimeH}
\begin{split}
\| \mathcal{H}(t) \|_{\dot{H}^2} &\lesssim \| \partial_t^2 \mathcal{H}(t) \|_{L^2} + \| \mathcal{H}(t) \|_{H^1} \\
&\lesssim \| (\partial_t^2 \mathcal{E},\partial_t^2 \mathcal{H})(0) \|_{L^2} + \| (\mathcal{E},\mathcal{H})(0) \|_{H^1} + \| \nabla \cdot (\mu \mathcal{H})(0) \|_{H^1}.
\end{split}
\end{equation}
By \eqref{eq:SecondTimeEEnergyEstimate} and \eqref{eq:SecondTimeHEnergyEstimate}, we obtain
\begin{equation}
\label{eq:SecondTimeAPriori}
\| (\partial_t^2 \mathcal{E}, \partial_t^2 \mathcal{H})(0) \|_{L^2} \lesssim \| (\mathcal{E},\mathcal{H})(0) \|_{H^2}.
\end{equation}
By taking \eqref{eq:SpaceTimeE}, \eqref{eq:SpaceTimeH}, and \eqref{eq:SecondTimeAPriori} together with a priori estimates for $(\mathcal{E},\mathcal{H})$ in $L^2$, we conclude
\begin{equation*}
\begin{split}
\| (\mathcal{E},\mathcal{H})(t) \|_{H^2} &\lesssim \|(\partial_t^2 \mathcal{E},\partial_t^2 \mathcal{H})(t) \|_{L^2} + \| (\mathcal{E},\mathcal{H})(t) \|_{L^2} + O(\|(\mathcal{E},\mathcal{H})(t) \|_{H^1}) \\
 &\lesssim \|(\partial_t^2 \mathcal{E},\partial_t^2 \mathcal{H})(0) \|_{L^2} + \| (\mathcal{E},\mathcal{H})(0) \|_{L^2} + O(\|(\mathcal{E},\mathcal{H})(0) \|_{H^1}) \\
 &\lesssim \| (\mathcal{E},\mathcal{H})(0) \|_{H^2}.
\end{split}
\end{equation*}
The proof of \eqref{eq:APrioriMaxwell3} is complete for $s=2$. For non-integer $s$, we prove the claim by interpolation.
\end{proof}

\section{Preliminaries}
\label{section:Preliminaries}
In this section we collect facts on pseudo-differential operators, which we rely on in the remainder of the paper. We denote derivatives by
\begin{equation*}
\partial_x^\alpha = \partial_{x_1}^{\alpha_1} \partial_{x_2}^{\alpha_2} \ldots \partial_{x_m}^{\alpha_m} \text{ and } D_\xi^\alpha = \partial_\xi^\alpha / (i^{|\alpha|}) \text{ for } \alpha \in \N_0^m.
\end{equation*}
Recall the H\"ormander class of symbols:
\begin{equation*}
S^m_{\rho,\delta} = \{ a \in C^\infty(\R^m \times \R^m : |\partial_x^\alpha \partial_\xi^\beta a(x,\xi)| \lesssim_{\alpha,\beta} \langle \xi \rangle^{m - |\beta| \rho + |\alpha| \delta} \}
\end{equation*}
with $m \in \R$,  $0 \leq \delta < \rho \leq 1$. The $L^p$-boundedness of symbols $a \in S^0_{1,\delta}$, $0 \leq \delta < \rho \leq 1$, is well-known (cf. \cite[Theorem~0.11.A]{Taylor1991}). We use the following quantization:
\begin{equation*}
a(x,D) f = (2 \pi)^{-m} \int_{\R^m} e^{ix.\xi} a(x,\xi) \hat{f}(\xi) d\xi \qquad (f \in \mathcal{S}'(\R^m)).
\end{equation*}
We recall the composition of pseudo-differential operators. 
\begin{proposition}[{\cite[Prop.~0.3C]{Taylor1991}}]
Given $P(x,\xi) \in S^{m_1}_{\rho_1,\delta_1}$, $Q(x,\xi) \in S^{m_2}_{\rho_2,\delta_2}$, suppose that
\begin{equation*}
0 \leq \delta_2 < \rho \leq 1 \text{ with } \rho = \min(\rho_1,\rho_2).
\end{equation*}
Then, $(P \circ Q)(x,D) \in OPS^{m_1+m_2}_{\rho,\delta}$ with $\delta = \max(\delta_1,\delta_2)$, and $P(x,D) \circ Q(x,D)$ satisfies the asymptotic expansion
\begin{equation*}
(P \circ Q)(x,D) = \sum_\alpha \frac{1}{\alpha!} (D_\xi^\alpha P \partial_x^\alpha Q)(x,D) + R,
\end{equation*}
where $R: \mathcal{S}' \to C^\infty$ is a smoothing operator.
\end{proposition}

The following lemma will be useful:
\begin{lemma}[{\cite[Lemma~2.3]{SchippaSchnaubelt2022}}]
\label{lem:PseudoBoundsFourierSeries}
Let $1 \leq p,q \leq \infty$, $s \geq 0$, and $a \in C^s_x C^\infty_c(\R^m \times \R^m)$ with $a(x,\xi) = 0$ for $\xi \notin B(0,2)$. Suppose that
\begin{equation*}
\sup_{x \in \R^m} \sum_{0 \leq |\alpha| \leq m+1} \| D_\xi^\alpha a(x,\cdot) \|_{L^1_\xi} \leq C.
\end{equation*}
Then the following estimate holds:
\begin{equation*}
\| a(x,D) f \|_{L^p L^q} \lesssim C \| f \|_{L^p L^q}.
\end{equation*}
\end{lemma}
In the quasilinear case, when the coefficients merely satisfy $\partial \varepsilon \in L_t^2 L_{x'}^\infty$, we have the following:
\begin{lemma}
\label{lem:PseudoBoundsQuasilinear}
Let $X = L_t^2 L^\infty_{x'}$ and $a \in X C^\infty_c(\R^m \times \R^m)$ with $a(x,\xi) = 0$ for $\xi \notin B(0,2)$. Suppose that
\begin{equation*}
\sup_{x \in \R^m} \sum_{|\alpha| \leq 2m} \| D_\xi^\alpha a(x,\cdot) \|_{L^1_\xi} \leq C. 
\end{equation*}
Then the following estimate holds:
\begin{equation*}
\| a(x,D) f \|_{L^2_{t,x}} \lesssim_{\| a \|_{X C^\infty_c}} \| f \|_{L_t^\infty L_x^2}.
\end{equation*}
\end{lemma}
The lemma is proved via a Fourier series trick. In the following we also write $X^1= \{ f : \partial f \in X \}$.


\section{Diagonalizing reflected Maxwell equations}
\label{section:Diagonalization}

The purpose of this section is to reduce the proof of Proposition \ref{prop:ExtendedMaxwell} to Strichartz estimates for half-wave equations with metric $\frac{g^{ij}}{\varepsilon \mu}$. Here $\varepsilon,\mu, g^{ij} \in C^\infty(\R^3_{ \geq 0})$ are extended evenly to the full space, introducing a Lipschitz-singularity of co-dimension~$1$. The following is due to Blair--Smith--Sogge \cite{BlairSmithSogge2009}:
\begin{theorem}
\label{thm:LipschitzSingularHalfWave}
Let $d \geq 2$ and $(g^{ij})_{1 \leq i,j \leq d} \subseteq C^\infty(\R^d_{\geq 0})$ be uniformly elliptic. 
Let $u: [0,1] \times \R^d \to \C$. Then the following estimate holds:
\begin{equation*}
\| u \|_{L_t^p([0,1], L_{x'}^q(\R^d))} \lesssim \| u \|_{L_t^\infty H^\gamma(\R^3)} + \| (i \partial_t + D_{\tilde{g}}) u \|_{L_t^1 H^\gamma}
\end{equation*}
with $\tilde{g}^{ij}$ denoting the even extension of $g^{ij}$ and
\begin{equation*}
D_{\tilde{g}} = Op \big( \sum_{i,j=1}^d \tilde{g}^{ij} \xi_i \xi_j \big)^{\frac{1}{2}}
\end{equation*}
provided that $2 \leq p,q \leq \infty$ and $\gamma$ satisfy
\begin{equation*}
\frac{3}{p} + \frac{2}{q} \leq 1, \quad q < \infty, \quad \gamma = 3 \big( \frac{1}{2} - \frac{1}{q} \big) - \frac{1}{p} .
\end{equation*}
\end{theorem}
The reduction to the above proceeds via diagonalization with pseudo-differential operators. However, the symbols are very rough, so extra care is required. 
\subsection{Littlewood-Paley decomposition and frequency truncation}
We begin with a paradifferential decomposition. Recall that
\begin{equation*}
\mathcal{P} = 
\begin{pmatrix}
\sqrt{g} g^{-1} \varepsilon \partial_t & - \nabla \times \\
\nabla \times & \sqrt{g} g^{-1} \mu \partial_t 
\end{pmatrix}
.
\end{equation*}
In the following we denote $u=(\mathcal{E},\mathcal{H}): \R \times \R^3 \to \R^3 \times \R^3$ and omit the tilde for the reflected quantities to lighten the notation. Let $(S_\lambda)_{\lambda \in 2^{\N_0}}$ denote a family of inhomogeneous Littlewood-Paley projections for space-time frequencies and $(S'_\lambda)_{\lambda \in 2^{\N_0}}$, $(S^{\tau}_\lambda)_{\lambda \in 2^{\N_0}}$ projections for spatial or temporal frequencies, respectively. We define
\begin{equation}
\label{eq:FrequencyTruncationA}
\varepsilon' = \sqrt{g} g^{-1} \varepsilon, \quad \mu' = \sqrt{g} g^{-1} \mu, \quad \mathcal{P}_{<\lambda} = 
\begin{pmatrix}
\varepsilon'_{<\lambda} \partial_t & - \nabla \times \\
\nabla \times & \mu'_{<\lambda} \partial_t
\end{pmatrix}
\end{equation}
through spatial frequency truncation: $\kappa_{<\lambda} = \sum_{\mu \leq \lambda/16} S'_\mu \kappa$ for $\kappa \in \{\varepsilon',\mu' \}$.

\medskip

For the proof of Proposition \ref{prop:ExtendedMaxwell} it suffices to prove the following estimate for frequency localized functions for $2^{\N_0} \ni \lambda \gg 1$: We can suppose that $\lambda \gg 1$ because low frequencies are easily estimated by Bernstein's inequality. Let $0 < \delta < \frac{3}{q}$.
\begin{align}
\label{eq:FrequencyLocalizedEstimateI}
\| S_{ \{|\tau| \sim |\xi'| \}} u \|_{L_t^p L_{x'}^q} &\lesssim \| \langle \partial_t \rangle^{\gamma + \delta} u \|_{L_t^\infty L_{x'}^2} + \| \langle \partial_t \rangle^{\gamma} \mathcal{P} u \|_{L^2_{x}}, \\
\label{eq:FrequencyLocalizedEstimateII}
\| S_{\{ |\tau| \gg |\xi'| \}} u \|_{L_t^p L_{x'}^q} &\lesssim \| \langle \partial_t \rangle^{\gamma - \frac{1}{2}+\delta} u \|_{L^2_{t,{x'}}} + \| \langle \partial_t \rangle^{\gamma - \frac{1}{2} + \delta} \mathcal{P} u \|_{L^2_{t,{x'}}}, \\
\label{eq:FrequencyLocalizedEstimateIII}
\| S_{ \{ |\tau| \ll |\xi'| \}} u \|_{L_t^p L_{x'}^q} &\lesssim \| \langle D' \rangle^{\gamma - \frac{1}{2}+\delta} u \|_{L^2_{x}} + \| \langle \partial_t \rangle^{\gamma - \frac{1}{2}+\delta} u \|_{L^2_{x}} \\
&\quad + \| \langle D' \rangle^{\gamma - \frac{1}{2}+\delta} \mathcal{P} u \|_{L^2_{x}} + \| \rho_e \|_{L_t^\infty H^{\gamma - 1 + \frac{1}{p}+\delta}}. \nonumber
\end{align}
In the following we implicitly consider $u$ compactly supported in $[0,T]$. This is strictly speaking not conserved by $S_\lambda$, but for $\lambda \gg 1$ up to Schwartz tails, which are neglected in the following. $S_{ \{ |\tau| \sim |\xi'| \}}$ denotes a space-time frequency projection to temporal frequencies comparable to spatial frequencies, $S_{ \{|\tau| \gg |\xi'| \} }$ a space-time frequency projection to temporal frequencies $\{|\tau| \gtrsim 1\}$ and spatial frequencies $\{|\xi'| \ll |\tau| \}$. Correspondingly, $S_{\{|\xi'| \gg |\tau| \}}$ denotes a projection for spatial frequencies dominating temporal frequencies. Estimates \eqref{eq:FrequencyLocalizedEstimateII} and \eqref{eq:FrequencyLocalizedEstimateIII} crucially rely on ellipticity of components of $\mathcal{P}$ after diagonalization. Since we can achieve estimates with regularity $\gamma - \frac{1}{2} < 1$, the commutator estimates for Lipschitz functions are applicable. We give the proof of \eqref{eq:FrequencyLocalizedEstimateII} shortly using the ellipticity away from the characteristic surface. The proof of \eqref{eq:FrequencyLocalizedEstimateI} is more involved and requires the use of the Strichartz estimates due to Blair--Smith--Sogge. However, if $\{|\tau| \sim |\xi'| \sim 1\}$, we can trade temporal for spatial frequencies.

\begin{lemma}
Let $2^{\N_0} \ni \lambda \gg 1$, $2 \leq p,q < \infty$, and $\delta >0$. The estimate
\begin{equation}
\label{eq:SpaceTimeDyadicEstimate}
\| S_\lambda^\tau S'_\lambda u \|_{L_t^p L_{x'}^q} \lesssim \lambda^\gamma ( \| S_\lambda^\tau S'_\lambda u \|_{L_t^\infty L_{x'}^2} + \| \mathcal{P}_{<\lambda} S_\lambda^\tau S'_\lambda u \|_{L^2_{x}} )
\end{equation}
implies
\begin{equation}
\label{eq:SpaceTimeEstimate}
\| S_{\{|\tau| \sim |\xi'| \sim 1\}} u \|_{L_t^p L_{x'}^q} \lesssim \| \langle \partial_t \rangle^{\gamma + \delta} u \|_{L_t^\infty L_{x'}^2} + \| \langle \partial_t \rangle^\gamma \mathcal{P} u \|_{L_{x}^2}.
\end{equation}
\end{lemma}
\begin{proof}
Littlewood-Paley decomposition and Minkowski's inequality give for $2 \leq p,q < \infty$
\begin{equation*}
\| u \|_{L^p L^q} \lesssim \big( \sum_{\lambda \geq 1} \| S_\lambda u \|^2_{L_t^p L_{x'}^q} \big)^{\frac{1}{2}}
\end{equation*}
which we can further decompose almost orthogonally into spatial and temporal frequencies. Summation of $\| \langle \partial_t \rangle^\gamma u \|_{L_t^\infty L_{x'}^2}$ is clear. Note that the lack of almost orthogonality in $L_t^\infty L_{x'}^2$ leads to the $\delta$-loss in derivatives. Now we write
\begin{equation*}
\mathcal{P}_{<\lambda} S'_\lambda v = S'_\lambda \mathcal{P}_{<\lambda} v+ [\mathcal{P}_{<\lambda}, S'_\lambda] v
\end{equation*}
and note that
\begin{equation*}
\| S_\lambda^\tau [ \mathcal{P}_{<\lambda}, S'_\lambda] \langle \partial_t \rangle^\gamma u \|_{L^2_{x}} \lesssim \| S_\lambda^\tau \langle \partial_t \rangle^\gamma u \|_{L^2_{x}}
\end{equation*}
because $\| [\kappa_{<\lambda}, S'_\lambda] \|_{L^2_{x'} \to L^2_{x'}} \lesssim \lambda^{-1}$ by a kernel estimate for $\kappa \in \{ \varepsilon', \mu' \}$.

We write
\begin{equation*}
S_\lambda^\tau S'_\lambda \mathcal{P}_{<\lambda} v = S_\lambda^\tau S'_\lambda \mathcal{P} v - S_\lambda^\tau S'_\lambda \mathcal{P}_{\gg \lambda} v - S_\lambda^\tau S'_\lambda \mathcal{P}_{\sim \lambda} v.
\end{equation*}
Clearly,
\begin{equation*}
\| S_\lambda^\tau S'_\lambda \mathcal{P}_{\sim \lambda} v \|_{L^2_{x}} \lesssim \| S_\lambda^\tau v \|_{L^2_{x}}
\end{equation*}
and similarly, by a fixed-time estimate,
\begin{equation*}
\| S_\lambda^\tau S'_\lambda (S'_{\gtrsim \lambda} \varepsilon \partial_t S'_{\gtrsim \lambda} v) \|_{L^2_{x}} \lesssim \lambda \| \varepsilon_{\gtrsim \lambda} \|_{L^\infty} \| S_\lambda^\tau v \|_{L^2_{x}} \lesssim \| \partial \varepsilon \|_{L^\infty} \| S^\tau_\lambda v \|_{L^2_{x}},
\end{equation*}
which estimates the second term. We remain with $S_\lambda^\tau S'_\lambda \mathcal{P} \langle \partial_t \rangle^\gamma u$ and conclude
\begin{equation*}
\| S_\lambda^\tau \mathcal{P}_{<\lambda} S'_\lambda \langle \partial_t \rangle^\gamma u \|_{L^2_{x}} \lesssim \| S_\lambda^\tau S'_\lambda \mathcal{P} \langle \partial_t \rangle^\gamma u \|_{L^2_{x}} + \| S_\lambda^\tau \langle \partial_t \rangle^\gamma u \|_{L^2_{x}}.
\end{equation*}
This is the commutator argument for the Maxwell operator. After summing the Littlewood-Paley blocks, we obtain \eqref{eq:SpaceTimeEstimate}.
\end{proof}


\begin{lemma}
Let $\lambda,\nu \in 2^{\N_0}$, $\lambda \ll \nu$. Let $2 \leq p,q < \infty$. The estimate
\begin{equation}
\label{eq:DyadicFrequencyLocalizedEstimateII}
\| S_\nu' S_\lambda^\tau u \|_{L_t^p L_{x'}^q} \lesssim \nu^{\gamma - \frac{1}{2}} \| S'_\nu S_\lambda^\tau \mathcal{P}_{< \nu} u \|_{L^2_{x}} + \nu^{\gamma - 1 +\frac{1}{p}} (\| \rho_{e \nu}' \|_{L_t^\infty L_{x'}^2} + \| \rho_{m \nu}' \|_{L_t^\infty L_{x'}^2} )
\end{equation}
with
\begin{equation*}
\rho'_{e \nu} = \nabla \cdot ( \varepsilon'_{<\nu} S'_\nu \mathcal{E}), \qquad \rho'_{m \nu} = \nabla \cdot( \mu'_{< \nu} S'_\nu \mathcal{H})
\end{equation*}
implies
\begin{equation}
\label{eq:FrequencyLocalizedEstimateIIConclusion}
\begin{split}
\| S_{\{|\tau| \ll |\xi'| \}} u \|_{L_t^p L_{x'}^q} &\lesssim \| \langle D' \rangle^{\gamma - \frac{1}{2}+\delta} u \|_{L^2_{x}} +
\| \langle \partial_t \rangle^{\gamma - \frac{1}{2} + \delta} u \|_{L^2_{x}} + \| \langle D' \rangle^{\gamma - \frac{1}{2}+\delta} \mathcal{P} u \|_{L^2_{x}} \\
&\quad + \| \rho_e \|_{L_t^\infty H^{\gamma - 1 + \frac{1}{p} + \delta}} + \| \rho_m \|_{L_t^\infty H^{\gamma - 1 + \frac{1}{p} + \delta}}
\end{split}
\end{equation}
for $\delta > 0$.
\end{lemma}
\begin{proof}
We have to carry out the summation
\begin{equation*}
\sum_{ \substack{ \nu \gg 1, \\ 1 \lesssim \lambda \ll \nu }} \nu^{\gamma - \frac{1}{2}} \| S_\lambda^\tau \mathcal{P}_{<\nu} S'_\nu u \|_{L^2_{x}} + \nu^{\gamma - 1 +\frac{1}{p}} ( \| \rho'_{e \nu} \|_{L_t^\infty L_{x'}^2} + \| \rho'_{m \nu} \|_{L_t^\infty L_{x'}^2}).
\end{equation*}
For the Maxwell operator, we use that $\gamma - \frac{1}{2} < 1$. First, we note that
\begin{equation*}
S_\lambda^\tau \mathcal{P}_{<\nu} S'_\nu u = \tilde{S}'_\nu S^\tau_\lambda \mathcal{P}_{<\nu} S'_\nu u.
\end{equation*}
Above and in the following $\tilde{S}'_\nu$ denotes a mildly enlarged spatial frequency projector. By $\mathcal{P}= \mathcal{P}_{<\mu} + \mathcal{P}_{\sim \mu} + \mathcal{P}_{\gg \mu}$ and $\tilde{S}'_\nu \mathcal{P}_{\gg \nu } S'_\nu = 0$, we can write
\begin{equation*}
\| S_\lambda^\tau \mathcal{P}_{<\nu} S'_\nu u \|_{L^2_{x}} \leq \| S_\lambda^\tau \tilde{S}'_\nu \mathcal{P} S'_\nu u \|_{L^2_{x}} + \| S_\lambda^\tau \tilde{S}'_\nu \mathcal{P}_{\sim \nu} S'_\nu u \|_{L^2_{x}}.
\end{equation*}
The latter term is clearly estimated by
\begin{equation*}
\| S_\lambda^\tau \tilde{S}'_\nu \mathcal{P}_{\sim \nu} S'_\nu u \|_{L^2_{x}} \lesssim \| S'_\nu u \|_{L^2_{x}}.
\end{equation*}
For the first term, we write
\begin{equation}
\label{eq:Com-Est}
\begin{split}
\nu^{\gamma - \frac{1}{2}} \| S_\lambda^\tau \tilde{S}'_\nu \mathcal{P} S'_\nu u \|_{L^2_{x}} &\lesssim \nu^{\gamma - \frac{1}{2}} \lambda \| S_\lambda^\tau \tilde{S}'_\nu [\varepsilon',S'_\nu] u \|_{L^2_{x}} + \nu^{\gamma - \frac{1}{2}} \lambda \|  S_\lambda^\tau \tilde{S}'_\nu [\mu',S'_\nu] u \|_{L^2_{x}} \\
&\quad + \| \langle D' \rangle^{\gamma - \frac{1}{2}} S_\lambda^\tau \tilde{S}'_\nu \mathcal{P}u \|_{L^2_{x}}.
\end{split}
\end{equation}
Furthermore,
\begin{equation}
\label{eq:ComEstimate-I}
\begin{split}
\| S_\lambda^\tau \tilde{S}'_\nu [\varepsilon',S'_\nu] u \|_{L^2_{x}} &= \| S_\lambda^\tau \tilde{S}'_\nu [\varepsilon',S'_\nu] \tilde{S}'_\nu u \|_{L^2_{x}} + \| S_\lambda^\tau \tilde{S}'_\nu [\varepsilon',S'_\nu] S'_{\ll \nu} u \|_{L^2_{x}} \\
&\quad + \| S_\lambda^\tau \tilde{S}'_\nu [\varepsilon',S'_\nu] S'_{\gg \nu} u \|_{L^2_{x}}.
\end{split}
\end{equation}
The estimate of the first term in \eqref{eq:ComEstimate-I} is straight-forward by the fixed-time commutator estimate $\| [\varepsilon',S'_\nu] \|_{L^2_{x} \to L^2_x} \lesssim \mu^{-1}$:
\begin{equation*}
\sum_{ \substack{ \nu \gg 1, \\ 1 \lesssim \lambda \ll \nu }} \nu^{\gamma - \frac{1}{2}} \lambda \| S_\lambda^\tau \tilde{S}'_\nu [\varepsilon',S'_\nu] \tilde{S}'_\nu u \|_{L^2_{x}} \lesssim \sum_{ \substack{ \nu \gg 1, \\ 1 \lesssim \lambda \ll \nu }} \nu^{\gamma - \frac{1}{2}} \lambda \nu^{-1} \| \tilde{S}'_\nu u \|_{L^2_{x}} \lesssim \| \langle D' \rangle^{\gamma - \frac{1}{2}+\delta} u \|_{L^2_{x}}.
\end{equation*}
For the second term in \eqref{eq:ComEstimate-I} we note that
\begin{equation*}
\begin{split}
\sum_{ \substack{ \nu \gg 1, \\ 1 \lesssim \lambda \ll \nu }} \nu^{\gamma - \frac{1}{2}} \lambda \| S_\lambda^\tau \tilde{S}'_\nu [\varepsilon',S'_\nu] S'_{\ll \nu} u \|_{L^2_{x}} &\lesssim \sum_{ \substack{ \nu \gg 1, \\ 1 \lesssim \lambda \ll \nu }}  \nu^{\gamma - \frac{1}{2}} \lambda \| \varepsilon'_{\sim \nu} S^\tau_\lambda S'_{\ll \nu} u \|_{L^2_{x}} \\
 &\lesssim \sum_{ \substack{ \nu \gg 1, \\ 1 \lesssim \lambda \ll \nu }} \nu^{\gamma - \frac{3}{2} } \lambda \| \partial \varepsilon \|_{L^\infty} \| S^\tau_\lambda u \|_{L^2_{x}} \\ &\lesssim \| \partial \varepsilon \|_{L^\infty} \| \langle \partial_t \rangle^{\gamma - \frac{1}{2}+\delta} u \|_{L^2_{x}}.
 \end{split}
\end{equation*}
For the third term in \eqref{eq:ComEstimate-I} we obtain similarly
\begin{equation*}
\begin{split}
\sum_{ \substack{ \nu \gg 1, \\ 1 \lesssim \lambda \ll \nu }} \nu^{\gamma - \frac{1}{2}} \lambda \| S_\lambda^\tau \tilde{S}'_\nu [\varepsilon',S'_\nu] S'_{\gg \nu} u \|_{L^2_{x}} &\lesssim \sum_{ \substack{ \nu \gg 1, \\ 1 \lesssim \lambda \ll \nu }}
\nu^{\gamma - \frac{1}{2}} \lambda \| S_\lambda^\tau \varepsilon'_{\gg \nu} S'_{\gg \nu} u \|_{L^2_{x}} \\
&\lesssim \| \partial \varepsilon \|_{L^\infty} \| \langle \partial_t \rangle^{\gamma - \frac{1}{2}+\delta} u \|_{L^2_{x}}.
\end{split}
\end{equation*}
Clearly, the second commutator in \eqref{eq:Com-Est} can be handled likewise.

\medskip

We turn to the charges: Recall that $\rho_e = \nabla \cdot (\varepsilon' \mathcal{E})$ with $\varepsilon' = \sqrt{g} g^{-1} \varepsilon$. Since we are working in geodesic normal coordinates, we have
\begin{equation*}
\varepsilon' =
\begin{pmatrix}
\varepsilon'_{11} & \varepsilon'_{12} & 0 \\
\varepsilon'_{21} & \varepsilon'_{22} & 0 \\
0 & 0 & \varepsilon_{33}'
\end{pmatrix}
.
\end{equation*}
To carry out the commutator argument, we separate
\begin{equation*}
\begin{split}
\rho'_{e \nu} &= \partial_1 ({\varepsilon'}^{11}_{< \nu} S'_\nu \mathcal{E}_1 + {\varepsilon'}^{21}_{<\nu} S'_\nu \mathcal{E}_2) + \partial_2 ({\varepsilon'}^{21}_{<\nu} S'_\nu \mathcal{E}_1 + {\varepsilon'}^{22}_{<\nu} S'_\nu \mathcal{E}_2) + \partial_3 ({\varepsilon'}^{33}_{<\nu} S'_\nu \mathcal{E}_3) \\
&= (\partial_1 {\varepsilon'}^{11}_{< \nu}) S'_\nu \mathcal{E}_1 + (\partial_1 {\varepsilon'}^{12}_{<\nu}) S'_\nu \mathcal{E}_2 + (\partial_2 {\varepsilon'}^{21}_{<\nu}) S'_\nu \mathcal{E}_1 + (\partial_2 {\varepsilon'}^{22}_{<\nu}) S'_\nu \mathcal{E}_2 + (\partial_3 {\varepsilon'}^{33}_{\nu}) \mathcal{E}_3 \\
&\quad + {\varepsilon'}^{11}_{<\nu} \partial_1 S'_\nu \mathcal{E}_1 + \varepsilon_{<\nu}'^{12} \partial_1 S'_\nu \mathcal{E}_2 + {\varepsilon'}^{21}_{<\nu} \partial_2 S'_\nu \mathcal{E}_1 + {\varepsilon'}^{22}_{<\nu} \partial_2 S'_\nu \mathcal{E}_2 + {\varepsilon'}^{33}_{<\nu} \partial_3 S'_\nu \mathcal{E}_3 \\
&=: {\rho'}_{e \nu}^{(1)} + {\rho'}_{e \nu}^{(2)}.
\end{split}
\end{equation*}
We can estimate terms with derivative acting on $\varepsilon'$ collected in ${\rho'}_{e \nu}^{(1)}$ directly by Lipschitz continuity. For example,
\begin{equation*}
\nu^{\gamma - 1 + \frac{1}{p}} \| \partial_1 \varepsilon'_{<\nu} S'_\nu \mathcal{E}_1 \|_{L_t^\infty L_{x'}^2} \lesssim \nu^{\gamma - \frac{1}{2} - \delta} \| S'_\nu \mathcal{E} \|_{L_t^\infty L_{x'}^2}.
\end{equation*}
The terms with derivative acting on $\mathcal{E}$ collected in ${\rho'}_{e \nu}^{(2)}$ are amenable to a commutator argument. Note that
\begin{equation*}
\nu^{ \gamma - 1 + \frac{1}{p}} \| \varepsilon'^{11}_{< \nu} \partial_1 S'_\nu \mathcal{E}_1 \|_{L^2_x} = \nu^{\gamma - 1 + \frac{1}{p}} \| \tilde{S}'_\nu \varepsilon'^{11}_{<\nu} \partial_1 S'_\nu \mathcal{E}_1 \|_{L^2_x}.
\end{equation*}
Since $\tilde{S}'_\nu \varepsilon'^{11}_{\gg \nu} S'_\nu = 0$, we can write
\begin{equation*}
\nu^{\gamma - 1 + \frac{1}{p}} \| \tilde{S}'_\nu \varepsilon'^{11}_{<\nu} \partial_1 S'_\nu \mathcal{E}_1 \|_{L^2_x} \leq \nu^{\gamma - 1 +\frac{1}{p}} \| \tilde{S}'_\nu \varepsilon'^{11}_{\sim \nu} \partial_1 S'_\nu \mathcal{E}_1 \|_{L^2_x} + \nu^{\gamma - 1 +\frac{1}{p}} \| \tilde{S}'_\nu \varepsilon'^{11} \partial_1 S'_\nu \mathcal{E}_1 \|_{L^2_x}.
\end{equation*}
The first expression is estimaetd by
\begin{equation*}
\nu^{\gamma - 1 + \frac{1}{p}} \| \tilde{S}'_\nu \varepsilon'^{11}_{\sim \nu} \partial_1 S'_\nu \mathcal{E}_1 \|_{L^2_x} \lesssim \| \partial \varepsilon'^{11} \|_{L_x^\infty} \nu^{\gamma - 1 +\frac{1}{p}} \| S'_\nu \mathcal{E}_1 \|_{L^2_x},
\end{equation*}
which is more than enough. For $\gamma - 1 + \frac{1}{p} > 0$, we obtain by the Coifman--Meyer estimate for the second term:
\begin{equation*}
\begin{split}
\sum_{\nu \geq 1} \| \langle D' \rangle^{\gamma - 1 + \frac{1}{p}} \tilde{S}_{\nu}' \varepsilon'^{11} \partial_1 S'_\nu \mathcal{E}_1 \|_{L^2_x} &\leq \sum_{\nu \geq 1} (\nu^{- \delta} \| \langle D' \rangle^{\gamma - 1 + \frac{1}{p} + \delta} \tilde{S}'_\nu [ \varepsilon'^{11}, S'_\nu] \partial_1 \mathcal{E}_1 \|_{L^2_x} \\
&\quad + \| \langle D' \rangle^{\gamma - 1 +\frac{1}{p}+\delta} (\varepsilon'^{11} \partial_1 \mathcal{E}_1) \|_{L^2_x} ) \\
&\lesssim \| \langle D' \rangle^{\gamma - 1 + \frac{1}{p} + \delta} \mathcal{E} \|_{L^2_x} + \| \langle D' \rangle^{\gamma - 1 + \frac{1}{p} + \delta} \varepsilon'^{11} \partial_1 \mathcal{E}_1 \|_{L^2_x}.
\end{split}
\end{equation*}
Let
\begin{equation*}
\rho_e'^{(2)} = \varepsilon'^{11} \partial_1 \mathcal{E}_1 + \varepsilon'^{12} \partial_1 \mathcal{E}_2 + \varepsilon'^{21} \partial_2 \mathcal{E}_1 + \varepsilon'^{22} \partial_2 \mathcal{E}_2 + \varepsilon'^{33} \partial_3 \mathcal{E}_3.
\end{equation*}
We obtain by the previous arguments:
\begin{equation*}
\sum_\nu \nu^{\gamma - 1 +\frac{1}{p}} \| \rho'_{e \nu} \|_{L^\infty_t L_{x}^2} \lesssim \| \langle D' \rangle^{\gamma - 1 + \frac{1}{p} + \delta} \mathcal{E} \|_{L_t^\infty L_x^2} + \| \langle D' \rangle^{\gamma - 1 + \frac{1}{p} + \delta} \rho_e'^{(2)} \|_{L_t^\infty L_x^2}.
\end{equation*}
The first term is acceptable. We estimate the second term by oddness of the function $\rho_e'^{(2)}$ switching to the half-space:
\begin{equation*}
\begin{split}
&\quad \| \langle D' \rangle^{\gamma - 1 + \frac{1}{p} + \delta} \rho_e'^{(2)} \|_{L_t^\infty L_x^2(\R^3)} \lesssim \| \langle D' \rangle^{\gamma - 1 + \frac{1}{p} + \delta} \rho_e'^{(2)} \|_{L_t^\infty L_x^2(\R^3_+)} \\
&\lesssim \| \langle D' \rangle^{\gamma - 1 + \frac{1}{p} + \delta} \rho_e' \|_{L_t^\infty L_x^2(\R^3_+)} + \| \langle D' \rangle^{\gamma - 1 + \frac{1}{p} + \delta} \mathcal{E} \|_{L_t^\infty L_x^2(\R^3_+)}.
\end{split}
\end{equation*}
For the ultimate estimate we used smoothness of the coefficients and invariance of Sobolev functions under multiplication with smooth functions. We remark that the estimate is easier for $\gamma - 1 + \frac{1}{p} < 0$ because it is not necessary to switch between half-space and full space. The estimate for $\rho'_{ m \nu}$ follows along the above lines. After summation of the Littlewood-Paley blocks, we obtain \eqref{eq:FrequencyLocalizedEstimateIIConclusion}.
\end{proof}

We turn to the proof of \eqref{eq:FrequencyLocalizedEstimateII}, which does not make use of the diagonalization of~$\mathcal{P}$.

\begin{proof}[Proof of \eqref{eq:FrequencyLocalizedEstimateII}]
Let $1 \ll \mu \ll \lambda$ and
\begin{equation*}
\tilde{\mathcal{P}} = 
\begin{pmatrix}
\partial_t & \varepsilon'^{-1} \nabla \times \\
- \mu'^{-1} \nabla \times & \partial_t
\end{pmatrix}
.
\end{equation*}
If $\{ \lambda \sim |\tau| \gg |\xi'| \sim \mu \}$, the operator $\tilde{P}_{< \mu}$ (obtained from frequency truncation of $\varepsilon'^{-1}$ and $\mu'^{-1}$) is elliptic and gains one derivative. We estimate by Bernstein's inequality and ellipticity of $\tilde{\mathcal{P}}_{< \mu}$ (note that $\tilde{\mathcal{P}}_{<\mu}$ has Lipschitz coefficients):
\begin{equation*}
\begin{split}
\| S_\lambda^\tau S'_\mu u \|_{L^p L^q} &\lesssim \lambda^{\frac{1}{2}-\frac{1}{p}} \mu^{3 \big( \frac{1}{2} - \frac{1}{q} \big)} \| S_\lambda^\tau S'_\mu u \|_{L^2_{t,x}} \\
&\lesssim \lambda^{-\frac{1}{2}-\frac{1}{p}} \mu^{3 \big( \frac{1}{2} - \frac{1}{q} \big)} \| \tilde{\mathcal{P}}_{< \mu} S_\lambda^\tau S'_\mu u \|_{L^2_{t,x}} \\
&\lesssim \lambda^{-\frac{1}{2}-\frac{1}{p}} \mu^{\frac{1}{p}+\frac{1}{2}} \| S_\lambda^\tau \langle D' \rangle^{\gamma - \frac{1}{2}} \tilde{S}'_\mu \tilde{\mathcal{P}}_{<\mu} S'_\mu u \|_{L^2_{t,x}}.
\end{split}
\end{equation*}
Above and in the following $\tilde{S}'_\mu$ denotes a mildly enlarged frequency projection around frequencies of size $\mu$. Now we write again $\tilde{\mathcal{P}} = \tilde{\mathcal{P}}_{<\mu} + \tilde{\mathcal{P}}_{\sim \mu} + \tilde{\mathcal{P}}_{\gg \mu}$ and note that
\begin{equation*}
\| S_\lambda^\tau \langle D' \rangle^{\gamma - \frac{1}{2}} \tilde{S}'_\mu \tilde{\mathcal{P}}_{\sim \mu} S'_\mu u \|_{L^2_{t,x}} \lesssim \mu^{\gamma - \frac{1}{2}} \| S'_\mu u \|_{L^2_{t,x}} \lesssim \| S_\lambda^\tau \langle D' \rangle^{\gamma - \frac{1}{2}} S'_\mu u \|_{L^2_{t,x}}.
\end{equation*}
Like above, $\tilde{S}'_\mu \tilde{\mathcal{P}}_{ \gg \mu} S'_\mu = 0$ by impossible frequency interaction. Summation over $\mu$ and $\lambda$ gives the acceptable contribution
\begin{equation*}
\lesssim \| \langle \partial_t \rangle^{\gamma - \frac{1}{2} + \delta} u \|_{L^2_{t,x}}.
\end{equation*}

For $\tilde{\mathcal{P}}$ we use the estimate
\begin{equation*}
\| [\kappa', S'_\mu ] \|_{L^2_{x'} \to L^2_{x'}} \lesssim \mu^{-1}.
\end{equation*}
We have
\begin{equation*}
\begin{split} 
&\quad \mu^{\gamma - \frac{1}{2}} \| S_\lambda^\tau \tilde{S}'_\mu [\kappa', S'_\mu] \nabla \times A \|_{L^2_{t,x}} \\
 &\lesssim \mu^{\gamma - \frac{1}{2}} \| S_\lambda^\tau \tilde{S}'_\mu [\kappa', S'_\mu] S'_{\lesssim \mu} \nabla \times A \|_{L^2_{t,x}} + \mu^{\gamma - \frac{1}{2}} \| S_\lambda^\tau \tilde{S}'_\mu [\kappa', S'_\mu] S'_{\gg \mu} \nabla \times A \|_{L^2_{t,x}} \\
 &\lesssim \mu^{\gamma - \frac{1}{2}} \| S_\lambda^\tau A \|_{L^2_{t,x}} + \mu^{\gamma - \frac{1}{2}} \| S_\lambda^\tau \tilde{S}'_\mu (\kappa'_{\gg \mu} S'_{\gg \mu} \nabla \times A) \|_{L^2_{t,x}}.
 \end{split}
\end{equation*}
The first term is already acceptable. The second term is rewritten as
\begin{equation*}
\tilde{S}'_\mu (\kappa'_{\gg \mu} S'_{\gg \mu} \partial S_\lambda^\tau A) = \tilde{S}'_\mu \partial ( \kappa'_{\gg \mu} S'_{\gg \mu} S_\lambda^\tau A) - \tilde{S}'_\mu (\partial \kappa'_{\gg \mu} S'_{\gg \mu} S_\lambda^\tau A).
\end{equation*}
For the first term we find
\begin{equation*}
\| \tilde{S}'_\mu \partial ( \kappa'_{\gg \mu} S'_{\gg \mu} S_\lambda^\tau A) \|_{L^2_{t,x}} \lesssim \mu \| \kappa'_{\gg \mu} \|_{L^\infty_{x'}} \| S'_{\gg \mu} S_\lambda^\tau A \|_{L^2_{t,x'}} \lesssim \| \partial \kappa' \|_{L^\infty_{x'}} \|S_\lambda^\tau A \|_{L^2_{t,x}}.
\end{equation*}
This yields an acceptable contribution after summation over $\mu \ll \lambda$ and $\lambda$. Clearly,
\begin{equation*}
\| \tilde{S}'_\mu (\partial \kappa'_{\gg \mu} S'_{\gg \mu} S_\lambda^\tau A \|_{L^2_{t,x}} \lesssim \| \partial \kappa' \|_{L^\infty} \| S_\lambda^\tau A \|_{L^2_{t,x}}.
\end{equation*}
This is likewise acceptable.

We summarize
\begin{equation}
\label{eq:FrequencyLocalizedEllipticEstimate}
\| S_{\{ |\tau| \gg |\xi'| \gtrsim 1 \} } u \|_{L^p L^q} \lesssim \| \langle \partial_t \rangle^{\gamma - \frac{1}{2} + \delta} u \|_{L^2_{t,x}}+ \| \langle \partial_t \rangle^{\gamma - \frac{1}{2}+\delta} \mathcal{P} u \|_{L^2_{t,x}}.
\end{equation}
This completes the proof.

\end{proof}

With the estimates for different regions in phase space at hand, we can finish the proof of Proposition \ref{prop:ExtendedMaxwell}.

\begin{proof}[Conclusion of the Proof of Proposition \ref{prop:ExtendedMaxwell}]
Taking \eqref{eq:FrequencyLocalizedEstimateI}-\eqref{eq:FrequencyLocalizedEstimateIII} together, we find
\begin{equation*}
\begin{split}
\| u \|_{L^p L^q} &\lesssim \| \langle D' \rangle^\gamma u \|_{L_t^\infty L_x^2} + \| \langle \partial_t \rangle^{\gamma + \delta} u \|_{L_t^\infty L_x^2} \\
&\quad + \| \langle \partial_t \rangle^{\gamma + \delta} \mathcal{P} u \|_{L^2_{t,x}} + \| \langle D' \rangle^{\gamma + \delta} \mathcal{P} u \|_{L^2_{t,x}} \\
&\quad + \| \rho_e \|_{L_t^\infty H^{\gamma - 1 + \frac{1}{p} + \delta}}.
\end{split}
\end{equation*}
By applying the estimate to homogeneous solutions, we obtain
\begin{equation*}
\| u \|_{L^p L^q} \lesssim \| \langle D' \rangle^\gamma u \|_{L_t^\infty L_x^2} + \| \langle \partial_t \rangle^{\gamma+\delta} u \|_{L^2_{t,x}} + \| \rho_e \|_{L_t^\infty H^{\gamma - 1+ \frac{1}{p} +\varepsilon}}.
\end{equation*}
For homogeneous solutions, we can trade the time derivatives for spatial derivatives and by the energy estimates of Section \ref{section:EnergyEstimates}, we obtain
\begin{equation*}
\| u \|_{L^p L^q} \lesssim \| \langle D' \rangle^{\gamma + \delta} u(0) \|_{L^2_x} + \| \rho_e \|_{L_t^\infty H^{\gamma - 1 + \frac{1}{p}+ \delta}}.
\end{equation*}
The conclusion follows from Duhamel's formula.
\end{proof}

\bigskip

The proofs of \eqref{eq:FrequencyLocalizedEstimateI} and \eqref{eq:FrequencyLocalizedEstimateIII} make use of the diagonalization of $\mathcal{P}_{<\lambda}$ via pseudo-differential operators. This is carried out in the following. Let $h = \big( \det(g_{ij}) \big)^{1/2}$ and denote $C(\xi')_{ij} = -\varepsilon_{ijk} \xi_k'$. The principal symbol (with rough coefficients) is given by
\begin{equation*}
p(x,\xi) =
i
\begin{pmatrix}
\xi_0 h g^{-1} \varepsilon & - C(\xi') \\
C(\xi') & h g^{-1} \mu \xi_0
\end{pmatrix}
.
\end{equation*}

\bigskip

We consider as truncated operator $\mathcal{P}_\lambda$ the following: Let $g^{-1} = A A^t$ denote the factorization into Jacobians (which we also extend such that these are Lipschitz along the boundary). Let $A_{< \lambda}$ denote the truncation of spatial frequencies of $A$ to frequencies less than $\lambda/8$. Let $h_{<\lambda} = \det (A_{<\lambda})$. We define
\begin{equation}
\label{eq:FrequencyTruncationB}
\mathcal{P}_{<\lambda} = 
\begin{pmatrix}
h_{<\lambda} A_{<\lambda} A_{<\lambda}^{t} \varepsilon_{<\lambda} \partial_t & - \nabla \times \\
\nabla \times & h_{<\lambda} A_{<\lambda} A_{<\lambda}^{t} \mu_{<\lambda} \partial_t
\end{pmatrix}
.
\end{equation}
Observe that $\| (\mathcal{P} - \mathcal{P}_{<\lambda}) S_\lambda u \|_{L^2 } \lesssim \| S_\lambda u \|_{L^2}$. 
Note that in $\rho_e'$ we can truncate $h$, $A$, $A^t$, and $\varepsilon$ in frequencies because we can write the difference as a telescoping sum
\begin{equation}
\label{eq:TelescopingSum1}
\begin{split}
&\quad \| S_\lambda (\nabla \cdot (h A A^t \varepsilon \mathcal{E})) - S_\lambda \nabla \cdot (h_{<\lambda} A_{<\lambda} A^t_{<\lambda} \varepsilon_{<\lambda} \mathcal{E}) \|_{L^2} \\
&= \| S_\lambda \nabla \cdot (h_{> \lambda} A A^t \varepsilon \mathcal{E} + h A_{>\lambda} A^t \varepsilon \mathcal{E} + \ldots ) \|_{L^2}.
\end{split}
\end{equation}
For instance, 
\begin{equation}
\label{eq:TelescopingSum2}
\| S_\lambda \nabla \cdot (h_{>\lambda} A A^t \varepsilon \mathcal{E} ) \|_{L^2} \lesssim \lambda \| h_{>\lambda } \|_{L^\infty} \| A \|_{L^\infty} \| A^t \|_{L^\infty} \| \varepsilon \|_{L^\infty} \| \mathcal{E} \|_{L^2}.
\end{equation}

After these reductions, we are dealing with symbols in $S^1_{1,1}$, which is a borderline case for symbol composition. But the considered symbols $a \in S^i_{1,1}$ actually satisfy
\begin{equation}
\label{eq:LipschitzEstimate}
|\partial_x a | \lesssim 1
\end{equation}
because the reflected Jacobians and coefficients are Lipschitz. This suffices for symbol composition to hold to first order. Accordingly, we make the following definition:
\begin{definition}
Let $k \in \N_0$. We define the symbol class
\begin{equation*}
 \tilde{S}^k_{1,1} = \{ a \in C^\infty(\R^d \times \R^d) \, : \,  |\partial_x^\alpha \partial_\xi^\beta a(x,\xi)| \lesssim \langle \xi \rangle^{k  - |\beta| + (|\alpha|-1)_+} \}.
\end{equation*}
\end{definition}
We have the following:
\begin{lemma}
 Let $m,n \in \R$, $a \in \tilde{S}^m_{1,1}$, $b \in \tilde{S}^n_{1,1}$. Then, we find the following estimate to hold:
 \begin{equation*}
  a(x,D) \circ b(x,D) = (ab)(x,D) + E
 \end{equation*}
 with $\| E \|_{H^s(\mathbb{R}^d) \to H^{s+m+n-1} ( \mathbb{R}^d)} \lesssim 1$.
\end{lemma}


\subsection{Diagonalizing the principal symbol}
\label{subsection:DiagonalizingLipschitz}
 In the following we carry out the formal computation to find suitable conjugation matrices for the operator $\mathcal{P}_\lambda$. The aim is to prove the following proposition:
\begin{proposition}
\label{prop:Diagonalization}
Let $2^{\N} \ni \lambda \gg \lambda_0$. There is a decomposition of phase space by projections
\begin{equation*}
S'_\lambda S_\lambda = S_{\lambda 1} + S_{\lambda 2} + S_{\lambda 3}
\end{equation*}
such that for every $i \in \{1,2,3\}$ there are $\mathcal{M}_{\lambda}^i \in OP \tilde{S}^{0}_{1,1}$, $\mathcal{N}^i_\lambda \in OP \tilde{S}^{0}_{1,1}$, and $\mathcal{D}^i_\lambda \in OP \tilde{S}^{1}_{1,1}$ such that
\begin{equation*}
\mathcal{P}_\lambda S_{\lambda i} = \mathcal{M}^i_\lambda \mathcal{D}^i_\lambda \mathcal{N}^i_\lambda S_{\lambda i} + E^i_\lambda
\end{equation*}
with $\| E^i_\lambda \|_{2 \to 2} \lesssim 1$ with implicit constant independent of $\lambda$.
\end{proposition}
Before we turn to the technical details, we carry out a formal diagonalization of
\begin{equation*}
p(x,\xi) = i
\begin{pmatrix}
h_{<\lambda} A_{<\lambda} A^t_{<\lambda} \varepsilon_{<\lambda} \xi_0 & - C(\xi') \\
C(\xi') & h_{<\lambda} A_{<\lambda} A_{<\lambda}^t \mu_{<\lambda} \xi_0
\end{pmatrix}
.
\end{equation*}
The symbol is in $\tilde{S}^{1}_{1,1}$. We diagonalize the principal symbol as follows:
\begin{equation*}
p(x,\xi) \pi(x,\xi) = m(x,\xi) d(x,\xi) n(x,\xi) \pi(x,\xi)
\end{equation*}
with $m,n \in \tilde{S}^0_{1,1}$ and $d \in \tilde{S}^1_{1,1}$, and $\pi \in \tilde{S}^0_{1,1}$ denoting a projection to a region in phase space to be determined. In the first step, we write
\begin{equation*}
\begin{split}
&\;
\begin{pmatrix}
h_{<\lambda} A_{<\lambda} A^t_{<\lambda} \varepsilon_{<\lambda} \xi_0 & - C(\xi') \\
C(\xi') & h_{<\lambda} A_{<\lambda} A^t_{<\lambda} \mu_{<\lambda} \xi_0
\end{pmatrix}
\\
&=
\begin{pmatrix}
 A_{< \lambda} &0 \\
 0 & A_{< \lambda}
 \end{pmatrix}
\begin{pmatrix}
h_{<\lambda} \varepsilon_{<\lambda} \xi_0 & - A_{<\lambda}^{-1} C(\xi') (A^t_{<\lambda})^{-1} \\
A_{<\lambda}^{-1} C(\xi') (A_{<\lambda}^t)^{-1} & h_{<\lambda} \mu_{<\lambda} \xi_0
\end{pmatrix}
\begin{pmatrix}
A_{<\lambda}^t & 0 \\
0 & A_{<\lambda}^t
\end{pmatrix}.
\end{split}
\end{equation*}
We recall the following:
\begin{lemma}
\label{lem:TransformationCurl}
Let $B \in \C^{3 \times 3}$. The following identity holds:
\begin{equation}
\label{eq:TransformationCurl}
B^t C(\xi') B = C(\text{ad} B \cdot \xi').
\end{equation}
In the above display $ad B$ denotes the adjugate matrix, i.e., 
\begin{equation*}
\text{ad} A = ((-1)^{i+j} A_{ji})_{i,j}
\end{equation*}
with $A_{ji}$ denoting the $(j,i)$-minor of $A$.
\end{lemma}
This yields by the definition of the adjugate matrix, $h_{<\lambda}$, and using Cramer's rule
\begin{equation*}
A_{<\lambda}^{-1} C(\xi') (A_{<\lambda}^t)^{-1} = C(h_{<\lambda} A_{<\lambda}^t \xi').
\end{equation*}
We write
\begin{equation*}
\begin{split}
&\quad \begin{pmatrix}
h_{<\lambda} \varepsilon_{<\lambda} \xi_0 & - A_{<\lambda}^{-1} C(\xi') (A_{<\lambda}^t)^{-1} \\
A_{<\lambda}^{-1} C(\xi') (A_{<\lambda}^t)^{-1} & h_{<\lambda} \mu_{<\lambda} \xi_0
\end{pmatrix}
\\
&= 
\begin{pmatrix}
\varepsilon_{<\lambda} \xi_0 & - C(A_{<\lambda}^t \xi') \\
C(A_{<\lambda}^t \xi') & \mu_{<\lambda} \xi_0
\end{pmatrix}
\begin{pmatrix}
h_{<\lambda} & 0 \\
0 & h_{<\lambda}
\end{pmatrix}
\\
&= 
\begin{pmatrix}
\varepsilon_{<\lambda}^{\frac{1}{2}} & 0 \\
0 & \mu_{<\lambda}^{\frac{1}{2}}
\end{pmatrix}
\begin{pmatrix}
\xi_0 & - C\big( \frac{A_{<\lambda}^t \xi'}{(\varepsilon_{<\lambda} \mu_{<\lambda})^{\frac{1}{2}}} \big) \\
C\big( \frac{A_{<\lambda}^t \xi'}{(\varepsilon_{<\lambda} \mu_{<\lambda})^{\frac{1}{2}}} \big) & \xi_0
\end{pmatrix}
\begin{pmatrix}
\varepsilon_{<\lambda}^{\frac{1}{2}} & 0 \\
0 & \mu_{<\lambda}^{\frac{1}{2}}
\end{pmatrix}
\begin{pmatrix}
h_{<\lambda} & 0 \\
0 & h_{<\lambda}
\end{pmatrix}
.
\end{split}
\end{equation*}

Hence, we have reduced to diagonalizing
\begin{equation}
\label{eq:ReducedMatrix}
B =
\begin{pmatrix}
\xi_0 & - C(\tilde{\xi}') \\
C(\tilde{\xi}') & \xi_0
\end{pmatrix}
.
\end{equation}

This reflects invariance of pseudo-differential operators under change of coordinates. Since the symbols are very rough, we prefer to carry out the computation directly.

In \cite{Schippa2021Maxwell3d,Schippa2022ResolventEstimates} the symbol was diagonalized in the more difficult case of partially anisotropic $\varepsilon$, i.e., $\varepsilon$ having possibly two different eigenvalues. In this case, the resulting expressions are fairly complicated. We take the opportunity to point out a simplification for isotropic $\varepsilon$ and $\mu$. Write $\xi' = (\xi_1,\xi_2,\xi_3)$. We begin with computing the characteristic polynomial of $p/i$ using the block matrix structure:
\begin{equation*}
q(y) = 
\begin{vmatrix}
y - \xi_0 & C(\xi') \\
- C(\xi') & y - \xi_0
\end{vmatrix} 
= \big\vert (y - \xi_0)^2 1_{3 \times 3} + C^2(\xi') \big\vert.
\end{equation*}
Hence, we have reduced to computing the eigenvalues of $C^2(\xi')$. Note that
\begin{equation*}
C^2(\xi') = 
\begin{pmatrix}
\xi_2^2 + \xi_3^2 & - \xi_1 \xi_2 & - \xi_1 \xi_3 \\
- \xi_1 \xi_2 & \xi_1^2 + \xi_3^2 & - \xi_2 \xi_3 \\
-\xi_1 \xi_3 & - \xi_2 \xi_3 & \xi_1^2 + \xi_2^2
\end{pmatrix}
= |\xi'|^2 1_{3 \times 3} - \xi \otimes \xi.
\end{equation*}
It follows that
\begin{equation*}
r(\lambda,\xi') = \det (\lambda 1_{3 \times 3} - C^2(\xi')) = ( \lambda - \| \xi \|^2)^2 \lambda.
\end{equation*}
This gives for the characteristic polynomial $q$
\begin{equation*}
q(\lambda) = (\lambda - \xi_0)^2 [(\lambda - (\xi_0 - \| \xi' \|))^2 (\lambda - (\xi_0 + \| \xi' \|))^2].
\end{equation*}
We conclude that the diagonalization is given by
\begin{equation}
\label{eq:DiagonalizationMatrix}
d(x,\xi) = i (\xi_0, \xi_0, \xi_0 - \| \xi' \|, \xi_0 + \| \xi' \|, \xi_0 - \| \xi' \|, \xi_0 + \| \xi' \|).
\end{equation}
In the following let $\xi_i^* = \frac{\xi_i}{\| \xi' \|}$ for $i=1,2,3$.
Eigenvectors of $\xi_0$ are clearly given by
\begin{equation*}
\begin{pmatrix}
\xi_1^* \\ \xi_2^* \\ \xi_3^* \\ 0 \\ 0 \\ 0
\end{pmatrix}, \qquad
\begin{pmatrix}
0 \\ 0 \\ 0 \\ \xi_1^* \\ \xi_2^* \\ \xi_3^*
\end{pmatrix}
.
\end{equation*}
\textbf{Eigenvectors of }$\xi_0 - \| \xi' \|$: We use the block matrix structure of $p(x,\xi)$. Let $v = (v_1, v_2)^t$ denote an eigenvector. We find the system of equations:
\begin{equation*}
\begin{pmatrix}
\| \xi' \| & C(\xi') \\
- C(\xi') & \| \xi' \|
\end{pmatrix}
\begin{pmatrix}
v_1 \\ v_2
\end{pmatrix}
= 0.
\end{equation*}
Iterating the above in the non-trivial case $\xi' = 0$ yields the eigenvector equation for $v_1$:
\begin{equation*}
\| \xi' \|^2 v_1 + C^2(\xi') v_1 = 0.
\end{equation*}
For this we find the zero-homogeneous eigenvectors:
\begin{equation}
\label{eq:EigenvectorsI}
\begin{pmatrix}
0 \\ - \xi_3^* \\ \xi_2^* 
\end{pmatrix}, \quad
\begin{pmatrix}
\xi_3^* \\ 0 \\ -\xi_1^*
\end{pmatrix}, \quad
\begin{pmatrix}
- \xi_2^* \\ \xi_1^* \\ 0
\end{pmatrix}
.
\end{equation}
The system of equations from above yields
\begin{equation*}
v_2 = \frac{C(\xi')}{\| \xi' \|} v_1.
\end{equation*}
This gives for $v_2$:
\begin{equation*}
\begin{pmatrix}
{\xi_2^*}^2 + {\xi_3^*}^2 \\ - \xi_1^* \xi_2^* \\ -\xi_1^* \xi_3^*
\end{pmatrix}, \qquad
\begin{pmatrix}
-\xi_1^* \xi_2^* \\ {\xi_1^*}^2 + {\xi_3^*}^2 \\ - \xi_2^* \xi_3^*
\end{pmatrix}
, \qquad
\begin{pmatrix}
-\xi_1^* \xi_3^* \\ -\xi_2^* \xi_3^* \\ {\xi_1^*}^2 + {\xi_2^*}^2
\end{pmatrix}
.
\end{equation*}
\textbf{Eigenvectors of }$\xi_0 + \| \xi' \|$: Again, we use the block matrix structure of $p(x,\xi)$, and let $v = (v_1, v_2)^t$ denote an eigenvector. This yields the system of equations:
\begin{equation*}
\begin{pmatrix}
- \| \xi' \| & C(\xi') \\
- C(\xi') & - \| \xi' \|
\end{pmatrix}
\begin{pmatrix}
v_1 \\ v_2
\end{pmatrix}
= 0.
\end{equation*}
We find again for $v_1$
\begin{equation*}
C^2(\xi') v_1 + \| \xi' \|^2 v_1 = 0,
\end{equation*}
and for $v_2$
\begin{equation*}
v_2 = - \frac{C(\xi') v_1}{\| \xi' \|}.
\end{equation*}
\textbf{Conjugation matrices:} 
We choose conjugation matrices depending on a non-vanishing direction of $\xi$. In the following suppose that $|\xi_3^*| \gtrsim 1$.
One choice of conjugation matrices according to \eqref{eq:DiagonalizationMatrix} is given by choosing the first eigenvector in \eqref{eq:EigenvectorsI}:
\begin{equation}
\label{eq:ConjugationMatrixI}
m_3(x,\xi) =
\begin{pmatrix}
\xi_1^* & 0 		& 0 		& 0 		& \xi_3^* & \xi_3^* \\
\xi_2^* & 0		&- \xi_3^* &- \xi_3^*	&	0 	&	0	\\
\xi_3^* & 0 		& \xi_2^* & \xi_2^* &- \xi_1^*	& - \xi_1^* \\
0	   & \xi_1^*	& {\xi_2^*}^2 + {\xi_3^*}^2 & - ({\xi_2^*}^2 + {\xi_3^*}^2)  & -\xi_1^* \xi_2^* &  \xi_1^* \xi_2^* \\
0		& \xi_2^* & - \xi_1^* \xi_2^* & \xi_1^* \xi_2^* & ({\xi_1^*}^2 + {\xi_3^*}^2) & - ({\xi_1^*}^2 + {\xi_3^*}^2) \\
0		& \xi_3^* & - \xi_1^* \xi_3^* &  \xi_1^* \xi_3^* & - \xi_2^* \xi_3^* & \xi_2^* \xi_3^*
\end{pmatrix}
.
\end{equation}
We have the following:
\begin{lemma}
 Let $m_3$ be given as in \eqref{eq:ConjugationMatrixI}. Then,
 \begin{equation}
  \label{eq:DetConjugationMatrixI}
  \det m_3(x,\xi) = {\xi_3^*}^2.
 \end{equation}
\end{lemma}
\begin{proof}
 By elementary column operations, that is adding and subtracting the third and fourth and fifth and sixth eigenvector, we compute the determinant to be
 \begin{equation*}
 \begin{split}
  \det m_3 &= 
  \begin{vmatrix}
   \xi_1^* & 0 & 0 & 0 & \xi_3^* & 0 \\
   \xi_2^* & 0 & - \xi_3^* & 0 & 0 & 0 \\
   \xi_3^* & 0 & \xi_2^* & 0 & -\xi_1^* & 0 \\
   0 & \xi_1^* & 0 & {\xi_2^*}^2 + {\xi_3^*}^2 & 0 & \xi_1^* \xi_2^* \\
   0 & \xi_2^* & 0 & -\xi_1^* \xi_2^* & 0 & -({\xi_1^*}^2 + {\xi_3^*}^2) \\
   0 & \xi_3^* & 0 & -\xi_1^* \xi_3^* & 0 & \xi_2^* \xi_3^*
  \end{vmatrix}
\\
&= \begin{vmatrix}
    \xi_1^* & 0 & \xi_3^* \\
    \xi_2^* & -\xi_3^* & 0 \\
    \xi_3^* & \xi_2^* & -\xi_1^*
   \end{vmatrix}
  \begin{vmatrix}
   {\xi_2^*}^2 + {\xi_3^*}^2 & \xi_1^* & \xi_1^* \xi_2^* \\
   -\xi_1^* \xi_2^* & \xi_2^* & -({\xi_1^*}^2 + {\xi_3^*}^2) \\
   - \xi_1^* \xi_3^* & \xi_3^* & \xi_2^* \xi_3^*
  \end{vmatrix}
.
\end{split}
 \end{equation*}
The ultimate line follows from permuting the columns and using block matrix structure. Then, it is straight-forward
\begin{equation*}
 \begin{vmatrix}
  \xi_1^* & 0 & \xi_3^* \\
  \xi_2^* & -\xi_3^* & 0 \\
  \xi_3^* & \xi_2^* & -\xi_1^*
 \end{vmatrix}
 = {\xi_1^*}^2 \xi_3^* + {\xi_2^*}^2 \xi_3^* + {\xi_3^*}^2 \xi_3^* = \xi_3^*.
\end{equation*}
Again by ${\xi_1^*}^2 + {\xi_2^*}^2 + {\xi_3^*}^2 = 1$, we find
\begin{equation*}
\begin{split}
&\quad \;
 \begin{vmatrix}
  {\xi_2^*}^2 + {\xi_3^*}^2 & \xi_1^* & \xi_1^* \xi_2^* \\
  -\xi_1^* \xi_2^* & \xi_2^* & -({\xi_1^*}^2 + {\xi_3^*}^2) \\
  -\xi_1^* \xi_3^* & \xi_3^* & \xi_2^* \xi_3^*
 \end{vmatrix}
=
\begin{vmatrix}
 1-{\xi_1^*}^2 & \xi_1^* & \xi_1^* \xi_2^* \\
 -\xi_1^* \xi_2^* & \xi_2^* & -({\xi_1^*}^2 + {\xi_3^*}^2) \\
 -\xi_1^* \xi_3^* & \xi_3^* & \xi_2^* \xi_3^*
\end{vmatrix}
\\
&= 
\begin{vmatrix}
 1 & \xi_1^* & \xi_1^* \xi_2^* \\
 0 & \xi_2^* & -({\xi_1^*}^2 + {\xi_3^*}^2) \\
 0 & \xi_3^* & \xi_2^* \xi_3^*
\end{vmatrix}
- \xi_1^*
\begin{vmatrix}
 \xi_1^* & \xi_1^* & \xi_1^* \xi_2^* \\
 \xi_2^* & \xi_2^* & -({\xi_1^*}^2 + {\xi_3^*}^2) \\
 \xi_3^* & \xi_3^* & \xi_2^* \xi_3^*
\end{vmatrix}
\\
&= \xi_3^*. 
\end{split}
\end{equation*}
This finishes the proof.
\end{proof}
Likewise, we define $m_1$ and $m_2$ by choosing the non-trivial eigenvectors for $|\xi_i^*| \gtrsim 1$, which leads us to conjugation matrices with determinant
\begin{equation*}
 \det m_i(x,\xi) = {\xi_i^*}^2.
\end{equation*}

By elementary column operations, that is adding and subtracting the third and fourth and fifth and sixth eigenvector, the determinant is computed to be
\begin{equation*}
\det m_3(x,\xi) = {\xi_3^*}^2.
\end{equation*}
We shall see that for $|\xi_3^*| \gtrsim 1$, we can choose the eigenvectors as an orthonormalbasis through linear combinations of the above. Let
\begin{equation*}
w_1 =
\begin{pmatrix}
0 \\ - \xi_3^* \\ \xi_2^* \\ {\xi_2^*}^2 + {\xi_3^*}^2 \\ - \xi_1^* \xi_2^* \\ - \xi_1^* \xi_3^*
\end{pmatrix}
, \qquad w_3 = 
\begin{pmatrix}
\xi_3^* \\ 0 \\ - \xi_1^* \\ - \xi_1^* \xi_2^* \\ {\xi_1^*}^2 + {\xi_3^*}^2 \\ \xi_2^* \xi_3^*
\end{pmatrix}
.
\end{equation*}
We have $\| w_1 \|^2 = 2 ({\xi_2^*}^2 + {\xi_3^*}^2)$, $\| w_3 \|^2 = 2 ({\xi_1^*}^2 + {\xi_3^*}^2)$, and normalize $w_i' = w_i / \| w_i \|$. We compute
\begin{equation*}
\langle w_1', w_3' \rangle = \frac{- \xi_1^* \xi_2^* (1+(\xi_3^*)^2)}{2({\xi_2^*}^2 + {\xi_3^*}^2)^{\frac{1}{2}} ({\xi_1^*}^2 + {\xi_3^*}^2)^{\frac{1}{2}}}.
\end{equation*}
Now we consider $\tilde{w}_3 = w_3' - \langle w_1',w_3' \rangle w_1'$:
\begin{equation*}
\tilde{w}_3 = \frac{1}{\sqrt{2}({\xi_1^*}^2 + {\xi_3^*}^2)} \begin{pmatrix}
\xi_3^* \\ 0 \\ - \xi_1^* \\ - \xi_1^* \xi_2^* \\ {\xi_1^*}^2 + {\xi_3^*}^2 \\ \xi_2^* \xi_3^*
\end{pmatrix}
- 
\frac{\langle w_1', w_3' \rangle}{\sqrt{2} ({\xi_2^*}^2 + {\xi_3^*}^2)^{\frac{1}{2}}}
\begin{pmatrix}
0 \\ - \xi_3^* \\ \xi_2^* \\ {\xi_2^*}^2 + {\xi_3^*}^2 \\ -\xi_1^* \xi_2^* \\ - \xi_1^* \xi_3^*
\end{pmatrix}
.
\end{equation*}
Clearly, $\| \tilde{w}_3 \|_2 \gtrsim 1$ for $|\xi_3^*| \gtrsim 1$. Hence, by renormalizing (and not changing notations for sake of brevity), we let
\begin{equation*}
\tilde{w}_3 \rightarrow \frac{\tilde{w}_3}{\| \tilde{w}_3 \|_2}.
\end{equation*}
Similarly, consider
\begin{equation*}
w_2 =
\begin{pmatrix}
0 \\ -\xi_3^* \\ \xi_2^* \\ -({\xi_2^*}^2 + {\xi_3^*}^2) \\ \xi_1^* \xi_2^* \\ \xi_1^* \xi_3^*
\end{pmatrix},
\qquad
w_4 =
\begin{pmatrix}
\xi_3^* \\ 0 \\ -\xi_1^* \\ \xi_1^* \xi_2^* \\ - ({\xi_1^*}^2 + {\xi_3^*}^2) \\ - \xi_1^* \xi_3^*
\end{pmatrix}
.
\end{equation*}
We compute
\begin{equation*}
\| w_2 \|_2^2 = 2({\xi_2^*}^2 + {\xi_3^*}^2), \qquad \| w_4 \|_2^2 = 2({\xi_3^*}^2 + {\xi_1^*}^2),
\end{equation*}
which allows for renormalization $w_i' = w_i / \| w_i \|_2$. Now we consider $\tilde{w}_4 = w_4' - \langle w_2', w_4' \rangle \tilde{w}_2'$, which yields after an additional renormalization eigenvectors of $\xi_0+ \| \xi' \|$.
We conclude that the matrix
\begin{equation*}
\tilde{m}_3(x,\xi) = 
\begin{pmatrix}
u_1 & u_2 & \tilde{w}_1 & \tilde{w}_2 & \tilde{w}_3 & \tilde{w}_4
\end{pmatrix}
\end{equation*}
consists of orthonormal eigenvectors to $d$ as in \eqref{eq:DiagonalizationMatrix} for $|\xi_3'| \gtrsim 1$. We summarize the accomplished diagonalization:
\begin{equation*}
\begin{split}
p(x,\xi) &=
\begin{pmatrix}
A_{<\lambda} & 0 \\
0 & A_{<\lambda}
\end{pmatrix}
\begin{pmatrix}
\varepsilon_{<\lambda}^{\frac{1}{2}} & 0 \\
0 & \mu_{<\lambda}^{\frac{1}{2}}
\end{pmatrix}
\tilde{m}_i(x,\xi_0,\tilde{\xi}') d(x,\xi_0,\tilde{\xi}') \\
&\quad \times \tilde{m}_i^t(x,\xi_0,\tilde{\xi}') 
\begin{pmatrix}
\varepsilon_{<\lambda}^{\frac{1}{2}} & 0 \\
0 & \mu_{<\lambda}^{\frac{1}{2}}
\end{pmatrix}
\begin{pmatrix}
A_{<\lambda}^t & 0 \\
0 & A_{<\lambda}^t
\end{pmatrix}
\text{ with } \tilde{\xi}' = \frac{A_{<\lambda}^t \xi'}{(\varepsilon_{<\lambda} \mu_{<\lambda})^{\frac{1}{2}}}.
\end{split}
\end{equation*}
in the phase space region $| {\tilde{\xi}_3}^* | \gtrsim 1$. Note that there is always $i \in \{1,2,3\}$ such that
\begin{equation*}
| \tilde{\xi}_i^* | \gtrsim 1.
\end{equation*}
We define phase-space projection operators by the function
\begin{equation*}
\pi_3(x,\xi) =  \chi(\lambda^{-1} \xi) \tilde{\chi}(\lambda^{-1} (A_{<\lambda}^t \xi')_3),
\end{equation*}
with $\chi, \, \tilde{\chi} \in C^\infty_c$ suitable bump functions. The corresponding projections are denoted by $S_\lambda S_{\lambda 3}$.
We let
\begin{equation*}
\begin{split}
N^3(x,\xi) &= \tilde{m}_i^t(x,\xi_0,\tilde{\xi}') 
\begin{pmatrix}
\varepsilon_{<\lambda}^{\frac{1}{2}}(x) & 0 \\
0 & \mu_{<\lambda}^{\frac{1}{2}}(x)
\end{pmatrix}
\begin{pmatrix}
A_{<\lambda}^t(x) & 0 \\
0 & A_{<\lambda}^t(x)
\end{pmatrix}
\\
&\quad \quad \times \chi(\lambda^{-1} \xi') \tilde{\chi}(\lambda^{-1} (A^t_{<\lambda} \xi')_3),
\end{split}
\end{equation*}
and
\begin{equation*}
D^3(x,\xi) = d(x,\xi_0,\tilde{\xi}') \chi(\lambda^{-1} \xi') \tilde{\chi}(\lambda^{-1} (A^t_{<\lambda} \xi')_3),
\end{equation*}
and
\begin{equation*}
M^3(x,\xi) =
\begin{pmatrix}
A_{<\lambda}(x) & 0 \\
0 & A_{<\lambda}(x)
\end{pmatrix}
\begin{pmatrix}
\varepsilon_{<\lambda}^{\frac{1}{2}}(x) & 0 \\
0 & \mu_{<\lambda}^{\frac{1}{2}}
\end{pmatrix}
\tilde{m}_i(x,\xi) \chi(\lambda^{-1} \xi') \tilde{\chi}(\lambda^{-1} (A^t_{<\lambda} \xi')_3).
\end{equation*}
The corresponding operators are defined by
\begin{equation*}
\mathcal{M}^3_\lambda(x,D) = Op(M^3(x,\xi)), \; \mathcal{D}^3_\lambda(x,D) = Op(D^3(x,\xi)), \; \mathcal{N}^3_\lambda(x,D) = Op(N^3(x,\xi)).
\end{equation*}
 By symbol composition, we can harmlessly insert frequency projectors after every factor. This makes the single factors bounded with symbols in $\tilde{S}^i_{1,1}$, $i \in \{0,1\}$. By Lemma \ref{lem:PseudoBoundsFourierSeries}, the claim follows, and the proof of Proposition \ref{prop:Diagonalization} is complete. $\hfill \Box$

\subsection{Conclusion of frequency localized estimate}
\label{subsection:Conclusion}
We have shown in Subsection \ref{subsection:DiagonalizingLipschitz} that after appropriate localization in phase space, the Maxwell system can be diagonalized to two degenerate and four non-degenerate half-wave equations. The degenerate equations correspond to stationary solutions, possibly induced by charges.

However, we note that $\mathcal{P}_{<\lambda}$ defined in \eqref{eq:FrequencyTruncationA} for the purpose of proving commutator estimates and $\mathcal{P}_{<\lambda}$ defined in \eqref{eq:FrequencyTruncationB} for the diagonalization are defined slightly differently. We denote by $\mathcal{P}^{(1)}_{<\lambda}$ defined in \eqref{eq:FrequencyTruncationA} and $\mathcal{P}^{(2)}_{<\lambda}$ defined in \eqref{eq:FrequencyTruncationB}. Similarly, we denote $\rho^{(1)}_{e \mu} = \nabla \cdot (\varepsilon'_{<\nu} S'_{\nu} \mathcal{E})$ and $\rho^{(1)}_{m \nu} = \nabla \cdot (\mu'_{<\nu} S'_\nu \mathcal{H})$ and $\rho^{(2)}_{e \nu} = \nabla \cdot (h_{<\nu} A_{<\nu} A^t_{<\nu} \varepsilon_{<\nu} S'_\nu \mathcal{E})$, and $\rho^{(2)}_{m \nu} = \nabla \cdot (h_{<\nu} A_{<\nu} A^t_{<\nu} \mu_{<\nu} S'_\nu \mathcal{H})$. We have the following lemma, which shows that we can indeed use $\mathcal{P}_{<\lambda}^{(1)}$ for commutator estimates and $\mathcal{P}_{<\lambda}^{(2)}$ for the diagonalization:
\begin{lemma}
The following estimates hold:
\begin{align}
\label{eq:FrequencyTruncationTransferI}
\| \mathcal{P}_{<\lambda}^{(2)} S_\lambda^\tau S'_\lambda u \|_{L^2_x} &\lesssim \| S_\lambda^\tau S'_\lambda u \|_{L^2_x} + \| \mathcal{P}_{<\lambda}^{(1)} S_\lambda^\tau S'_\lambda u \|_{L^2_x}, \\
\label{eq:FrequencyTruncationTransferII}
\sum_{\substack{1 \leq \lambda \ll \nu, \\ \nu}} \nu^{\gamma - \frac{1}{2}} \| S'_\nu S^\tau_\lambda \mathcal{P}^{(2)}_{\nu} u \|_{L^2_x} &\lesssim \sum_{\substack{1 \leq \lambda \ll \nu, \\ \nu}} \nu^{\gamma - \frac{1}{2}} \| S'_\nu S^\tau_\lambda \mathcal{P}_{<\nu}^{(1)} u \|_{L^2_x} + \| \langle \partial_t \rangle^{\gamma - \frac{1}{2} + \varepsilon} u \|_{L^2_x}, \\
\label{eq:FrequencyTruncationTransferIII}
\| \rho^{(2)}_{e \nu} \|_{L^2_{x'}} &\lesssim \| \rho^{(1)}_{e \nu} \|_{L^2_{x'}} + \| S'_\nu \mathcal{E} \|_{L^2_{x'}}, \\
\label{eq:FrequencyTruncationTransferIV}
\| \rho^{(2)}_{m \nu} \|_{L^2_{x'}} &\lesssim \| \rho^{(1)}_{e \nu} \|_{L^2_{x'}} + \| S'_\nu \mathcal{H} \|_{L^2_{x'}}.
\end{align}
\end{lemma}
\begin{proof}
We begin with the proof of \eqref{eq:FrequencyTruncationTransferI}. By the triangle inequality, we find
\begin{equation*}
\| \mathcal{P}_{<\lambda}^{(2)} S_\lambda^\tau S'_\lambda u \|_{L^2_x} \lesssim \| (\mathcal{P} - \mathcal{P}_{<\lambda}^{(2)}) S_\lambda^\tau S'_\lambda u \|_{L^2_x} + \| (\mathcal{P} - \mathcal{P}_{<\lambda}^{(1)}) S_\lambda^\tau S'_\lambda u \|_{L^2_x} + \| S_\lambda^\tau S'_\lambda u \|_{L^2_x}.
\end{equation*}
For the error estimates we note that
\begin{equation*}
\begin{split}
\| (\mathcal{P} - \mathcal{P}^{(1)}_{<\lambda}) S^\tau_\lambda S'_\lambda u \|_{L^2_x} &\lesssim ( \| \varepsilon'_{\gtrsim \lambda} \|_{L^\infty_{x'}} + \| \mu'_{\gtrsim \lambda} \|_{L^\infty_{x'}} ) \lambda \| S_\lambda^\tau S'_\lambda u \|_{L^2_x} \\
&\lesssim ( \| \partial \varepsilon' \|_{L^\infty_{x'}} + \| \partial \mu' \|_{L^\infty_{x'}}) \| S_\lambda^\tau S'_\lambda u \|_{L^2_{x}}.
\end{split}
\end{equation*}
Similarly, we have by the telescoping sum argument \eqref{eq:TelescopingSum1} and \eqref{eq:TelescopingSum2}:
\begin{equation*}
\| (\mathcal{P} - \mathcal{P}^{(2)}_{<\lambda}) S^\tau_\lambda S'_\lambda u \|_{L^2_x} \lesssim \| S^\tau_\lambda S'_\lambda u \|_{L^2_x}.
\end{equation*}
This finishes the proof of \eqref{eq:FrequencyTruncationTransferI}. We turn to the proof of \eqref{eq:FrequencyTruncationTransferII}, for which we write again
\begin{equation*}
\| S'_\nu S^\tau_\lambda \mathcal{P}^{(2)}_{<\nu} u \|_{L^2_x} \lesssim \| S'_\nu S^\tau_\lambda (\mathcal{P}_{<\nu}^{(2)} - \mathcal{P}) u \|_{L^2_x} + \| S'_\nu S^\tau_\lambda (\mathcal{P}_{<\nu}^{(1)} - \mathcal{P}) u \|_{L^2_x} + \| S'_\nu S^\tau_\lambda \mathcal{P}^{(1)}_{<\nu} u \|_{L^2_x}.
\end{equation*}
We find
\begin{equation*}
\begin{split}
\sum_{\substack{1 \leq \lambda \ll \nu, \\ \nu}} \nu^{\gamma - \frac{1}{2}} \| S'_\nu S^\tau_\lambda (\mathcal{P}^{(1)}_{\nu} - \mathcal{P}) u \|_{L^2_x} &\lesssim \sum_{\substack{1 \leq \lambda \ll \nu, \\ \nu}} \lambda \nu^{\gamma - \frac{1}{2}} \| S^\tau_\lambda S'_\nu \kappa'_{\gtrsim \nu} u \|_{L^2_{x}} \\
&\lesssim \sum_{\substack{1 \leq \lambda \ll \nu, \\ \nu}} \lambda \nu^{\gamma - \frac{1}{2}} (\| \varepsilon'_{\gtrsim \nu} \|_{L^\infty_{x'}} + \| \mu'_{\gtrsim \nu} \|_{L^\infty_{x'}}) \| S_\lambda^\tau u \|_{L^2_x} \\
&\lesssim \sum_{\substack{1 \leq \lambda \ll \nu, \\ \nu}} \lambda \nu^{\gamma - \frac{3}{2}} (\| \partial \varepsilon' \|_{L_{x'}^\infty} + \| \partial \mu' \|_{L^\infty_{x'}}) \| S_\lambda^\tau u \|_{L^2_x} \\
 &\lesssim \| \langle \partial_t \rangle^{\gamma - \frac{1}{2} + \delta} u \|_{L^2_x}.
\end{split}
\end{equation*}
Similarly, by the telescoping sum argument, we obtain
\begin{equation*}
\begin{split}
&\quad \sum_{\substack{1 \leq \lambda \ll \nu, \\ \nu}} \nu^{\gamma - \frac{1}{2}} \| S'_\nu S^\tau_\lambda (\mathcal{P}^{(2)}_{\nu} - \mathcal{P}) u \|_{L^2_x} \\
 &\lesssim \sum_{\substack{1 \leq \lambda \ll \nu, \\ \nu}} \nu^{\gamma - \frac{1}{2}} \lambda \| (h A A^t \varepsilon - h_{<\nu} A_{<\nu} A^t_{<\nu} \varepsilon_{<\nu}) S_\lambda^\tau \mathcal{E} \|_{L^2_x} \lesssim \| \langle \partial_t \rangle^{\gamma - \frac{1}{2} + \delta} u \|_{L^2_x}.
\end{split}
\end{equation*}
We turn to the proof of \eqref{eq:FrequencyTruncationTransferIII}. We estimate
\begin{equation*}
\| \nabla \cdot (\varepsilon'_{\gtrsim \nu} S'_\nu \mathcal{E}) \|_{L^2_{x'}} \lesssim \| \partial \varepsilon' \|_{L^\infty_{x'}} \| S'_\nu \mathcal{E} \|_{L^2_{x'}} + \nu \| \varepsilon'_{\gtrsim \nu} \|_{L^\infty_{x'}} \| S'_\nu \mathcal{E} \|_{L^2_{x'}},
\end{equation*}
and by telescoping sum argument,
\begin{equation*}
\begin{split}
&\quad \| \nabla \cdot (h_{<\nu} A_{<\nu} A^t_{<\nu} \varepsilon_{<\nu} S'_\nu \mathcal{E}) - \nabla \cdot (h A A^t \varepsilon S'_\nu \mathcal{E}) \|_{L^2_{x'}} \\
 &= \| \nabla \cdot (C^{(1)}_{>\nu} C^{(2)} \ldots C^{(k)}) S'_\nu \mathcal{E} \|_{L^2_{x'}} \\
&\lesssim \| \partial C \|_{L^\infty_{x'}} \| S'_\nu \mathcal{E} \|_{L^2_{x'}} + \| C^{(1)}_{> \nu} \|_{L^\infty_{x'}} \| C^{(2)} \|_{L^\infty_{x'}} \ldots \| C^{(k)} \|_{L^\infty_{x'}} \nu \| S'_\nu \mathcal{E} \|_{L^2_{x'}} \lesssim \| S'_\nu \mathcal{E} \|_{L^2_{x'}}.
\end{split}
\end{equation*}
Hence, \eqref{eq:FrequencyTruncationTransferIII} is a consequence of the triangle inequality. The proof of \eqref{eq:FrequencyTruncationTransferIV} is carried out \emph{mutatis mutandis}.
\end{proof}

 We use this to finish the proof of Proposition \ref{prop:ExtendedMaxwell} by showing the following estimates:
\begin{align}
\label{eq:FrequencyLocalizedEstimateA}
\| S_\lambda^\tau S'_\lambda u \|_{L^p L^q} &\lesssim \lambda^\gamma (\| S_\lambda^\tau S'_\lambda u \|_{L_t^\infty L_x^2} + \| \mathcal{P}^{(2)}_{<\lambda} S^\tau_\lambda S'_\lambda u \|_{L^2_{t,x}} ), \\
\label{eq:FrequencyLocalizedEstimateB}
\| S'_\nu S^\tau_\lambda u \|_{L^p L^q} &\lesssim \nu^{\gamma - \frac{1}{2}} \| S'_\nu S_\lambda^\tau \mathcal{P}^{(2)}_{<\nu} u \|_{L^2_{t,x}} \\
&\quad + \nu^{\gamma - 1 + \frac{1}{p}} ( \| \rho_{e \nu}^{(2)} \|_{L_t^\infty L_x^2} + \| \rho^{(2)}_{m \nu} \|_{L_t^\infty L_x^2} ). \nonumber
\end{align}
To use the diagonalization, we need the following:
\begin{lemma}
\label{lem:Diagonalization}
For $i \in \{1,2,3\}$ and $\lambda \gg 1$, we find the following estimates to hold:
\begin{equation*}
\begin{split}
\| S_{\lambda i} u \|_{L^p L^q} &\lesssim \| \mathcal{N}_\lambda^i S_{\lambda i} u \|_{L^p L^q} + \lambda^{\gamma - \frac{1}{2}} \| S_{ \lambda i } u \|_{L^2}, \\
\| S_{\lambda i } u \|_{L^p L^q} &\lesssim \| \mathcal{M}_\lambda^i S_{\lambda i } u \|_{L^2}.
\end{split}
\end{equation*}
\end{lemma}
\begin{proof}
For the proof of the first estimate, we observe for the composed symbols of $\mathcal{M}_\lambda^i$ and $\mathcal{N}_\lambda^i$:
\begin{equation*}
\begin{split}
&\; \begin{pmatrix}
A_{<\lambda} & 0 \\
0 & A_{<\lambda}
\end{pmatrix}
\begin{pmatrix}
\varepsilon_{<\lambda}^{\frac{1}{2}} & 0 \\
0 & \mu_{<\lambda^{\frac{1}{2}}}
\end{pmatrix}
\tilde{m}_i(x,\xi_0,\tilde{\xi}') \tilde{m}_i^t(x,\xi_0,\tilde{\xi}') \\
&\times \begin{pmatrix}
\varepsilon_{<\lambda}^{\frac{1}{2}} & 0 \\
0 & \mu_{<\lambda}^{\frac{1}{2}}
\end{pmatrix}
\begin{pmatrix}
A_{<\lambda}^t & 0 \\
0 & A_{<\lambda}^t
\end{pmatrix}
=
\begin{pmatrix}
\varepsilon_{<\lambda} A_{<\lambda} A_{<\lambda}^t & 0 \\
0 & \mu_{<\lambda} A_{<\lambda} A_{<\lambda}^t
\end{pmatrix}
.
\end{split}
\end{equation*}
Hence, we find
\begin{equation*}
\mathcal{M}_\lambda^i \mathcal{N}_\lambda^i S_{\lambda i} = 
\begin{pmatrix}
\varepsilon_{<\lambda} A_{<\lambda} A^t_{<\lambda} & 0 \\
0 & \mu_{<\lambda} A_{<\lambda} A_{<\lambda}^t
\end{pmatrix}
S_{\lambda i}
+R_i(x,D)
\end{equation*}
with $\| R_i(x,D) \|_{L^2 \to L^2} \lesssim \lambda^{-1}$. This allows us to estimate
\begin{equation*}
\begin{split}
\| S_{\lambda i} u \|_{L^p L^q} &\lesssim \big\|
\begin{pmatrix}
\varepsilon_{<\lambda} A_{<\lambda} A_{<\lambda}^t & 0 \\
0 & \mu_{<\lambda} A_{<\lambda} A_{<\lambda}^t
\end{pmatrix}
S_{\lambda i} u \|_{L^p L^q} \\
&\lesssim \| \mathcal{M}_\lambda^i \mathcal{N}_\lambda^i S_{\lambda i} u \|_{L^p L^q} + \| R^i(x,D) S_{\lambda i} u \|_{L^p L^q} \\
&\lesssim \| \mathcal{N}_\lambda^i S_{\lambda i} u \|_{L^p L^q} + \lambda^{\gamma - \frac{1}{2}} \| S_{\lambda i} u \|_{L^2}
\end{split}
\end{equation*}
by Minkowski's inequality and Sobolev embedding.\\
For the proof of the second estimate, we argue similarly
\begin{equation*}
\| S_{\lambda i} u \|_{L_{t,x}^2} \lesssim \| (1 + R_i) S_{\lambda i} u \|_{L_{t,x}^2} = \| \mathcal{N}_\lambda^i \mathcal{M}_\lambda^i S_{\lambda i} u \|_{L_{t,x}^2} \lesssim \| \mathcal{M}_\lambda^i S_{\lambda i} u \|_{L_{t,x}^2}.
\end{equation*}
The proof is complete.
\end{proof}

We can finally show \eqref{eq:FrequencyLocalizedEstimateA} and \eqref{eq:FrequencyLocalizedEstimateB}:

\begin{proof}[Proof~of~\eqref{eq:FrequencyLocalizedEstimateA}]
We split $S_\lambda^\tau S'_\lambda u = \sum_{i=1}^3 S_\lambda^\tau S_{\lambda i} u $ with $S_\lambda^\tau S_{\lambda i} u$ being amenable to the diagonalization of $\mathcal{P}$ provided by $\mathcal{M}_\lambda^i$ and $\mathcal{N}_\lambda^i$. We write
\begin{equation*}
\| S_\lambda^\tau S_{\lambda i} u \|_{L^p L^q} \lesssim \| S_\lambda^\tau \mathcal{M}_{\lambda}^i \mathcal{N}^i_\lambda S_{\lambda i } u \|_{L^p L^q} + \| S_\lambda^\tau R(x,D') S_\lambda' u \|_{L^p L^q}.
\end{equation*}
Since $R(x,D')$ is smoothing of order $-1$, we can use Sobolev embedding to find
\begin{equation*}
\begin{split}
\| S_\lambda^\tau R(x,D') S_\lambda' u \|_{L^p L^q} &\lesssim \lambda^{\frac{1}{2}-\frac{1}{p}} \lambda^{3 \big( \frac{1}{2}- \frac{1}{q} \big) - 1} \| S_\lambda^\tau S_\lambda' u \|_{L^2_{t,x}} \\
&\lesssim \lambda^{-\varepsilon} \| \langle D' \rangle^\gamma S_\lambda' u \|_{L^2_{t,x}},
\end{split}
\end{equation*}
which is acceptable. By Lemma \ref{lem:Diagonalization}, we have
\begin{equation*}
\| S_\lambda^\tau \mathcal{M}_\lambda^i \mathcal{N}_\lambda^i S_{\lambda i} u \|_{L^p L^q} \lesssim \| S_\lambda^\tau \mathcal{N}_\lambda^i S_{\lambda i} u \|_{L^p L^q}.
\end{equation*}
We estimate the components $\| [S_\lambda^\tau \mathcal{N}_\lambda^i S_{\lambda i} u]_j \|_{L^p L^q}$ separately. The degenerate components $[\mathcal{D}_\lambda]_{jj}$, $j=1,2$, are elliptic. This yields by Sobolev embedding the estimate:
\begin{equation*}
\begin{split}
\| [\mathcal{N}_\lambda^i S_\lambda^\tau S'_\lambda u]_j \|_{L_t^p L_x^q} &\lesssim \lambda^{\gamma - \frac{1}{p}} \| [ \mathcal{N}_\lambda^i S_\lambda^\tau S'_\lambda u]_j \|_{L^2_{t,x}} \\
&\lesssim \lambda^{\gamma - 1+ \frac{1}{p}} \| [\mathcal{D}_\lambda \mathcal{N}_\lambda^i S_\lambda^\tau S_{\lambda i} u]_j \|_{L^2_{t,x}} \\
&\lesssim \lambda^\gamma \| S_\lambda^\tau \mathcal{D}_\lambda \mathcal{N}_\lambda^i S_{\lambda i} u \|_{L^2_{t,x}}.
\end{split}
\end{equation*}
Another application of Lemma \ref{lem:Diagonalization} and Proposition \ref{prop:Diagonalization} yields
\begin{equation*}
\begin{split}
\lambda^\gamma \| S_\lambda^\tau \mathcal{D}_\lambda \mathcal{N}_\lambda^i S_{\lambda i} u \|_{L^2_{t,x}} &\lesssim \lambda^\gamma \| S_\lambda^\tau \mathcal{M}_\lambda^i \mathcal{D}_\lambda \mathcal{N}_\lambda^i S_{\lambda i } u \|_{L^2_{t,x}} \\
&\lesssim \lambda^\gamma \| S_\lambda^\tau \mathcal{P}_\lambda u \|_{L^2_{t,x}} + \lambda^\gamma \| S_\lambda^\tau u \|_{L^2_{t,x}}.
\end{split}
\end{equation*}
The non-degenerate components $j=3,\ldots,6$ are estimated by \cite[Eq.~(2.1)]{BlairSmithSogge2009}:
\begin{equation*}
\| S_\lambda^\tau \mathcal{N}_\lambda^i S_{\lambda i } u \|_{L^p L^q} \lesssim \lambda^\gamma ( \| S_\lambda^\tau \mathcal{N}_\lambda^i S_{\lambda i} u \|_{L_t^\infty L_x^2} + \| S_\lambda^\tau \mathcal{D}_\lambda \mathcal{N}_{\lambda i} S_{\lambda i} u \|_{L^2_{t,x}}).
\end{equation*}
By another application of Lemma \ref{lem:Diagonalization} and Proposition \ref{prop:Diagonalization}, we find
\begin{equation*}
\| S_\lambda^\tau \mathcal{N}_\lambda^i S_{\lambda i } u \|_{L^p L^q} \lesssim \lambda^\gamma (\| S_\lambda^\tau S'_\lambda u \|_{L_t^\infty L^2_x} + \| S_\lambda^\tau \mathcal{P}_\lambda S_\lambda u \|_{L^2_{t,x}}).
\end{equation*}
We passed from $S_{\lambda i}$ to $S_\lambda'$ above by first order symbol composition. This finishes the proof.
\end{proof}

\begin{proof}[Proof~of~\eqref{eq:FrequencyLocalizedEstimateB}]
If $\{|\tau| \ll |\xi'| \}$ and $\{ |\xi'| \gtrsim 1 \}$, we see that after diagonalization, the operator $\mathcal{P}$ is elliptic up to the charges. Let $\lambda \sim |\tau| \ll |\xi'| \sim \nu$. We make an additional localization in phase space: $S'_\nu u = \sum_{i=1}^3 S_{\nu i} u$.

\begin{equation*}
\| S_\lambda^\tau S_{\nu i} u \|_{L^p L^q} \lesssim \| S_\lambda^\tau \mathcal{M}^i_\nu \mathcal{N}_{\nu}^i S_{\nu i} u \|_{L^p L^q} + \| S_\lambda^\tau R(x,D') S_{\nu i} u \|_{L^p L^q}
\end{equation*}
with $R(x,D')$ being smoothing of order $-1$, we can use Sobolev embedding to find
\begin{equation*}
\| S_\lambda^\tau R(x,D') S_{\nu i} u \|_{L^p L^q} \lesssim \lambda^{\frac{1}{2}-\frac{1}{p}} \mu^{3 \big( \frac{1}{2}-\frac{1}{q} \big) - 1} \| S_\lambda^\tau S'_\nu u \|_{L^2_{t,x}} \lesssim \| \langle D' \rangle^\gamma S'_\nu u \|_{L^2_{t,x}}.
\end{equation*}
By Lemma \ref{lem:Diagonalization},
\begin{equation*}
\| S_\lambda^\tau \mathcal{M}_\nu^i \mathcal{N}_\nu^i S_{\nu i} u \|_{L^p L^q} \lesssim \| S_\lambda^\tau \mathcal{N}_\nu^i S_{\nu i} u \|_{L^p L^q}.
\end{equation*}
For $[\mathcal{N}_\nu^i S_{\nu i} u ]_j$ and $j=1,2$ we can use Sobolev embedding and definition of charges. For this purpose, recall the symbol of $\mathcal{N}_\nu^i$. With $\tilde{\xi}' = A^t_{<\nu} \xi'$, we find for $v \in \C^6$, $v=(v_1,v_2)^t$, with $v_i \in \C^3$:
\begin{equation*}
[ \tilde{m}_i^t(x,\xi_0,\tilde{\xi'}) \begin{pmatrix}
\varepsilon_{<\nu}^{\frac{1}{2}} & 0 \\
0 & \mu_{<\nu}^{\frac{1}{2}}
\end{pmatrix}
\begin{pmatrix}
A_{<\nu}^t & 0 \\
0 & A_{<\nu}^t
\end{pmatrix}
v]_1 = 
\frac{(\xi')^t}{\mu_{<\nu}^{\frac{1}{2}} \| \xi' \|} A_{<\nu} A_{<\nu}^t v_1.
\end{equation*}
Moreover,
\begin{equation*}
[ \tilde{m}_i^t(x,\xi_0,\tilde{\xi}') \begin{pmatrix}
\varepsilon_{<\nu}^{\frac{1}{2}} & 0 \\
0 & \mu_{<\nu}^{\frac{1}{2}}
\end{pmatrix}
\begin{pmatrix}
A_{<\nu}^t & 0 \\
0 & A_{<\nu}^t
\end{pmatrix}
v]_2 = 
\frac{(\xi')^t}{\varepsilon_{<\nu}^{\frac{1}{2}} \| \xi' \|} A_{<\nu} A_{<\nu}^t v_2.
\end{equation*}
Consequently, we can write
\begin{equation*}
[ \mathcal{N}_\nu^i S_\nu u ]_1 = \frac{1}{h_{<\nu} \varepsilon_{<\nu} \mu_{<\nu}^{\frac{1}{2}}} \frac{1}{|\nabla_{x'}|} \nabla \cdot ( h_{<\nu} \varepsilon_{<\nu} A_{<\nu} A_{<\nu}^t \mathcal{E}) + R_1(x,D) \mathcal{E}
\end{equation*}
with $\| R_1 \|_{L^2 \to L^2} \lesssim \nu^{-1}$. Therefore, the estimate for the first component follows from Sobolev embedding:
\begin{equation*}
\| [ \mathcal{N}_\nu^i S_{\nu i} u ]_1 \|_{L^p L^q} \lesssim \nu^{\gamma - 1+ \frac{1}{p}} \big( \| S'_\nu \rho_e \|_{L^\infty L^2} + \| S'_\nu u \|_{L^\infty L^2}).
\end{equation*}
Similarly,
\begin{equation*}
[ \mathcal{N}_\nu^i S_{\nu i} u ]_2 = \frac{1}{\varepsilon_{<\nu}^{\frac{1}{2}} h_{<\nu} \mu_{<\nu}^{\frac{1}{2}}} \frac{1}{|\nabla_{x'}|} \nabla \cdot ( h_{<\nu} \mu_{<\nu} A_{<\nu} A_{<\nu}^t \mathcal{H}) + R_2(x,D) \mathcal{H}
\end{equation*}
with $\| R_2 \|_{L^2 \to L^2} \lesssim \nu^{-1}$. We find by definition of $\rho'_{m \mu}$ and another Sobolev embedding yields
\begin{equation*}
\| [ \mathcal{N}_\nu^i S_{\nu i} u ]_2 \|_{L^p L^q} \lesssim \nu^{\gamma - \frac{1}{2}} \| S_\nu u \|_{L^2}.
\end{equation*}
For the components $i=3,\ldots,6$ $[\mathcal{D}_\nu]_{ii}$ is elliptic:
\begin{equation*}
\| S_\lambda^\tau \mathcal{N}_\nu S'_\nu u \|_{L^p L^q} \lesssim \nu^{3 \big( \frac{1}{2} - \frac{1}{q} \big)} \lambda^{\frac{1}{2}-\frac{1}{p}} \nu^{-1} \| S_\lambda^\tau \mathcal{D}_\nu^i S'_{\nu i} u \|_{L^2_{t,x}}.
\end{equation*}
Consequently, we obtain
\begin{equation*}
\| S_\lambda^\tau \mathcal{N}_\mu S'_\mu u \|_{L^p L^q} \lesssim \mu^{3 \big( \frac{1}{2} - \frac{1}{q} \big) - \frac{1}{p} - \frac{1}{2} + \varepsilon} \big( \frac{\lambda}{\mu} \big)^{\frac{1}{2}-\frac{1}{p}} \mu^{-\varepsilon} \| S_\lambda^\tau \mathcal{D}^i_{\mu} [\mathcal{N}_\mu S'_\mu u]_i \|_{L^2_{t,x}}.
\end{equation*}
By another application of Lemma \ref{lem:Diagonalization} and Proposition \ref{prop:Diagonalization}, we conclude the proof.

\end{proof}

\section{Diagonalizing reflected Maxwell equations in two dimensions}
\label{section:Diagonalization2d}

This section is devoted to the proof of Strichartz estimates in the two-dimensional case. We want to reduce to previously established results for half-wave equations with structured Lipschitz coefficients. We have already reduced to Proposition \ref{prop:ExtendedMaxwell2d} in Section \ref{section:MaxwellManifolds}, which states that it suffices to prove Strichartz estimates for the extended fields in geodesic coordinates $\tilde{u} = (\tilde{\mathcal{E}},\tilde{\mathcal{H}})$ close to the boundary:
\begin{equation*}
\| \tilde{u} \|_{L_T^p L_{x'}^q} \lesssim \| \tilde{u} \|_{L_T^\infty H^{\gamma + \delta}} + \| \tilde{\mathcal{J}}_e \|_{L_T^2 H^{\gamma + \delta}} + \| \tilde{\rho}_e \|_{L_T^\infty H^{\gamma - 1 + \frac{1}{p} + \delta}}
\end{equation*}
for $p,q \geq 2$, $q < \infty$ satisfying
\begin{equation*}
\frac{3}{p} + \frac{1}{q} \leq \frac{1}{2}, \quad \gamma = 2 \big( \frac{1}{2} - \frac{1}{q} \big) - \frac{1}{p}, \quad 0 <\delta < \frac{1}{2}.
\end{equation*}
We omit the $\tilde{\;}$ in the following for the extended quantities to lighten the notation.

For the diagonalization we can rely on results from \cite{SchippaSchnaubelt2022,Schippa2022ResolventEstimates}. Interestingly, in the two-dimensional case, there are no symmetry assumptions on the permittivity (the permeability is scalar anyway) required for a diagonalization with $L^p$-bounded multipliers to hold. Thus, we simply redenote the permittivity and permeability decorated with the cometric and $\sqrt{g}$ by $\varepsilon$ and $\mu$ to arrive at the Maxwell operator:
\begin{equation*}
\mathcal{P} = 
\begin{pmatrix}
\partial_t (\varepsilon^{11} \cdot) & 0 & -\partial_2 \\
0 & \partial_t (\varepsilon^{22} \cdot) & \partial_1 \\
-\partial_2 & \partial_1 & \partial_t (\mu \cdot)
\end{pmatrix}
.
\end{equation*}
The principal symbol of $\mathcal{P}$ with rough coefficients is given by
\begin{equation*}
\begin{split}
p(x,\xi) &= i
\begin{pmatrix}
\xi_0 \varepsilon^{11} & 0 & -\xi_2 \\
0 & \xi_0 \varepsilon^{22} & \xi_1 \\
-\xi_2 & \xi_1 & \xi_0 \mu
\end{pmatrix}
\\
&= i
\begin{pmatrix}
\xi_0 & 0 & -\xi_2/\mu \\
0 & \xi_0 & \xi_1/\mu \\
-\xi_2 \varepsilon_{11} & \xi_1 \varepsilon_{22} & \xi_0
\end{pmatrix}
\begin{pmatrix}
\varepsilon^{11} & 0 & 0 \\
0 & \varepsilon^{22} & 0 \\
0 & 0& \mu
\end{pmatrix}
.
\end{split}
\end{equation*}
On the level of the equation, the above factorization corresponds to rewriting the equation in terms of $(\mathcal{D},\mathcal{B})$ instead of $(\mathcal{E},\mathcal{H})$. It turns out that this facilitates to find conjugation matrices.
For the proof of Proposition \ref{prop:ExtendedMaxwell2d} it suffices to show the following estimate for frequency localized functions for $1 \ll \lambda \in 2^{\N_0}$:

\begin{proposition}
\label{prop:Microlocalization2d}
The following dyadic estimate holds:
\begin{equation}
\label{eq:DyadicEstimate2d}
\| S'_\lambda S_\lambda u \|_{L_T^p L_{x'}^q(\R^2)} \lesssim \lambda^\gamma ( \| S_\lambda S'_\lambda u \|_{ L^\infty_T L^2_{x'}} + \| \mathcal{P}_\lambda S_\lambda S'_\lambda u \|_{L^2_{T,x'}} ) + \lambda^{\gamma - 1 + \frac{1}{p}} \| \rho'_{e \lambda} \|_{L^\infty_T L^2_{x'}}
\end{equation}
with $\rho'_{e \lambda} = \nabla \cdot (\varepsilon_{< \lambda} S'_\lambda \mathcal{E})$.
\end{proposition}
\eqref{eq:DyadicEstimate2d} handles the contribution of the phase space region $\{|\tau| \lesssim |\xi'| \}$. The commutator arguments to remove the frequency localization are easier than in three dimensions because $\gamma < 1$ and thus, omitted. The estimate for $\{|\tau | \gg |\xi'| \}$ follows from ellipticity of $\mathcal{P}$ in this region in phase space and is carried out like in three dimensions. Secondly, note that we do not distinguish between definitions of $\mathcal{P}_\lambda^{(1)}$ or $\mathcal{P}_\lambda^{(2)}$ like in Section \ref{section:Diagonalization} because we actually do not use the internal structure of $\varepsilon' = \sqrt{g} g^{-1} \varepsilon$.

\subsection{Diagonalizing the principal symbol}
We use the diagonalization established in \cite{SchippaSchnaubelt2022} (see also \cite[Lemma~2.2]{Schippa2022ResolventEstimates}) to show the following:
\begin{proposition}
\label{prop:Diagonalization2d}
Let $2^{\N} \ni \lambda \gg \lambda_0$. There are operators $\mathcal{M}_\lambda \in OP \tilde{S}^0_{1,1}$, $\mathcal{N}_\lambda \in OP \tilde{S}^0_{1,1}$, and $\mathcal{D}_\lambda \in OP \tilde{S}^1_{1,1}$ such that
\begin{equation*}
\mathcal{P}_\lambda S_\lambda S'_\lambda = \mathcal{M}_\lambda \mathcal{D}_\lambda \mathcal{N}_\lambda + E_\lambda
\end{equation*}
with $\| E_\lambda \|_{L^2 \to L^2} \lesssim 1$ and implicit constant independent of $\lambda$. The principal symbols are given by
\begin{equation*}
\begin{split}
m(x,\xi) &=
\begin{pmatrix}
\varepsilon_{22} \xi_1^* & - \xi_2^* / \mu & \xi_2^* / \mu \\
\varepsilon_{11} \xi_2^* & \xi_1^* / \mu & - \xi_1^*/\mu \\
0 & - 1 & - 1
\end{pmatrix}, \\
n(x,\xi) &= 
\begin{pmatrix}
\mu^{-1} \xi_1^* & \mu^{-1} \xi_2^* & 0 \\
\frac{- \xi_2^* \varepsilon_{11}}{2} & \frac{\xi_1^* \varepsilon_{22}}{2} & -\frac{1}{2} \\
\frac{\xi_2^* \varepsilon_{11}}{2} & \frac{- \xi_1^* \varepsilon_{22}}{2} & - \frac{1}{2}
\end{pmatrix}
\begin{pmatrix}
\varepsilon^{11} & 0 & 0 \\
0 & \varepsilon^{22} & 0 \\
0 & 0 & \mu
\end{pmatrix},
\\
d(x,\xi) &= i \text{diag} (\xi_0, \xi_0 - \| \xi \|_{\varepsilon'}, \xi_0 + \| \xi \|_{\varepsilon'})
\end{split}
\end{equation*}
with $\| \xi \|^2_{\varepsilon'} = \langle \xi, \mu^{-1} \det(\varepsilon)^{-1} \varepsilon \xi \rangle$, $\xi^* = \xi / \| \xi \|_{\varepsilon'}$. All coefficients in the above definitions are frequency truncated at $\lambda$.
\end{proposition}
The diagonalization is substantially easier than in three dimensions because it does not require an additional localization in phase space.

\subsection{Conclusion of the proof}
To finish the proof of Theorem \ref{thm:StrichartzEstimatesMaxwell2d} like in Section~\ref{section:Diagonalization}, we have to check that the contribution of the charges is ameliorated like before:
\begin{proposition}
With the notations from Proposition \ref{prop:Diagonalization2d}, the following estimate holds:
\begin{equation}
\label{eq:DiagonalizedEstimate2d}
\| \mathcal{N}_\lambda S_\lambda u \|_{L_t^p L_{x'}^q} \lesssim \lambda^\gamma ( \| \mathcal{N}_\lambda S_\lambda u \|_{L_{x}^2} + \| \mathcal{D}_\lambda \mathcal{N}_\lambda S_\lambda u \|_{L_{x}^2} ) + \lambda^{\gamma - 1+ \frac{1}{p}} \| \rho_{e\lambda} \|_{L_t^\infty L^2_{x'}}.
\end{equation}
\end{proposition}
\begin{proof}
We show \eqref{eq:DiagonalizedEstimate2d} componentwise. For the first component we have to use the divergence condition: We have
\begin{equation*}
\begin{split}
[n(x,\xi)]_{11} &= \frac{\xi_1 \varepsilon^{11}}{\mu \| \xi \|_{\varepsilon'}}, \quad [n(x,\xi)]_{12} = \frac{\xi_2 \varepsilon^{22}}{\mu \| \xi \|_{\varepsilon'}}, \\
[n(x,\xi)]_{13} &= 0.
\end{split}
\end{equation*}
This gives
\begin{equation*}
[\mathcal{N}_\lambda S_\lambda u]_1 = \frac{1}{\mu |\nabla_{\varepsilon'}|} [\nabla \cdot (\varepsilon S_\lambda u )] + R_1(x,D) \mathcal{E}
\end{equation*}
with $\| R_1 \|_{2 \to 2} \lesssim \lambda^{-1}$. This yields the estimate for the first component by Sobolev embedding:
\begin{equation*}
\| [\mathcal{N}_\lambda S_\lambda u]_1 \|_{L_t^p L_{x'}^q} \lesssim \lambda^{\gamma - 1 + \frac{1}{p}} ( \| S_\lambda \rho_e \|_{L_t^\infty L_{x'}^2} + \| S_\lambda u \|_{L_t^\infty L_{x'}^2} ).
\end{equation*}
The non-degenerate components $i=2,3$ are estimated by Theorem \ref{thm:LipschitzSingularHalfWave}:
\begin{equation*}
\| [\mathcal{N}_\lambda S_\lambda u]_i \|_{L_t^p L_{x'}^q} \lesssim \lambda^\gamma ( \| S_\lambda u \|_{L_t^\infty L_{x'}^2} + \| \mathcal{D}^i_\lambda [\mathcal{N}_\lambda S_\lambda u]_i \|_{L_x^2} ).
\end{equation*}
The proof is complete.
\end{proof}
We record the corresponding result of Lemma \ref{lem:Diagonalization} to complete the proof of Theorem \ref{thm:StrichartzEstimatesMaxwell2d}.
\begin{lemma}
\label{lem:DiagonalizationErrorEstimatesII}
For $\lambda \gg 1$, the following estimates hold:
\begin{equation*}
\begin{split}
\| S_{\lambda} u \|_{L_t^p L_x^q} &\lesssim \| \mathcal{N}_\lambda S_{\lambda} u \|_{L_t^p L_x^q} + \lambda^{\gamma - \frac{1}{2}} \| S_{\lambda} u \|_{L_{t,x}^2}, \\
\| S_{\lambda} u \|_{L_{t,x}^2} &\lesssim \| \mathcal{M}_\lambda S_{\lambda } u \|_{L_{t,x}^2}.
\end{split}
\end{equation*}
\end{lemma}
The lemma is proved like in the previous section.

\section{Improved local well-posedness for the Kerr system in two dimensions}
\label{section:ImprovedLWP}
This section is devoted to the proof of the following theorem.
\begin{theorem}
\label{thm:ImprovedLWPKerr2d}
Let $\Omega \subseteq \R^2$ be a smooth domain with compact boundary and $s \in (11/6,2]$. Then the Kerr system in two dimensions
\begin{equation}
\label{eq:KerrSystem2d}
\left\{ \begin{array}{cllcl}
\partial_t (\varepsilon \mathcal{E}) &= \nabla_\perp \mathcal{H}, &\quad [ \mathcal{E} \wedge \nu ]_{x' \in \partial \Omega} &=& 0, \quad (t,x') \in \R \times \Omega, \\
\partial_t \mathcal{H} &= -( \partial_1 \mathcal{E}_2 - \partial_2 \mathcal{E}_1), &\quad \text{tr}_{\partial \Omega}(\rho_e) &=& 0
\end{array} \right.
\end{equation}
with $\varepsilon(\mathcal{E}) = 1 + |\mathcal{E}|^2$ and $(\mathcal{E},\mathcal{H})(0) = (\mathcal{E}_0,\mathcal{H}_0) \in \mathcal{H}_0^s(\Omega)$ is locally well-posed provided that $\| (\mathcal{E}_0,\mathcal{H}_0) \|_{H^s} \leq \delta \ll 1$ and $\| \rho_e(0) \|_{H^{\tilde{s}}} \leq D < \infty$ for some $\tilde{s} > \frac{13}{12}$. This means there is $T=T(\|(\mathcal{E}_0,\mathcal{H}_0) \|_{H^s},D)$ such that the solution $(\mathcal{E},\mathcal{H})$ to \eqref{eq:KerrSystem2d} exists for $0 \leq t \leq T$ and for initial data $(\mathcal{E}^{(i)},\mathcal{H}^{(i)}) \in \mathcal{H}_0^s(\Omega)$, $i=1,2$ with $\|(\mathcal{E}^i,\mathcal{H}^i) \|_{H^s} \leq \delta$ and $\| \rho_e^i(0) \|_{H^{\tilde{s}}} \leq D$ we have
\begin{equation*}
\sup_{t \in [0,T]} \| (\mathcal{E}^1,\mathcal{H}^1)(t) - (\mathcal{E}^2,\mathcal{H}^2)(t) \|_{H^s(\Omega)} \to 0
\end{equation*}
for $\| (\mathcal{E}^1, \mathcal{H}^1) - (\mathcal{E}^2,\mathcal{H}^2) \|_{H^s} \to 0$.
\end{theorem}

Spitz proved continuous dependence in $\mathcal{H}^3(\Omega)$ in three dimensions: Recall that this refers to data in $H^3(\Omega)$ satisfying the compatibility conditions up to second order derived from the boundary condition.
In Appendix \ref{appendix:LWPHighRegularity} we establish local well-posedness in $\mathcal{H}^3(\Omega)$:
\begin{theorem}
\label{thm:LocalWellposednessH3Maxwell2d}
Let $\Omega \subseteq \R^2$ be a smooth domain with compact boundary. Then \eqref{eq:KerrSystem2d} is locally well-posed in $\mathcal{H}^3(\Omega)$. This means there is $T=T(\| (\mathcal{E}_0,\mathcal{H}_0) \|_{H^3(\Omega)})$ such that solutions $(\mathcal{E}^i,\mathcal{H}^i)$, $i=1,2$ exists for $0 \leq t \leq T$ and depend continuously on the initial data: We have
\begin{equation*}
\sup_{t \in [0,T]} \| (\mathcal{E}^1,\mathcal{H}^1)(t) - (\mathcal{E}^2,\mathcal{H}^2)(t) \|_{H^3(\Omega)} \to 0
\end{equation*}
for $\| (\mathcal{E}^1,\mathcal{H}^1)(0) - (\mathcal{E}^2,\mathcal{H}^2)(0) \|_{H^3(\Omega)} \to 0$.
\end{theorem}

In the following we take the local existence of these (sufficiently smooth) solutions for granted and want to examine behavior in rougher topologies. The argument to show local well-posedness for $11/6 < s \leq 2$ proceeds in three steps:
\begin{itemize}
\item[1.] We require estimates for solutions
\begin{equation}
\label{eq:EnergyEstimatesLWP}
\| u(t) \|_{H^s(\Omega)} \lesssim_{\delta,T} e^{\int_0^T \| \partial_x \mathcal{E}(\tau) \|_{L^\infty(\Omega)} d\tau} \| u(0) \|_{H^s(\Omega)}
\end{equation}
for $s \in [0,2]$. These were already proved in Section \ref{section:EnergyEstimates}. By using Strichartz estimates, these give a priori estimates for $s \in ( \frac{11}{6},2]$ (see Proposition \ref{prop:APrioriEstimate2dKerr}).
\item[2.] We prove Lipschitz-continuous dependence in $L^2$ for initial data in $\mathcal{H}^s(\Omega)$, $s \in ( \frac{11}{6} , 2]$ (see Proposition \ref{prop:L2Lipschitz}).
\item[3.] We show continuous dependence, but no uniform continuous dependence via frequency envelopes (Subsection \ref{subsection:FrequencyEnvelopes}). Here we use a regularization, which respects the compatibility conditions. This is facilitated by working in $\mathcal{H}_0^s(\Omega)$ and would be more delicate in $\mathcal{H}^s(\Omega)$. This is the only step in the proof, which uses that the initial data are in the smaller space $\mathcal{H}_0^s(\Omega)$.
\end{itemize}
We note that for $s>2$ we have $\| \partial_x \mathcal{E} \|_{L^\infty(\Omega)} \lesssim \| \mathcal{E} \|_{H^s(\Omega)}$ by Sobolev embedding, and Strichartz estimates are not required.
 For this reason, we shall only prove Theorem \ref{thm:ImprovedLWPKerr2d} as formulated for $s \in (11/6,2]$.

\medskip

\subsection{A priori control of solutions via Strichartz estimates}

In this subsection we show the following:
\begin{proposition}
\label{prop:APrioriEstimate2dKerr}
Let $11/6 < s \leq 2$ and $\tilde{s} > \frac{13}{12}$. Then there is $\delta>0$ such that for an $\mathcal{H}^3(\Omega)$-solution to \eqref{eq:KerrSystem2d} with $\| u(0) \|_{H^s(\Omega)} \leq \delta$ we have
\begin{equation}
\label{eq:APrioriEstimate2dKerr}
\sup_{t \in [0,T]} \| u(t) \|_{H^s(\Omega)} \lesssim \| u(0) \|_{H^s(\Omega)}
\end{equation}
with $T=T(\| u(0) \|_{H^s(\Omega)}, \| \rho_e(0) \|_{H^{\tilde{s}}})$.
\end{proposition}
We first argue how the proof is finished with Strichartz estimates at hand:
\begin{proof}[Conclusion of proof with Strichartz estimates]
By finite speed of propagation, it suffices to prove the claim in charts. The interior of $\Omega$ can be handled like in $\R^2$ (see \cite{SchippaSchnaubelt2022}). It suffices to prove \eqref{eq:APrioriEstimate2dKerr} in a chart and written in geodesic coordinates. To this end, let $\sup_{t \in [0,T]} \| u(t) \|_{H^s(\Omega)} \leq \tilde{\delta} \ll 1$. We use a bootstrap argument based on the estimates:
\begin{align}
\label{eq:StrichartzEstimatesKerr2d}
\| \partial_x u \|_{L_T^4 L_{x'}^\infty} &\leq C_1(\tilde{\delta},T,s_1) ( \| u \|_{L_T^\infty H^{s_1}} + \| \rho_e(0) \|_{H^{\tilde{s}}}), \quad \frac{11}{6} < s_1 \leq 2, \\
\label{eq:EnergyEstimateKerr2d}
\sup_{t \in [0,T]} \| u(t) \|_{H^s(\Omega)} &\leq C(\tilde{\delta},T) e^{C_2(\tilde{\delta}) \int_0^T \| \partial_x u(t') \|_{L^\infty(\Omega)} dt'} \| u(0) \|_{H^s(\Omega)}, \quad s \in [0,2], \\
\label{eq:EnergyEstimateKerrHigh}
\sup_{t \in [0,T]} \| u(t) \|_{H^3(\Omega)} &\leq C(\tilde{\delta},T) e^{C_3(\tilde{\delta},T) e^{\int_0^T \| \partial_x u(t') \|_{L^\infty(\Omega)} dt'} (\| u(0) \|_{H^2(\Omega)} + 1) } \| u(0) \|_{H^3(\Omega)}.
\end{align}
The crucial Strichartz estimate \eqref{eq:StrichartzEstimatesKerr2d} will be proved below. The energy estimate \eqref{eq:EnergyEstimateKerr2d} was established in the proof of Proposition \ref{prop:APrioriMaxwell2Kerr}. We need \eqref{eq:EnergyEstimateKerrHigh} to argue that the $H^3$-solutions exist for the same time like the solutions at lower regularity.

Recall that \eqref{eq:EnergyEstimateKerr2d} and \eqref{eq:EnergyEstimateKerrHigh} require a smallness condition, which is ensured by $\tilde{\delta} \ll 1$. Once we control the $H^s$-norm via a continuity argument and keep it small, the assumptions on the $H^s$-norm will be satisfied (this requires smallness of the initial data). Moreover, we can choose the constants uniform in $\tilde{\delta}$ and $T$, if these parameters are bounded. Taking \eqref{eq:StrichartzEstimatesKerr2d} and \eqref{eq:EnergyEstimateKerr2d} together, we find
\begin{equation}
\label{eq:BootstrapA}
\| \partial_x u \|_{L_T^4 L_{x'}^\infty(\Omega)} \leq C_1 C e^{C_2 \| \partial_x u \|_{L_T^1 L_{x'}^\infty}} \big( \| u(0) \|_{H^s(\Omega)} + \| \rho_e(0) \|_{H^{\tilde{s}}} \big)
\end{equation}
for $\frac{11}{6} < s \leq 2$ and $\tilde{s} > \frac{13}{12}$. We argue as follows: Suppose that the $H^3$-solution exists on $[0,\tau]$.
By continuity of $\partial_x u$ in $L^\infty(\Omega)$ (recall that $u \in C([0,\tau],H^3(\Omega))$ and use Sobolev embedding), there is $T_0$ such that $\| \partial_x u \|_{L_{T_0}^4 L_{x'}^\infty(\Omega)} \leq e C_1 C ( \| u(0) \|_{H^s(\Omega)} + \| \rho_e(0) \|_{H^{\tilde{s}}} )$. We want to extend this badly quantified time interval to $[0,T_{\max}]$, $T_{\max} = T(\| u(0) \|_{H^s(\Omega)}, \| \rho_e(0) \|_{H^{\tilde{s}}})$\footnote{The time moreover depends on $\Omega$, but this is suppressed in the following.} such that
\begin{equation*}
T_{\max}^{3/4} C_2 C_1 C ( \| u(0) \|_{H^s(\Omega)} + \| \rho_e(0) \|_{H^{\tilde{s}}}) \leq \frac{1}{8}.
\end{equation*}
Suppose that $T_0 < T_{\max}$ (otherwise we are done). Then we have for $T_0$ actually the improved estimate
\begin{equation*}
\begin{split}
\| \partial_x u \|_{L_{T_0}^4 L_{x'}^\infty} &\leq C_1 C e^{C_2 T_0^{\frac{3}{4}} C C_1 ( \| u_0 \|_{H^s(\Omega)} + \| \rho_e(0) \|_{H^{\tilde{s}}}) } (\| u(0) \|_{H^s(\Omega)} + \| \rho_e(0) \|_{H^{\tilde{s}}}) \\
&\leq C_1 C e^{\frac{1}{8}} (\| u(0) \|_{H^s(\Omega)} + \| \rho_e(0) \|_{H^{\tilde{s}}}). 
\end{split}
\end{equation*}
Moreover, by finiteness of $\| \partial_x u \|_{L_{T_0}^1 L^\infty_{x'}}$ we have that $\tau \geq T_0$.

This allows us to continue up to a time $T_1$ such that
\begin{equation*}
\| \partial_x u \|_{L_{T_1}^4 L_{x'}^\infty} \leq 2 C_1 C (\| u(0) \|_{H^s(\Omega)} + \| \rho_e(0) \|_{H^{\tilde{s}}}).
\end{equation*}
This can be bootstrapped up to a time $T_{\max} = T(\| u_0 \|_{H^s(\Omega)}, \| \rho_e(0) \|_{H^{\tilde{s}}})$ and gives the estimate
\begin{equation*}
\| \partial_x u \|_{L_{T_{\max}}^4 L_{x'}^\infty(\Omega)} \lesssim ( \| u(0) \|_{H^s(\Omega)} + \| \rho_e(0) \|_{H^{\tilde{s}}}).
\end{equation*}
By \eqref{eq:EnergyEstimateKerr2d} for $s=3$, we also infer existence of solutions in $H^3$ for the same time. The a priori estimate is immediate from \eqref{eq:EnergyEstimateKerr2d}.
\end{proof}

The Strichartz estimates for $(\mathcal{E},\mathcal{H})$ in \eqref{eq:StrichartzEstimatesKerr2d} play the key role in the argument. We resolve the Kerr system
\begin{equation*}
\left\{ \begin{array}{clclcl}
\partial_t (\varepsilon \mathcal{E}) &= \nabla_\perp \mathcal{H}, &\quad \nabla \cdot (\varepsilon \mathcal{E}) &= \rho_e, &\quad (t,x') &\in \R \times \Omega, \\
\partial_t \mathcal{H} &= - (\nabla \times \mathcal{E})_3, &\quad [\mathcal{E}_{||} ]_{x' \in \partial \Omega} &= 0, &\quad (\mathcal{E},\mathcal{H})(0) &= (\mathcal{E}_0,\mathcal{H}_0) \in \mathcal{H}^s(\Omega)
\end{array} \right.
\end{equation*}
in geodesic coordinates as follows: Let $\varphi: \Omega \to \R^2_{>0}$ denote the change of coordinates $x' = \varphi(x)$ and $J(x') = \frac{\partial \varphi}{\partial x}(\varphi^{-1}(x'))$ the Jacobian. The change of coordinates from Section \ref{section:MaxwellManifolds} reads
\begin{equation*}
\mathcal{E}'(x') = (J^t)^{-1} \mathcal{E}(\varphi^{-1}(x')), \quad \mathcal{H}'(x') = \mathcal{H}(\varphi^{-1}(x')), \quad \varepsilon' = \frac{J \varepsilon J^t}{\det J}.
\end{equation*}
The cometric is given by
\begin{equation*}
g^{-1} =  JJ^t = 
\begin{pmatrix}
g^1 & 0 \\
0 & g^2
\end{pmatrix}.
\end{equation*}
The change of coordinates leads us to
\small
\begin{equation*}
\left\{ \begin{array}{clclcl}
\partial_t (\varepsilon' \mathcal{E}') &= \nabla_\perp \mathcal{H}', &\quad \frac{1}{\sqrt{g}} \nabla \cdot (\varepsilon' \mathcal{E}') &= \rho_e', &\quad (t,x') &\in \R \times \R^2_{>0}, \\
\partial_t (\mu_1 \mathcal{H}')= &= - (\nabla \times \mathcal{E}')_3, &\quad [\mathcal{E}_1']_{x_2'= 0} &= 0, &\quad (\mathcal{E}',\mathcal{H}')(0) &= (\mathcal{E}_0',\mathcal{H}_0').
\end{array} \right.
\end{equation*}
\normalsize
We compute
\begin{equation*}
\mu_1 = \sqrt{g}, \quad \varepsilon' = \frac{J (1+|\mathcal{E}|^2) J^t}{\det J} = \frac{g^{-1} ( 1 + \langle \mathcal{E}', JJ^t \mathcal{E}' \rangle)}{\det J}= \frac{g^{-1} (1 + \langle \mathcal{E}', g^{-1} \mathcal{E}' \rangle)}{\det J}.
\end{equation*}
We extend the system to $\R^2$ by reflecting $\mathcal{E}'_1$ oddly and $\mathcal{E}_2'$ and $\mathcal{H}'$ evenly (according to the boundary conditions). Moreover, the coefficients of $g$ are extended evenly. Consequently, $\varepsilon'$ is extended evenly.

It will be important to work in non-divergence form, to which end we compute
\begin{equation*}
\partial_t (\varepsilon' \mathcal{E}') = \varepsilon_1 \partial_t \mathcal{E}' \text{ with } \varepsilon_1 = \sqrt{g}( \varepsilon' + 2 (g^{-1} \mathcal{E}') \otimes (g^{-1} \mathcal{E}')).
\end{equation*}
Note that $\varepsilon_1$ is still diagonal at $x_2 = 0$. The diagonal components are reflected evenly, the off-diagonal ones oddly. Importantly, we observe that for $(\mathcal{E},\mathcal{H}) \in C_t H^3$ we have $\partial_x \varepsilon_1 \in L_t^2 L_{x'}^\infty$ by Sobolev embedding. Using Strichartz estimates for coefficients $\partial_x \varepsilon_1 \in L_t^2 L_{x'}^\infty$ and energy estimates, we will prove bounds depending on $\| u(0) \|_{H^s}$. It is non-trivial to verify
\begin{equation*}
\| \partial_x \varepsilon_1 \|_{L_T^2 L_{x'}^\infty} \leq C(\| (\mathcal{E},\mathcal{H}) \|_{H^s(\Omega)}),
\end{equation*}
which is carried out below.

We let
\begin{equation}
\label{eq:DefinitionMaxwell2dAniso}
P_1 = \begin{pmatrix}
\partial_t (\varepsilon_1 \cdot) & - \nabla_\perp \\
(\nabla \times \cdot)_3 & \partial_t (\mu \cdot)
\end{pmatrix}
\text{ with } \partial_x (\varepsilon_1,\mu) \in L_t^2 L_{x'}^\infty, \text{ and } \rho_e' = \nabla \cdot (\varepsilon_1 \mathcal{E}').
\end{equation}

The analysis of \cite{SchippaSchnaubelt2022} provides us with the following Strichartz estimates; see Corollary \ref{cor:Strichartz2dL2Lipschitz} in Appendix \ref{appendix:StrichartzRevisited}:
\begin{proposition}[Strichartz~estimates~for~anisotropic~permittivity~in~two~dimensions]
\label{prop:Strichartz2dAnisotropic}
Let $2 \leq p,q \leq \infty$, $\frac{2}{p} + \frac{1}{q} \leq \frac{1}{2}$, $\rho = 2 \big( \frac{1}{2} - \frac{1}{q} \big) - \frac{1}{2}$, $P_1$ like in \eqref{eq:DefinitionMaxwell2dAniso}, and $\delta > 0$. Then the following estimate holds:
\begin{equation*}
\begin{split}
\| \langle D' \rangle^{-\rho - \frac{1}{3p} - \delta} (\mathcal{E}',\mathcal{H}') \|_{L_T^p L_x^q} &\lesssim_{T,\delta, \| \partial (\varepsilon_1,\mu_1) \|_{L_T^2 L_x^\infty}} \| (\mathcal{E}',\mathcal{H}') \|_{L_T^\infty L_x^2} + \| P_1(\mathcal{E}',\mathcal{H}') \|_{L_T^1 L_x^2} \\
&\quad + \| \langle D' \rangle^{-1+\frac{1}{p}} \rho'_e \|_{L_T^\infty L_{x'}^2} + \| \langle D' \rangle^{-1+\frac{1}{p}} \partial_t \rho'_e \|_{L_T^1 L_{x'}^2}.
\end{split}
\end{equation*}
\end{proposition}
Since the argument is essentially contained in \cite{SchippaSchnaubelt2022} and Appendix \ref{appendix:StrichartzRevisited} up to the change of variables $\mathcal{D} = \varepsilon_1 \mathcal{E}$, $\mathcal{B} = \mu_1 \mathcal{H}$, we shall be brief. 
\begin{proof}[Sketch~of~Proof~of~Proposition~\ref{prop:Strichartz2dAnisotropic}]
The idea is to show the estimates first for coefficients $\varepsilon$ with $\partial_x^2 \varepsilon \in L_t^1 L_x^\infty$ and then use paradifferential truncation. In the following we omit $'$ to lighten the notations. After standard reductions, which are detailed in \cite{SchippaSchnaubelt2022}, we find that it suffices to prove:
\begin{equation}
\label{eq:DyadicEstimate2dAnisotropic}
\lambda^{-\rho} \| S_\lambda u \|_{L_T^p L_{x'}^q} \lesssim \| S_\lambda u \|_{L_T^\infty L_{x'}^2} + \| P_1^\lambda S_\lambda u \|_{L^2_x} + \lambda^{-1+\frac{1}{p}} \| S'_\lambda \rho_e \|_{L_t^\infty L^2_{x'}}.
\end{equation}
Above we require $\lambda \gtrsim 1$ with Fourier support of $\varepsilon$ contained in $\{ |\xi| \leq \lambda^{\frac{1}{2}} \}$ and $u$ essentially supported in the unit cube. Moreover, we can suppose that $\| \partial_x^2 \varepsilon \|_{L_T^1 L_x^\infty} \lesssim 1$, that the coefficients of $P_1^\lambda$ are truncated at frequencies $\lambda^{\frac{1}{2}}$, and we suppose by an elliptic estimate away from the characteristic surface that the support of the space-time Fourier transform of $u$ is in $\{ |\xi_0| \lesssim |(\xi_1,\xi_2)| \}$. The estimate \eqref{eq:DyadicEstimate2dAnisotropic} is then proved by diagonalization with pseudo-differential operators:
We factorize the principal symbol
\begin{equation*}
\begin{split}
\tilde{p}(x,\xi) &= 
\begin{pmatrix}
i \xi_0 \varepsilon^{11} & i \xi_0 \varepsilon^{12} & - i \xi_2 \\
i \xi_0 \varepsilon^{21} & i \xi_0 \varepsilon^{22} & i \xi_1 \\
i \xi_2 & - i \xi_1 & i \xi_0 \mu
\end{pmatrix}
\\
&=
\begin{pmatrix}
i \xi_0 & 0 & - i \frac{\xi_2}{\mu} \\
0 & i \xi_0 & i \frac{\xi_1}{\mu} \\
i (\varepsilon_{11} \xi_2 - \varepsilon_{12} \xi_1) & i (\varepsilon_{21} \xi_2 - \varepsilon_{22} \xi_1) & i \xi_0
\end{pmatrix}
\begin{pmatrix}
\varepsilon^{11} & \varepsilon^{12} & 0 \\
\varepsilon^{21} & \varepsilon^{22} & 0 \\
0 & 0 & \mu
\end{pmatrix}
\\
 &=: p(x,\xi) (\varepsilon \otimes \mu).
 \end{split}
\end{equation*}
The symbol $p(x,\xi)$ was diagonalized in \cite{SchippaSchnaubelt2022} (see also Appendix \ref{appendix:StrichartzRevisited}) as
\begin{equation*}
p(x,\xi) = m(x,\xi) \text{diag} (i \xi_0, i (\xi_0 + \| \xi' \|_{\varepsilon'}), i(\xi_0 - \| \xi' \|_{\varepsilon'})) m^{-1}(x,\xi)
\end{equation*}
with quantizations of $m$ and $m^{-1}$ giving bounded operators in $L^p L^q$. We can write
\begin{equation}
\label{eq:DecompositionPEps}
P_1^\lambda = P^\lambda (\varepsilon_{\leq \lambda^{\frac{1}{2}}} \otimes \mu_{\leq \lambda^{\frac{1}{2}}}) + A, \quad \| A \|_{L^2 \to L^2} \lesssim 1.
\end{equation}
We define new variables $S_\lambda (\mathcal{D},\mathcal{H}) = (\varepsilon_{\leq \lambda^{\frac{1}{2}}} \otimes \mu_{\leq \lambda^{\frac{1}{2}}}) S_\lambda (\mathcal{E},\mathcal{H})$, to which we apply Theorem \ref{thm:2dStrichartzFull}:
\begin{equation*}
\begin{split}
\lambda^{-\rho} \| S_\lambda (\mathcal{D},\mathcal{H}) \|_{L^p L^q} &\lesssim \| S_\lambda (\mathcal{D},\mathcal{H}) \|_{L^\infty_T L_{x'}^2} + \| P^{\lambda} S_\lambda (\mathcal{D},\mathcal{H}) \|_{L^2_x} \\
&\quad + \lambda^{-1+\frac{1}{p}} \| S'_\lambda S_\lambda (\partial_1 \mathcal{D}_1 + \partial_2 \mathcal{D}_2) \|_{L^\infty_t L_{x'}^2}.
\end{split}
\end{equation*}
By \eqref{eq:DecompositionPEps} and straight-forward error estimates, we find
\begin{equation*}
\lambda^{-\rho} \| S_\lambda (\mathcal{E},\mathcal{H}) \|_{L^p L^q} \lesssim \| S_\lambda (\mathcal{E},\mathcal{H}) \|_{L^\infty_T L_{x'}^2} + \| \tilde{P}_{\leq \lambda^{\frac{1}{2}}} S_\lambda (\mathcal{E},\mathcal{H}) \|_{L^2_x} + \lambda^{-1+\frac{1}{p}} \| S'_\lambda \rho_e \|_{L^\infty_t L_{x'}^2}.
\end{equation*}
\end{proof}
We are ready to conclude the proof of Proposition \ref{prop:APrioriEstimate2dKerr} by applying the Strichartz estimates.
\begin{proof}[Proof~of~\eqref{eq:StrichartzEstimatesKerr2d}~via~Proposition~\ref{prop:Strichartz2dAnisotropic}]
It suffices to prove \eqref{eq:StrichartzEstimatesKerr2d} in geodesic coordinates. Using Maxwell equations, we have 
\begin{equation*}
\| \partial_t (\mathcal{E},\mathcal{H}) \|_{L_T^4 L_{x'}^\infty} \lesssim_{\| \mathcal{E} \|_{L_T^\infty L^\infty_{x'}}} \| \partial_{x'} (\mathcal{E},\mathcal{H}) \|_{L_T^4 L_{x'}^\infty}.
\end{equation*}

Since we can control the $L^\infty_{x'}$-norm of $(\mathcal{E},\mathcal{H})$ by Sobolev embedding, it suffices to show an estimate for the spatial derivatives of $(\mathcal{E},\mathcal{H})$. We apply Strichartz estimates due to Proposition \ref{prop:Strichartz2dAnisotropic} to $(\mathcal{E},\mathcal{H})$ resolved in geodesic normal coordinates and extended appropriately to the full space (see above). Recall that the extended fields are denoted with $(\mathcal{E}',\mathcal{H}')$. We aim for the estimate:
\begin{equation}
\label{eq:StrichartzEstimateSpatialDerivative}
\begin{split}
\| \langle D' \rangle (\mathcal{E}',\mathcal{H}') \|_{L_T^4 L_{x'}^\infty(\R^2)} &\lesssim_{T,\alpha} (1 + T^\kappa C(\| (\mathcal{E}',\mathcal{H}') \|_{L_T^\infty H^{\alpha + 1}}, \| \partial_{x'} ( \mathcal{E}',\mathcal{H}') \|_{L_T^4 L_{x'}^\infty} )) \\
&\quad \times \| (\mathcal{E}',\mathcal{H}') \|_{L_T^\infty H^{\alpha+1}} + T^{\frac{1}{4}} \| \rho_e(0) \|_{H^{\tilde{s}}}.
\end{split}
\end{equation}

Applying Proposition \ref{prop:Strichartz2dAnisotropic} requires uniform ellipticity of $\varepsilon_1$ and an estimate of $\| \partial_x \varepsilon_1 \|_{L_T^2 L_{x'}^\infty}$ with $T=T(\| (\mathcal{E}_0,\mathcal{H}_0) \|_{\mathcal{H}^s(\Omega)})$.

\medskip

\emph{Control of $\| \partial_x \varepsilon_1 \|_{L_T^2 L_{x'}^\infty(\R^2)}$ and uniform ellipticity of $\varepsilon_1$:} We require an estimate $\| \partial_x \varepsilon_1 \|_{L_T^2 L_{x'}^\infty(\R^2)} \leq C(\| (\mathcal{E}_0,\mathcal{H}_0) \|_{H^s(\Omega)})$ for $\frac{11}{6} < s \leq 2$ and $T=T(\| (\mathcal{E}_0,\mathcal{H}_0) \|_{H^s(\Omega)})$. By the bootstrap assumption it suffices to prove that there is $\kappa > 0$ such that
\begin{equation*}
\| \partial_x \varepsilon_1 \|_{L_T^2 L_{x'}^\infty} \lesssim T^\kappa C( \| \partial_{x'} (\mathcal{E}',\mathcal{H}') \|_{L_T^4 L_{x'}^\infty(\R^2)}, \; \| (\mathcal{E}',\mathcal{H}') \|_{L_T^\infty H^s}).
\end{equation*}
The time derivative of $\varepsilon_1$ is handled by Sobolev embedding, H\"older's inequality, and Maxwell equations:
\begin{equation*}
\| \partial_t \varepsilon_1 \|_{L_T^2 L_{x'}^\infty} \lesssim T^{\frac{1}{4}} \| \mathcal{E}' \|_{L_T^\infty L_{x'}^\infty} \| \partial_t \mathcal{E}' \|_{L_T^4 L_{x'}^\infty} \lesssim T^{\frac{1}{4}} \| (\mathcal{E}',\mathcal{H}') \|_{L_T^\infty H^s} \| \partial_{x'} (\mathcal{E}',\mathcal{H}') \|_{L_T^4 L_{x'}^\infty}.
\end{equation*}
Similarly, we find for the spatial derivatives:
\begin{equation*}
\begin{split}
\| \partial_{x'} \varepsilon_1 \|_{L_T^2 L_{x'}^\infty(\R^2)} &\lesssim T^{\frac{1}{2}} \| \mathcal{E}' \|_{L_T^\infty L_{x'}^\infty}^2 + \| \partial_{x'} \mathcal{E}' \|_{L_T^2 L_{x'}^\infty(\R^2)} \| \mathcal{E}' \|_{L_T^\infty L_{x'}^\infty(\R^2)} \\
&\lesssim T^{\frac{1}{2}} \| (\mathcal{E}',\mathcal{H}') \|_{L_T^\infty H^s}^2 + T^{\frac{1}{4}} \| \partial_{x'} \mathcal{E}' \|_{L_T^4 L_{x'}^\infty(\R^2)} \| (\mathcal{E}',\mathcal{H}') \|_{L_T^\infty H^s}.
\end{split}
\end{equation*}

Ellipticity of $\varepsilon_1$ follows from recalling $g^{-1} = JJ^t$ and writing
\begin{equation*}
\begin{split}
\varepsilon_1 &= \sqrt{g}( J (1 + |\mathcal{E}'|^2_{g^{-1}}) J^t + 2 J (J^t \mathcal{E}') (\mathcal{E}')^t J J^t )\\
&= \sqrt{g} ( J (( 1+ \langle J^t \mathcal{E}', J^t \mathcal{E}' \rangle) + 2 (J^t \mathcal{E}') (J^t \mathcal{E}')^t) J^t).
\end{split}
\end{equation*}
Hence, uniform ellipticity follows by uniform invertibility of $J$, and uniform ellipticity of $1 + |\mathcal{E}|^2 + 2 \mathcal{E} \otimes \mathcal{E}$. The uniform invertibility of $J$ can clearly be assumed in one chart and uniform ellipticity of $1 + |\mathcal{E}|^2 + 2 \mathcal{E} \otimes \mathcal{E}$ is given since the eigenvalues are $1 + 3 |\mathcal{E}|^2$ and $1+|\mathcal{E}|^2$, and $\| \mathcal{E} \|_{L_T^\infty L_{x'}^\infty} \lesssim \| \mathcal{E} \|_{L_T^\infty H^s}$.

\medskip

\emph{Applying Strichartz estimates and switching to non-divergence form:}

\noindent Let 
\begin{equation*}
\tilde{P}_1 = 
\begin{pmatrix}
\varepsilon_1 \partial_t & - \nabla_\perp \\
(\nabla \times \cdot)_3 & \mu \partial_t
\end{pmatrix}
, \quad \eta(x,D) \mathcal{E}' = \sum_{i,j=1}^2 (\varepsilon_1)_{ij} \partial_i \mathcal{E}_j'.
\end{equation*}
We apply Strichartz estimates from Proposition \ref{prop:Strichartz2dAnisotropic} with $P_1$ to find
\begin{equation*}
\begin{split}
\| \langle D' \rangle^{-\alpha} (\mathcal{E}',\mathcal{H}') \|_{L_T^4 L_{x'}^\infty} &\lesssim_{\alpha, T} \| (\mathcal{E}',\mathcal{H}') \|_{L_T^\infty L_{x'}^2} + \| P_1 (\mathcal{E}',\mathcal{H}') \|_{L_T^2 L_{x'}^2} \\
&\quad + \| \langle D' \rangle^{-\frac{3}{4}} \nabla \cdot (\varepsilon_1 \mathcal{E}') \|_{L_T^2 L_{x'}^2}
\end{split}
\end{equation*}
for $\alpha > \frac{5}{6}$, $\| \partial_x \varepsilon_1 \|_{L_T^2 L_{x'}^\infty} \leq C < \infty$.

\medskip

Noting that $\nabla \cdot (\varepsilon_1 \mathcal{E}') = (\partial_{x'} \varepsilon_1) \mathcal{E}' + \eta(x,D) \mathcal{E}'$ and applying H\"older's inequality gives
\begin{equation}
\label{eq:StrichartzEstimatesNegativeDerivatives}
\begin{split}
\| \langle D' \rangle^{-\alpha} (\mathcal{E}',\mathcal{H}') \|_{L_T^4 L_{x'}^\infty(\R^2)} &\lesssim_{T, \alpha} \| (\mathcal{E}',\mathcal{H}') \|_{L_T^\infty L_{x'}^2} + \| \tilde{P}_1 (\mathcal{E}',\mathcal{H}') \|_{L_T^2 L_{x'}^2} \\
&\quad + \| \langle D' \rangle^{-\frac{3}{4}} \eta(x,D) \mathcal{E}' \|_{L_T^2 L_{x'}^2}.
\end{split}
\end{equation}

\medskip

\emph{Commutator estimates:}

\noindent Finally, we use commutator estimates to find an estimate for $\| \langle D' \rangle (\mathcal{E}',\mathcal{H}') \|_{L_T^4 L_{x'}^\infty}$. We apply \eqref{eq:StrichartzEstimatesNegativeDerivatives} to $\langle D' \rangle^{\alpha + 1} (\mathcal{E}',\mathcal{H}')$ to find
\begin{equation*}
\begin{split}
\| \langle D' \rangle (\mathcal{E}',\mathcal{H}') \|_{L_T^4 L_{x'}^\infty(\R^2)} &\lesssim_{T, \alpha} \| (\mathcal{E}',\mathcal{H}') \|_{L_T^\infty H^{\alpha + 1}} + \| P_1 ( \langle D' \rangle^{\alpha + 1} (\mathcal{E}',\mathcal{H}')) \|_{L_T^2 L_{x'}^2} \\
 &\quad + \| \langle D' \rangle^{-\frac{3}{4}} \eta(x,D) \langle D' \rangle^{\alpha + 1} \mathcal{E}' \|_{L_T^2 L_{x'}^2}.
 \end{split}
\end{equation*}
We note that
\begin{equation*}
\begin{split}
&\quad \mu (\partial_t \langle D' \rangle^{\alpha + 1} \mathcal{H}' + (\nabla \times \langle D' \rangle^{\alpha + 1} \mathcal{E}')_3 \\
&= \mu (\partial_t \langle D' \rangle^{\alpha + 1} \mathcal{H}' + \mu^{-1} (\nabla \times \langle D' \rangle^{\alpha + 1} \mathcal{E}')_3 ) \\
&= \mu (\langle D' \rangle^{\alpha + 1} (\partial_t \mathcal{H}' + \mu^{-1} (\nabla \times \mathcal{E}')_3)) + \mu [\mu^{-1}, \langle D' \rangle^{\alpha +1 }] (\nabla \times \mathcal{E}')_3 \\
&= \mu [\mu^{-1}, \langle D' \rangle^{\alpha +1 }] (\nabla \times \mathcal{E}')_3.
\end{split}
\end{equation*}
The ultimate estimate follows from $P_1 (\mathcal{E}',\mathcal{H}') = 0$. For the second term we write
\begin{equation*}
\varepsilon_1 \partial_t \langle D' \rangle^{\alpha + 1} \mathcal{E}' - \nabla_\perp \langle D' \rangle^{\alpha + 1} \mathcal{H}' = \varepsilon_1 \langle D' \rangle^{\alpha +  1} (\partial_t \mathcal{E}' - \varepsilon_1^{-1} \nabla_\perp \mathcal{H}') - \varepsilon_1 [\langle D' \rangle^{\alpha + 1}, \varepsilon_1^{-1}] \nabla_\perp \mathcal{H}. 
\end{equation*}
We infer from the Kato-Ponce commutator estimate (see \cite{KatoPonce1988}) and $\| (\varepsilon_1, \mu) \|_{L^\infty_{x'}} \lesssim_{\| \mathcal{E}' \|_{L^\infty_{x'}}} 1$,
\begin{equation*}
\begin{split}
&\quad \| P_1 (\langle D' \rangle^{\alpha + 1} (\mathcal{E}',\mathcal{H}')) \|_{L_{x'}^2} \\
 &\lesssim \| [\mu^{-1}, \langle D' \rangle^{\alpha +1 }] (\nabla \times \mathcal{E}')_3 \|_{L^2_{x'}} + \| [\langle D' \rangle^{\alpha + 1}, \varepsilon_1^{-1}] \nabla_\perp \mathcal{H}' \|_{L^2_{x'}} \\
&\lesssim \| \partial_{x'} \mu^{-1} \|_{L^\infty_{x'}} \| \langle D' \rangle^{\alpha + 1} \mathcal{E}' \|_{L^2_{x'}} + 
\| \langle D' \rangle^{\alpha + 1} \mu^{-1} \|_{L^2_{x'}} \| (\nabla \times \mathcal{E}')_3 \|_{L^\infty_{x'}} \\
&\quad + \| \partial_{x'} \varepsilon_1^{-1} \|_{L^\infty_{x'}} \| \langle D' \rangle^{\alpha + 1} \mathcal{H}' \|_{L^2_{x'}} + \| \langle D' \rangle^{\alpha + 1} \varepsilon_1^{-1} \|_{L^2_{x'}} \| \nabla_\perp \mathcal{H}' \|_{L^\infty_{x'}}.
\end{split}
\end{equation*}
We have $\| \partial_{x'} \varepsilon_1^{-1} \|_{L^\infty_{x'}} \lesssim C ( \| \mathcal{E} \|_{L^\infty_{x'}}, \| \partial_{x'} \mathcal{E} \|_{L^\infty_{x'}} )$ and therefore by Sobolev embedding and H\"older's inequality, there is $\kappa > 0$ such that
\begin{equation*}
\| \partial_{x'} \varepsilon_1^{-1} \|_{L_T^2 L_{x'}^\infty} \lesssim T^\kappa C ( \| \mathcal{E}' \|_{L_T^\infty H^s}, \| \partial_{x'} \mathcal{E}' \|_{L_T^4 L_{x'}^\infty}).
\end{equation*}
Secondly, by Moser estimates, ellipticity of $\varepsilon_1$, and the fractional Leibniz rule we have
\begin{equation*}
\| \langle D' \rangle^{\alpha + 1} (\varepsilon_1^{-1}) \|_{L^2_{x'}} \lesssim_{ \| \mathcal{E}' \|_{L^\infty_{x'}}} \| \langle D' \rangle^{\alpha + 1} \varepsilon_1 \|_{L^2_{x'}}.
\end{equation*}
We write the components of $\varepsilon_1$ as product of $g_1$, $\sqrt{g_1}$, and $\mathcal{E}_1'$, $\mathcal{E}_2'$. Recall that $(\varepsilon_1)_{11}$, $(\varepsilon_1)_{22}$ are reflected evenly, whereas $(\varepsilon_1)_{21}$, $(\varepsilon_1)_{12}$ are reflected oddly due to their internal structure. We use continuity of $\text{ext}_N$ and $\text{ext}_D$ to estimate for $\alpha + 1 \leq 2$:
\begin{equation*}
\| (\varepsilon_1)_{ij} \|_{H^{\alpha + 1}(\R^2)} \lesssim \| (\varepsilon_1)_{ij} \|_{H^{\alpha + 1}(\R^2_{>0})}.
\end{equation*}
On the half-space, we can use smoothness of $g$ and invariance of Sobolev spaces under multiplication with smooth functions to find 
\begin{equation*}
\begin{split}
\| (\varepsilon_1)_{ij} \|_{H^{\alpha + 1}(\R^2_{>0})} &\lesssim \sum_{m,n} \| \mathcal{E}'_m \mathcal{E}'_n \|_{H^{\alpha + 1}(\R^2_{>0})} \\
 &\lesssim \sum_{m,n} \| \mathcal{E}'_m \mathcal{E}'_n \|_{H^{\alpha + 1}(\R^2)} \lesssim \sum_{m,n} \| \mathcal{E}_m' \|_{H^{\alpha + 1}(\R^2)} \| \mathcal{E}_n' \|_{H^{\alpha + 1}(\R^2)}.
 \end{split}
\end{equation*}
In the penultimate estimate, we switched back to the full space by appropriate extension of the components of $\mathcal{E}'$. In the ultimate estimate, we used that $H^s(\R^2)$ is a Banach algebra for $s>1$. 

The second estimate
\begin{equation*}
\| \langle D' \rangle^{\alpha + 1} \mu^{-1} \|_{L^2_{x'}} \lesssim \| \langle D' \rangle^{\alpha + 1} g \|_{L^2_{x_2' > 0 }} \lesssim 1
\end{equation*}
follows likewise by switching back to the half space and assuming that $g \in C^\infty_c(\R^2_{>0})$. Here we use that we are dealing with a compact boundary.

Consequently,
\begin{equation*}
\begin{split}
\| \| \langle D' \rangle^{\alpha + 1} \mu^{-1} \|_{L^2_{x'}} \| (\nabla \times \mathcal{E}')_3 \|_{L^\infty_{x'}} \|_{L_T^2} &\lesssim T^{\frac{1}{4}} \| \partial_{x'} \mathcal{E}' \|_{L_T^4 L_{x'}^\infty}, \\
\| \| \langle D' \rangle^{\alpha + 1} (\varepsilon_1^{-1}) \|_{L^2_{x'}} \| \nabla_\perp \mathcal{H}' \|_{L^\infty_{x'}} \|_{L^2_T} &\lesssim_{\| \mathcal{E}' \|_{L_T^\infty L^\infty_{x'}}} T^{\frac{1}{4}} \| \mathcal{E}' \|^2_{L_T^\infty H^{\alpha+1}} \| \partial_{x'} \mathcal{H}' \|_{L_T^4 L_{x'}^\infty}.
\end{split}
\end{equation*}
Hence, for $T$ small enough we can absorb this into the left-hand side to find
\begin{equation*}
\| \langle D' \rangle (\mathcal{E}',\mathcal{H}') \|_{L_T^4 L_{x'}^\infty} \lesssim_{T,\alpha} \| (\mathcal{E}',\mathcal{H}') \|_{L_T^\infty H^{\alpha + 1}} + \| \langle D' \rangle^{-\frac{3}{4}} \eta(x,D) \langle D' \rangle^{\alpha + 1} \mathcal{E}' \|_{L^2_T L^2_{x'}}.
\end{equation*}

It remains to find an estimate of $\eta(x,D) \langle D' \rangle^{\alpha + 1} \mathcal{E}'$ in terms of the charges and lower order terms $\| \langle D' \rangle^{\alpha + 1} (\mathcal{E}',\mathcal{H}') \|_{L_T^\infty L^2_{x'}}$. We write
\begin{equation*}
\eta(x,D) \langle D' \rangle^{\alpha + 1} \mathcal{E}' = \langle D' \rangle^{\alpha + 1} \eta(x,D) \mathcal{E}' + [\varepsilon_1 , \langle D' \rangle^{\alpha + 1}] \partial_{x'} \mathcal{E}'.
\end{equation*}
The commutator
\begin{equation*}
\| \langle D' \rangle^{-\frac{3}{4}} [ \varepsilon_1, \langle D' \rangle^{\alpha + 1}] \partial_{x'} \mathcal{E}' \|_{L_T^2 L_{x'}^2} \lesssim \| [\varepsilon_1, \langle D' \rangle^{\alpha + 1}] \partial_{x'} \mathcal{E}' \|_{L^2_T L^2_{x'}}
\end{equation*}
can be handled like above with the Kato-Ponce commutator estimate. We need to relate $\eta(x,D) \mathcal{E}'$ to the charges $\nabla \cdot ( \varepsilon \mathcal{E}')= \frac{1}{\sqrt{g}} \nabla \cdot (\sqrt{g} g^{-1} (1 + |\mathcal{E}'|^2_{g^{-1}}) \mathcal{E}')$. A straight-forward computation yields
\begin{equation*}
\nabla \cdot (\sqrt{g} g^{-1} (1+|\mathcal{E}'|^2_{g^{-1}}) \mathcal{E}') = \eta(x,D) \mathcal{E}' + O(\partial g (\mathcal{E}')^3).
\end{equation*}
Therefore,
\begin{equation*}
\| \langle D' \rangle^{\alpha + \frac{1}{4}} \eta(x,D) \mathcal{E}' \|_{L_T^\infty L_{x'}^2} \lesssim \| \rho_e(0) \|_{H^{\alpha + \frac{1}{4}}} + \| \langle D' \rangle^{\alpha + \frac{1}{4}} ( \partial g (\mathcal{E}')^3) \|_{L_T^\infty L^2_{x'}}.
\end{equation*}

We split the second term into monomials of $\partial g$ and $\mathcal{E}'$
\begin{equation*}
\partial g (\mathcal{E}')^3 = \sum_{k,i_1,i_2,i_3} (\partial g)_k \mathcal{E}_{i_1}' \mathcal{E}_{i_2}' \mathcal{E}_{i_3}',
\end{equation*}
which gives a decomposition into even or odd functions. By continuity of the respective extension for $\alpha + \frac{1}{2} \leq 2$, we find
\begin{equation*}
\begin{split}
\| (\partial g)_k \mathcal{E}'_{i_1} \mathcal{E}'_{i_2} \mathcal{E}'_{i_3} \|_{L_T^\infty H^{\alpha + \frac{1}{4}}(\R^2)} &\lesssim \| (\partial g)_k \mathcal{E}'_{i_1} \mathcal{E}'_{i_2} \mathcal{E}'_{i_3} \|_{L_T^\infty H^{\alpha + \frac{1}{4}}(\R^2_{>0})} \\
&\lesssim \| \mathcal{E}'_{i_1} \mathcal{E}'_{i_2} \mathcal{E}'_{i_3} \|_{L_T^\infty H^{\alpha + \frac{1}{4}}(\R^2_{>0})}.
\end{split}
\end{equation*}
Here we used again invariance of Sobolev spaces under multiplication with smooth functions. Like above, we extend again to the full space to apply the fractional Leibniz rule and Sobolev embedding
\begin{equation*}
\| \mathcal{E}'_{i_1} \mathcal{E}'_{i_2} \mathcal{E}'_{i_3} \|_{L_T^\infty H^{\alpha + \frac{1}{4}}(\R^2_{>0})} \lesssim \prod_{j=1}^3 \| \mathcal{E}_{i_j}' \|_{L_T^\infty H^{\alpha + \frac{1}{4}}(\R^2)} \lesssim \| \mathcal{E}' \|_{L_T^\infty H^{\alpha  +1}}^3.
\end{equation*}
We conclude for $\frac{5}{6} < \alpha \leq 1$ the existence of $\kappa > 0$ such that
\begin{equation*}
\begin{split}
&\quad \| \langle D' \rangle (\mathcal{E}',\mathcal{H}') \|_{L_T^4 L_{x'}^\infty(\R^2)} \\
 &\lesssim_{T,\alpha} (1 + T^\kappa C(\| (\mathcal{E}',\mathcal{H}') \|_{L_T^\infty H^{\alpha + 1}}, \| \partial_{x'} ( \mathcal{E}',\mathcal{H}') \|_{L_T^4 L_{x'}^\infty} ) \| (\mathcal{E}', \mathcal{H}') \|_{L_T^\infty H^{\alpha + 1}} \\
&\quad + T^{\frac{1}{4}} \| \rho_e(0) \|_{H^{\alpha + \frac{1}{4}}}.
\end{split}
\end{equation*}
This finishes the proof letting $\alpha \to \frac{5}{6}$ and setting $s=\alpha + 1$.

\end{proof}

\subsection{$L^2$-Lipschitz continuous dependence}
We turn to $L^2$-Lipschitz continuous dependence for initial data at higher regularities.
\begin{proposition}
\label{prop:L2Lipschitz}
Let $s \in (\frac{11}{6},2]$, $u_0^1$, $u_0^2 \in \mathcal{H}^3(\Omega)$, with $\| u_0^i \|_{H^s(\Omega)} \leq \delta \ll 1$, $i=1,2$, and for some $\tilde{s} > \frac{13}{12}$ and $D > 0$, $\| \rho_e^i(0) \|_{H^{\tilde{s}}} \leq D$, and denote the corresponding solutions to \eqref{eq:KerrSystem2d} with $u^i$. Then the following estimate holds:
\begin{equation*}
\| u^1 - u^2 \|_{L_T^\infty L^2(\Omega)} \leq C(\| u_0^i \|_{H^s(\Omega)},D) \| u_0^1 - u_0^2 \|_{L^2(\Omega)}.
\end{equation*}
\end{proposition}
\begin{proof}
We analyze the equation satisfied by the differences of solutions $v = u^1 - u^2 = (\mathcal{E}_\Delta, \mathcal{H}_\Delta)$:
\begin{equation*}
\left\{ \begin{array}{cl}
\partial_t (\tilde{\varepsilon} \mathcal{E}_\Delta) &= \nabla_\perp \mathcal{H}_\Delta, \qquad \qquad (t,x') \in \R \times \Omega, \\
\partial_t \mathcal{H}_\Delta &= - \nabla \times \mathcal{E}_\Delta, \quad \qquad [(\mathcal{E}_\Delta)_{||}]_{x' \in \partial \Omega} = 0
\end{array} \right.
\end{equation*}
with $\tilde{\varepsilon} = 1 + \mathcal{E}_2^t \mathcal{E}_1 + |\mathcal{E}_1 - \mathcal{E}_2|^2 + \mathcal{E}_1 \otimes \mathcal{E}_2 + \mathcal{E}_2 \otimes \mathcal{E}_1$. Hence, the integration-by-parts-argument from the proof of Proposition \ref{prop:APrioriMaxwell2Kerr}\footnote{We proved Proposition \ref{prop:APrioriMaxwell2Kerr} for the diagonal Kerr permittivity, but it is straight-forward to obtain the below estimate for possibly off-diagonal $\tilde{\varepsilon}$.} yields
\begin{equation*}
\| v(t) \|_{L^2(\Omega)} \leq e^{C(\tilde{\delta}) \int_0^t \| \partial_t (\mathcal{E}_\Delta,\mathcal{H}_\Delta)(s) \|_{L^\infty(\Omega)} ds} \| v(0) \|_{L^2(\Omega)}
\end{equation*}
with $\tilde{\delta} = \| u^1 \|_{L_T^\infty L_{x'}^\infty(\Omega)} + \| u^2 \|_{L_T^\infty L_{x'}^\infty(\Omega)}$.\\
By the proof of Proposition \ref{prop:APrioriEstimate2dKerr} and Minkowski's inequality, we obtain for $T=T(\| u_0^i \|_{H^s(\Omega)})$, $\frac{11}{6} < s \leq 2$, and provided that $\tilde{\delta}$ is small enough, the following for some $\kappa > 0$:
\begin{equation*}
\begin{split}
\| \partial_t (\mathcal{E}_\Delta,\mathcal{H}_\Delta) \|_{L_T^2 L_{x'}^\infty(\Omega)} &\leq \| \dot{u}_1 \|_{L_T^2 L_{x'}^\infty(\Omega)} + \| \dot{u}_2 \|_{L_T^2 L_{x'}^\infty(\Omega)} \\
&\lesssim T^\kappa ( \| u_1(0) \|_{H^s(\Omega)} + \| u_2(0) \|_{H^s(\Omega)} + D).
\end{split}
\end{equation*}
The proof is complete.
\end{proof}

\subsection{Proof of continuous dependence via frequency envelopes}
\label{subsection:FrequencyEnvelopes}
In this section we want to extend the data-to-solution mapping for \eqref{eq:KerrSystem2d} from $\mathcal{H}^3(\Omega)$ to $\mathcal{H}^s(\Omega)$ for $\frac{11}{6} < s \leq 2$ under a boundedness condition on the charges. We shall use frequency envelopes, for which a regularization is required, which is consistent with the compatibility conditions (to use the local existence provided in $\mathcal{H}^3(\Omega)$). The  compatibility conditions for the Kerr nonlinearity are computed in the proof of Theorem \ref{thm:LocalWellposednessH3Maxwell2dAppendix}. In geodesic coordinates the boundary conditions are given by
\begin{align}
\label{eq:NonlinearCompatibility0}
 [\mathcal{E}_1]_{x'_2 = 0} &= 0, \\
 \label{eq:NonlinearCompatibility1}
 [\partial_2 \mathcal{H}]_{x'_2 = 0} &= 0, \\
 \label{eq:NonlinearCompatibility2}
  [\partial_2 (\sqrt{g}^{-1} (\partial_2 \mathcal{E}_1 - \partial_1 \mathcal{E}_2))]_{x'_2 = 0} &= 0.
\end{align}

For this reason \eqref{eq:NonlinearCompatibility0} and \eqref{eq:NonlinearCompatibility2} are automatically satisfied for $\mathcal{E}_0 \in \overline{C^\infty_c(\Omega)^2}^{\| \cdot \|_{H^s}}$, and $\mathcal{H}$ satisfies Neumann boundary conditions in $H^3(\Omega)$. We suppose in the following by density and compactness of $\partial \Omega$ that in geodesic coordinates 
\begin{equation*}
\text{dist}( \text{supp}(\mathcal{E}_0), \{x_2' = 0 \}) > 0.
\end{equation*}

The components are extended to the full space by reflection: $\mathcal{E}_1$ is reflected oddly, and $\mathcal{E}_2$, $\mathcal{H}$ is reflected evenly. We regularize the components as follows:
\begin{equation*}
f_n(x_1',x_2') = \int_{\R^2} f(x_1'-y_1,x_2'-y_2) n^2 \psi(n y_1) \psi(n y_2) dy_1 dy_2
\end{equation*}
with $\psi \in C^\infty_c(\R)$, which is symmetric at the origin, $\psi \geq 0$, and $\int_\R \psi(y) dy = 1$.

This regularization preserves the boundary conditions (for $\mathcal{E}_{0n}$ we have to choose $n$ large enough) and clearly $(\mathcal{E}_{0n},\mathcal{H}_{0n}) \in H^k(\R^2)^3$ for any $k \in \N$. The regularization with parameter $N \in 2^{\N}$ corresponds to a frequency truncation at frequencies $N$. However, presently this is not sharp in frequency space, but with a Schwartz tail. Sharp frequency truncation does not preserve the boundary conditions.

Now, for $N \in 2^{\mathbb{N}_0}$, we let 
\begin{equation*}
 P_{\leq N} \mathcal{E}_i = \mathcal{E}_{iN}, \quad P_{\leq N} \mathcal{H} = \mathcal{H}_N, \quad P_N = P_{\leq 2 N} - P_{\leq N/2}.
\end{equation*}
We introduce the notation for $M,N \in 2^{\mathbb{N}_0}$:
\begin{equation*}
\big[ \frac{M}{N} \big] = \min( \frac{M}{N}, \frac{N}{M} ).
\end{equation*}

We can now define frequency envelopes for the problem at hand:
\begin{definition}
 $(c_N)_{N \in 2^{\N_0}} \in \ell^2$ is called a frequency envelope for a function $u \in H^s(\Omega)$, if it has the following properties:
 \begin{itemize}
  \item[a)] Energy bound:
  \begin{equation}\label{feenergy}
   \| P_N u \|_{H^s} \leq c_N.
  \end{equation}
  \item[b)] Slowly varying: For a suitable choice of  $\delta > 0$ we have for all $J,K \in 2^{\N_0}$:
  \begin{equation}\label{feslow}
   \frac{c_K}{c_J} \lesssim \max\big( \frac{J}{K}, \frac{K}{J} \big)^\delta.
  \end{equation}
 \end{itemize}
The envelopes are called sharp, if they also satisfy
\begin{equation*}
 \| u \|^2_{H^s} \approx \sum_N c_N^2.
\end{equation*}
\end{definition}
Sharp frequency envelopes exist for any $\delta >0$. Indeed, we can start taking  $\widetilde {c}_N = N^{2s} \| P_N f \|^2_{L^2}$ as an ansatz and define (cf. \cite{IfrimTataru2020})
$$ c_N = \sup_{M} \min \big[ \frac{N}{M} \big]^\delta \widetilde{c}_M.$$ 
Indeed, \eqref{feenergy} is clearly satisfied while ~\eqref{feslow} follows from 
\begin{equation}\label{fe}
c_K =  \sup_{M} \big[ \frac{M}{N} \big]^\delta \widetilde{c}_M \\
\leq \sup_{M} \big[ \frac{M}{N} \big]^\delta \widetilde{c}_M  \times \max \big( \frac{K}{N}, \frac{N}{K} \big)^\delta .
\end{equation}
Finally, we have 
$$ c_N^2= \sup_{M} \big[ \frac{M}{N} \big]^{2\delta} \widetilde{c}^2_M\leq \sum_{M} \big[ \frac{M}{N} \big]^{2\delta} \widetilde{c}^2_M,$$ 
which implies
$$\sum_{N} c_N^2 \leq C_\delta \sum_N \widetilde{c}_N^2 = \| u\|_{H^s}^2.
$$
Lastly, the definition of the frequency envelope~\eqref{fe} satisfies the additional compactness property:
\begin{lemma}\label{fecomp}
Assume that the sequence $(u_n)_{n \in \N_0}$ converges to $u$ in $H^s$. Then for any $\delta>0$ the family of frequency envelopes defined by
 \begin{equation*}
c_{K,n} =  \sup_{M} \big[ \frac{M}{N} \big]^\delta \widetilde{c}_{M,n} ,\qquad  \widetilde{c}_{M,n} = \| P_M u_n\|_{H^s}
\end{equation*}
converge in $\ell ^2$ to $(c_K)$.
\end{lemma}
\begin{proof}
We observe that
\begin{equation*}
 |c_{K_n} - c_K| = | \sup_{M} \big[ \frac{M}{N} \big]^\delta \widetilde{c}_{M,n} - \sup_{M} \big[ \frac{M}{N} \big]^\delta \widetilde{c}_M |  \leq \sup_{M} \big[ \frac{M}{N} \big]^\delta |\widetilde{c}_{M,n} - \widetilde{c}_M|.
\end{equation*}
\end{proof}

We recall the following:

\medskip

\emph{Regularization:} Let $u_0 \in H^s(\tilde{\Omega})$ and $(c_N)_{N \in 2^{\N_0}}$ be a sharp frequency envelope for $u_0$ in $H^s$. Then, we have for $u_0^N = P_{\leq N} u_0$ the following bounds:
\begin{itemize}
 \item [i)] Uniform bounds: 
 \begin{equation}\label{unif}
 \| P_K u_0^N \|_{H^s} \lesssim c_K.
 \end{equation}
 \item [ii)] High frequency bounds: If $0< \delta <j$, we have  
 \begin{equation}
\label{HF}
\| u_0^N \|_{H^{s+j}} \lesssim N^j c_N \quad (j \geq 0).
\end{equation}
  By the slow varying property of frequency envelopes ~\eqref{feslow}: 
 $$ \| u_0^N \|^2_{H^{s+j}}= \sum_{K\leq N} K^{2j} \|P_K u_0\|_{H^s}^2 \leq \sum_{K\leq N} K^{2j}  c^2_K \leq \sum_{K\leq N} K^{2(j- \delta)} N^{2\delta} c^2_N\sim N^{2j} c^2_N.$$
 \item [iii)] Difference bounds: 
\begin{equation}\label{diff}
\| u_0^{2N} - u_0^N \|_{L^2} \lesssim N^{-s} c_N \quad (s \geq 0).
\end{equation}
 \item [iv)] Limit as $n \to \infty$: $u_0 = \lim_{N \to \infty} u_0^N$ in $H^s$.
\end{itemize}

Presently, the above properties are relevant for $\frac{11}{6} < s < s + j \leq 3$ (for a proper choice of $\delta >0$). The regularized initial data give rise to a family of solutions in $H^3$ by the local existence result in Theorem \ref{thm:LocalWellposednessH3Maxwell2d}.

\medskip

\emph{Uniform bounds:} Proposition \ref{prop:APrioriEstimate2dKerr} yields a time interval of length 
$$T = T(\| u_0 \|_{H^s}, \| \rho_e \|_{H^{\frac{13}{12}+\epsilon}}),$$
 on which the solution exists, with $T$ independent of the regularization parameter $N$, provided that we can prove a uniform bound for the regularized charges. We show the following:
\begin{lemma}
 Let $\tilde{s}=\frac{13}{12} + \epsilon$ with $\epsilon$ chosen small enough. We have the following estimate:
 \begin{equation*}
 \| \rho_{e N} \|_{H^{\tilde{s}}} \lesssim \| \mathcal{E} \|_{H^{\tilde{s}+\frac{1}{2}+\epsilon}}^3 + \| \rho_e \|_{H^{\tilde{s}}}.
 \end{equation*}
\end{lemma} 
 \begin{proof}
To carry out commutator estimates, we change to non-divergence form:
\begin{equation*}
\begin{split}
\rho_{eN} &= \frac{1}{\sqrt{g}} \partial_1 (\sqrt{g} g_1 \varepsilon_{N} \mathcal{E}_{1N}) + \frac{1}{\sqrt{g}} \partial_2 ( \sqrt{g} g_2 \varepsilon_N \mathcal{E}_{2N}) \\
&= \frac{1}{\sqrt{g}} \partial_1 (\sqrt{g} g_1 (1+g_1 |\mathcal{E}_{1N}|^2+ g_2|\mathcal{E}_{2N}|^2) \mathcal{E}_{1N}) \\
&\quad + \frac{1}{\sqrt{g}} \partial_2 ( \sqrt{g}g_2 (1+ g_1 |\mathcal{E}_{1N}|^2 + g_2 |\mathcal{E}_{2N}|^2) \mathcal{E}_{2N}) \\
&= \tilde{\varepsilon}_{ij N} \partial_j \mathcal{E}_{i N} + O(\partial g \mathcal{E}^3).
\end{split}
\end{equation*} 
The second term is clearly lower order.
By the fractional Leibniz rule we find
\begin{equation*}
\| \langle D' \rangle^{\tilde{s}} ( \tilde{\varepsilon}_{ij N} \partial_j \mathcal{E}_{i N} ) \|_{L^2(\R^2)} \lesssim \| \langle D' \rangle^{\tilde{s}} \tilde{\varepsilon}_{ij N} \partial_j \mathcal{E}_{ij N} \|_{L^2(\R^2)} + \| \tilde{\varepsilon}_{ij N} \langle D' \rangle^{\tilde{s}} \partial_j \mathcal{E}_{i N} \|_{L^2(\R^2)}.
\end{equation*}
The first term is acceptable by H\"older's inequality and Sobolev embedding as
\begin{equation*}
\| \langle D' \rangle^{\tilde{s}} \tilde{\varepsilon}_{ij N} \partial_j \mathcal{E}_{ij N} \|_{L^2(\R^2)} \lesssim \| \langle D' \rangle^{\tilde{s}+\frac{1}{2}} \mathcal{E} \|^2_{L^2} \| \langle D' \rangle^{\frac{3}{2}} \mathcal{E} \|_{L^2}.
\end{equation*}
We turn to the second term:
\begin{equation*}
\| \tilde{\varepsilon}_{ij N} \langle D' \rangle^{\tilde{s}} \partial_j \mathcal{E}_{iN} \|_{L^2} \leq \| (1-P_{ \leq N}) \tilde{\varepsilon}_{ij} \partial_j \langle D' \rangle^{\tilde{s}} \mathcal{E}_{i N} \|_{L^2} + \| \tilde{\varepsilon}_{ij} \partial_j \langle D' \rangle^{\tilde{s}} \mathcal{E}_{i N} \|_{L^2}.
\end{equation*}
We compute for the first term
\begin{equation*}
\begin{split}
\| (1-P_{\leq N}) \tilde{\varepsilon}_{ij} \partial_j \langle D' \rangle^{\tilde{s}} \mathcal{E}_{iN} \|_{L^2} &\lesssim \| (1-P_{\leq N}) \tilde{\varepsilon}_{ij} \|_{L^4} N \| \langle D' \rangle^{\tilde{s}} \mathcal{E}_{i N} \|_{L^4} \\
&\lesssim \| \partial_x (1-P_{\leq N}) \tilde{\varepsilon}_{ij} \|_{L^4} \| \langle D' \rangle^{\tilde{s}+\frac{1}{2}} \mathcal{E}_i \|_{L^2} \\
&\lesssim \| \langle D' \rangle^{\tilde{s}+ \frac{1}{2}} \mathcal{E} \|_{L^2}^3.
\end{split}
\end{equation*}

Regarding the second term, we find
\begin{equation*}
\| \tilde{\varepsilon}_{ij} \partial_j \langle D' \rangle^{\tilde{s}} \mathcal{E}_{i N} \|_{L^2} \leq \| P_{\leq N} ( \tilde{\varepsilon}_{ij} \partial_j \langle D' \rangle^{\tilde{s}} \mathcal{E}_i ) \|_{L^2} + \| [P_{\leq N}, \tilde{\varepsilon}_{ij}] \partial_j \langle D' \rangle^{\tilde{s}} \mathcal{E}_i \|_{L^2}.
\end{equation*}
The first term is estimated by H\"older's inequality and Sobolev embedding as
\begin{equation*}
\begin{split}
\| P_{\leq N} (\tilde{\varepsilon}_{ij} \partial_j \langle D' \rangle^{\tilde{s}} \mathcal{E}_i) \|_{L^2} &\leq \| \tilde{\varepsilon}_{ij} \partial_j \langle D' \rangle^{\tilde{s}} \mathcal{E}_i \|_{L^2} \\
&\leq \| \langle D' \rangle^{\tilde{s}} (\tilde{\varepsilon}_{ij} \partial_j \mathcal{E}_i) \|_{L^2} + \| \langle D' \rangle^{\tilde{s}} \tilde{\varepsilon}_{ij} \partial_j \mathcal{E}_i \|_{L^2} \\
&\lesssim \| \rho_e \|_{H^{\tilde{s}}} + \| \langle D' \rangle^{\tilde{s}+\frac{1}{2}} \mathcal{E} \|_{L^2}^3.
\end{split}
\end{equation*}
For the second term we introduce an additional Littlewood-Paley decomposition. Here $P_K'$ denotes the usual frequency localization in $\R^2$. We find
\begin{equation*}
\begin{split}
[P_{\leq N}, \tilde{\varepsilon}_{ij} ] \partial_j \langle D' \rangle^{\tilde{s}} \mathcal{E}_i &= \sum_{M_1 \ll M_2} [P_{\leq N}, P'_{M_1} \tilde{\varepsilon}_{ij} ] \partial_j \langle D' \rangle^{\tilde{s}} P'_{M_2} \mathcal{E}_i \\
&\qquad + \sum_{M_1 \gtrsim M_2} [P_{\leq N}, P'_{M_1} \tilde{\varepsilon}_{ij} ] \partial_j \langle D' \rangle^{\tilde{s}} P'_{M_2} \mathcal{E}_i.
\end{split}
\end{equation*}
The contribution of the second term is estimated like above by H\"older's inequality and distributing derivatives:
\begin{equation*}
\begin{split}
\big\| \sum_{M_1 \gtrsim M_2} [ P_{\leq N}, P'_{M_1} \tilde{\varepsilon}_{ij} ] \partial_j \langle D' \rangle^{\tilde{s}} P'_{M_2} \mathcal{E}_i \big\|_{L^2} &\lesssim \sum_{M_1 \gtrsim M_2} \| P'_{M_1} \tilde{\varepsilon}_{ij} \|_{L^4} M_2 \| \langle D' \rangle^{\tilde{s}} \mathcal{E}_i \|_{L^4} \\
&\lesssim \| \langle D' \rangle^{\frac{3}{2}+\epsilon} \tilde{\varepsilon}_{ij} \|_{L^2} \| \langle D' \rangle^{\tilde{s}+\frac{1}{2}} \mathcal{E}_i \|_{L^2} \\
&\lesssim \| \langle D' \rangle^{\tilde{s}+\frac{1}{2}} \mathcal{E} \|_{L^2}^3.
\end{split}
\end{equation*}
We turn to the first term, for which we can suppose that $M_2 \lesssim N$ since $P_{\leq N} P'_{M_2}$ is smoothing for $M_2 \gg N$. By a standard kernel estimate, we find
\begin{equation*}
\begin{split}
&\quad \| [P_{\leq N}, P'_{M_1} \tilde{\varepsilon}_{ij} ] \partial_j \langle D' \rangle^{\tilde{s}} P'_{M_2} \mathcal{E}_i \|_{L^2} \\
&\lesssim N^{-1} M_2^{\frac{1}{2}} \| \partial_x P'_{M_1} \tilde{\varepsilon}_{ij} \|_{L^\infty} \| \langle D' \rangle^{\tilde{s}+\frac{1}{2}} \mathcal{E} \|_{L^2} \\
&\lesssim N^{-1} M_2^{\frac{1}{2}} \| \partial_x P_{M_1}' \tilde{\varepsilon}_{ij} \|_{L^\infty} \| \langle D' \rangle^{\tilde{s}+\frac{1}{2}} \mathcal{E} \|_{L^2} \\
&\lesssim \| \langle D' \rangle^{\tilde{s}+\frac{1}{2}} \mathcal{E} \|^3_{L^2}.
\end{split}
\end{equation*}
The proof is complete.
\end{proof}

We can now come back to the proof of the continuity of the flow. Proposition~\ref{prop:L2Lipschitz} yields $L^2$-bounds on the difference:
\begin{itemize}
 \item[i)] High frequency bounds for solutions (from~\eqref{HF}):
 \begin{equation*}
  \| u^N \|_{C([0,T],H^{s+j})} \lesssim N^j c_N \qquad \big( \frac{11}{6} < s < s + j \leq 3 \big),
 \end{equation*}
  \item[ii)] Difference bounds (from~\eqref{diff})
  \begin{equation*}
   \| u^{2N} - u^N \|_{C([0,T],L^2)} \lesssim N^{-s} c_N.
  \end{equation*}
\end{itemize}
Interpolation gives
\begin{equation*}
 \| u^{2N} - u^N \|_{C([0,T],H^m)} \lesssim c_N N^{-(s-m)}
\end{equation*}
for $0 \leq m \leq 3$. By the $L^2$-bound and telescoping sum
\begin{equation*}
 u - u^N = \sum_{M \geq N} u_{2M} - u_M,
\end{equation*}
we obtain the estimates
\begin{equation*}
 \| u - u^N \|_{C_T L^2} \lesssim N^{-s}.
\end{equation*}
This yields convergence in $L^2$. Moreover, the frequencies of $u^{2N} - u^N$ are decaying exponentially off $N$. Indeed, let $M \ll N$. Then, we find by the $L^2$-estimate
\begin{equation*}
 \| P_M (u^{2N} - u^N) \|_{C_T H^s} \lesssim M^s \| u^{2N} - u^N \|_{L_T^\infty L^2} \lesssim \big( \frac{M}{N} \big)^s c_N.
\end{equation*}
For $M \gg N$, we can use the high frequency bounds for the solutions:
\begin{equation*}
 M^j \| P_M ( u^{2N} - u^N) \|_{C_T H^s} \sim \| P_M (u^{2N} - u^N) \|_{C_T H^{s+j}} \lesssim N^j c_N.
\end{equation*}
Therefore,
\begin{equation*}
 \| P_M (u^{2N} - u^N) \|_{C_T H^s} \lesssim \big( \frac{N}{M} \big)^j c_N.
\end{equation*}
and as a consequence, 
$$ \| P_M ( u^{2N} - u^N) \|_{C_T H^s}  \lesssim \big[ \frac{M}{N} \big]^\kappa c_N,\qquad \kappa = \min (j,s),
$$
which implies 
\begin{equation*}
\begin{split}
 \| P_M ( u - u^N) \|_{C_T H^s}  &\lesssim \sum_{P\geq N}  \| P_M ( u^{2P} - u^P) \|_{C_T H^s} \\
 &\lesssim \sum_{P\geq N} \big[ \frac PM  \big]^\kappa c_P =  \sum_{P} \big[ \frac PM  \big]^\kappa c_P1_{P\geq N}
 \end{split}
\end{equation*}
By Schur's Lemma, the operator with kernel $K(M,P)= \big[ \frac{M}{P} \big]^\kappa$ is bounded on $\ell^2(2^{\mathbb{N}_0})$, which implies 
\begin{equation}\label{feest}
\|( u - u^N) \|^2_{C_T H^s} \leq C \sum_{M} \| P_M ( u - u^N) \|^2_{C_T H^s}  \leq C \sum_{P\geq N} c_P^2 \rightarrow _{N\rightarrow + \infty} 0.
\end{equation}
 This also proves convergence in $H^s$ to $u \in C_T H^s$.
  The proof of continuous dependence follows from this estimate and Lemma~\ref{fecomp}. Let $u_{0,n}\rightarrow u_0$ in $H^s$. Denote by $c_{N,n}$ a family of frequency envelopes satisfying the conclusions of Lemma~\ref{fecomp}. Then, since the sequence $u_{0,n} $ is bounded in $H^s$, we get for $T$ fixed 
  $$ \| u_n - u_n^N\|_{C_T H^s} \leq C \sum_{P\geq N} c_{P,n}^2 , \qquad  \| u - u^N\|_{C_T H^s} \leq C \sum_{P\geq N} c_{P}^2.$$
  From Lemma~\ref{fecomp} for any $\epsilon >0$, we can choose $N>0$ and $n_0$  such that 
  $$ \forall n\geq n_0: \, \| u_n - u_n^N\|_{C_T H^s} \leq \epsilon,  \| u - u^N\|_{C_T H^s} \leq \epsilon.$$
  Now, writing
  \begin{equation*}
  \begin{split}
   \| u_n - u \|_{C_T H^s} &\leq  \| u_n - u_n^N\|_{C_T H^s} +  \| u_n^N - u^N\|_{C_T H^s} +  \| u^N - u\|_{C_T H^s} \\
   &\leq  2 \epsilon +  \| u_n^N - u^N\|_{C_T H^s},
  \end{split}
  \end{equation*}
  since the initial data $u_{0,n}^N$ and $u_0^N$  are bounded in $H^{s+j}$ ($N$ fixed, $n\rightarrow + \infty$), we get from interpolation between the a priori estimate in $H^{s+j}$ and contraction in $L^2$ that 
  $$ \lim_{n\rightarrow + \infty } \| u_n^N - u^N\|_{C_T H^s}=0.$$
   The proof of Theorem \ref{thm:ImprovedLWPKerr2d} is complete.

$\hfill \Box$

\newpage

\appendix

\section{Local well-posedness for Maxwell equations in two dimensions at high regularity}
\label{appendix:LWPHighRegularity}
In the following we prove local well-posedness of the Maxwell system with Kerr nonlinearity in two dimensions. Let $\Omega \subseteq \R^2$ be a smooth domain with compact boundary. We consider the Kerr system
\begin{equation}
\label{eq:Maxwell2dAppendix}
\left\{
\begin{array}{cl}
\partial_t (\varepsilon \mathcal{E}) &= \nabla_{\perp} \mathcal{H}, \quad [ \mathcal{E} \wedge \nu ]_{x' \in \partial \Omega} = 0, \\
\partial_t (\mu \mathcal{H}) &= -( \partial_1 \mathcal{E}_2 - \partial_2 \mathcal{E}_1), \quad (t,x') \in \R \times \Omega
\end{array}
\right.
\end{equation}
with $\nabla_\perp = (\partial_2, - \partial_1)$ and $\varepsilon = 1+|\mathcal{E}|^2$. Moreover, we shall prove that the solution satisfies finite speed of propagation. To this end, we perceive the two-dimensional case as projection of the three-dimensional case. This allows us to apply the local well-posedness theory established by Spitz \cite{Spitz2019} and transfer the finite speed of propagation from the three-dimensional to the two-dimensional case. We remark that a simpler version of the detailed argument likewise yields local well-posedness in $\mathcal{H}^3(\Omega)$ in the autonomous case:
\begin{equation}
\left\{
\begin{array}{cl}
\partial_t (\varepsilon \mathcal{E}) &= \nabla_{\perp} \mathcal{H}, \quad (t,x') \in \R \times \Omega, \quad [ \mathcal{E} \wedge \nu ]_{x' \in \partial \Omega} = 0, \\
\partial_t \mathcal{H} &= -( \partial_1 \mathcal{E}_2 - \partial_2 \mathcal{E}_1)
\end{array}
\right.
\end{equation}
with $\varepsilon, \mu \in C^\infty(\Omega;\R_{>0})$, which satisfy ellipticity conditions:
\begin{equation*}
\exists \lambda, \Lambda > 0: \forall x' \in \Omega: \lambda \leq \kappa(x') \leq \Lambda, \quad \kappa \in \{ \varepsilon; \mu \}.
\end{equation*}

We show the following:
\begin{theorem}
\label{thm:LocalWellposednessH3Maxwell2dAppendix}
There is $\delta > 0$ such that \eqref{eq:Maxwell2dAppendix} is locally well-posed in $\mathcal{H}^3(\Omega)$ for initial data $\| (\mathcal{E},\mathcal{H})_0 \|_{H^3(\Omega)} \leq \delta $. This means there is $T=T(\| (\mathcal{E},\mathcal{H})_0 \|_{H^3(\Omega)}, \Omega)$ such that solutions $(\mathcal{E}^i,\mathcal{H}^i)$, $i=1,2$ exist for $0<t \leq T$ and depend continuously on the initial data: We have
\begin{equation*}
\| (\mathcal{E}^1,\mathcal{H}^1)(t) - (\mathcal{E}^2,\mathcal{H}^2)(t) \|_{H^s(\Omega)} \to 0
\end{equation*}
for $\| (\mathcal{E}^1,\mathcal{H}^1)(0) - (\mathcal{E}^2,\mathcal{H}^2)(0) \|_{H^s(\Omega)} \to 0$.
\end{theorem}

The proof is carried out in the following steps:
\begin{itemize}
\item We extend the two-dimensional system to three dimensions by introducing a cylindrical tangent direction.
\item We find the compatibility conditions for the Kerr nonlinearity in two and three dimensions and see that the cylindrical extensions satisfies the compatibility conditions, if the compatibility conditions in two dimensions are satisfied.
\item Finally, we recover the solutions to the two-dimensional system by restricting the solutions to the cylindrical extension and check the regularity.
\end{itemize}

\subsubsection*{Extension to the three-dimensional case}

We consider the non-compact cylinder $\tilde{\Omega} = \Omega \times \R$, which is still smooth and its boundary can be covered by finitely many charts. This makes the extended Maxwell system on $\tilde{\Omega}$ still amenable to Spitz's local well-posedness theory:
\begin{equation}
\label{eq:Maxwell3dAppendix}
\left\{
\begin{array}{cl}
\partial_t (\varepsilon \tilde{\mathcal{E}}) &= \nabla \times \tilde{\mathcal{H}}, \quad [\tilde{\mathcal{E}} \wedge \nu]_{x' \in \partial \tilde{\Omega}} = 0, \quad [\tilde{\mathcal{H}} . \nu]_{x' \in \partial \tilde{\Omega}} =0, \\
\partial_t \tilde{\mathcal{H}} &= - \nabla \times \tilde{\mathcal{E}}, \quad (t,x') \in \R \times \tilde{\Omega}.
\end{array}
\right.
\end{equation}
Let the two-dimensional initial data be given by $(\mathcal{E}_0,\mathcal{H}_0) \in \mathcal{H}^3(\Omega)$. We extend the initial data using a smooth cut-off in the cylindrical direction. Let $\varphi \in C^\infty_c(\R)$ with $\varphi(x') = 1$ for $|x'| \leq R$ and $\varphi \in C^\infty_c(B(0,2R))$. The extended data is given by
\begin{equation}
\label{eq:InitialData3d}
\tilde{\mathcal{E}}_0(x'_1,x'_2,x'_3) = \varphi(x'_3)
\begin{pmatrix}
\mathcal{E}_{01}(x'_1,x'_2) \\
\mathcal{E}_{02}(x'_1,x'_2) \\
0
\end{pmatrix}, \qquad
\tilde{\mathcal{H}}_0(x'_1,x'_2,x'_3) = \varphi(x'_3)
\begin{pmatrix}
0 \\
0 \\
\mathcal{H}_0(x'_1,x'_2)
\end{pmatrix}
.
\end{equation}
For $(\mathcal{E}_0,\mathcal{H}_0) \in H^3(\Omega)$ we clearly have $(\tilde{\mathcal{E}}_0,\tilde{\mathcal{H}}_0) \in H^3(\tilde{\Omega})$. 
\subsubsection*{Compatibility conditions in two and three dimensions}

We turn to compatibility conditions. First, we record the compatibility conditions in two dimensions by changing to geodesic coordinates:
\begin{equation}
\label{eq:Maxwell2dGeodesic}
\left\{ \begin{array}{cl}
\partial_t (\tilde{\varepsilon} \mathcal{E}) &= \sqrt{g}^{-1} g \nabla_\perp \mathcal{H}, \quad \tilde{\varepsilon} = 1+ \langle \mathcal{E}, g^{-1} \mathcal{E} \rangle, \quad g^{-1} = 
\begin{pmatrix}
g^{1} & 0 \\
0 & 1
\end{pmatrix}, \\
\partial_t \mathcal{H} &= - \sqrt{g}^{-1} (\partial_1 \mathcal{E}_2 - \partial_2 \mathcal{E}_1), \quad (t,x') \in \R \times \R^2_{>0}.
\end{array} \right.
\end{equation}
The tangential direction in $\R^2_{>0} = \{ (x'_1,x'_2) \in \R^2 : x'_2 > 0 \}$ is $e_1$. The boundary condition is given by $[ \mathcal{E}_1 ]_{x'_2 = 0} = 0$. 
Since $\partial_t$ is a tangential derivative, we obtain
\begin{equation*}
\partial_t (\tilde{\varepsilon} \mathcal{E}_1) = \sqrt{g}^{-1} g^{1} \partial_2 \mathcal{H} \Rightarrow [\partial_2 \mathcal{H}]_{x'_2 = 0} = 0.
\end{equation*}
This means that first order compatibility conditions are Neumann boundary conditions for $\mathcal{H}$.
We take an additional time derivative in geodesic coordinates to find the second order compatibility condition:
\begin{equation*}
0 = [\partial_t^2 (\tilde{\varepsilon} \mathcal{E})_1]_{x_2' = 0} = [\partial_2 (\sqrt{g}^{-1} (\partial_1 \mathcal{E}_2 - \partial_2 \mathcal{E}_1)]_{x_2' =0}.
\end{equation*}

\medskip

 We turn to the nonlinear compatibility conditions for the Kerr nonlinearity in three dimensions.
The conditions are computed again in geodesic normal coordinates with cometric given by
\begin{equation*}
g^{-1} = 
\begin{pmatrix}
g^{11} & g^{12} & 0 \\
g^{21} & g^{22} & 0 \\
0 & 0 & 1
\end{pmatrix}
.
\end{equation*}

The equations are given by
\begin{equation}
 \label{eq:KerrLocalized}
\left\{ \begin{array}{cl}
 \partial_t ((1 + \langle \mathcal{E}, g^{-1} \mathcal{E} \rangle) \mathcal{E}) &= \sqrt{g}^{-1} g \nabla \times \mathcal{H},
\\
\partial_t \mathcal{H} &= - \sqrt{g}^{-1} g \nabla \times \mathcal{E}.
\end{array} \right.
\end{equation}
The boundary conditions (zeroth order compatibility conditions) read
\begin{equation}
\label{eq:NonlinearCompatibility0Appendix}
 [\mathcal{E}_1]_{x'_3 = 0} = [\mathcal{E}_2]_{x'_3 = 0} = [\mathcal{H}_3]_{x'_3 = 0} = 0.
\end{equation}
We compute the compatibility conditions of first order by taking \eqref{eq:KerrLocalized} and \eqref{eq:NonlinearCompatibility0} together:
\begin{equation*}
 [\partial_2 \mathcal{H}_3 - \partial_3 \mathcal{H}_2] = 0 , \quad [\partial_3 \mathcal{H}_1 - \partial_1 \mathcal{H}_3] = 0.
\end{equation*}
Since $\partial_1$ and $\partial_2$ are tangential derivatives, it follows that we have Neumann boundary conditions for $\mathcal{H}_1$ and $\mathcal{H}_2$:
\begin{equation}
\label{eq:NonlinearCompatibility1Appendix}
 [\partial_3 \mathcal{H}_1] = [\partial_3 \mathcal{H}_2] = 0.
\end{equation}
The time derivative of the third component of $\mathcal{H}_3$ yields no additional condition on $\mathcal{E}_1$ and $\mathcal{E}_2$.

\begin{equation*}
\partial_t^2 (\sqrt{g} g^{-1} \varepsilon \mathcal{E}) = \nabla \times (\sqrt{g}^{-1} g (\nabla \times \mathcal{E})).
\end{equation*}
Expanding gives
\begin{equation*}
\begin{split}
\partial_t^2 (\sqrt{g} g^{-1} \varepsilon \mathcal{E})_1 &= \partial_2 (\sqrt{g}^{-1} (\partial_1 \mathcal{E}_2 - \partial_2 \mathcal{E}_1)) \\
&\quad - \partial_3 (\sqrt{g}^{-1} (g_{21} (\partial_2 \mathcal{E}_3 - \partial_3 \mathcal{E}_2) + g_{22} (\partial_3 \mathcal{E}_1 -\partial_1 \mathcal{E}_3))), \\
\partial_t^2 (\sqrt{g} g^{-1} \varepsilon \mathcal{E})_2 &= \partial_3 (\sqrt{g}^{-1} (g_{11}(\partial_2 \mathcal{E}_3 - \partial_3 \mathcal{E}_2) + g_{12} (\partial_3 \mathcal{E}_1 - \partial_1 \mathcal{E}_3))) \\
&\quad - \partial_1 (\sqrt{g}^{-1} (\partial_1 \mathcal{E}_2 - \partial_2 \mathcal{E}_1)).
\end{split}
\end{equation*}
The left hand-side is vanishing at the boundary. Moreover, $\partial_1$ and $\partial_2$ are tangential derivatives, which means that for $i=1,2$
\begin{equation*}
[\partial_i (\sqrt{g}^{-1} (\partial_1 \mathcal{E}_2 - \partial_2 \mathcal{E}_1))]_{x_3' = 0} = 0.
\end{equation*}
We find that the second order compatibility conditions are given by
\begin{equation*}
\begin{split}
&\quad [\partial_3 (\sqrt{g}^{-1} (g_{21} (\partial_2 \mathcal{E}_3 - \partial_3 \mathcal{E}_2) + g_{22} (\partial_3 \mathcal{E}_1 -\partial_1 \mathcal{E}_3)))]_{x_3' = 0} = 0, \\
&\quad [\partial_3 (\sqrt{g}^{-1} (g_{11}(\partial_2 \mathcal{E}_3 - \partial_3 \mathcal{E}_2) + g_{12} (\partial_3 \mathcal{E}_1 - \partial_1 \mathcal{E}_3)))]_{x_3' = 0} = 0.
\end{split}
\end{equation*}

\begin{equation*}
 [\partial_3 \partial_t \mathcal{H}_1] = [\partial_3^2 \mathcal{E}_1] - [\partial_1 \partial_3 \mathcal{E}_3] = [\partial_3^2 \mathcal{E}_1] = 0.
\end{equation*}

\medskip

\subsubsection*{Relating the solutions in two and three dimensions}

To relate the boundary conditions in two and three dimensions after cylindrical extension, we note that we can extend the geodesic coordinates in two dimensions
\begin{equation*}
g^{-1} = 
\begin{pmatrix}
g^{1} & 0 \\
0 & 1
\end{pmatrix}
\end{equation*}
such that $e_1$ denotes the tangential and $e_2$ the normal direction trivially as
\begin{equation*}
\tilde{g}^{-1} = 
\begin{pmatrix}
g^{1} & 0 & 0 \\
0 & 1 & 0 \\
0 & 0 & 1
\end{pmatrix}
\end{equation*}
such that $e_3$ denotes the second tangential direction. The boundary conditions in geodesic coordinates in two dimensions are given by
\begin{equation}
\label{eq:2dBoundaryConditionsAppendix}
\begin{split}
[ \mathcal{E}_1]_{x'_2 = 0} &= 0, \\
 [\partial_2 \mathcal{H}]_{x'_2 = 0} &= 0, \\
[\partial_2 (\sqrt{g}^{-1}(\partial_1 \mathcal{E}_2 - \partial_2 \mathcal{E}_1)]_{x'_2 = 0} &= 0.
  \end{split}
\end{equation}
In three dimensions we find
\begin{equation}
\label{eq:3dBoundaryConditionsAppendix}
\begin{split}
[ \tilde{\mathcal{E}}_1]_{x'_2 = 0} &= [\tilde{\mathcal{E}}_3]_{x'_2 = 0} = [\tilde{\mathcal{H}}_2]_{x'_2 = 0} = 0, \\
[\partial_2 \tilde{\mathcal{H}}_1]_{x'_2 = 0} &= [\partial_2 \tilde{\mathcal{H}}_3]_{x'_2 = 0} = 0, \\
 [\partial_2 (\sqrt{g}^{-1} (\partial_1 \tilde{\mathcal{E}}_2 - \partial_2 \tilde{\mathcal{E}}_1))]_{x_2'= 0} &= 0, \quad [\partial_2(\sqrt{g}^{-1} g_1 (\partial_2 \tilde{\mathcal{E}}_3 - \partial_3 \tilde{\mathcal{E}}_2))]_{x_2' = 0} = 0.
\end{split}
\end{equation}
It turns out that the conditions in three dimensions either follow from the conditions in two dimensions or from cylindrical extension. Indeed, 
\begin{equation*}
[\tilde{\mathcal{E}}_1]_{x'_2 = 0} = [\partial_2 (\sqrt{g}^{-1} (\partial_1 \tilde{\mathcal{E}}_2 - \partial_2 \tilde{\mathcal{E}}_1))]_{x_2'= 0} = 0
\end{equation*}
is immediate from \eqref{eq:2dBoundaryConditionsAppendix}. Since $e_3$ is the cylindrical direction, we find by the definition of the extended fields
\begin{equation*}
\begin{split}
[\tilde{\mathcal{E}}_3]_{x'_2 = 0 } = [\tilde{\mathcal{H}}_2]_{x'_2 = 0} =  [\partial_2 \tilde{\mathcal{H}}_1]_{x'_2 = 0} &= [\partial_2 \tilde{\mathcal{H}}_3]_{x'_2 = 0} = 0, \\
[\partial_2(\sqrt{g}^{-1} g_1 (\partial_2 \tilde{\mathcal{E}}_3 - \partial_3 \tilde{\mathcal{E}}_2))]_{x_2' = 0} &= 0.
\end{split}
\end{equation*}

This yields local-in-time solutions to \eqref{eq:Maxwell3dAppendix} by applying \cite[Theorem~5.3]{Spitz2019}. Next, we argue that the solutions $(\tilde{\mathcal{E}},\tilde{\mathcal{H}})$ in $\Omega \times \R$ to \eqref{eq:Maxwell3dAppendix} yield solutions to \eqref{eq:Maxwell2dAppendix}. 

By finite speed of propagation there is $T = T(\tilde{\Omega}) = T(\Omega)$ such that for $0 \leq t \leq T$
\begin{equation*}
\tilde{\mathcal{E}}_3(t,x'_1,x'_2,0) = \tilde{\mathcal{H}}_1(t,x'_1,x'_2,0) = \tilde{\mathcal{H}}_2(t,x'_1,x'_2,0) = 0.
\end{equation*}
Secondly, we argue that $\tilde{\mathcal{E}}_1$, $\tilde{\mathcal{E}}_2$, and $\tilde{\mathcal{H}}_3$ do not depend on the cylindrical coordinate $x'_3$ for $0 \leq t \leq T$. For the components $i=1,2$ of $\tilde{\mathcal{E}}$ this follows from \eqref{eq:Maxwell3dAppendix}:
\begin{equation*}
\begin{split}
\partial_t \tilde{\mathcal{H}}_1 &= 0 = \partial_2 \tilde{\mathcal{E}}_3 - \partial_3 \tilde{\mathcal{E}}_2 \; \Rightarrow \; \partial_3 \tilde{\mathcal{E}}_2 = 0, \\
\partial_t \tilde{\mathcal{H}}_2 &= 0 = \partial_1 \tilde{\mathcal{E}}_3 - \partial_3 \tilde{\mathcal{E}}_1 \; \Rightarrow \; \partial_3 \tilde{\mathcal{E}}_1 = 0.
\end{split}
\end{equation*}
For $\tilde{\mathcal{H}}$, we observe that $\nabla \cdot \tilde{\mathcal{H}}_0(x'_1,x'_2,0) = 0$. It follows from the time evolution that $\nabla \cdot \tilde{\mathcal{H}}(t,x'_1,x'_2,0) = 0$, and consequently, for $0 < t \leq T$ we find
\begin{equation*}
\partial_1 \tilde{\mathcal{H}}_1 + \partial_2 \tilde{\mathcal{H}}_2 + \partial_3 \tilde{\mathcal{H}}_3 = 0 \Rightarrow \partial_3 \tilde{\mathcal{H}}_3 = 0.
\end{equation*}
Hence, we retrieve local-in-time solutions to \eqref{eq:Maxwell2dAppendix} by restricting solutions to \eqref{eq:Maxwell3dAppendix}:
\begin{equation*}
(\mathcal{E}(t,x'_1,x'_2),\mathcal{H}(t,x'_1,x'_2)) = (\tilde{\mathcal{E}}_1(t,x'_1,x'_2,0),\tilde{\mathcal{E}}_2(t,x'_1,x'_2,0),\tilde{\mathcal{H}}_3(t,x'_1,x'_2,0)).
\end{equation*}
The restriction is well-defined by Sobolev embedding, which yields that $(\tilde{\mathcal{E}},\tilde{\mathcal{H}})(t) \in C^{1,\frac{1}{2}}(\tilde{\Omega})$. It remains to check that the solutions obtained from restriction are in $H^3(\Omega)$. To this end, we note that by independence of $x'_3$ for $0 \leq t \leq T$ and $x'_3 \in [-\varepsilon,\varepsilon]$:
\begin{equation*}
(\tilde{\mathcal{E}}_1,\tilde{\mathcal{E}}_2,\tilde{\mathcal{H}}_3)(t,x'_1,x'_2,x'_3) = 
(\tilde{\mathcal{E}}_1,\tilde{\mathcal{E}}_2,\tilde{\mathcal{H}}_3)(t,x'_1,x'_2,0).
\end{equation*}
Hence, $(\tilde{\mathcal{E}},\tilde{\mathcal{H}})(t) \in H^3(\tilde{\Omega})$ implies that $(\mathcal{E},\mathcal{H})(t) \in H^3(\Omega)$.

$\hfill \Box$

\section{Strichartz estimates for Maxwell equations in the full space revisited}
\label{appendix:StrichartzRevisited}
\subsection{Strichartz estimates for Maxwell equations in three dimensions}

Here we show Strichartz estimates away from the boundary based on Strichartz estimates in the full space. Recall that we can find $T$ small enough such that $(\mathcal{E},\mathcal{H})(t)$ within $\Omega^{\text{int}} = \{ x \in \Omega : d(x) > \varepsilon /2 \}$ only depends on initial data in the interior $\tilde{\Omega}^{\text{int}} = \{ x \in \Omega : d(x) > \varepsilon /4 \}$, and the solution does not reach the boundary for times $t \leq T$. We prove that
\begin{equation}
\label{eq:StrichartzInteriorAppendix}
 \| (\mathcal{E},\mathcal{H}) \|_{L_T^p L^q(\Omega^{\text{int}})} \lesssim \| (\mathcal{E}_0,\mathcal{H}_0) \|_{H^s(\Omega)} + \| \rho_e(0) \|_{H^{s-1+\frac{1}{p}}}.
\end{equation}

A difficulty in applying the results from \cite{Schippa2021Maxwell3d} arises as these are formulated in terms of $\mathcal{D} = \varepsilon \mathcal{E}$ and $\mathcal{B}= \mu \mathcal{H}$ and with negative derivatives on the left-hand side, which requires the use of commutator estimates to find Strichartz estimates like in \eqref{eq:StrichartzInteriorAppendix}.
\begin{theorem}[{\cite[Theorem~1.3]{Schippa2021Maxwell3d}}]
\label{thm:3dStrichartzFullSpace}
Let $\varepsilon_1,\mu_1 \in C^1(\R \times \R^3;\R)$ and define $\varepsilon = \text{diag}(\varepsilon_1,\varepsilon_1,\varepsilon_1): \R \times \R^3 \to \R^{3 \times 3}$, $\mu = \text{diag}(\mu_1,\mu_1,\mu_1): \R \times \R^3 \to \R^{3 \times 3}$ be matrix-valued functions, which satisfy \eqref{eq:Ellipticity} and $\partial_x^2 \varepsilon \in L_t^1 L_{x'}^\infty$ and $\partial_x^2 \mu \in L_t^1 L_{x'}^\infty$. Let $(s,p,q)$ be wave Strichartz admissible in three dimensions, i.e.,
\begin{equation*}
2 \leq p \leq \infty, \; 2 \leq q < \infty\footnote{For $q=\infty$, one has to consider the Besov refinement $B_{\infty,2}$ in \eqref{eq:3dStrichartzFullSpace} on the left-hand side.}, \quad \frac{2}{p} + \frac{2}{q} \leq 1, \quad (p,q) \neq (2, \infty), \quad s = 3 \big( \frac{1}{2} - \frac{1}{q} \big) - \frac{1}{p}.
\end{equation*}
Let $u = (u_1,\ldots,u_6)= (u^{(1)},u^{(2)}) : \R \times \R^3 \to \R^3 \times \R^3$, and
\begin{equation*}
\tilde{P} = 
\begin{pmatrix}
\partial_t & - \nabla \times (\mu^{-1} \cdot) \\
\nabla \times (\varepsilon^{-1} \cdot) & \partial_t
\end{pmatrix}
.
\end{equation*}
The following estimate holds:
\begin{equation}
\label{eq:3dStrichartzFullSpace}
\begin{split}
\| |D'|^{-s} u \|_{L_t^p(0,T;L_{x'}^q)} &\lesssim \nu^{\frac{1}{p}} \| u \|_{L_t^\infty L_{x'}^2} + \nu^{-\frac{1}{p'}} \| \tilde{P}(x,D) u \|_{L_t^1 L_{x'}^2} \\
&\quad + T^{\frac{1}{p}} \big( \| |D'|^{-1+\frac{1}{p}} (\nabla \cdot u^{(1)}, \nabla \cdot u^{(2)}) \|_{L_T^\infty L_{x'}^2} \\
&\quad + \| |D'|^{-1+\frac{1}{p}} \partial_t (\nabla \cdot u^{(1)}, \nabla \cdot u^{(2)}) \|_{L_T^1 L_{x'}^2}),
\end{split}
\end{equation}
whenever the right hand-side is finite, provided that $\nu \geq 1$, and $T \| \partial_x^2 \varepsilon \|_{L_t^1 L_{x'}^\infty} + T \| \partial_x^2 \mu \|_{L_t^1 L_{x'}^\infty} \leq \nu^2$.
\end{theorem}
For the present application, we can fix $T$ and $\nu$ such that the solutions stay away from the boundary (using finite speed of propagation). We have by Theorem \ref{thm:3dStrichartzFullSpace}
\begin{equation}
\label{eq:StrichartzEstimate3dFullSpaceI}
\begin{split}
\| \langle D' \rangle^{-s} u \|_{L_t^p(0,T;L_{x'}^q(\R^3))} &\lesssim_{T, \varepsilon, \mu} \| u \|_{L_t^\infty L_{x'}^2} + \| \tilde{P}_1(x,D)u \|_{L_t^1 L_{x'}^2} \\
&\quad + \| |D'|^{-1+\frac{1}{p}} (\nabla \cdot (u^{(1)}, \nabla \cdot u^{(2)}) \|_{L_t^\infty L_{x'}^2} \\
&\quad + \| |D'|^{-1+\frac{1}{p}} \partial_t (\nabla \cdot u^{(1)}, \nabla \cdot u^{(2)}) \|_{L_t^1 L_{x'}^2}
\end{split}
\end{equation}
with
\begin{equation*}
\tilde{P}_1 = 
\begin{pmatrix}
\partial_t & - \mu^{-1} \nabla \times \\
\varepsilon^{-1} \nabla \times & \partial_t
\end{pmatrix}
\end{equation*}
by H\"older's inequality. We apply \eqref{eq:StrichartzEstimate3dFullSpaceI} to $u = \langle D' \rangle^s (\varepsilon \mathcal{E}, \mu \mathcal{H})$ to find
\begin{equation*}
\begin{split}
\| (\varepsilon \mathcal{E},\mu \mathcal{H}) \|_{L_t^p (0,T; L_{x'}^q(\R^3))} &\lesssim_{T,\varepsilon,\mu} \| \langle D' \rangle^s (\varepsilon \mathcal{E},\mu \mathcal{H}) \|_{L_t^\infty L_{x'}^2} + \| \tilde{P}_1(\langle D' \rangle^s (\varepsilon \mathcal{E},\mu \mathcal{H})) \|_{L_t^1 L_{x'}^2} \\
&\; + \| \langle D' \rangle^{s} |D'|^{-1+\frac{1}{p}} \rho_{em} \|_{L_t^\infty L_{x'}^2} + \| \langle D' \rangle^{s} |D'|^{-1+\frac{1}{p}} \partial_t \rho_{em} \|_{L_t^1 L_{x'}^2}
\end{split}
\end{equation*}
with $\rho_{em} = (\nabla \cdot (\varepsilon \mathcal{E}), \nabla \cdot (\mu \mathcal{H}))$. By uniform ellipticity of $\varepsilon$ and $\mu$, we find
\begin{equation*}
\| (\varepsilon \mathcal{E},\mu \mathcal{H}) \|_{L_t^p(0,T;L_{x'}^q)} \gtrsim \| (\mathcal{E},\mathcal{H}) \|_{L_t^p(0,T;L_{x'}^q)},
\end{equation*}
and by the fractional Leibniz rule, we obtain
\begin{equation*}
\| \langle D' \rangle^s (\varepsilon \mathcal{E},\mu \mathcal{H}) \|_{L_{x'}^2} \lesssim_{\| (\varepsilon,\mu) \|_{C_x^{\lceil s \rceil +1}}} \| \langle D' \rangle^s (\mathcal{E},\mathcal{H}) \|_{L_{x'}^2}.
\end{equation*}
Moreover, the charge terms are already in suitable form. By commutator estimates we shall argue that
\begin{equation}
\label{eq:CommutatorInterior}
\| \tilde{P}_1 \langle D' \rangle^s (\varepsilon \mathcal{E},\mu \mathcal{H}) \|_{L_{x'}^2} \lesssim \| \langle D' \rangle^s P (\varepsilon \mathcal{E},\mu \mathcal{H}) \|_{L_{x'}^2} + \| \langle D' \rangle^s (\mathcal{E},\mathcal{H}) \|_{L_{x'}^2}.
\end{equation}
Above $P$ denotes the Maxwell operator for $(\mathcal{E},\mathcal{H})$:
\begin{equation*}
P = 
\begin{pmatrix}
\partial_t (\varepsilon \cdot) & - \nabla \times \\
\nabla \times & \partial_t (\mu \cdot)
\end{pmatrix}
.
\end{equation*}
To prove \eqref{eq:CommutatorInterior}, we use the fractional Leibniz rule (cf. \cite{GrafakosOh2014})
\begin{equation*}
\| \langle D' \rangle^{\rho} (fg) \|_{L_{x'}^2} \lesssim \| \langle D' \rangle^{\rho} f \|_{L^2_{x'}} \| g \|_{L^\infty_{x'}} + \| \langle D' \rangle^{\rho} g \|_{L^2_{x'}} \| f \|_{L^\infty_{x'}} \quad (\rho \geq 0)
\end{equation*}
 and the following elementary commutator estimate:
\begin{lemma}
\label{lem:CommutatorLemmaFractionalLeibniz}
Let $X = B^1_{\infty,2}(\R^d) \cap C^{0,1}(\R^d)$.
The following estimate holds:
\begin{align}
\label{eq:EstRhoSmall}
\| [\varepsilon, \langle D' \rangle^\rho ] \partial_{x'} u \|_{L^2_{x'}(\R^d)} &\lesssim \| \varepsilon \|_{X} \| \langle D' \rangle^\rho u \|_{L_{x'}^2(\R^d)} \text{ for } 0 < \rho \leq 1, \\
\label{eq:EstRhoLarge}
\| [\varepsilon, \langle D' \rangle^\rho] \partial_{x'} u \|_{L^2_{x'}(\R^d)} &\lesssim \| \varepsilon \|_{B^\rho_{\infty,2}(\R^d)} \| \langle D' \rangle^\rho u \|_{L^2(\R^d)} \text{ for } \rho > 1.
\end{align}
\end{lemma}
\begin{proof}
We use a Littlewood-Paley decomposition:
\begin{equation*}
\| [ \varepsilon, \langle D' \rangle^\rho ] \partial_{x'} u \|_{L^2_{x'}}^2 = \sum_{N \geq 1} \| P_N [\varepsilon, \langle D' \rangle^\rho] \partial_{x'} u \|_{L^2_{x'}}^2.
\end{equation*}
We write
\begin{equation}
\label{eq:DecompositionA}
\begin{split}
P_N [\varepsilon, \langle D' \rangle^\rho] \partial_{x'} u &= P_N [\varepsilon_{< N/8}, \langle D' \rangle^\rho] \partial_{x'} \tilde{P}_N u + P_N [\varepsilon_{\sim N}, \langle D' \rangle^\rho] \partial_{x'} P_{<N/8} u \\
&\quad + P_N [\varepsilon_{\gtrsim N}, \langle D' \rangle^\rho] \partial_{x'} P_{\gtrsim N} u.
\end{split}
\end{equation}
We estimate the first term in \eqref{eq:DecompositionA}: Rewrite
\begin{equation}
\label{eq:DecA1}
P_N [\varepsilon_{<N/8}, \langle D' \rangle^\rho] \tilde{P}_N \partial_{x'} u = P_N \partial_{x'} [\varepsilon_{<N/8}, \langle D' \rangle^\rho] \tilde{P}_N u - P_N [\partial_{x'} \varepsilon_{<N/8}, \langle D' \rangle^\rho] \tilde{P}_N u.
\end{equation}
The second term in \eqref{eq:DecA1} is directly estimated by
\begin{equation*}
\| P_N [\partial_{x'} \varepsilon_{<N/8}, \langle D' \rangle^\rho] \tilde{P}_N u \|_{L^2_{x'}} \lesssim N^\rho \| \partial_{x'} \varepsilon \|_{L^\infty_{x'}} \| \tilde{P}_N u \|_{L^2_{x'}}.
\end{equation*}
For the first term in \eqref{eq:DecA1} we write
\begin{equation*}
\| P_N \partial_{x'} [ \varepsilon_{<N/8}, \langle D' \rangle^\rho ] \tilde{P}_N u \|_{L^2_{x'}} \lesssim N \| P_N [\varepsilon_{<N/8}, \langle D' \rangle^\rho \tilde{\tilde{P}}_N ] \tilde{P}_N u \|_{L^2_{x'}}.
\end{equation*}
Let $K_N$ denote the kernel of $\langle D' \rangle^\rho \tilde{\tilde{P}}_N$:
\begin{equation*}
\langle D' \rangle^\rho \tilde{\tilde{P}}_N f(x) = \int_{\R^d} K_N(x-y) f(y) dy.
\end{equation*}
We have the pointwise kernel estimate:
\begin{equation}
\label{eq:KernelEstimate}
|K_N(x)| \lesssim N^\rho N^d (1+N|x|)^{-M} \text{ for any } M \geq 1. 
\end{equation}
This follows from $K_N(x) = \int e^{ix.\xi} a(N^{-1} \xi) \langle \xi \rangle^\rho d\xi$ with $a \in C^\infty_c(B(0,4) \backslash B(0,1/4))$, rescaling, and non-stationary phase.

Therefore, by the mean-value theorem,
\begin{equation*}
|\varepsilon_{<N/8}(x) K_N(x-y) - K_N(x-y) \varepsilon_{<N/8}(y)| \lesssim |x-y| |K_N(x-y)| \| \partial_{x'} \varepsilon \|_{L_{x'}^\infty}.
\end{equation*}
An application of Young's inequality gives
\begin{equation*}
\| P_N [\varepsilon_{<N/8}, \langle D' \rangle^\rho] \tilde{P}_N u \|_{L^2_{x'}} \lesssim N^{\rho - 1} \| \partial_{x'} \varepsilon \|_{L^\infty_{x'}} \| \tilde{P}_N u \|_{L^2_{x'}}.
\end{equation*}
This shows \eqref{eq:EstRhoSmall} and \eqref{eq:EstRhoLarge} in the considered cases by square summation. We turn to the second term in \eqref{eq:DecompositionA}, where we distinguish $0<\rho \leq 1$ and $\rho > 1$. We have for $0<\rho \leq 1$
\begin{equation*}
\begin{split}
\| P_N [\varepsilon_{\sim N}, \langle D' \rangle^\rho] \partial_{x'} P_{<N/8} u \|_{L^2_{x'}} &\lesssim N \| \varepsilon_{\sim N} \|_{L^\infty_{x'}} \| \langle D' \rangle^\rho P_{<N/8} u \|_{L^2_{x'}} \\
&\lesssim N \| \varepsilon_{\sim N} \|_{L^\infty} \| \langle D' \rangle^\rho u \|_{L^2_{x'}}
\end{split}
\end{equation*}
with straight-forward square summation.

For $\rho > 1$, we obtain
\begin{equation*}
\| P_N [\varepsilon_{\sim N}, \langle D' \rangle^\rho] \partial_{x'} P_{< N/8} u \|_{L^2_{x'}} \lesssim N^\rho \| \varepsilon_{\sim N} \|_{L^\infty_{x'}} \| \partial_{x'} P_{\leq N/8} u \|_{L^2_{x'}}.
\end{equation*}

It remains to estimate the third term in \eqref{eq:DecompositionA}, which is rewritten as
\begin{equation}
\label{eq:DecompositionA2}
P_N [\varepsilon_{\gtrsim N}, \langle D' \rangle^\rho] \partial_{x'} P_{\gtrsim N} u = P_N \partial_{x'} [\varepsilon_{\gtrsim N}, \langle D' \rangle^\rho] P_{\gtrsim N} u - P_N [\partial_{x'} \varepsilon_{\gtrsim N}, \langle D' \rangle^\rho] P_{\gtrsim N} u.
\end{equation}
We write the first term in \eqref{eq:DecompositionA2} as
\begin{equation}
\label{eq:DecompositionA21}
P_N \partial_{x'} [\varepsilon_{\gtrsim N}, \langle D' \rangle^\rho] P_{\gtrsim N} u = P_N \partial_{x'} (\varepsilon_{\gtrsim N} \langle D' \rangle^\rho  P_{\gtrsim N} u) - P_N \partial_{x'} \langle D' \rangle^\rho ( \varepsilon_{\gtrsim N} P_{\gtrsim N} u).
\end{equation}
For the first term in \eqref{eq:DecompositionA21} we find by the Cauchy-Schwarz inequality
\begin{equation}
\label{eq:AuxComAppendix}
\sum_{N \geq 1} \| P_N \partial_{x'} (\varepsilon_{\gtrsim N} \langle D' \rangle^\rho P_{\gtrsim N} u) \|_{L^2}^2 \lesssim \sum_{N \geq 1} N^2 \sum_{M \gtrsim N} \| \varepsilon_M \|_{L^\infty}^2 \| \langle D' \rangle^\rho u \|_{L^2}^2 .
\end{equation}
Indeed, we have
\begin{equation*}
\begin{split}
\| \sum_{M \gtrsim N} \varepsilon_M \langle D' \rangle^\rho P_M u \|_{L^2}^2 &\lesssim \big\| \big( \sum_{M \gtrsim N} |\varepsilon_M |^2 \big)^{\frac{1}{2}} \big( \sum_{M \gtrsim N} |P_M \langle D' \rangle^\rho u |^2 \big)^{\frac{1}{2}} \big\|_{L^2}^2 \\
&\lesssim \big\| \big( \sum_{M \gtrsim N} | \varepsilon_M |^2 \big)^{\frac{1}{2}} \big\|_{L^\infty}^2 \big\| \big( \sum_{M \gtrsim N} |P_M \langle D' \rangle^\rho u |^2 \big)^{\frac{1}{2}} \big\|_{L^2}^2 \\
&\lesssim \big( \sum_{M \gtrsim N} \| \varepsilon_M \|_{L^\infty}^2 \big) \| P_{\gtrsim N} \langle D' \rangle^\rho u \|_{L^2}^2.
\end{split}
\end{equation*}
We conclude the above estimate by changing the summation as
\begin{equation*}
\begin{split}
\eqref{eq:AuxComAppendix} &\lesssim \sum_{M \gtrsim 1} \| \varepsilon_M \|_{L^\infty}^2 \sum_{1 \lesssim N \lesssim M} N^2 \| \langle D' \rangle^\rho u \|_{L^2}^2 \lesssim \sum_{M \gtrsim 1} M^2 \| \varepsilon_M \|_{L^\infty}^2 \| \langle D' \rangle^\rho u \|_{L^2}^2 \\
 &\lesssim \| \varepsilon \|_{B^1_{\infty,2}}^2 \| u \|_{H^\rho}^2.
\end{split}
\end{equation*}
For the second term in \eqref{eq:DecompositionA21} we find
\begin{equation*}
\begin{split}
\| P_N \partial_{x'} \langle D' \rangle^\rho (\varepsilon_{\gtrsim N} P_{\gtrsim N} u) \|_{L^2} &\lesssim N^{\rho + 1} \| P_N (\sum_{M \gtrsim N} \varepsilon_M P_M u) \|_{L^2} \\
&\lesssim N^{\rho + 1} \| P_N \big( \sum_{M \gtrsim N} |\varepsilon_M|^2 \big)^{\frac{1}{2}} \big( \sum_{M \gtrsim N} |P_M u |^2 \big)^{\frac{1}{2}} \|_{L^2} \\
&\lesssim N^{\rho + 1} \big( \sum_{M \gtrsim N} \| \varepsilon_M \|_{L^\infty}^2 \big)^{\frac{1}{2}} \| P_{\gtrsim N} u \|_{L^2}.
\end{split}
\end{equation*}
For this reason we have
\begin{equation*}
\begin{split}
\sum_{N \geq 1} \| P_N \partial_{x'} \langle D' \rangle^\rho ( \varepsilon_{\gtrsim N} P_{\gtrsim N} u) \|_{L^2}^2 
&\lesssim \sum_{N \geq 1} N^{2 \rho} N^2 \sum_{M \gtrsim N} \| \varepsilon_M \|_{L^\infty}^2 \| P_{\gtrsim N} u \|_{L^2}^2 \\
&\lesssim \sum_{N \geq 1} N^2 \sum_{M \gtrsim N} \| \varepsilon_M \|_{L^\infty}^2 \| \langle D' \rangle^\rho u \|_{L^2}^2 \\
&\lesssim \| \varepsilon \|_{B^1_{\infty,2}}^2 \| \langle D' \rangle^\rho u \|_{L^2}^2.
\end{split}
\end{equation*}
We turn to the second term in the High-High-interaction \eqref{eq:DecompositionA2}, which is written as
\begin{equation}
\label{eq:DecompositionA22}
P_N [ \partial_{x'} \varepsilon_{\gtrsim N}, \langle D' \rangle^\rho] P_{\gtrsim N} u = P_N (\partial_{x'} \varepsilon_{\gtrsim N})(\langle D' \rangle^\rho P_{\gtrsim N} u) - P_N \langle D' \rangle^\rho (\partial_{x'} \varepsilon_{\gtrsim N}) P_{\gtrsim N} u.
\end{equation}
For the first term in \eqref{eq:DecompositionA22} we have by Plancherel's theorem and the Cauchy-Schwarz inequality
\begin{equation*}
\begin{split}
\sum_N \| P_N (\partial_{x'} \varepsilon_{\gtrsim N}) (\langle D' \rangle^\rho P_{\gtrsim N} u) \|_{L^2}^2 
&= \| \sum_{M \geq 1} (\partial_{x'} \varepsilon_M) (\langle D' \rangle^\rho P_M u) \|_{L^2}^2 \\
&\leq \big\| \big( \sum_M |\partial_{x'} \varepsilon_M |^2 \big)^{\frac{1}{2}} \big( \sum_M |\langle D' \rangle^\rho P_M u |^2 \big)^{\frac{1}{2}} \big\|_{L^2}^2 \\
&\leq \big\| \big( \sum_M |\partial_{x'} \varepsilon_M |^2 \big)^{\frac{1}{2}} \|_{L^\infty}^2 \big\| \big( \sum_M |\langle D' \rangle^\rho P_M u |^2 \big)^{\frac{1}{2}} \big\|_{L^2}^2.
\end{split}
\end{equation*}
We use the characterization of BMO by the Littlewood-Paley square function and the embedding $L^\infty \hookrightarrow BMO$ to find
\begin{equation*}
\lesssim \| \partial_{x'} \varepsilon \|^2_{L^\infty} \| \langle D' \rangle^\rho u \|_{L^2}^2.
\end{equation*}
The second term in \eqref{eq:DecompositionA22} is better behaved than the first term and can be estimated similarly. This finishes the proof.

\end{proof}
We now turn to the proof of \eqref{eq:CommutatorInterior}: For the time derivatives there is no commutator, but we have
\begin{equation*}
\mu^{-1} \nabla \times (\langle D' \rangle^s (\mu \mathcal{H})) = [\mu_1^{-1}, \langle D' \rangle^s] \nabla \times (\mu \mathcal{H}) + \langle D' \rangle^s (\mu^{-1} \nabla \times (\mu \mathcal{H})).
\end{equation*}
Furthermore, $\mu^{-1} \nabla \times (\mu \mathcal{H}) = O(\mu^{-1} (\partial \mu) \mathcal{H}) + \nabla \times \mathcal{H}$. In both steps, we use that $\mu = \mu_1 1_{3 \times 3}$. By Lemma \ref{lem:CommutatorLemmaFractionalLeibniz} and the fractional Leibniz rule, we find
\begin{equation*}
\| \mu^{-1} \nabla \times (\langle D' \rangle^s (\mu \mathcal{H})) - \langle D' \rangle^s \nabla \times \mathcal{H} \|_{L_{x'}^2} \lesssim_{\| \mu \|_{C_{x'}^{\lceil s \rceil + 1}}} \| \langle D' \rangle^s \mathcal{H} \|_{L^2_{x'}}.
\end{equation*}
This also shows the corresponding estimate
\begin{equation*}
\| \varepsilon^{-1} \nabla \times (\langle D' \rangle^s (\varepsilon \mathcal{E})) - \langle D' \rangle^s \nabla \times \mathcal{E} \|_{L^2_{x'}} \lesssim_{\| \varepsilon \|_{C_{x'}^{\lceil s \rceil + 1}}} \| \langle D' \rangle^s \mathcal{E} \|_{L^2_{x'}}.
\end{equation*}
We conclude \eqref{eq:StrichartzEstimate3dFullSpaceI}, which implies
\begin{equation}
\label{eq:StrichartzEstimate3dFullSpaceII}
\begin{split}
\| (\mathcal{E},\mathcal{H}) \|_{L_t^p(0,T;L_{x'}^q)} &\lesssim_{T,\varepsilon,\mu} \| \langle D' \rangle^s (\mathcal{E},\mathcal{H}) \|_{L_t^\infty L_{x'}^2} + \| \langle D' \rangle^s P (\mathcal{E},\mathcal{H}) \|_{L_t^1 L_{x'}^2} \\
&\quad + \| \langle D' \rangle^{s} |D'|^{-1 + \frac{1}{p}} \rho_{em} \|_{L_t^\infty L_{x'}^2} + \| \langle D' \rangle^s |D'|^{-1 + \frac{1}{p}} \partial_t \rho_{em} \|_{L_t^1 L_{x'}^2}.
\end{split}
\end{equation}
This we apply to the homogeneous solution $(\mathcal{E},\mathcal{H})$, which remains in the interior up to time $T>0$ to find with $P(\mathcal{E},\mathcal{H}) = 0$ and $\partial_t \rho_{em}(0) = 0$:
\begin{equation*}
\| (\mathcal{E}, \mathcal{H}) \|_{L_t^p(0,T;L_{x'}^q)} \lesssim \| \langle D' \rangle^s(\mathcal{E},\mathcal{H})(0) \|_{L_{x'}^2} + \| \langle D' \rangle^{s - 1 + \frac{1}{p}} \rho_{em}(0) \|_{L_{x'}^2}.
\end{equation*}
Note that we used the trivial estimate 
\begin{equation*}
\| |D'|^{-1+\frac{1}{p}} \langle D' \rangle^s P_{\lesssim 1} \nabla \cdot \mathcal{E}(0) \|_{L^2_{x'}} \lesssim \| \mathcal{E}(0) \|_{L^2_{x'}}
\end{equation*}
to recast the homogeneous derivatives in \eqref{eq:StrichartzEstimate3dFullSpaceII} for low frequencies as inhomogeneous derivatives. This finishes the proof of Strichartz estimates for the interior part.

\subsection{Strichartz estimates for Maxwell equations in two dimensions}

The purpose of this section is to show Strichartz estimates for Maxwell equations with rough coefficients in the full space in two dimensions, which are suitable for the arguments of this paper. 

The following is the analog of Theorem \ref{thm:3dStrichartzFullSpace} in two dimensions:
\begin{theorem}
\label{thm:2dStrichartzFull}
Let $\varepsilon_{ij}$, $\mu_1 \in C^1(\R \times \R^2;\R)$ for $i=1,2$ such that $(\varepsilon_{ij})_{i,j = 1,2}$ satisfies \eqref{eq:Ellipticity2d} and $\partial^2_{x} \varepsilon_{ij} \in L_t^1 L_{x'}^\infty$ and $\partial^2_x \mu_1 \in L_t^1 L_{x'}^\infty$. Let $(s,p,q)$ be wave Strichartz admissible in two dimensions, i.e.,
\begin{equation*}
2 \leq p \leq \infty, \; 2 \leq q <\infty \footnote{For $q=\infty$, one has to consider the Besov refinement $B_{\infty,2}$ in \eqref{eq:Strichartz2dFullSpace} on the left-hand side.}, \quad \frac{2}{p} + \frac{1}{q} \leq \frac{1}{2}, \quad s = 2 \big( \frac{1}{2} - \frac{1}{q} \big) - \frac{1}{p}.
\end{equation*}
Let $u=(u_1,u_2,u_3) = (u^{(1)},u^{(2)}): \R \times \R^2 \to \R^2 \times \R$, and
\begin{equation*}
\tilde{P} = 
\begin{pmatrix}
\partial_t & 0 & -\partial_2 (\mu_1 \cdot) \\
0 & \partial_t & \partial_1 (\mu_1 \cdot) \\
\partial_1 (\varepsilon_{21} \cdot) - \partial_2 (\varepsilon_{11} \cdot) & \partial_1( \varepsilon_{22} \cdot) - \partial_2(\varepsilon_{12} \cdot) & \partial_t
\end{pmatrix}
.
\end{equation*}
The following estimate holds:
\begin{equation}
\label{eq:Strichartz2dFullSpace}
\begin{split}
\| |D'|^{-s} u \|_{L_t^p(0,T;L^q_{x'})} &\lesssim \nu^{\frac{1}{p}} \| u \|_{L_t^\infty L_{x'}^2} + \nu^{-\frac{1}{p'}} \| \tilde{P}(x,D) u \|_{L_t^1 L_{x'}^2} \\
&\quad + T^{\frac{1}{p}} \big( \| |D'|^{-1 + \frac{1}{p}} \nabla \cdot u^{(1)} \|_{L_t^\infty L_{x'}^2} + \| |D'|^{-1+\frac{1}{p}} \partial_t \nabla \cdot u^{(1)} \|_{L_t^1 L_{x'}^2} \big),
\end{split}
\end{equation}
whenever the right hand-side is finite, provided that $\nu \geq 1$, and $T \sum_{i,j} \| \partial_x^2 \varepsilon_{ij} \|_{L_t^1 L_{x'}^\infty} + T \| \partial^2_x \mu_1 \|_{L_t^1 L_{x'}^\infty} \leq \nu^2$.
\end{theorem}
A simplified variant of the estimate \eqref{eq:Strichartz2dFullSpace} was proved in \cite[Theorem~1.3]{SchippaSchnaubelt2022} with $\mu_1 \equiv 1$ and an inferior estimate for the charges. For the proof, we revisit the analysis of \cite{SchippaSchnaubelt2022} and improve and generalize it using arguments from \cite{Schippa2021Maxwell3d,Schippa2022ResolventEstimates}, but shall be brief to avoid repitition.
\begin{proof}
In the first step we reduce to the frequency localized estimate
\begin{equation*}
\lambda^{-s} \| S_\lambda u \|_{L^p L^q} \lesssim \| S_\lambda u \|_{L^\infty L^2} + \| \tilde{P} S_\lambda u \|_{L^2} + \lambda^{-1+\frac{1}{p}} \| S_\lambda \tilde{\rho}_e \|_{L^2_x}
\end{equation*}
for $\lambda \gtrsim 1$, where $u$ is essentially supported in the unit cube and its space-time Fourier transform is supported in $\{ |\xi_0| \lesssim |(\xi_1,\xi_2)| \}$. Moreover, we denote $\tilde{\rho}_e = \nabla \cdot u^{(1)}$. This is carried out like in \cite[Section~3.4]{SchippaSchnaubelt2022}. Now we frequency truncate the operator $\tilde{P}$ to coefficients, which have space-time Fourier transform $B(0,\lambda^{\frac{1}{2}})$. We let
\begin{equation*}
\tilde{P}^\lambda = 
\begin{pmatrix}
\partial_t & 0 & - \partial_2 (\mu_1^{\lambda^{\frac{1}{2}}} \cdot) \\
0 & \partial_t & \partial_1 (\mu_1^{\lambda^{\frac{1}{2}}} \cdot) \\
- \partial_2 (\varepsilon_{11}^{\lambda^{\frac{1}{2}}} \cdot) + \partial_1 (\varepsilon_{21}^{\lambda^{\frac{1}{2}}} \cdot) & \partial_1 (\varepsilon^{\lambda^{\frac{1}{2}}}_{22} \cdot) -  \partial_2(\varepsilon_{12}^{\lambda^{\frac{1}{2}}} \cdot) & \partial_t
\end{pmatrix}
.
\end{equation*}
It suffices to show
\begin{equation}
\label{eq:Maxwell2dFullTruncated}
\lambda^{-s} \| S_\lambda u \|_{L^p L^q} \lesssim \| S_\lambda u \|_{L^\infty L^2} + \| \tilde{P}^\lambda S_\lambda u \|_{L^2} + \lambda^{-1+\frac{1}{p}} \| S_\lambda \tilde{\rho}_e \|_{L_t^\infty L^2_{x'}}.
\end{equation}

$\tilde{P}^\lambda$ was diagonalized with pseudo-differential operators for $\mu_1 \equiv 1$ in \cite[Section~3.1]{SchippaSchnaubelt2022} as
\begin{equation}
\label{eq:DiagonalizationMaxwell2d}
\tilde{P}^\lambda = \mathcal{M}_\lambda \mathcal{D}_\lambda \mathcal{N}_\lambda + E_\lambda \text{ with } \| E_\lambda \|_{L^2_x \to L^2_x} \lesssim 1.
\end{equation}
In \cite{Schippa2022ResolventEstimates} the diagonalization was carried out in the constant-coefficient case for $\mu_1 = \mu^{-1} \neq 0$. This will determine the principal symbols of the operators in \eqref{eq:DiagonalizationMaxwell2d}. The principal symbol of $\tilde{P}^\lambda$ reads (omitting the frequency truncation and $x$-dependence to lighten notations):
\begin{equation*}
p(x,\xi) = 
i 
\begin{pmatrix}
\xi_0 & 0 & -\xi_2 \mu_1 \\
0 & \xi_0 & \xi_1 \mu_1 \\
\xi_1 \varepsilon_{12} - \xi_2 \varepsilon_{11} & \xi_1 \varepsilon_{22} - \xi_2 \varepsilon_{12} & \xi_0
\end{pmatrix}
.
\end{equation*}
We let
\begin{equation*}
\| \xi' \|^2_{\varepsilon'} = \langle \xi', \mu_1 \det(\varepsilon)^{-1} \varepsilon \xi' \rangle, \quad \varepsilon = ((\varepsilon_{ij})_{i,j})^{-1}, \quad \xi^* = \xi' / \| \xi' \|_{\varepsilon'}.
\end{equation*}
The following diagonalization holds for almost all $\xi' \in \R^2$ by \cite[Lemma~2.2]{Schippa2022ResolventEstimates}: $p(x,\xi) = m(x,\xi') d(x,\xi) m^{-1}(x,\xi')$ with
\begin{equation*}
\begin{split}
m(x,\xi') &= 
\begin{pmatrix}
\varepsilon_{22} \xi_1^* - \varepsilon_{12} \xi_2^* & - \xi_2^* \mu_1 & \xi_2^* \mu_1 \\
\varepsilon_{11} \xi_2^* - \varepsilon_{21} \xi_1^* & \xi_1^* \mu_1 & - \xi_1^* \mu_1 \\
0 & - 1 & -1
\end{pmatrix}, \\
m^{-1}(x,\xi') &=
\begin{pmatrix}
\mu_1 \xi_1^* & \mu_1 \xi_2^* & 0 \\
\frac{\xi_1^* \varepsilon_{21} - \xi_2^* \varepsilon_{11}}{2} & \frac{\varepsilon_{22} \xi_1^* - \varepsilon_{21} \xi_2^*}{2} & - \frac{1}{2} \\
\frac{\xi_2^* \varepsilon_{11} - \xi_1^* \varepsilon_{12}}{2} & \frac{\xi_2^* \varepsilon_{12} - \xi_1^* \varepsilon_{22}}{2} &  - \frac{1}{2}
\end{pmatrix},
\end{split}
\end{equation*}
and $d(x,\xi) = i \text{diag}(\xi_0, \xi_0 - \| \xi' \|_{\varepsilon'}, \xi_0 + \| \xi' \|_{\varepsilon'} )$. The error estimates for the diagonalization with the standard quantization can be proved like in \cite[Section~3.3]{SchippaSchnaubelt2022}; see also \cite[Section~3.2]{Schippa2021Maxwell3d} for a simplification of arguments.

In the following we sketch the conclusion of the proof with the diagonalization at hand. By $\mathcal{M}_\lambda \mathcal{N}_\lambda S_\lambda S'_\lambda = S_\lambda S'_\lambda + O_{L^2_x \to L^2_x}(\lambda^{-1})$ and Sobolev embedding, we obtain
\begin{equation}
\label{eq:Diagonalization2dFullI}
\lambda^{-\rho} \| S_\lambda u \|_{L^p_t L^q_{x'}} \lesssim \lambda^{-\rho} \| \mathcal{N}_\lambda S_\lambda u \|_{L_t^p L_{x'}^q} + \| S_\lambda u \|_{L^2_x}.
\end{equation}
Moreover, by Fourier support of $u$, we have $S_\lambda u = S_\lambda \tilde{S}'_\lambda u $ and straight-forward estimates for pseudo-differential operators (cf. Section \ref{section:Preliminaries})
\begin{equation}
\label{eq:Charges2dFull}
\begin{split}
\lambda^{-\rho} \| [\mathcal{N}_\lambda S_\lambda u]_1 \|_{L_t^p L_{x'}^q} &\lesssim \lambda^{-(\rho + 1)} \| \nabla \cdot (\tilde{S}'_\lambda u ) \|_{L^p_t L_{x'}^q} \\
&\lesssim \lambda^{-1 + \frac{1}{p}} \| \nabla \cdot \tilde{S}'_\lambda u \|_{L_t^p L_{x'}^q} \\
&\lesssim \lambda^{-1 + \frac{1}{p}} \| \nabla \cdot \tilde{S}'_\lambda u \|_{L_t^\infty L_{x'}^2}.
\end{split}
\end{equation}
This controls $[\mathcal{N}_\lambda S_\lambda u]_1$ in terms of the charges. $[\mathcal{N}_\lambda S_\lambda u]_i$ are estimated like in \cite{SchippaSchnaubelt2022} with the estimates for (rough) half-wave equations. This yields
\begin{equation*}
\lambda^{-s} \| S_\lambda u \|_{L_t^p L_{x'}^q} \lesssim \| S_\lambda u \|_{L_t^\infty L_{x'}^2} + \| \mathcal{D}_\lambda \mathcal{N}_\lambda S_\lambda u \|_{L^2_x} + \lambda^{-1 + \frac{1}{p}} \| S_\lambda \tilde{\rho}_e \|_{L_t^\infty L_{x'}^2}.
\end{equation*}
The proof of \eqref{eq:Maxwell2dFullTruncated} can be concluded now by another error estimate
\begin{equation*}
\mathcal{N}_\lambda \mathcal{M}_\lambda S_\lambda S'_\lambda = S_\lambda S'_\lambda + O_{L^2_x \to L^2_x}(\lambda^{-1}),
\end{equation*}
and invoking \eqref{eq:DiagonalizationMaxwell2d}.
\end{proof}

We obtain the following corollary by paradifferential truncation (cf. \cite[Corollary~1.7]{SchippaSchnaubelt2022}):
\begin{corollary}
\label{cor:Strichartz2dL2Lipschitz}
Let notations be like in Theorem \ref{thm:2dStrichartzFull}. Assume that $\| \partial_x \varepsilon \|_{L_T^2 L_{x'}^\infty} \lesssim 1$. Then the solution $u$ to
\begin{equation*}
\left\{ \begin{array}{cl}
\tilde{P}(x,D) u &= f, \qquad \quad \nabla \cdot u^{(1)} = \tilde{\rho}_e, \\
u(0) &= u_0
\end{array} \right.
\end{equation*}
satisfies
\begin{equation*}
\begin{split}
\| \langle D' \rangle^{-\alpha} u \|_{L^p(0,T;L^q(\R^2)} &\lesssim_{T,\alpha} \| u_0 \|_{L^2(\R^2)} + \| f \|_{L^1(0,T;L^2(\R^2))} \\
&\quad + \| \langle D' \rangle^{-1+\frac{2}{3p}} \tilde{\rho}_e(0) \|_{L^2_{x'}} + \| \langle D' \rangle^{-1 + \frac{2}{3p}} \partial_t \rho_e \|_{L^1(0,T:L^2)}.
\end{split}
\end{equation*}
\end{corollary}

\section{Helmholtz decompositions}
\label{appendix:Helmholtz}
In this Appendix we collect facts on Helmholtz decompositions. Let $\mathcal{E}: \R^d \supseteq \Omega \to \R^d$ denote a sufficiently smooth vector field with $d \in \{2,3\}$. The question is under which assumptions on the domain $\Omega \subseteq \R^d$, the vector field $\mathcal{E}$, and $s \geq 0$ we find the equivalence of norms to hold:
\begin{equation*}
\| \mathcal{E} \|_{H^{s+1}(\Omega)} \sim \| \nabla \times \mathcal{E} \|_{H^s(\Omega)} + \| \nabla \cdot \mathcal{E} \|_{H^s(\Omega)} + \| \mathcal{E} \|_{L^2(\Omega)}.
\end{equation*}
Suitable results for connected bounded domains with smooth boundary were proved in \cite[Chapter~IX,~Section~§1]{DautrayLions1990}. We shall see how these results extend to domains with compact boundary.

We have the following for $d=2$:
\begin{proposition}[{\cite[Proposition~6',~p.~237]{DautrayLions1990}}]
\label{prop:Helmholtz2dBounded}
Let $k \in \N_0$, and $\Omega$ be a connected bounded open set in $\R^2$ with smooth boundary. Then,
\begin{equation*}
H^{k+1}(\Omega)^2 = \{ \mathcal{E} \in L^2(\Omega), \; (\nabla \times \mathcal{E})_3 \in H^k(\Omega), \; \nabla \cdot \mathcal{E} \in H^k(\Omega), \; u \wedge \nu \big|_{\partial \Omega} \in H^{k+\frac{1}{2}}(\partial \Omega)^2 \}
\end{equation*}
and correspondingly,
\begin{equation}
\label{eq:HelmholtzRegularCase}
\| \mathcal{E} \|_{H^{k+1}(\Omega)} \sim \| (\nabla \times \mathcal{E})_3 \|_{H^k(\Omega)} + \| \nabla \cdot \mathcal{E} \|_{H^k(\Omega)} + \| \mathcal{E} \|_{L^2(\Omega)}.
\end{equation}
\end{proposition}
From this we deduce the following for domains with compact boundary:
\begin{proposition}
\label{prop:Helmholtz2d}
Let $\Omega \subseteq \R^2$ be an open set with compact smooth boundary and $s \geq 0$. Then, for $\mathcal{E} \in H^{s+1}(\Omega)^2$ with $[\mathcal{E} \times \nu ]_{x' \in \partial \Omega} = 0$ we have the equivalence of norms:
\begin{equation}
\label{eq:HelmholtzDecompositionConnected}
\| \mathcal{E} \|_{H^{s+1}(\Omega)} \sim \| (\nabla \times \mathcal{E})_3 \|_{H^s(\Omega)} + \| \nabla \cdot \mathcal{E} \|_{H^s(\Omega)} + \| \mathcal{E} \|_{L^2(\Omega)}.
\end{equation}
\end{proposition}
\begin{proof}
It suffices to show the claim for $s \in \N_0$ and connected $\Omega$ as the norms disentangle by disjointness of the supports. Indeed, for $\Omega = \bigcup_i \Omega_i$ denoting the decomposition into connected components, we have $H^s(\Omega) = \bigoplus_i H^s(\Omega_i)$. Moreover, the estimate
\begin{equation*}
\| (\nabla \times \mathcal{E})_3 \|_{H^s(\Omega)} + \| \nabla \cdot \mathcal{E} \|_{H^s(\Omega)} + \| \mathcal{E} \|_{L^2(\Omega)} \lesssim \| \mathcal{E} \|_{H^{s+1}(\Omega)}
\end{equation*}
is immediate. We turn to the reverse inequality. This will follow from a localization argument and Proposition \ref{prop:Helmholtz2dBounded}. Let $(U_i,\varphi_i)_{i=1,\ldots,n}$ be bounded charts, which cover a neighbourhood of the boundary and $(U_0,\varphi_0 = id)$ be the trivial chart for the interior. Let $\sum_{i=0}^n \psi_i = 1_{\Omega}$ be a partition of unity of $\Omega$ with $\text{supp}(\psi_i) \subseteq U_i$. We have for $i=1,\ldots,n$ by Proposition \ref{prop:Helmholtz2dBounded}:
\begin{equation*}
\| \mathcal{E} \psi_i \|_{H^{s+1}} \sim \| (\nabla \times (\psi_i \mathcal{E}))_3 \|_{H^s(\Omega)} + \| \nabla \cdot (\psi_i \mathcal{E}) \|_{H^s} + \| \psi_i \mathcal{E} \|_{L^2},
\end{equation*}
and for $i=0$:
\begin{equation*}
\begin{split}
\| \mathcal{E} \psi_0 \|_{H^{s+1}(\Omega)} &= \| \mathcal{E} \psi_0 \|_{H^{s+1}(\R^2)} \\
 &\sim \| (\nabla \times (\mathcal{E} \psi_0))_3 \|_{H^s(\R^2)} + \| \nabla \cdot (\mathcal{E} \psi_0) \|_{H^s(\R^2)} + \| \mathcal{E} \psi_0 \|_{L^2(\R^2)} \\
&=  \| (\nabla \times (\mathcal{E} \psi_0))_3 \|_{H^s(\Omega)} + \| \nabla \cdot (\mathcal{E} \psi_0) \|_{H^s(\Omega)} + \| \mathcal{E} \psi_0 \|_{L^2(\Omega)}.
 \end{split}
\end{equation*}
Here we used that \eqref{eq:HelmholtzDecompositionConnected} holds for $\Omega = \R^d$ as can readily be verified by changing to Fourier space and $\text{dist}(\text{supp}(\mathcal{E} \psi_0), \partial \Omega) > 0$, which allows us to change back and forth between $H^s(\Omega)$ and $H^s(\R^2)$ for the involved functions.

Furthermore, for $i=0,\ldots,n$ we have
\begin{equation*}
\| (\nabla \times (\mathcal{E} \psi_i))_3 \|_{H^s(\Omega)} \leq \| \psi_i (\nabla \times \mathcal{E})_3 \|_{H^s(\Omega)} + \| \partial \psi_i \mathcal{E} \|_{H^s(\Omega)}.
\end{equation*}
Moreover,
\begin{equation*}
\| \psi_i (\nabla \times \mathcal{E})_3 \|_{H^s(\Omega)} \leq C \| (\nabla \times \mathcal{E})_3 \|_{H^s(\Omega)}
\end{equation*}
follows from $\partial_{x'}^\alpha \psi_i \in C^\infty_c(\Omega)$ for $|\alpha| \geq 1$ and $\| \psi_i (\nabla \times \mathcal{E})_3 \|_{L^2(\Omega)} \leq \| (\nabla \times \mathcal{E})_3 \|_{L^2(\Omega)}$ because $|\psi_i(x)| \leq 1$ for any $x \in \Omega$. Since $\partial_{x'} \psi \in C^\infty_c(\Omega)$, we have
\begin{equation*}
\| ( \partial_{x'} \psi ) \mathcal{E} \|_{H^s(\Omega)} \lesssim \| \mathcal{E} \|_{H^s(\Omega)}.
\end{equation*}
We proved
\begin{equation*}
\begin{split}
\| ( \nabla \times (\mathcal{E} \psi_i))_3 \|_{H^s(\Omega)} &\lesssim \| (\nabla \times \mathcal{E})_3 \|_{H^s(\Omega)} + \| \mathcal{E} \|_{H^s(\Omega)} \\
&\leq C \| (\nabla \times \mathcal{E})_3 \|_{H^s(\Omega)} + \varepsilon \| \mathcal{E} \|_{H^{s+1}(\Omega)} + C_\varepsilon \| \mathcal{E} \|_{L^2(\Omega)}.
\end{split}
\end{equation*}
By the same arguments it holds
\begin{equation*}
\| \nabla \cdot (\mathcal{E} \psi_i) \|_{H^s(\Omega)} \leq C \| \nabla \cdot \mathcal{E} \|_{H^s(\Omega)} + \varepsilon \| \mathcal{E} \|_{H^{s+1}(\Omega)} + C_\varepsilon \| \mathcal{E} \|_{L^2(\Omega)}.
\end{equation*}
Hence, we can conclude
\begin{equation*}
\begin{split}
\| \mathcal{E} \|_{H^{s+1}(\Omega)} &\leq \sum_{i=0}^n \| \mathcal{E} \psi_i \|_{H^{s+1}(\Omega)} \\
&\leq C \sum_{i=0}^n \big(  \| (\nabla \times (\mathcal{E} \psi_i))_3 \|_{H^s(\Omega)} + \| \nabla \cdot (\mathcal{E} \psi_i) \|_{H^s(\Omega)} + \| \psi_i \mathcal{E} \|_{L^2(\Omega)} \big) \\
&\leq C (n+1) \big( \| (\nabla \times \mathcal{E})_3 \|_{H^s(\Omega)} + \| \nabla \cdot \mathcal{E} \|_{H^s(\Omega)} + \| \mathcal{E} \|_{L^2(\Omega)} \big) \\
&\quad + C (n+1) \varepsilon \| \mathcal{E} \|_{H^{s+1}(\Omega)} + C_\varepsilon \| \mathcal{E} \|_{L^2(\Omega)}.
\end{split}
\end{equation*}
Choosing $\varepsilon = \frac{1}{2(n+1)C}$ we have proved that
\begin{equation*}
\| \mathcal{E} \|_{H^{s+1}(\Omega)} \lesssim \| (\nabla \times \mathcal{E})_3 \|_{H^s(\Omega)} + \| \nabla \cdot \mathcal{E} \|_{H^s(\Omega)} + \| \mathcal{E} \|_{L^2(\Omega)}.
\end{equation*}
\end{proof}

For $d=3$, one can argue like in Proposition \ref{prop:Helmholtz2d} to extend the results due to Dautray--Lions \cite[Proposition~6',~p.~237]{DautrayLions1990} for connected bounded domains with smooth boundary likewise in the three-dimensional case to domains with compact and smooth boundary. We record the following, which suffices for the purposes of this paper:

\begin{proposition}
\label{prop:Helmholtz3d}
Let $\Omega \subseteq \R^3$ be a smooth domain with compact boundary. Let $\mathcal{E} \in H^3(\Omega;\R^3)$ be a vector field. Suppose that either the tangential components satisfy Dirichlet boundary conditions and the normal component satisfies Neumann boundary conditions or vice versa. Then the following estimate holds:
\begin{equation}
\label{eq:HelmholtzDecomposition3dAppendix}
\| \mathcal{E} \|_{H^1(\Omega)} \sim \| \mathcal{E} \|_{H_{curl}(\Omega)} + \| \mathcal{E} \|_{H_{div}(\Omega)} + \| \mathcal{E} \|_{L^2(\Omega)}.
\end{equation}
\end{proposition}

\section*{Acknowledgements}
R.S. acknowledges financial support by the German Research Foundation (DFG) -- Project-Id 258734477 -- SFB 1173.  The second author would like to thank the Institut de Mathématique d'Orsay for kind hospitality, where much of this research was carried out in spring 2022, and Roland Schnaubelt (KIT) for helpful discussions on Maxwell equations on domains and putting the results into context.

\bibliographystyle{plain}

\end{document}